\documentclass[10pt, reqno]{amsart}
\pdfoutput=1
\usepackage{graphicx,amsmath,amssymb,epsfig}
\usepackage{xcolor}
\usepackage{graphicx}
\usepackage{multirow}
\usepackage[square,sort,comma,numbers]{natbib}
\usepackage{ifpdf}

\usepackage{caption}
\usepackage{subcaption}
\usepackage{floatrow}
\usepackage[flushleft]{threeparttable}
\usepackage{hyperref}
\usepackage{array}
\long\def\comment#1{}
\providecommand{\tabularnewline}{\\}

\long\def\comment#1{}
\newtheorem{theorem}{Theorem}
\newtheorem{algorithm}{Algorithm}[section]

\newtheorem{corollary}{Corollary}
\newtheorem{lemma}{Lemma}

\newtheorem{assumption}{Assumption}

\newtheorem*{condition*}{\conditionnumber}
\providecommand{\conditionnumber}{}


\theoremstyle{definition}
\newtheorem{definition}{Definition}

\newtheorem{remark}{Comment}[section]
\numberwithin{remark}{section}
\newtheorem{example}{Example}

\newcommand{\be}{\begin{eqnarray}}
\newcommand{\ee}{\end{eqnarray}}

\newcommand{\ba}{\begin{array}}
\newcommand{\ea}{\end{array}}
\newcommand{\bs}{\begin{align}\begin{split}\nonumber}
\newcommand{\bsnumber}{\begin{align}\begin{split}}
\newcommand{\es}{\end{split}\end{align}}

\renewcommand{\[}{\left[}
\renewcommand{\]}{\right]}
\renewcommand{\hat}{\widehat}

\newcommand{\cc}{\mathbf{c}}

\newcommand{\F}{\mathcal{F}}

\newcommand{\Gn}{\mathbb{G}_n}

\newcommand{\Pn}{\mathbb{P}_n}

\newcommand{\Ep}{{\mathrm{E}}}

\newcommand{\En}{{\mathbb{E}_n}}

\newcommand{\UU}{\mathcal{U}}
\newcommand{\pp}{\tilde p}
\newcommand{\dn}{d_{\UU}}

\renewcommand{\Pr}{{\mathrm{P}}}

\newcommand{\barf}{\overline{f}}

\newcommand{\diam}{{\text{diam}}}

\def\RR{ {\mathbb{R}}}

\def\supp{{\rm support}}
\newcommand{\floor}[1]{\left\lfloor #1 \right\rfloor}
\newcommand{\ceil}[1]{\left\lceil #1 \right\rceil}
\newcommand{\semin}[1]{\phi_{{\rm min}}(#1)}
\newcommand{\semax}[1]{\phi_{{\rm max}}(#1)}
\newcommand{\semaxtilde}[1]{\tilde{\phi}_{{\rm max}}(#1)}
\renewcommand{\hat}{\widehat}
\renewcommand{\leq}{\leqslant}
\renewcommand{\geq}{\geqslant}

\newcommand{\diag}{{\rm diag}}
\newcommand{\underf}{{\underline{f}}}
\newcommand{\G}{{G}}

\newcommand{\mT}{{\mathcal{H}}}
\newcommand{\mF}{{\mathcal{F}}}
\newcommand{\mG}{{\mathcal{G}}}
\newcommand{\mP}{{\mathcal{P}}}
\newcommand{\mU}{{\mathcal{U}}}
\newcommand{\un}{{u_n}}

\begin{document}

 \title[Quantile Graphical Models]
{Quantile Graphical Models: Prediction and Conditional Independence with Applications to Systemic Risk}\thanks{\smaller We would like to thank Don Andrews, Debopam Bhattacharya, Peter Bickel, Marianne Bitler, Peter B{\"u}hlmann, Colin Cameron, Karim Chalak, Xu Cheng, Valentina Corradi, Alan Crawford, Francis Diebold, Peng Ding, Mirko Draca, Iv{\'a}n Fern{\'a}ndez-Val, Bulat Gafarov, Jean-Jacques Forneron, Kenji Fukumizu, Dalia Ghanem, Bryan Graham, Jiaying Gu, Ran Gu, Han Hong, Jungbin Hwang, Cheng Hsiao, Hidehiko Ichimura, Michael Jansson, Oscar Jorda, Chihwa Kao, Kengo Kato, Shakeeb Khan, Roger Koenker, Brad Larsen, Chenlei Leng, Arthur Lewbel, Michael Leung, Song Liu, Francesca Molinari, Whitney Newey, Dong Hwan Oh,  David Pacini, Andrew Patton, Aureo de Paula, Victor de la Pena, Elisabeth Perlman, Stephen Portnoy, Demian Pouzo, James Powell, Geert Ridder, Joe Romano, Stephen Ross, Shu Shen, Senay Sokullu, Sami Stouli, Aleksey Tetenov, Takuya Ura, Tiemen Woutersen, Zhijie Xiao, Wenkai Xu, Chaoran Yu, and Yanos Zylberberg for comments and discussions. We would also like to thank the seminar and workshop participants from Aarhus University, Boston College, 2015 Warwick Summer Workshop, 11th World Congress of the Econometric Society, 2017 UCL Workshop on the Theory of Big Data, 2017 California Econometrics Conference, 2018 International Symposium on Financial Engineering and
Risk Management, 2018 Shanghai Econometrics Workshop, 2018 York Econometrics Symposium, Bank of England Modelling with Big Data and Machine Learning, London School of Economics, University of Bristol, University of Connecticut, Humboldt University Berlin, the Institute of Statistical Mathematics, USC, UC Berkeley, UC Davis, Stanford University, Warwick Statistics Department and University of Tokyo.}

\author[Belloni]{Alexandre Belloni$^*$}\thanks{$^*$Duke University, e-mail:abn5@duke.edu.}
\author[Chen]{Mingli Chen$^\ddag$}\thanks{$^\ddag$University of Warwick, e-mail:m.chen.3@warwick.ac.uk}
\author[Chernozhukov]{Victor Chernozhukov$^\dag$}\thanks{$^\dag$Massachusetts Institute of Technology, e-mail:vchern@mit.edu}
\date{First version: November 2012,  this version October 28, 2019.}

\begin{abstract}\small
We propose two types of Quantile Graphical Models (QGMs) --- Conditional Independence Quantile Graphical Models (CIQGMs) and Prediction Quantile Graphical Models (PQGMs). CIQGMs characterize the conditional independence of distributions by evaluating the distributional dependence structure at each quantile index. As such, CIQGMs can be used for validation of the graph structure in the causal graphical models (\cite{pearl2009causality, robins1986new, heckman2015causal}). One main advantage of these models is that we can apply them to large collections of variables driven by non-Gaussian and non-separable shocks. PQGMs characterize the statistical dependencies through the graphs of the best linear predictors under asymmetric loss functions. PQGMs make weaker assumptions than CIQGMs as they allow for misspecification. Because of QGMs’ ability to handle large collections of variables and focus on specific parts of the distributions, we could apply them to quantify tail interdependence. The resulting tail risk network can be used for measuring systemic risk contributions that help make inroads in understanding international financial contagion and dependence structures of returns under downside market movements.

We develop estimation and inference methods for QGMs focusing on the high-dimensional case, where the number of variables in the graph is large compared to the number of observations. For CIQGMs, these methods and results include valid simultaneous choices of penalty functions, uniform rates of convergence, and confidence regions that are simultaneously valid. We also derive analogous results for PQGMs, which include new results for penalized quantile regressions in high-dimensional settings to handle misspecification, many controls, and a continuum of additional conditioning events.

\end{abstract}

\maketitle

\section{Introduction}
Co-movements, dependence and influence between variables are fundamental in economics and finance for decision and policy making as well as prediction. To this end, we propose Quantile Graphical Models (QGMs) as a modeling framework and consider their usefulness in three main applications. First, empirical auction models often rely on independent private values or affiliated private values, detecting collusion in these auctions is a form of testing conditional independence. Examples of studying entry, market power, or collusion can be found in \cite{bajari2003deciding, porter2005detecting, harrington2008detecting}. Second, QGMs can be used to identify and measure systemic tail risk. The recent bank and sovereign crisis in the US and Europe have also boosted the interest in the important role of network spill-over effects in contagion and shaping systemic risk (\cite{Acemoglu2010, Acemoglu2013, elliott2014financial, hansen2014challenges}). Many measures of systemic risk focus on spill-overs fit naturally in our QGM setting (\cite{Adrian2011, Andersen2013, hautsch2014forecasting, hautsch2015financial, hardle2016tenet}). We apply these insights to re-evaluate international financial contagion in volatilities (\cite{Claessens2001}).\footnote{Works on economic and financial networks include \cite{Billio2010, bonaldi2015empirical, kastl2017recent}. We refer to \cite{de2017econometrics} for an excellent review on the econometrics literature on networks.} Third, QGMs can measure dependence between stock returns for hedging strategies. In financial management settings, risk quantification is crucial, and advanced hedging decisions are typically focused on the tail of the distribution of stock returns rather than the mean. Moreover, such strategies aiming to reduce risk are critical precisely during the market downside movement. Empirical evidence (\cite{Ang2006,Ang2002,Patton2004}) points to the non-Gaussianity of the distribution of stock returns, especially during market downturns. Therefore, it is also instructive to understand how dependence (and policy impact) would change as the downside movement of the market becomes more extreme. The proposed QGMs are flexible enough to cover all these cases.

QGMs can be viewed as part of graphical models which have been successively applied to estimate and visualize relationship  (\cite{maathuis2018handbook,heckman2015causal}). Graphical models are widely used in machine learning, statistical learning, and social science to model the statistical dependence among the components a $d$-dimensional random vector $X_{V}$, in the form of a graph or network $G=(V, E)$. Here $V$ is the node set contains the labels of the components and $E$ is the edge set represents \textit{unknown} statistical relationships that need to be estimated, thus poses novel problems of statistical inference. In the case of Gaussian Graphical Models (GGMs), which assuming $X_{V}$ are jointly Gaussian distributed, the conditional independence structure is completely characterized by the support of the inverse of the covariance matrix of $X_{V}$. Notably in this case, the same graph will characterize conditional independence and the best linear prediction. However, in \textit{non-Gaussian} settings, not only it is harder to characterize conditional independence, there are no reasons for the same graph to characterize both conditional independence and best linear prediction.

QGMs provide an alternative route to learn conditional independence and prediction under asymmetric loss functions which is appealing in non-Gaussian settings. As in non-Gaussian cases these notions do not coincide and there are needs for different estimation approaches. We propose two different QGMs to handle different types of applications. First, we propose Conditional Independence Quantile Graphical Models (CIQGMs) to characterize the conditional independence of distributions through evaluating the distributional dependence structure at each quantile index. Second, we propose Prediction Quantile Graphical Models (PQGMs) in which predictive relationship is the main focus. Note, QGMs also enable us to focus on \textit{specific parts of the distributions} of variables, which play an important role in applications like financial contagion and measuring systemic risk contributions where extreme events are the main interests for practitioners.

CIQGMs can be used for validation of the graph structure in the causal graphical models (\cite{pearl2009causality, robins1986new, heckman2015causal}).  Conditional independence has a long history in statistical models with consequences towards parameter identification, causal inference, prediction sufficiency, and many others, see \cite{dawid1979conditional}. CIQGMs aim to characterize conditional independence via the conditional quantile functions. In such models, we consider a flexible specification that can approximate well the conditional quantile functions (up to a vanishing approximation error). In turn, this allows detecting which variables have a strong or near zero impact on others which can then be used to provide guidance on conditional independence.
 
PQGMs focus on the prediction of a variable based on linear combinations of other variables (a reduced form relation) under asymmetric losses. An important motivation for proposing PQGMs is to allow for misspecification as the conditional quantile function is typically non-linear in non-Gaussian settings. The linear specification is widely used in practice despite possible misspecification which motivates an analysis for accomodating these issues. We characterize the uniform prediction properties under a family of asymmetric loss functions, which this family enables practitioners to investigate different parts of the tail distribution. Other papers investigated the impact of misspecification on quantile functions are \cite{abadie2014inference, angrist2006quantile, knight2008asymptotics, lee2009efficiency}. Our analysis also contributes to the high-dimensional quantile regression by allowing non-vanishing misspecification.

Broadly speaking, QGMs enhance our understanding of statistical dependence among $X_V$. For example, for each quantile index $\tau$, they provide visualization of the dependence via graphs whose edges represent conditional (quantile) relationships. Given that for each specific quantile index $\tau$ we will obtain one such graph, we could have a \textit{graph process} indexed by $\tau \in (0,1)$. The structure represented by a $\tau$-quantile graph represents a local relation and can be valuable in cases where tail interdependence might be of special interest.\footnote{This is similar to the contrast between quantile regression and linear regression, where the latter provides information only on the conditional mean, while the former can provide a more complete description of the distribution of the outcome.} The graph process induced by QGMs has several important features. First, a $\tau$-quantile graph enables different values of edge strength in different directions. This is important because for undirected networks, the distinction is unclear. Second, QGMs can capture the tail interdependence through estimating at a high or low quantile index. \footnote{Examples of high or low quantile index can be $\tau = 0.95$ or $\tau = 0.05$ respectively. The analysis extends to the case of $s^3\log^5p = o(n \tau (1-\tau))$. The case of $n \tau = C$, even in the fixed dimension case, leads to a substantially different analysis and limiting distributions, as shown in \cite{chernozhukov2005extremal}. Here $s$ is the sparsity parameter, $p$ is the dimension of conditional variables, $n$ is the sample size, $C$ is a constant.} Third, QGMs can capture the asymmetric dependence structure at different
quantiles, which can be particularly useful in empirical applications (e.g., stock market dependence, exchange rate dependence). By considering
all the quantiles at once we can characterize conditional independence structure for a set of variables that are not jointly Gaussian.
 
We also provide and study the estimation procedures that allow us to learn QGMs from the observed data. Our techniques are geared for covering high-dimensional settings where the size of the model is potentially larger than the sample size. These techniques are based on $\ell_{1}$-penalized quantile regression and Neyman orthogonal equations. For CIQGMs, under
mild regularity conditions, we provide rates of convergence and edge properties of the estimated graph that hold \textit{uniformly} over a large class of data generating processes. We provide \textit{simultaneously valid} confidence regions (post-selection) for the coefficients of the CIQGM that are uniformly valid, despite of possible model selection mistakes. Based on proper thresholding, recovery of CIQGMs patterns is possible when coefficients are well separated from zero which parallel the results for graph recovery in the Gaussian case.\footnote{Similar to graph recovery in the Gaussian
case such exact recovery is subject to the lack of uniformity validity critiques of
Leeb and P\"otscher \cite{leeb2008can}.} For PQGMs, we provide an estimator that achieves an adaptive rate of convergence, which might differ under different conditioning events. Therefore we contribute to the recent active literature on simultaneous valid confidence regions post-model selection,  \cite{BCCH2012sparse,BCH2014inference,BCFH2013program,Farrell:JMP,CanerKock:HDCI,CHS:PnP} \cite{GRD2014,zhang2014confidence,belloni2015uniform,BCK2013robustQR,jankova2015confidence,ning2017general,wang2016inference};
in particular, the penalty choices and theoretical results are uniformly valid and adaptive to the relevant conditioning events.

Although we build upon the quantile regression literature (\citep{Koenker2005,BCK2013robustQR}), we derive new results for penalized quantile regression in high dimensional settings that are uniformly valid, robust to small coefficients (e.g. allowing for model selection mistakes), allow possibly non-vanishing misspecification, many controls and a continuum of additional conditioning events. These results contribute to a growing literature that relies on quantile based models to characterize the data generating process. \cite{zheng2015globally} considers a globally adaptive quantile regression model, establishes oracle properties, and improved rates of convergence for the high-dimensional case. Screening procedures based on moment conditions motivated by the quantile models have been proposed and analyzed in \cite{he2013quantile, wu2015conditional} in the high-dimensional case. \cite{joe1997multivariate} considers tail dependence defined via conditional probabilities in a low dimensional setting.   

Finally, we view QGMs as complementary to a large body of works on GGMs (\cite{Dempster1972,Lauritzen1996,Cox1996,Edwards2000,Drton2004,Drton2007,Drton2008, meinshausen2006high, Yuan2007,Banerjee2008,Friedman2008, Yuan2010,Cai2011,Liu2012a,Sun2012,Liu2012b,chiong2017estimation}). Our work is also complementary to other works trying to relax the joint Gaussian assumption. \cite{Liu2009,Liu2012,Xue2012} work with so-called nonparanormal models or semiparametric Gaussian copula models, i.e., the variables follow a joint Gaussian distribution after monotone transformations. \cite{Ravikumar2011, peng2009partial} work with Sub-Gaussian data, which restricted fatness of the tails. \cite{ravikumar2010high, Xue2012a, loh2012structure} work with discrete-valued random variable, and few types of exponential families. \cite{yang2015graphical} provides results for M-estimators for a subclass of exponential family graphical models. QGMs allow for different sets of distributions.

The rest of the paper is organized as follows. Section \ref{sec:example} provides main motivating examples. Section \ref{sec:QGMs} lays out the foundation of the conceptual framework of QGMs. Section \ref{sec:Estimator-and-Consistency} contains estimators for QGMs while Section \ref{Sec:MainTheory} contains the theoretical guarantees of the estimators. Section \ref{sec:Empirical} provides an empirical application of QGMs to measure systemic risk contribution. Finally, the appendix contains proofs, simulations, and implementation details of the estimators.

{\bf Notation}. For an integer $k$, we let $[k]:=\{1,\ldots,k\}$ denote the set of integers from $1$ to $k$. For a random variable $X$ we denote by $\mathcal{X}$ its support. We use the notation $a\vee b=\max\{a,b\}$ and $a\wedge b=\min\{a,b\}$. We use $\|v\|_{p}$ to denote the $p$-norm of a vector $v$. In particular, the $\ell_2$-norm is denoted by $\Vert \cdot \Vert$; the $\ell_{0}$-``norm'' $\Vert\cdot\Vert_{0}$ denotes the number of non-zero components. Given a vector $\delta\in\mathbb{R}^{p}$, and a set of indices $T\subset\{1,...,d\}$, we denote by $\delta_{T}$ the vector in which $\delta_{Tj}=\delta_{j}$ if $j\in T$, $\delta_{Tj}=0$ if $j\notin T$. We use $\mathbb{E}_{n}$ to abbreviate the notation $n^{-1}\sum_{i=1}^n$; for example, $\mathbb{E}_{n}[f]:=\mathbb{E}_{n}[f(\omega_i)]:=n^{-1}\sum_{i=1}^nf(\omega_i)$.

\section{Motivating Examples} \label{sec:example}

\subsection{Screening for Collusion} One application of CIQGM is in the empirical auction literature on the examination of entry or bidding prices patterns to detect coordinating groups. As shown in \cite{bajari2003deciding}, firms’ bids, after controlling for all information about costs, are \textit{jointly independent} under the competitive model and lack of independence is taken as evidence consistent with collusion. Although, collusion is only one alternative explanation, testing conditional independence can be viewed as a screening device to determine whether further investigation is warranted.

Mathematically, denote $Y_{a,t}$ as the amount bid by firm $a$ on project $t$, and $Z_{a,t}$ as covariates observed in dataset. We define $X_{a,t}$ to be residuals after projecting out $Z_{a,t}$. \cite{bajari2003deciding} test whether $X_{a}$ is independent of $X_{b}$, using Fisher's Z-transformation of the coefficient of correlation between $X_{a}$ and $X_{b}$. For two Gaussian random variables, this is equivalent to test pairwise independence. In our terminology, this corresponds to an edge $(a,b)$ not being contained in the graph if and only if
\begin{equation}
X_{a}\perp X_{b}.
\end{equation}
Pairwise independence, however, needs not imply \textit{joint independence}. CIQGM can also handle joint independence in the \textit{non-Gaussian} setting. This is because CIQGM works with the following case
\begin{equation}
X_{a}\perp X_{b} \ \ \vert \ \ X_{V\backslash\{a,b\}},
\end{equation} namely an edge $(a,b)$ is not
contained in the graph if and only if $X_{b}$ and $X_{a}$ are independent conditional on all remaining
variables $X_{V\backslash\{a,b\}}=\{X_{k};k\in V\backslash\{a,b\}\}$, via the equivalence between conditional probabilities and conditional quantiles (details can be found in Section \ref{sec:Quantile-Graphical-Independence}).

\subsection{Systemic Tail Risk} \label{sub: Network-CoVaR}
Measuring systemic risk taking into account tail risk network spillover effects is another application of QGMs, e.g. our framework complements to the systemic risk measure CoVaR \cite{Adrian2011} which ignore tail risk dependencies induced by the underlying financial network structure. Our framework also allows for large scale networks.

Traditional tail risk measures, such as Value of Risk (VaR), focus on the loss of an individual institution. CoVaR attempts to measure the VaR of the whole financial system or a particular financial institution
by conditioning on another being in distress. Formally, \cite{Adrian2011} define institution $b$'s CoVaR at level $\tau$ conditional
on a particular outcome of institution $a$, as the value of $CoVaR_{\tau}^{b\vert a}$
that solves
\begin{equation}
\Pr(X_{b}\leq CoVaR_{\tau}^{b\vert a}\vert\mathbb{C}(X_{a}))  =  \tau,
\end{equation}
for some event $ \mathbb{C}(X_{a}) $ based on $X_{a}$. A special case is $\mathbb{C}(X_{a})=\{ X_a = VaR_{\tau}^{a}\}$ which, as interpreted by \cite{Adrian2011}, means with probability $\tau$ institution $b$ is in trouble given that institution $a$ is in trouble.
 
QGMs can work with the case
\begin{equation}
\mathrm{P}(X_{b}\leq CoVaR_{\tau}^{b\vert a,V\backslash\{a,b\}}\vert\mathbb{C}(X_{a}, X_{V\backslash\{a,b\}}))=\tau,
\end{equation}
with the main difference here is the conditioning events (or variables), i.e. from  $\mathbb{C}(X_{a})$ to $\mathbb{C}(X_{a}, X_{V\backslash\{a,b\}})$ with the latter could be high dimensional. Hence, our QGMs take into account risk spillovers from other institutions driving the CoVaR. The identified risk spillovers between all financial institutions constitute a financial tail risk network which, as shown later, can be used for measuring institutions' systemic tail risk contributions. In summary, QGMs can take into account the system-wide network spillover effects via incorporating tail network spillover effects into risk measuring, thus relate systemic risk to tail spillover effects from individual institutions to the whole system.

Another related definitions of tail risk, $\Delta CoVaR$, is defined as the change in the VaR of the whole financial system conditional on a institution being under distress relative to its median state. In terms of estimation, replacing covariate $X_{a}$ by the difference between its $\tau$-th quantile (denoted as $VaR_{\tau}^{a}$), and its median (denoted as $VaR_{50\%}^{a}$), yields
\begin{equation}
\Delta CoVaR_{\tau}^{b\vert a}=\hat{\beta}_a^{b}(\tau)(VaR_{\tau}^{a}-VaR_{50\%}^{a}).
\end{equation}
where $\hat{\beta}_a^{b}(\tau)$ comes from pairwise quantile regression of $X_b$ on $X_a$, and
\begin{equation} \label{Eq:deltaInf}
\Delta CoVaR_{\tau}^{b\vert a,V\backslash\{a,b\}}={\check{\beta}}_{a}^{b}(\tau)(VaR_{\tau}^{a}-VaR_{50\%}^{a}),
\end{equation}
where ${\check{\beta}}^{b}(\tau)$
is estimated via Algorithm \ref{Alg:1} or \ref{Alg:2}, inference procedures are based on Corollary \ref{theorem: general bs}.

After learning a tail risk network, we can use our new network-cooperated $\Delta CoVaR$ to measure the systemic risk contribution of each institution. The systemic risk contribution of institution $a$ can be measured by it "to" or "from" degree, see \cite{Andersen2013}; or by other network centrality measures, see \cite{jackson2010social, Kolaczyk2009, newman2010networks}.
To-degrees measure contributions of individual institutions to the overall risk
of systemic network events, for institution $a$ is defined as $\delta_{a}^{to} = \sum_{k}\Delta CoVaR_{\tau}^{k\vert a,V\backslash\{a,k\}}$. From-degrees measure exposure of individual institutions to systemic shocks from
the network, for institution $a$ it is defined as $\delta_{a}^{from} = \sum_{k}\Delta CoVaR_{\tau}^{a\vert k,V\backslash\{a,k\}}$. The net contribution of institution $a$ is defined as \textit{net}-$\Delta CoVaR^{a}=\mbox{\ensuremath{\delta}}_{a}^{to}-\delta_{a}^{from}$.
 
In Section \ref{sec:Empirical}, we revisit the analysis of international financial contagion through the volatility spillovers perspective. We visualize the tail risk interdependence via PQGMs as they allow for heteroskedasticity and asymmetric responses, can be used to model nonlinear tail interdependence, and to visualize potential asymmetric changes in conditional correlations. The estimated contagion network taking into account global interconnectedness is important for Eurozone financial regulators who want to identify globally systemically important EU countries, or for global financial portofolio diversification. Our the systemic tail risk analysis tools mentioned in previous paragraphs, can help with achieve those goals.

\subsection{Stock Returns Under Market Downside Movements}\label{Ex:Downside} Hedging decisions rely on the dependence of various stocks returns. Moreover, hedging is even more relevant during market downside movements, which motivates us to understand interdependence conditional on those events. Stock returns are in general non-Gaussian in those settings, as shown in the empirical finance literature, e.g. \cite{Ang2002, Longin2001,Patton2006}. We can parameterize the downside movements by using a random variable $W$, which could be the market index, and conditional on the event $\Omega_\varpi = \{W \leq \varpi\}$. This allows us to define a $\varpi$-conditional-CIQGM as $G^I(\tau,\varpi)=(V,E^I(\tau,\varpi))$ and a $\varpi$-conditional-PQGM as $G^P(\tau,\varpi)=(V,E^P(\tau,\varpi))$, for each $\varpi \in \mathcal{W}$. We might be interest in a fixed $\varpi$ or on a family of values $\varpi\in(-\bar \varpi, 0]$. The latter induces $\mathcal{W} = \{\Omega_\varpi=\{W\leq \varpi\} : \varpi\in(-\bar \varpi, 0] \}$.

Figure \ref{Figure:StockNew} provides an example using $ \mathcal{W}$-Conditional QGM with
\sloppy$W \text{= \{Market Index Returns\}}$, $\varpi$ as the $\tau_m$-th quantile of the
market index returns, and $\tau_m = \{0.15, 0.5, 0.75, 0.9\}$. We obtain daily stock returns from CRSP and use S\&P 500 as the market index. The full sample consists of 2769 observations for 86 stocks from Jan 2, 2003 to December 31, 2013. The total number of stocks is 86 due to data availability. We define market movement as when the market index returns
are below a pre-specified level (e.g. $\tau_m$-th quantile), hence conditioning on a particular $\varpi$ corresponds to consider the subsample based on whether the corresponding date's market return is less equal to the $\tau_m$-th quantile of the
market index returns. The results show higher interdependence under market downside moments, pose different hedging decisions.

\begin{figure}
\begin{minipage}{\textwidth}
\centering
\begin{subfigure}[b]{0.485\textwidth}
\centering
    {{\small     $\tau_m = 0.15$}}
        	\vskip -2.2cm
    \includegraphics[height=1.5\textwidth]{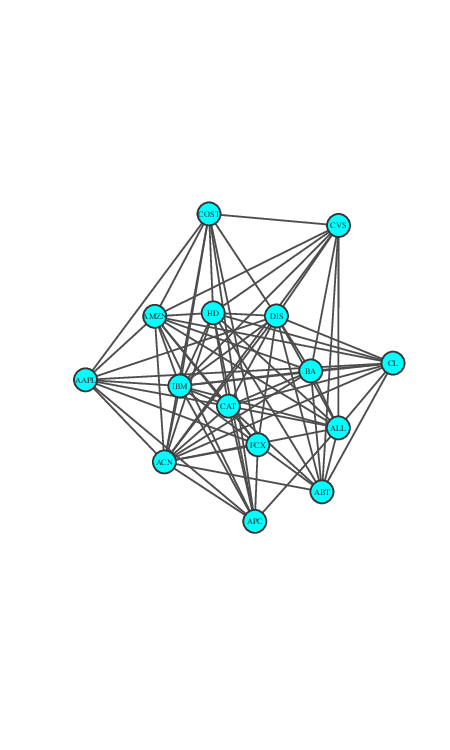}
\end{subfigure}
\hfill
 \begin{subfigure}[b]{0.485\textwidth}
 \centering
  	{{\small  $\tau_m = 0.5$}}
          	\vskip -2.2cm
	\includegraphics[height=1.5\textwidth]{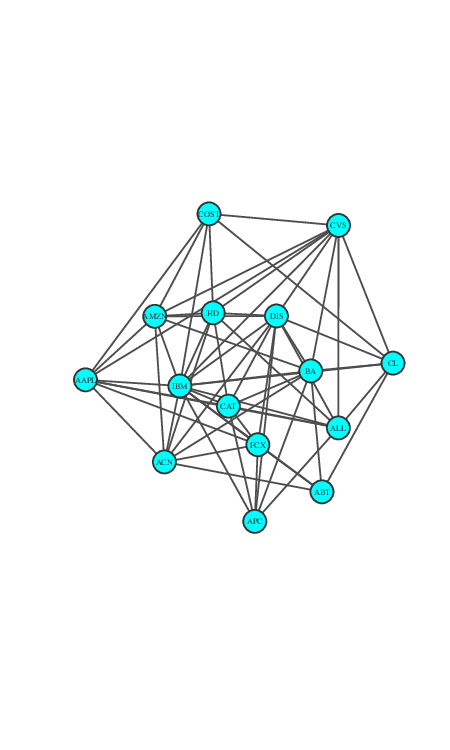}
\end{subfigure}
 \vskip -2.2cm
 \begin{subfigure}[b]{0.485\textwidth}
 \centering
    	{{\small  $\tau_m = 0.75$}}
        	\vskip -2.2cm
       \includegraphics[height=1.5\textwidth]{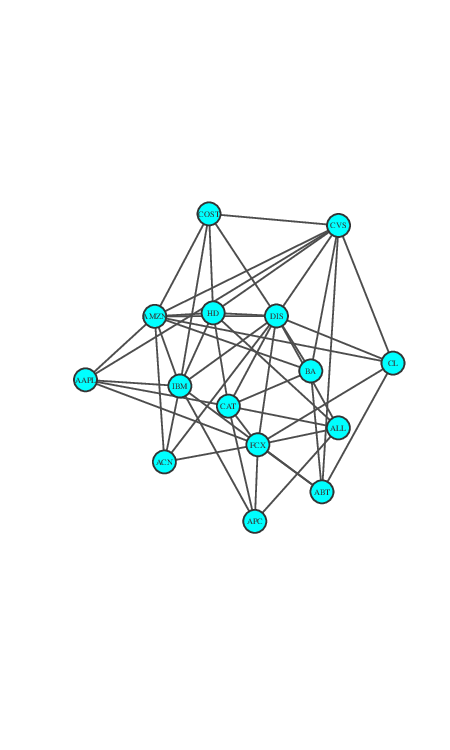}
\end{subfigure}
\quad
\begin{subfigure}[b]{0.485\textwidth}
\centering
	{{\small  $\tau_m = 0.9$}}
	            	\vskip -2.2cm
	\includegraphics[height=1.5\textwidth]{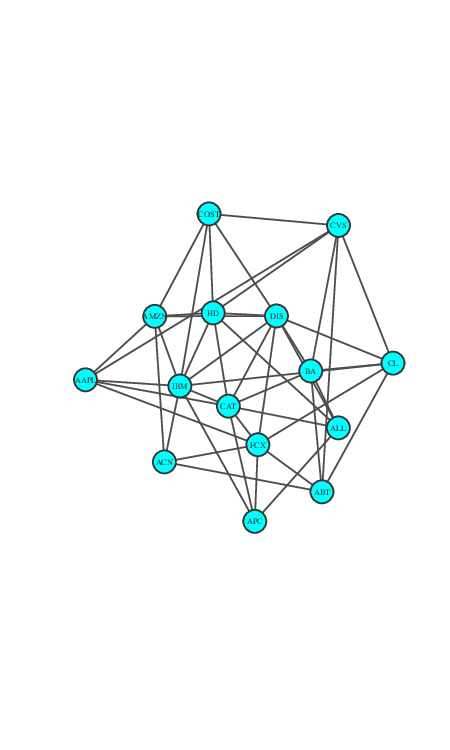}
\end{subfigure}
\vskip -2.2cm
     \caption[Stock Returns Interdependence under Different Market Conditions]
    {\small Stock Returns Interdependence under Different Market Conditions. Note: Presence of edges between nodes indicate that these two nodes (or stocks) are conditionally dependent.}
       \label{Figure:StockNew}
\end{minipage}
\end{figure}

\section{Quantile Graphical Models}\label{sec:QGMs}

In this section we describe quantile graphical models associated with a $d$-dimensional random vector $X_V$ where the set $V=[d]=\{1,\ldots,d\}$ denotes the labels of the components. These models aim to provide a description of the dependence between the random variables in $X_V$. In particular, these models induce graphs that allow for visualizing dependence structures. Nonetheless, because of the non-Gaussianity, we consider two fundamentally distinct models (one geared towards conditional independence and one geared towards prediction).

\subsection{Conditional Independence Quantile Graphical Models}\label{sec:Quantile-Graphical-Independence}

Conditional independence graphs have been used to provide visualization and insight on the dependence structure between random variables. Each node of the graph is associated with a component of $X_V$. We denote the conditional independence graph as $G^I=(V,E^I)$ where $G^I$ is an undirected graph with vertex set $V$ and edge set $E^I$ which is represented by an adjacency matrix ($E^I_{a,b}=1$ if the edge $(a,b)\in G^I$, and $E^I_{a,b}=0$ otherwise). An edge $(a,b)$ is not
contained in the graph if and only if
\begin{equation}\label{Def:CondInd}
X_{a}\perp X_{b} \ \ \vert \ \ X_{V\backslash\{a,b\}},
\end{equation}namely  $X_{b}$ and $X_{a}$ are independent conditional on all remaining
variables $X_{V\backslash\{a,b\}}=\{X_{k};k\in V\backslash\{a,b\}\}$.

\begin{remark}[Conditional Independence Under Gaussianity] In the case that $X_V$ is jointly Gaussian distributed,
 $X_V\sim N(0,{\Sigma})$ with $\Sigma$ as
the covariance matrix of $X_V$, the conditional independence
structure between two components is determined by the inverse of the covariance
matrix, i.e. the precision matrix $\Theta=\Sigma^{-1}$. It follows that the non-zero elements in the precision matrix corresponds
to the non-zero coefficients of the associated (high dimensional) mean
regression. The family of Gaussian distributions with this property
is known as a Gauss-Markov random field with respect to the graph
$G$. This observation has motivated a large literature \cite{Lauritzen1996} and interesting extensions that allow for transformations of Gaussian variables \cite{Liu2009,Liu2012}.  \end{remark}

In order to achieve a tractable concept for non-Gaussian settings, we use that
(\ref{Def:CondInd})
occurs if and only if \begin{equation}\label{Def:Find}F_{X_a}(\cdot\vert X_{V\backslash\{a\}})=F_{X_a}(\cdot\vert X_{V\backslash\{a,b\}}) \ \ \mbox{for all} \ \ X_{V\backslash\{a\}} \in \mathcal{X}_{V\backslash\{a\}}.\end{equation} In turn, by the equivalence between conditional probabilities and conditional quantiles to characterize a random variable, we have that (\ref{Def:CondInd})
occurs if and only if
\begin{equation}\label{Def:Qind}
\begin{array}{c}Q_{X_{a}}(\tau\vert X_{V\backslash\{a\}})  =  Q_{X_{a}}(\tau\vert X_{V\backslash\{a,b\}}) \ \ \ \mbox{for\ all} \  \tau\in(0,1), \ \ \mbox{and} \ \ X_{V\backslash\{a\}} \in \mathcal{X}_{V\backslash\{a\}}.
\end{array}
\end{equation}

For a quantile index $\tau \in (0,1)$, the $\tau$-quantile conditional independence graph is
a directed graph $G^I(\tau)=(V,E^I(\tau))$ with vertex set $V$ and edge set
$E^I(\tau)$. An edge $(a,b)$ is not contained in the edge set $E^I(\tau)$ if and only if
\begin{equation}\label{Def:Qindtau}
\begin{array}{c}Q_{X_{a}}(\tau\vert X_{V\backslash\{a\}})  = Q_{X_{a}}(\tau\vert X_{V\backslash\{a,b\}}) \ \ \mbox{for all} \ \ X_{V\backslash\{a\}} \in \mathcal{X}_{V\backslash\{a\}}.
\end{array}
\end{equation}

By the equivalence between (\ref{Def:Find}) and (\ref{Def:Qind}), the union of $\tau$-quantile graphs over $\tau\in(0,1)$ represents
the conditional independence structure of $X$, namely $E^I=\cup_{\tau \in (0,1)} E^I(\tau)$. We also consider a relaxation of (\ref{Def:CondInd}). For a set of quantile indices $\mathcal{T}\subset (0,1)$, we say that \begin{equation}\label{Def:TauInd}X_{a}\perp_{\mathcal{T}} X_{b} \ \ \vert \ \ X_{V\backslash\{a,b\}},\end{equation}
$X_{a}$ and $X_{b}$ are $\mathcal{T}$-conditionally independent given $X_{V\backslash\{a,b\}}$, if (\ref{Def:Qindtau}) holds for all $\tau \in \mathcal{T}$. Thus, we have that (\ref{Def:CondInd}) implies (\ref{Def:TauInd}).
We define the $\mathcal{T}$-quantile graph as $G^I(\mathcal{T})=(V,E^I(\mathcal{T}))$ where
$$ E^I(\mathcal{T})= \cup_{\tau \in \mathcal{T}}E^I(\tau).$$
Although the conditional independence concept relates to all quantile indices, the quantile characterization described above also lends itself to quantile specific impacts which can be of independent interest.\footnote{For example, we might be interested in some extreme events which typically correspond to crises in financial systems.}

\subsection{Prediction Quantile Graphical Models}\label{sec:Quantile-Graphical-Prediction}

Prediction Quantile Graphical Models (PQGMs) are motivated by prediction accuracy under an asymmetric loss function (instead of conditional independence as in Section \ref{sec:Quantile-Graphical-Independence}). More precisely, for each $a\in V$, we are interested in predicting $X_a$ based on linear combinations of the remaining variables, $X_{V\backslash\{a\}}$, where accuracy is measured with respect to an asymmetric loss function. Formally, PQGMs measure accuracy  as
 \begin{equation}\label{Def:Asym} \mathcal{L}_a (\tau \mid V\backslash\{a\}) = \min_\beta \Ep[ \rho_\tau( X_a-X_{-a}'\beta)]\end{equation}
where $X_{-a}=(1,X_{V\backslash\{a\}}')'$, and the asymmetric loss function $\rho_{\tau}(t)=(\ensuremath{\tau-1\{t\leq0\}})t$ is the check function used in \cite{KoenkerBassett1978}.

Importantly, PQGMs are concerned with the best linear predictor under the asymmetric loss function $\rho_\tau$ which is a specification that is widely used in practice. This is a fundamental distinction with respect to CIQGMs discussed in Section \ref{sec:Quantile-Graphical-Independence} where the specification of the conditional quantile was approximately a linear function of transformations of $X_{V\backslash \{a\}}$.\footnote{In Section \ref{sec:Quantile-Graphical-Independence} the vector $Z^a$ in equation (\ref{Def:CondQ}) collects the functions of the vector $X_{V\backslash \{a\}}$.} Indeed, we note that under suitable conditions the linear predictor that solves the minimization problem in (\ref{Def:Asym}) approximates the conditional quantile regression as shown in \cite{BCF2011}. (In fact, the conditional quantile function would be linear if $X_V$ was jointly Gaussian distributed.) However, PQGMs do not assume that the conditional quantile function of $X_a$ is well approximated by a linear function and instead it focuses on the best linear predictor.

We define that $X_b$ is predictively uninformative for $X_a$ given $X_{V\backslash\{a,b\}}$ if
$$ \mathcal{L}_a(\tau \mid V\backslash\{a\}) =  \mathcal{L}_a(\tau \mid V\backslash\{a,b\}) \ \ \ \mbox{for all} \ \ \tau \in (0,1),$$
i.e., considering a linear function of $X_b$ will not improve our performance of predicting $X_a$  with respect to the asymmetric loss function $\rho_\tau$.

Again we can visualize the predictive relationship through a graph process indexed by $\tau \in (0,1)$. That is, for each $\tau \in (0,1)$ we have a directed graph $G^P(\tau)=(V,E^P(\tau))$, where an edge $(a,b)\in G^P(\tau)$ only if $X_b$ is predictively informative for $X_a$ given $X_{V\backslash\{a,b\}}$ at the quantile $\tau$. Finally, it is also convenient to define the PQGM associated with a subset $\mathcal{T}\subset(0,1)$ as $G^P(\mathcal{T})=(V,E^P(\mathcal{T}))$ where $$E^P(\mathcal{T})=\cup_{\tau \in \mathcal{T}}E^P(\tau).$$

\subsection{$\mathcal{W}$-Conditional Quantile Graphical Models}\label{Sec:CondQGM}

In what follows, we discuss an extension of the QGMs discussed in Sections \ref{sec:Quantile-Graphical-Independence} and \ref{sec:Quantile-Graphical-Prediction} to allows for conditioning on a (possible infinity) family of  events $\varpi \in \mathcal{W}$.\footnote{With a slight abuse of notation, we let $\varpi$ to denote the event and also the index of such event. For example, we write $\Pr(\varpi)$ as a shorthand for $\Pr( W  \in \Omega_\varpi)$.} Such extension is motivated by several applications in which the interdependence between the random variables in $X_V$ maybe substantially impacted by additional observable events (e.g. downside movements of the market). This general framework allows different forms of conditioning. The main implication of this extension is that QGMs are now graph processes indexed by $\tau \in \mathcal{T} \subset (0,1)$ and $\varpi \in \mathcal{W}$.

We define $X_a$ and $X_b$ are
$(\mathcal{T},\varpi)$-conditionally independent,
 \begin{equation}\label{Def:CondTauInd}X_{a}\perp_{\mathcal{T}} X_{b} \ \ \vert \ \ X_{V\backslash\{a,b\}}, \varpi\end{equation}
if for all $\tau \in \mathcal{T}$ we have
\begin{equation}\label{Def:CondQind}
\begin{array}{c}Q_{X_{a}}(\tau\vert X_{V\backslash\{a\}},\varpi)  =  Q_{X_{a}}(\tau\vert X_{V\backslash\{a,b\}},\varpi).\end{array}
\end{equation}
The conditional independence edge set associated with $(\tau,\varpi)$ is defined analogously as before. We denote them by $E^I(\tau,\varpi)$ and $E^I(\mathcal{T},\varpi) = \cup_{\tau \in \mathcal{T}} E^I(\tau,\varpi)$ for each $\varpi \in \mathcal{W}$.

The extension of PQGMs proceeds by defining the accuracy under the asymmetric loss function conditionally on $\varpi$. More precisely, we define
 \begin{equation}\label{Def:CondAsym} \mathcal{L}_a (\tau \vert V\backslash\{a\}, \varpi) = \min_\beta \Ep[ \rho_\tau( X_a-X_{-a}'\beta)\mid \varpi ].\end{equation}
The prediction edge set associated with $(\tau,\varpi)$ is also defined analogously as before. We denote them by $E^P(\tau,\varpi)$ and $E^P(\mathcal{T},\varpi) = \cup_{\tau \in \mathcal{T}} E^P(\tau,\varpi)$, for each $\varpi \in \mathcal{W}$.

\section{Estimators for High-Dimensional Quantile Graphical Models}\label{sec:Estimator-and-Consistency}

In this section, we propose and discuss estimators for QGMs introduced in Section \ref{sec:QGMs}. Throughout it is assumed that we observe a $d$-dimensional i.i.d. random vector $X_V$, namely $\{X_{iV} : i=1,\ldots,n\}$. Based on the data observed, unless additional assumptions are imposed we cannot estimate the quantities of interest for all $\tau \in (0,1)$. Instead, in what follows we will consider a (compact) set of quantile index $\mathcal{T}\subset (0,1)$. The estimators are intended to handle high dimensional models and a continuum of conditioning events in $\mathcal{W}$.  

\subsection{Estimators for CIQGMs}
We discuss the specification and propose an estimator for CIQGMs. Although in general it is potentially hard to correctly specify coherent models, the following are simple examples.

\begin{example}[Multivariate Gaussian Distribution]\label{ExGauss}
Consider the Gaussian case, $X_V \sim N(\mu,\Sigma)$. It follows that for each $a\in V$, the conditional distribution $X_a\mid X_{V\backslash \{a\}}$ satisfies $$X_a\mid X_{V\backslash \{a\}} \sim N\left(\mu_a - \sum_{j\in V\backslash\{a\}} \frac{(\Sigma^{-1})_{aj}}{(\Sigma^{-1})_{aa}}(X_j-\mu_j), \frac{1}{(\Sigma^{-1})_{aa}}\right).$$
Therefore the conditional quantile function of $X_a$ is linear in $X_{V\backslash \{a\}}$ and is given by $$Q_{X_a}(\tau\vert X_{V\backslash \{a\}}) = \frac{\Phi^{-1}(\tau)}{(\Sigma^{-1})_{aa}^{1/2}}+\mu_a -\sum_{j\in V\backslash\{a\}} \frac{(\Sigma^{-1})_{aj}}{(\Sigma^{-1})_{aa}}(X_j-\mu_j).$$
\end{example}

\begin{example}[Multivariate $t$-Distribution]\label{ExT}
Consider the multivariate $t$ distribution case, $X_{V}\sim t_{p}(\mu,\Sigma,v)$,
with location $\mu$, scale matrix $\Sigma$, and degrees of freedom
$v$, as in \cite{ding2016conditional}. It follows that for each
$a\in V$, the conditional distribution $X_{a}\vert X_{V\backslash\{a\}}$
satisfies
$$
X_{a}\mid X_{V\backslash\{a\}}\sim t_{p-1}\left(\mu_{a}-\sum_{j\in V\backslash\{a\}}\frac{(\Sigma^{-1})_{aj}}{(\Sigma^{-1})_{aa}}(X_{j}-\mu_{j}),\frac{v+d_{1}}{v+1}\frac{1}{(\Sigma^{-1})_{aa}},v+1\right).
$$
Therefore the conditional quantile function of $X_{a}$ is given by
$$Q_{X_a}(\tau\vert X_{V\backslash \{a\}})=\sqrt{\frac{v+d_{1}}{v+1}}\frac{F_{t_{v+1}}^{-1}(\tau)}{(\Sigma^{-1})_{aa}^{1/2}}+\mu_{a}-\sum_{j\in V\backslash\{a\}}\frac{(\Sigma^{-1})_{aj}}{(\Sigma^{-1})_{aa}}(X_{j}-\mu_{j}),$$
here $d_{1}=(X_{V\backslash\{a\}}-\mu_{V\backslash\{a\}})^{T}\Sigma_{aa}^{-1}(X_{V\backslash\{a\}}-\mu_{V\backslash\{a\}})$.
\end{example}

\begin{example}[Multiplicative Error Model]\label{ExMultiplicative}
Consider $d=2$ so that $V=\{1,2\}$. Assume that $X_2$ and $\varepsilon$ are independent positive random variables. Assume further that they relate to $X_1$ as $$X_1 = \alpha + \varepsilon X_2.$$ In this case, we have that the conditional quantile functions are linear and given by $$Q_{X_1}(\tau\vert X_2) = \alpha + F^{-1}_\varepsilon(\tau) X_2 \ \ \ \mbox{and}\ \ \  Q_{X_2}(\tau\vert X_1) = (X_1 -\alpha)/F^{-1}_\varepsilon(1-\tau).$$
 \end{example}

\begin{example}[Additive Error Model]\label{ExAdditive}
Consider $d=2$ so that $V=\{1,2\}$. Let $X_2 \sim U(0,1)$ and $\varepsilon\sim U(0,1)$ be independent random variables. Also define the random variable $X_1$ as $$X_1 = \alpha + \beta X_2 + \varepsilon.$$ It follows that $Q_{X_1}(\tau\vert X_2) = \alpha + \beta X_2+\tau$. However, if $\beta = 0$, we have $ Q_{X_2}(\tau\vert X_1)=\tau$, and for $\beta > 0$, direct calculations yield that
$$ Q_{X_2}(\tau\vert X_1)= \left\{\begin{array}{l}
\frac{\tau}{\beta}(X_1-\alpha), \ \ \mbox{if } \ \ X_1 \leq \alpha + \beta\\
\tau + (1-\tau)(X_1-\alpha-\beta),\ \ \mbox{if } \ \ X_1  \geq \alpha + \beta\end{array}\right. $$ where we note that $X_1 \in [ \alpha, 1+\alpha+\beta]$. \end{example}

\begin{example}[Mixture of Gaussians]\label{ExMixGauss}
Similar to the prior example, consider the case $X_V \mid \varpi \sim N(\mu_\varpi,\Sigma_\varpi)$ for each $\varpi \in \mathcal{W}$. It follows that for $a\in V$, the conditional distribution satisfies $$X_a\mid X_{V\backslash \{a\}}, \varpi \sim N\left(\mu_{\varpi a} - \sum_{j\in V\backslash\{a\}} \frac{(\Sigma^{-1})_{\varpi aj}}{(\Sigma^{-1})_{\varpi aa}}(X_j-\mu_{\varpi j}), \frac{1}{(\Sigma^{-1})_{\varpi aa}}\right).$$
Again the conditional quantile function of $X_a$ is linear in $X_{V\backslash \{a\}}$ and is given by $$Q_{X_a}(\tau\vert X_{V\backslash \{a\}},\varpi) = \frac{\Phi^{-1}(\tau)}{(\Sigma^{-1})_{\varpi aa}^{1/2}}+\mu_{\varpi a} -\sum_{j\in V\backslash\{a\}} \frac{(\Sigma^{-1})_{\varpi aj}}{(\Sigma^{-1})_{\varpi aa}}(X_j-\mu_{\varpi j}).$$
\end{example}

\begin{example}[Monotone Transformations]\label{ExMonotoneGauss}
Consider the Gaussian case, for each $a\in V$, $X_a = h_a(Y_a)$  and $Y_V \sim N(\mu,\Sigma)$. It follows that for each $a\in V$, the conditional quantile function satisfies $$Q_{X_a}(\tau\vert X_{V\backslash \{a\}}) = h_a\left(\frac{\Phi^{-1}(\tau)}{(\Sigma^{-1})_{aa}^{1/2}}+\mu_a -\sum_{j\in V\backslash\{a\}} \frac{(\Sigma^{-1})_{aj}}{(\Sigma^{-1})_{aa}}(h_j^{-1}(X_j)-\mu_j)\right).$$
In particular, if $(h_a:a\in V)$ are monotone polynomials, the expression above is a sum of monomials with fractional and integer exponents.
\end{example}

Although a linear specification is correct for Examples \ref{ExGauss} and \ref{ExMultiplicative}, Example \ref{ExT} and \ref{ExAdditive} illustrate that we need to consider a more general transformation of the covariates $X_V$ in the specification for each conditional quantile function. Nonetheless, specifications with additional non-linear terms can approximate non-drastic departures from normality.

We will consider a conditional quantile representation for each $a\in V$. It is based on transformations  of the original covariates $X_{V\backslash\{a\}}$ that create a $p$-dimensional random vector $Z^a=Z^a(X_{V\backslash\{a\}})$ such that
\begin{equation}\label{Def:CondQ}
Q_{X_{a}}(\tau\vert X_{V\backslash\{a\}})=Z^a\beta_{a\tau}+r_{a\tau}, \ \ \beta_{a\tau}\in\mathbb{R}^{p}, \ \ \mbox{for all} \ \tau\in\mathcal{T},
\end{equation} where $r_{a\tau}$ denotes a small approximation error. For $b\in V\backslash\{a\}$ we let $I_a(b):= \{ j : Z^a_{j} \ \mbox{depends on} \ X_b\}$. That is, $I_a(b)$ contains the components of $Z^a$ that are functions of $X_b$. Under correct specification, if $X_a$ and $X_b$ are conditionally independent, we have $\beta_{a\tau j} = 0$ for all $j\in I_a(b)$, $\tau \in (0,1)$.

This allows us to connect the conditional independence quantile graph estimation problem with model selection with quantile regression. Indeed, the representation (\ref{Def:CondQ}) has been used in several quantile regression models, see \cite{Koenker2005}. Under mild conditions this model allows us to identify the process $(\beta_{a\tau})_{\tau \in \mathcal{T}}$ as the solution of the following moment equation%
\begin{equation}\label{Def:QRreg}
\Ep[(\tau - 1\{X_a \leq Z^a\beta_{a\tau} + r_{a\tau}\})Z^a] = 0. \\
\end{equation}
In order to allow for a flexible specification, so that the approximation errors are negligible, it is attractive to consider a high-dimensional $Z^a$ where its dimension $p$ is possibly larger than the sample size $n$. In turn, having a large number of technical controls creates an estimation challenge if the number of coefficients $p$ is not negligible with respect to the sample size $n$. In such a high dimensional setting, a widely applicable condition that makes estimation possible is
 approximate sparsity \cite{fan2011sparse,BCCH2012sparse,BCH2014inference}. Formally we require
\begin{equation}\label{Def:Sparsity}
\max_{a\in V}\sup_{\tau \in \mathcal{T}}\|\beta_{a\tau}\|_0 \leq s,  \ \ \ \max_{a\in V}\sup_{\tau \in \mathcal{T}} \{ \Ep[r^2_{a\tau}] \}^{1/2} \lesssim \sqrt{s/n}, \ \ \ \mbox{and} \ \  \max_{a\in V}\sup_{\tau \in \mathcal{T}} |\Ep[ f_{a\tau}r_{a\tau}Z^a ]| = o(n^{-1/2}),
\end{equation} where the sparsity parameter $s$ of the model is allowed to grow (at a slower rate) as $n$ grows, and  $f_{a\tau}=f_{X_a\mid X_{V\backslash\{a\}}}(Q_{X_a}(\tau\vert X_{V\backslash\{a\}})\vert X_{V\backslash\{a\}})$ denotes the conditional density function evaluated at the corresponding conditional quantile value. This sparsity also has implications on the maximum degree of the associated quantile graph.

Algorithm \ref{Alg:1} below contains our proposal to estimate $\beta_{a\tau}$, $a\in V$, $\tau \in \mathcal{T}$. It is based on three procedures in order to overcome high-dimensionality. In the first step, we apply a (post-)$\ell_1$-penalized quantile regression. The second step applies (post-)Lasso where the data is weighted by the conditional density function at the conditional quantile.\footnote{We note that an estimate for $f_{a\tau}$ is available from $\ell_1$-penalized quantile regression estimators for $\tau+h$ and $\tau-h$ where $h$ is a bandwidth parameter, see \cite{Koenker2005,BCK2013robustQR} and Comment \ref{rem:f}.} Finally, the third step relies on constructing (orthogonal) score function that provides immunity to (unavoidable) model selection mistakes.

There are several parameters that need to be specified for Algorithm \ref{Alg:1}.  The penalty parameter $\lambda_{V\mathcal{T}}$ is chosen to be larger than the $\ell_\infty$-norm of the (rescaled) score at the true quantile function. The work in \cite{BC-SparseQR} exploits the fact that this quantity is pivotal in their setting. Here, additional correlation structure would have an impact and the distribution is pivotal only for each $a\in V$. The penalty is based on the maximum of the quantiles of the following random variables (each with pivotal distribution), for $a\in V$
 \begin{equation}\label{DefWpenalty} \Lambda_{a\mathcal{T}} = \sup_{\tau \in \mathcal{T}} \max_{j\in[p]} \frac{| \En[ (1\{U \leq \tau\} -\tau)Z^a_j] |}{\sqrt{\tau(1-\tau)}\hat\sigma_{aj}^Z} \end{equation}
where $\{U_i: i=1,\ldots,n\}$ are i.i.d. uniform $(0,1)$ random variables, and $\hat\sigma_{aj}^Z=\{\En[(Z_j^a)^2]\}^{1/2}$ for $j\in[p]$. The penalty parameter $\lambda_{V\mathcal{T}}$ is defined as
$$\lambda_{V\mathcal{T}} := \max_{a\in V} \Lambda_{a\mathcal{T}}(1-\xi/|V| \mid Z^a ), $$
that is, the maximum of the $1-\xi/|V|$ conditional quantile of $\Lambda_{a\mathcal{T}}$ given in (\ref{DefWpenalty}). Regarding the penalty term for the weighted Lasso in Step 2, we recommend a (theoretically valid) iterative choice. We refer to Appendix \ref{Appendix:Implementation} for the implementation details of the algorithm. We denote $\|\beta\|_{1,\hat\sigma^Z}:=\sum_j \hat\sigma_{aj}^Z|\beta_j|$ the standardized version of the $\ell_1$-norm.

\begin{algorithm}\label{Alg:1}{\rm (CIQGM Estimator.)} For each $a \in V$, $\tau \in \mathcal{T}$, and $j\in [p]$ \\
\enspace \emph{Step 1}. Compute $\hat\beta_{a\tau}$ from $\|\cdot\|_{1,\hat\sigma^Z}$-penalized $\tau$-quantile regression of $X_a$ on  $Z^a$ with penalty $\lambda_{V\mathcal{T}}\sqrt{\tau(1-\tau)}$.\\
\indent\indent\indent Compute $\widetilde\beta_{a\tau}$ from $\tau$-quantile regression of $X_a$ on  $\{Z^a_k : | \hat\beta_{a\tau k} | \geq \lambda_{V\mathcal{T}}\sqrt{\tau(1-\tau)} / \hat\sigma_{ak}^Z\}$.\\
\enspace \emph{Step 2}. Compute $\widetilde\gamma_{a\tau}^j$ from the post-Lasso estimator of $f_{a\tau} Z_{j}^a$ on $f_{a\tau} Z_{-j}^a$.\\
\enspace \emph{Step 3}. Construct the score function $\hat \psi_{ i}(\alpha)= (\tau - 1\{X_{ia}\leq Z^a_{ij}\alpha + Z_{i, -j}^a\widetilde\beta_{a\tau, -j}\})   f_{ia\tau}(Z_{ij}^a-Z_{i,-j}^a\widetilde\gamma_{a\tau}^j)$ and for \\
\indent\indent\indent $L_{a\tau j}(\alpha) = |\En[\hat \psi_i(\alpha)]|^2/\En[\hat \psi_i^2(\alpha)]$, set $\check \beta_{a\tau j} \in \arg\min_{\alpha \in \mathcal{A}_{a\tau j}} L_{a\tau j}(\alpha)$.
\end{algorithm}

Algorithm \ref{Alg:1} above has been studied in \cite{BCK2013robustQR} where it is applied to a single triple $(a,\tau, j)$, and we have used the following parameter space for $\alpha$, $\mathcal{A}_{a\tau j} = \{ \alpha \in \RR : |\alpha - \widetilde \beta_{a\tau j}| \leq 10 / \{\hat \sigma_{aj}^Z \log n \} \}$. Under similar conditions, results that hold uniformly over $(a,\tau, j) \in V\times \mathcal{T}\times [p]$ are achievable (as shown in the next sections) building upon the tools developed in \cite{BC-SparseQR} and \cite{chernozhukov2012gaussian}. Algorithm \ref{Alg:1} is tailored to achieve good rates of convergence in the $\ell_\infty$-norm. In particular, under standard regularity conditions, with probability approaching to 1 we have
$$ \sup_{\tau \in \mathcal{T}} \| \beta_{a\tau} - \check\beta_{a\tau} \|_\infty \lesssim \sqrt{\frac{\log (p|V|n)}{n}}.$$
In order to create an estimate of $E^I(\tau)=\{ (a,b) \in V\times V: \max_{j\in I_a(b)}|\beta_{a\tau j}|>0\}$, we define $$\hat E^I(\tau) =  \left\{ (a,b)\in V\times V : \  \max_{j\in I_a(b)}\frac{|\check\beta_{a\tau j}|}{\mbox{se}(\check\beta_{a\tau j})} > \overline{\rm cv}\right\}$$ where $\mbox{se}(\check\beta_{a\tau j})=\{\tau(1-\tau)\En[\widetilde v^2_{ia\tau j}]^{-1}\}^{1/2}$ with $\tilde{v}_{ia\tau j} = \hat{f}_{ia\tau}\{ Z^a_{ij} - Z^a_{i,-j}\tilde{\gamma}^j_{a\tau}\}$, is an estimate of the standard deviation of the estimator, and the critical value $\overline{\rm cv}$ is set to account for the uniformity over $a\in V$, $\tau \in \mathcal{T}$, and $j\in[p]$. We discuss in the following sections a data driven procedure based on multiplier bootstrap that is theoretically valid in this high dimensional setting.

\begin{remark}[Stepdown Procedure for $\overline{\rm cv}$] Setting a critical value $\overline{\rm cv}$ that accounts for the multiple hypotheses being tested plays an important role to estimate the graph $\hat E^I(\tau)$. Further improvements can be obtained by considering the stepdown procedure of \cite{romano2005exact} for multiple hypothesis testing that was studied for the high-dimensional case in \cite{chernozhukov2013gaussian}. The procedure iteratively creates a suitable sequence of decreasing critical values. In each step only null hypotheses that were not rejected are considered to determine the critical value. Thus, as long as any hypothesis is rejected at a step, the critical value decreases and we continue to the next iteration. The procedure stops when no hypothesis in the current active set is rejected.   \end{remark}

\begin{remark}[Estimation of Conditional Density Function]\label{rem:f}
The algorithm above requires the conditional density function $f_{a\tau}$ which typically needs to be estimated in practice. It turns out that estimation of conditional quantiles yields a natural estimator for the conditional density function as
$$ f_{a\tau} = \frac{1}{\partial Q_{X_a}(\tau\vert Z^a)/\partial \tau}.$$
Therefore, based on  $\ell_1$-penalized quantile regression estimates at the $\tau+h_n$ and $\tau-h_n$ quantile, where $h = h_n \to 0$ denotes a bandwidth parameter, we have\begin{equation}\label{Eq:hatf} \hat f_{a\tau} = \frac{2h}{\hat  Q_{X_a}(\tau+h\vert Z^a)-\hat  Q_{X_a}(\tau-h\vert Z^a)}\end{equation}
as an estimator of $f_{a\tau}$. Under smoothness conditions, it has a bias of order $h^2$. See \cite{BCK2013robustQR} and the references therein for additional comments and estimators.
 \end{remark}

\subsection{Estimators for PQGMs}

In this section we propose an estimator for PQGMs in which case we are interested in the prediction of $X_a$, $a\in V$, using a linear combination of $X_{V\backslash\{a\}}$ under the asymmetric loss discussed in (\ref{Def:Asym}). We will add an intercept as one of the variables for the sake of notation so that $X_{-a}=(1,X_{V\backslash\{a\}}')'$. Given the loss function $\rho_\tau$, the target $d$-dimensional vector of parameters  $\beta_{a\tau}$ is defined as (part of) the solution of the following optimization problem%
\begin{equation}\label{Def:LQRreg}
\beta_{a\tau} \in \arg\min_{\beta} \ \Ep[\rho_{\tau}(X_{a}-X_{-a}'\beta)].
\end{equation}
As we are interested in the case that $d$ is large, the use of high-dimensional tools to achieve consistent estimators is needed. The estimation procedure we proposed is based on $\ell_1$-penalized quantile regression but additional issues need to be considered to cope with the (non-vanishing) difference between the best linear predictor and the conditional quantile function.  Again we consider models that satisfy an approximately sparse condition. Formally, we require the existence of sparse coefficients $\{\bar\beta_{a\tau}:a\in V, \tau \in \mathcal{T}\}$ such that
\begin{equation}\label{Def:LQSparsity}
\max_{a\in V}\sup_{\tau \in \mathcal{T}}\|\bar \beta_{a\tau}\|_0 \leq s \ \ \ \mbox{and} \ \ \ \max_{a\in V}\sup_{\tau \in \mathcal{T}} \{ \Ep[\{X_{-a}'(\beta_{a\tau}-\bar\beta_{a\tau})\}^2] \}^{1/2} \lesssim \sqrt{s/n},
\end{equation} where (again) the sparsity parameter $s$ of the model is allowed to grow as $n$ grows. The high-dimensionality prevents us from using (standard) quantile regression methods and regularization methods are needed to achieve good prediction properties.

A key issue is to set the penalty parameter properly so that it bounds from above
 \begin{equation}\label{DefScorePred}\max_{a\in V} \sup_{\tau \in \mathcal{T}} \max_{j\in [d]}|\En[(1\{X_{a}\leq X_{-a}'\beta_{a\tau}\}-\tau)X_{-a,j}]|.\end{equation}
However, it is important to note that we do not assume that the conditional quantile of $X_a$ is a linear function of $X_{-a}$. Under correct linear specification of the conditional quantile function, $\ell_1$-penalized quantile regression estimator has been studied in \cite{BC-SparseQR}. The case that the conditional quantile function differs from a linear specification by vanishing approximation errors has been considered in \cite{kato2011} and \cite{BCK2013robustQR}. The analysis proposed here aims to allow for non-vanishing misspecification of the quantile function relative to a linear specification while still guarantees good rates of convergence in the $\ell_2$-norm to the best linear specification. Thus the penalty parameter in the penalized quantile regression needs to account for such misspecification and is no longer pivotal as in \cite{BC-SparseQR}.

In order to handle this issue we propose a two step estimation procedure. In the first step, the penalty parameter $\lambda_0$ is conservative and is set via bounds constructed based on symmetrization arguments, similar in spirit to \cite{vdGeer2008,Volume2013}. This leads to $\lambda_0 = 2(1+1/16) \sqrt{2\log(8|V|^2/\xi)/n}$. Although this is conservative, under mild conditions this would lead to estimates that can be leverage to fine tune the penalty choice. The second step uses the preliminary estimator to bootstrap (\ref{DefScorePred}) based on the tools in \cite{chernozhukov2013gaussian} as follows. Specifically, for estimates $\hat\varepsilon_{ia\tau}$ of the ``noise" $\varepsilon_{ia\tau}=1\{X_{ia} \leq X_{i,-a}'\beta_{a\tau}\}-\tau$ for $i\in[n]$, for $a\in V$ define \begin{equation}\label{Def:NewLambda}\bar{\Lambda}_{a\mathcal{T}} := 1.1 \sup_{\tau \in \mathcal{T}} \max_{j\in [d]}\frac{| \En[ g_i\hat\varepsilon_{ia\tau}X_{i,-aj}]|}{ \{\En[\hat\varepsilon_{ia\tau}^2X_{i,-aj}^2]\}^{1/2}}\end{equation} where $(g_i)_{i=1}^n$ is a sequence of i.i.d. standard Gaussian random variables. The new penalty parameter $\bar\lambda_{V\mathcal{T}}$ is defined as
\begin{equation} \label{barlambda}
\bar{\lambda}_{V\mathcal{T}} := \max_{a\in V} \bar{\Lambda}_{a\mathcal{T}} (1-\xi |X_{-a})
\end{equation}
that is, the maximum of the $(1-\xi)$ conditional quantile of $\bar{\Lambda}_{a\mathcal{T}}$. The penalty choice above adapts to the unknown correlation structure across components and quantile indices. The following algorithm states the procedure where we denote weighted $\ell_1$-norms by $\|\beta\|_{1,\hat\sigma^X}:=\sum_j\hat{\sigma}_{aj}^X|\beta_j|$ with $\hat{\sigma}_{aj}^X=\{\En[X_{j}^2]\}^{1/2}$ the standardized version of the $\ell_1$-norm and $\|\beta\|_{1,\hat\varepsilon}:=\sum_j \hat{\sigma}_{a\tau j}^{\varepsilon X}|\beta_j|$ with $ \hat{\sigma}_{a\tau j}^{\varepsilon X} = \{\En[\hat\varepsilon_{a\tau}^2X_{-a,j}^2]\}^{1/2}$ a norm based on the estimated residuals.

\begin{algorithm}\label{Alg:2}{\rm (PQGM Estimator.)} For each $a \in V$,  and $\tau \in \mathcal{T}$\\
\enspace \emph{Step 1}. Compute $\hat\beta_{a\tau}$ from $\|\cdot\|_{1,\hat\sigma^X}$-penalized $\tau$-quantile regression of $X_a$ on  $X_{-a}$ with penalty $\lambda_0$.\\
\indent\indent\indent Compute $\widetilde\beta_{a\tau}$ from $\tau$-quantile regression of $X_a$ on  $\{X_k : | \hat\beta_{a\tau k} | \geq \lambda_{0} / \hat\sigma_{ak}^X \}$.\\
\enspace \emph{Step 2}. For $\hat\varepsilon_{ia\tau}=1\{X_{ia} \leq X_{i,-a}'\widetilde\beta_{a\tau}\}-\tau$ for $i\in[n]$, and $\xi=1/n$, compute $\bar \lambda_{V\mathcal{T}}$ via (\ref{barlambda}).\\
 \emph{Step 3}. Recompute $\hat\beta_{a\tau}$ from $\|\cdot\|_{1,\hat\varepsilon}$-penalized $\tau$-quantile regression of $X_a$ on  $X_{-a}$ with penalty $\bar \lambda_{V\mathcal{T}} $.\\
\indent\indent\indent Compute $\check\beta_{a\tau}$ from $\tau$-quantile regression of $X_a$ on  $\{X_k : | \hat\beta_{a\tau k} | \geq \bar \lambda_{V\mathcal{T}}  / \hat{\sigma}_{a\tau k}^{\varepsilon X} \}$.
\end{algorithm}

Under regularity conditions stated in Section \ref{Sec:MainTheory}, with probability approaching 1, we have
$$ \max_{a\in V} \sup_{\tau \in \mathcal{T}} \| \beta_{a\tau} - \check\beta_{a\tau} \| \lesssim \sqrt{\frac{s\log (|V|n)}{n}}.$$
The estimate of the prediction quantile graph is given by the support of $(\check\beta_{a\tau})_{a\in V, \tau \in \mathcal{T}}$, namely $$\hat E^P(\tau) = \left\{ (a,b)\in V\times V : \  |\widehat\beta_{a\tau b}| > \bar \lambda_{V\mathcal{T}} /  \hat{\sigma}_{a\tau b}^{\varepsilon X} \right \}.$$ That is, it is induced by covariates selected by the $\ell_1$-penalized estimator. Those thresholded estimators not only have the same rates of convergence as of the original penalized estimators but also possess additional sparsity guarantees.

\subsection{Estimators for $\mathcal{W}$-Conditional Quantile Graphical Models}

In order to handle the additional conditioning events $\Omega_\varpi$, $\varpi\in\mathcal{W}$, we propose to modify Algorithms \ref{Alg:1} and \ref{Alg:2} based on kernel smoothing. To that extent, we assume the observed data is of the form $\{(X_{iV}, W_i) : i=1,\ldots, n\}$, where $W_i$ might be defined through additional variables. Furthermore, we assume for each conditioning event $\varpi\in \mathcal{W}$ we have access to a kernel function $K_\varpi$ that is applied to $W$, to represent the relevant observations associated with $\varpi$ (recall that we denote $\Pr( W \in \Omega_\varpi)$ as $\Pr(\varpi)$). We assume that $K_\varpi(W) = 1\{W \in \Omega_\varpi\}$.

\begin{example}[Stock Returns Under Market Downside Movements, continued] In Example \ref{Ex:Downside}, we have $W$ as the market return and the conditioning event as $\Omega_\varpi=\{W\leq \varpi\}$ which is parameterized by $\varpi\in\mathcal{W}$, a closed interval in $\RR$. We might be interest in a fixed $\varpi$ or on a family of values $\varpi\in(-\bar \varpi, 0]$. The latter induces $\mathcal{W} = \{\Omega_\varpi=\{W\leq \varpi\} : \varpi\in(-\bar \varpi, 0] \}$. The kernel function is simply $K_{\varpi}(t) = 1\{t\leq \varpi\}$.
\end{example}

This framework encompasses the previous framework by having $K_\varpi(W) = 1$ for all $W$. Moreover, it allows for a richer class of estimands which require estimators whose properties should hold uniformly over $\varpi \in \mathcal{W}$ as well. Next we propose estimators for this setting, i.e. we generalize the previous methods to account for the additional conditioning on $\varpi\in \mathcal{W}$. In what follows, with a slight abuse of notation we use $\varpi$ to denote not only the index but also the event $\Omega_\varpi$. For further notational convenience, we denote $u=(a,\tau,\varpi)\in \mathcal{U} := V\times\mathcal{T}\times\mathcal{W}$ so that the set $\mathcal{U}$ collects all the three relevant indices. With $\hat{\sigma}^{Z}_{a\varpi j} = \{\En[K_\varpi(W)(Z_j^a)^2]\}^{1/2}$, we define the following weighted $\ell_1$-norm $ \|\beta\|_{1,\varpi} = \sum_{j\in [p]} \hat{\sigma}^{Z}_{a\varpi j} |\beta_j|.$ This norm is $\varpi$ dependent and provides the proper adjustments as we condition on different events associated with different $\varpi$'s.

We first consider estimators of CIQGMs conditional on the events in $\mathcal{W}$. In this setting, the model is correctly specified up to small approximation errors. The definition of the penalty parameter will be based on the random variable
$$ \Lambda_{a\mathcal{T}\mathcal{W}} = \sup_{\tau \in \mathcal{T}, \varpi \in \mathcal{W}} \max_{ j\in[p]} \left|\frac{\En[K_\varpi(W)(1\{U \leq \tau\} - \tau) Z^a_j]}{\sqrt{\tau(1-\tau)} \hat{\sigma}^{Z}_{a\varpi j}}\right| $$
where $U_i$ are independent uniform $(0,1)$ random variables, and set the penalty $$\lambda_{V\mathcal{T}\mathcal{W}} = \max_{a\in V}\Lambda_{a\mathcal{T}\mathcal{W}}(1-\xi/\{|V|n^{1+2d_W}\}\vert Z^a, W),$$
that is, the maximum of the $(1-\xi/\{|V|n^{1+2d_W}\})$ conditional quantile of $\Lambda_{a\mathcal{T}\mathcal{W}} $. Algorithm \ref{Alg:1prime} provides the definition of the estimator. Here $\mathcal{A}_{uj} = \{ \alpha \in \RR :  |\alpha - \widetilde \beta_{uj}| \leq 10 / \{\hat{\sigma}^{Z}_{a\varpi j} \log n \} \}$, and denote $\lambda_u := \lambda_{V\mathcal{T}\mathcal{W}}\sqrt{\tau(1-\tau)}$.

\begin{algorithm}\label{Alg:1prime}{\rm ($\mathcal{W}$-Conditional CIQGM Estimator.)} For $(a,\tau,\varpi)\in V\times \mathcal{T}\times\mathcal{W}$ and $j\in [p]$\\
\enspace \emph{Step 1}. Compute $\hat\beta_{u}$ from  $\|\cdot\|_{1,\varpi}$-penalized $\tau$-quantile regression of $K_{\varpi}(W)(X_a;Z^a)$ with penalty $\lambda_u$.\\
\indent\indent\indent Compute $\widetilde{\beta}_{u}$ from $\tau$-quantile regression of $K_\varpi(W)(X_a; \{Z^a_k : | \hat\beta_{uk} | \geq   \lambda_u/ \hat{\sigma}^{Z}_{a\varpi j} \})$.\\
\enspace \emph{Step 2}. Compute $\widetilde\gamma_{u}^j$ from the post-Lasso estimator of $K_\varpi(W)f_{u}Z_{j}^a$ on $K_\varpi(W)f_{u}Z_{-j}^a$.\\
\enspace \emph{Step 3}. Construct the score function $\hat \psi_{i}(\alpha)= K_\varpi(W_i)(\tau - 1\{X_{ia}\leq Z^a_{ij}\alpha + Z_{i,-j}^a\widetilde\beta_{u, -j}\})   f_{iu}(Z_{ij}^a-Z_{i,-j}^a\widetilde\gamma_{u}^j)$  \\
\indent\indent\indent and for $L_{uj}(\alpha) = |\En[\hat \psi_i(\alpha)]|^2/\En[\hat \psi_i^2(\alpha)]$, set $\check \beta_{uj} \in \arg\min_{\alpha \in \mathcal{A}_{uj}} L_{uj}(\alpha)$ .
\end{algorithm}

Next we consider estimators of PQGMs conditional on the events in $\mathcal{W}$. Similar to the previous case, for $a\in V$ define \begin{equation} \bar{\Lambda}_{a\mathcal{T}\mathcal{W}} := 1.1 \sup_{\tau \in \mathcal{T}, \varpi \in \mathcal{W}} \max_{j \in [d]}\frac{| \En[ K_\varpi(W) g\hat\varepsilon_{a\tau\varpi}X_{-a,j}]|}{\{ \En[K_\varpi(W)\hat\varepsilon_{a\tau\varpi}^2X_{-a,j}^2]\}^{1/2}}\end{equation}
where $(g_i)_{i=1}^n$ is a sequence of i.i.d. standard Gaussian random variables. The new penalty parameter $\bar\lambda_{V\mathcal{T}}$ is defined as
\begin{equation}\label{Def:NewLambda2} \bar{\lambda}_{V\mathcal{T}\mathcal{W}} := \max_{a\in V} \bar{\Lambda}_{a\mathcal{T}\mathcal{W}} (1-\xi |X_{-a})
\end{equation}
that is, the maximum of the $(1-\xi)$ conditional quantile of $\bar{\Lambda}_{a\mathcal{T}\mathcal{W}}$.
It will also be useful to define another weighted $\ell_1$-norm, $\|\beta\|_{1,\varpi\hat\varepsilon}:=  \sum_{j} \hat{\sigma}_{a\tau \varpi j}^{\varepsilon X}|\beta_j|$ with $ \hat{\sigma}_{a\tau \varpi j}^{\varepsilon X} = \{\En[K_\varpi(W)\widehat\varepsilon_{a\tau\varpi}^2X_{-a, j}^2]\}^{1/2}$. We also denote $\hat{\sigma}^{X}_{a\varpi j}=\{\En[K_\varpi(W)X_{-a,j}^2]\}^{1/2}$.  The penalty choice and weighted $\ell_1$-norm adapt to the unknown correlation structure across components and quantile indices. The following algorithm states the procedure, with $ \lambda_{0\mathcal{W}} = 2(1+1/16) \sqrt{2\log(8|V|^2\{ne/d_W\}^{2d_W}/\xi)/n}$.

\begin{algorithm}\label{Alg:2prime}{\rm ($\mathcal{W}$-Conditional PQGM Estimator.)} For $(a,\tau,\varpi) \in V\times \mathcal{T}\times\mathcal{W}$\\
\enspace \emph{Step 1}. Compute $\hat\beta_{u}$ from $\|\cdot\|_{1,\varpi}$-penalized $\tau$-quantile regression of $X_a$ on  $X_{-a}$ with penalty $\lambda_{0\mathcal{W}}$.\\
\indent\indent\indent Compute $\widetilde\beta_{u}$ from $\tau$-quantile regression of   $K_{\varpi}(W)(X_a; \{X_{-a, k}: |\hat\beta_{uk} | \geq \lambda_{0\mathcal{W}} /\hat{\sigma}^{X}_{a\varpi k} \})$.\\
\enspace \emph{Step 2}. For $\hat\varepsilon_{iu}=1\{X_{ia} \leq X_{i,-a}'\widetilde\beta_{u}\}-\tau$ for $i\in[n]$, and $\xi=1/n$, compute $\bar \lambda_{V\mathcal{T}\mathcal{W}}$ via (\ref{Def:NewLambda2}).\\
\emph{Step 3}. Recompute $\hat\beta_{u}$ from $\|\cdot\|_{1,\varpi\widehat\varepsilon}$-penalized $\tau$-quantile regression of $K_\varpi(W)(X_a; X_{-a})$ with penalty $\bar \lambda_{V\mathcal{T}\mathcal{W}}$. \\
\indent\indent\indent Compute $\check\beta_{u}$ from $\tau$-quantile regression of $K_\varpi(W)(X_a; \{X_{-a, k}: |\hat\beta_{uk} | \geq \bar \lambda_{V\mathcal{T}\mathcal{W}} /\hat{\sigma}_{uk}^{\varepsilon X}\})$.\\
\end{algorithm}

\begin{remark}[Computation of Penalty Parameter over $\mathcal{W}$] The penalty choices require one to maximize over $a\in V$, $\tau\in\mathcal{T}$ and $\varpi\in\mathcal{W}$. The set $V$ is discrete and does not pose a significant challenge. However both other sets are continuous and additional care is needed. In most applications we are concerned with the case that $\mathcal{W}$ is a low dimensional VC class of sets and it impacts the calculation only through indicator functions, which is precisely the case of $\mathcal{T}$. It follows that only a polynomial number (in $n$) of different values of $\tau$ and $\varpi$ would need to be considered. \footnote{
A class of sets is said to be a VC class, if the VC dimension is finite. In what follows we use that the VC dimension provides a way to control how much we can overfit the data and it will also lead to (theoretically valid) recommendations for the penalty parameters. For the formal definition of VC class, see \cite{vdV-W}).}

\end{remark}

\section{Main Theoretical Results}\label{Sec:MainTheory}

This section is devoted to theoretical guarantees associated with the proposed estimators. We will establish rates of convergence results for the proposed estimators as well as the (uniform) validity of confidence regions. These results build upon and contribute to an increasing literature on the estimation of many processes of interest with (high-dimensional) nuisance parameters.

Throughout, we will provide results for the estimators of the $\mathcal{W}$-conditional quantile graphical models as those can be generalized the other models by setting $K_\varpi(W) = 1$. Although some of the tools are similar, CIQGMs and PQGMs require different estimators and are subject to different assumptions. Thus, substantial different analyses are required.

\subsection{$\mathcal{W}$-Conditional CIQGM}
For $u=(a,\tau,\varpi)\in \mathcal{U}$, define the $\tau$-conditional quantile function of $X_a$ given $X_{V\backslash\{a\}}$ and $\varpi$ as
\begin{equation}\label{Def:ModelCI} Q_{X_{a}}(\tau\vert X_{V\backslash \{a\}}, \varpi ) = Z^a\beta_u + r_u,\end{equation}
where $Z^a$ is a $p$-dimensional vector of (known) transformations of $X_{V\backslash \{a\}}$, and $r_u$ is an approximation error. The event $\varpi \in \mathcal{W}$ will be used for further conditioning through the function $K_\varpi(W) = 1\{ W \in \varpi\}$.

We let $f_{X_a\mid X_{V\backslash \{a\}}, \varpi }(\cdot\vert X_{V\backslash \{a\}},\varpi)$  denote the conditional density function of $X_a$ given $X_{V\backslash \{a\}}$ and $\varpi\in\mathcal{W}$. We define $f_{u}:=f_{X_a\vert X_{V\backslash \{a\}}, \varpi }(Q_{X_{a}}(\tau\vert X_{V\backslash \{a\}}, \varpi )\vert  X_{V\backslash \{a\}}, \varpi )$ as the value of the conditional density function evaluated at the $\tau$-conditional quantile. In our analysis we will consider  for $u\in\mathcal{U}$
 \begin{equation}\label{Def:fuQuantileTrue}\underline f_{u} = \inf_{\|\delta\|=1} \frac{\Ep[f_u \{Z^a\delta\}^{2}\vert \varpi]}{\Ep[\{Z^a\delta\}^{2}\vert \varpi ]} \ \ \mbox{and} \ \ \underline{f}_\mathcal{U} =\min_{u\in\mathcal{U}} \underline f_u.\end{equation}
Moreover, for each $u\in\mathcal{U}$ and $j\in[p]$ we define
\begin{equation}\label{Def:Model12} \gamma_{u}^j = \arg\min_{\gamma} \Ep[ f_{u}^2K_\varpi(W)(Z^a_j - Z^a_{-j}\gamma)^2].  \end{equation}
This provides a weighted projection to construct the residuals
$$ v_{uj} = f_u(Z^a_j - Z^a_{-j}\gamma_{u}^j) $$
that satisfy $\Ep[ f_{u} Z^a_{-j} v_{uj}  \vert \varpi] =0$ for each $(u,j)\in \mathcal{U}\times[p]$.

The estimands of interest are $\beta_u \in \RR^p $, $u \in  \mathcal{U}$, and can be written as the solution of (a continuum of) moment equations. Letting $\beta_{uj}$ denote the $j$th component of $\beta_u$ so that $\beta_{uj} \in \RR$ solves
$$ \Ep[\psi_{uj}(X,W,\beta,\eta_{uj}) ] = 0, $$
where the function $\psi_{uj}$ is given by
$$ \psi_{uj}(X,W,\beta,\eta_{uj}) = K_\varpi(W)( \tau - 1\{ X_a \leq Z^a_j \beta + Z^a_{-j}\eta_{uj}^{(1)} + \eta_{uj}^{(3)} \})f_u(Z_j^a-Z^a_{-j}\eta_{uj}^{(2)}),$$
and the true value of the nuisance parameter is given by $\eta_{uj}=(\eta_{uj}^{(1)},\eta_{uj}^{(2)},\eta_{uj}^{(3)})$ with $\eta_{uj}^{(1)}=\beta_{u,-j}$, $\eta_{uj}^{(2)}=\gamma^j_u$, and $\eta_{uj}^{(3)}=r_u$. In what follows $c, C$ denote some fixed constant, $\delta_n$ and $\Delta_n$ denote sequences go to zero with $\delta_n = n^{-\mu}$ for some sufficiently small $\mu$. Denote $\mu_\mathcal{W}=\inf_{\varpi\in\mathcal{W}}\Pr(\varpi)$.

{\bf Condition CI.} \textit{Let $u=(a, \tau, \varpi) \in \mathcal{U} := V \times \mathcal{T}\times \mathcal{W}$ and $(X_i,W_i)_{i=1}^n$ denote a sequence of independent and identically distributed random vectors generated accordingly to models (\ref{Def:ModelCI}) and (\ref{Def:Model12}):}

\textit{(i) Suppose $\sup_{u\in\mathcal{U}, j\in[p]}\{\|\beta_u\|+\|\gamma^j_u\|\}\leq C$ and $\mathcal{T}$ is a fixed compact set:
 (a) there exists $s=s_n$ such that $\sup_{u\in\mathcal{U}, j\in[p]}\{\|\beta_u\|_0 + \|\bar \gamma^j_u\|_0\}\leq s$, $\sup_{u\in\mathcal{U}, j\in[p]}\|\bar \gamma^j_u-\gamma^j_u\|+s^{-1/2}\|\bar \gamma^{j}_u-\gamma^{j}_u\|_1\leq C \{n^{-1}s\log(|V| p n)\}^{1/2}$, where $\bar \gamma^j_u$ is approximately sparse;
 (b) the conditional distribution function of $X_a$ given $X_{V\backslash \{a\}}$ and $\varpi$ is absolutely continuous with continuously differentiable density $f_{X_a\vert  X_{V\backslash \{a\}},\varpi}(t\vert X_{V\backslash \{a\}},\varpi)$ bounded by $\bar f$ and its derivative bounded by $\bar f'$ uniformly over $u\in\mathcal{U}$; (c) $|f_u-f_{u'}|\leq L_f\|u-u'\|$, $\|\beta_u-\beta_{u'}\|\leq L_\beta \|u-u'\|^\kappa$ with $\kappa \in [1/2,1]$,  and $\Ep[ |K_\varpi(W)-K_{\varpi'}(W)|] \leq L_K\|\varpi-\varpi'\|$; (d) the VC dimension $d_W$ of the set $\mathcal{W}$ is fixed, $\{Q_{X_a}(\tau\vert X_{V\backslash \{a\}},\varpi) : (\tau,\varpi)\in \mathcal{T} \times \mathcal{W}\}$ is a VC-subgraph with VC-dimension $1+Cd_W$ for every $a\in V$; }

\textit{(ii) The following moment conditions hold uniformly over $u\in\UU$ and $j\in[p]$: $\Ep[|f_uv_{uj}Z^a_k|^2\vert \varpi]^{1/2}\leq C\underf_u$, $\min_{a\in V}\inf_{\|\delta\|=1}\Ep[\{(X_a,Z^a)\delta\}^2 \vert \varpi]\geq c$, $ \max_{a\in V}\sup_{\|\delta\|=1}\Ep[ \{(X_a,Z^a)\delta\}^4 \vert \varpi ]\leq C$, $\Ep[f_u^2 (Z^a\delta)^{2}\vert \varpi] \leq C\underf_u^2\Ep[(Z^a\delta)^{2}\vert \varpi ]$, $\max_{j,k}\frac{\Ep[|f_uv_{uj}Z^a_k|^3\vert \varpi]^{1/3}}{\Ep[|f_uv_{uj}Z^a_k|^2\vert \varpi]^{1/2}}\log^{1/2}(pn|V|) \leq \delta_n \{n \Pr(\varpi)\}^{1/6}$;}

\textit{(iii) Furthermore, for some fixed $q\geq 4\vee (1+2d_W)$, $\sup_{u\in\,\mathcal{U},\|\delta\|=1} \Ep[|(X_a,Z^a)\delta|^2r_u^2 \vert\varpi]\leq C \Ep[r_u^2 \vert \varpi] \leq Cs/n$, $\max_{u\in\mathcal{U}, j\in[p]}|\Ep[f_ur_uv_{uj}\vert \varpi]|\leq \delta_n n^{-1/2}$, $\Ep[ \max_{i\leq n} \sup_{u\in \mathcal{U}}|K_\varpi(W)r_{iu}|^q ] \leq C$, and with probability $1-\Delta_n$, uniformly over $u\in \UU, j\in[p]$: $\En[r_u^2v_{uj}^2\vert\varpi] + \En[r_u^2\vert\varpi]\lesssim n^{-1}s\log(p|V| n)$, $\En[K_\varpi(W)\{|r_u|+r_u^2\}(Z^a\delta)^2] \leq \delta_n \En[K_\varpi(W)f_u(Z^a\delta)^2]$;}

\textit{(iv) For a fixed $q \geq 4 \lor (1+d_W)$,  $\mathrm{diam}(\mathcal{W}) \leq n^{1/2q}$, $\Ep[ \max_{i\leq n}\|X_{iV}\|_\infty^q\vee \max_{a\in V}\|Z^a_i\|_\infty^q]^{1/q}/\mu_\mathcal{W}\leq M_n$, $\Ep[\max_{i\leq n} \sup_{u\in\UU, j\in[p]} |v_{iuj}|^q ]^{1/q}\leq L_n $, $ (L_f+L_K)^2M_n^2\log^2(p|V|n)\leq \delta_n n\mu_\mathcal{W}^3\underf_\mathcal{U}^6 $, $M_n^{4}\log(p|V|n) \log n\leq \delta_n^2 n \mu_\mathcal{W}^2\underf_\mathcal{U}^2$, $s^2\log^2(p|V|n)\leq \delta_n^2 n \underf_\mathcal{U}^4\mu_\mathcal{W}^6$, 
$s^3\log^3(p|V|n)\leq \delta_n^4 n \underline{f}_{\mathcal{U}}^2\mu_\mathcal{W}^3$,
$L_n^2s\log^{3/2} (p|V|n) \leq \delta_n \underline{f}_{\mathcal{U}}(n\mu_\mathcal{W})^{1/2}$, $M_n s \sqrt{\log( p|V|n)} \leq \delta_n n^{1/2}\mu_\mathcal{W}\underf_\mathcal{U}$.}

Condition CI assumes various conditional moment conditions to allow for the estimation to be conditional on $\varpi \in \mathcal{W}$. Those are analogous to the (unconditional) conditions in the high-dimensional literature in quantile regression
models, \cite{BCK2013robustQR}. In particular, condition CI(i) assumes smoothness of the density function, and of coefficients. Condition CI(ii) assumes conditions on the (conditional) population design matrices such as the ratio between eigenvalues. Condition CI(iii) pertains to the approximations errors and assumes mild moment conditions. Finally Condition CI(iv) provides sufficient conditions on the allowed growth of the model via $p$ and $|V|$ relative to the available sample size $n$. Note, Condition CI(iii) also assume $d_W$ is bounded by fixed $q$, and the proof can easily be extended to other cases.

Condition CI is a high level condition intended to allow approximate sparse models, approximation errors, tail events in $\mathcal{W}$, and to require only $q$ moments (going beyond sub-Gaussian variables). When applied to the special case of sub-Gaussian, exactly sparse, and singleton $\mathcal{W}$, it becomes a relatively standard assumption. For example, without approximation error, e.g. the multivariate Gaussian
case, Condition CI(iii) can be removed entirely. By allowing for a large number of variables (and transformations) the approximation errors can be controlled when the quantile functions belong to some smooth function class (e.g. Sobolev space). Although it is outside the scope of the current work, the ideas and results can be generalized
to dependent data, using results from \cite{chernozhukov2014clt}.

Based on Condition CI, we derive our main results regarding the proposed estimator.
Moreover, we also establish new results for $\ell_1$-penalized quantile regression methods that hold uniformly over the indices $u \in \mU$. The following theorems summarize these results.

\begin{theorem}[{Uniform Rates of Convergence for $\mathcal{W}$-Conditional Penalized Quantile Regression}]\label{theorem:rateQRgraph}
Under Condition CI, we have that with probability at least $1-o(1)$
$$ \|\hat\beta_u - \beta_u\| \lesssim  \sqrt{\frac{s(1+d_W)\log(p|V|n)}{n \underline{f}_u\Pr(\varpi)}}, \ \ \ \mbox{uniformly over $u=(a,\tau,\varpi) \in\mathcal{U}$}$$
Moreover, the thresholded estimator $\hat\beta^{\bar\lambda}$, with $\bar\lambda=\sqrt{(1+d_W)\log(p|V|n)/n}$ and $\hat\beta_{uj}^{\bar\lambda}=\hat\beta_{uj}1\{|\hat\beta_{uj}|>\\
\bar\lambda  \hat{\sigma}^{Z}_{a\varpi j}\}$, satisfies the same rate and $\|\hat\beta^{\bar\lambda}\|_0\lesssim s$.
\end{theorem}

Theorem \ref{theorem:rateQRgraph} builds upon ideas in \cite{BC-SparseQR} however the proof strategy is designed to derive rates that are adaptive to each $u\in \mathcal{U}$. Indeed the rates of convergence are $u$-dependent and they show a slower rate for rare events $\varpi \in \mathcal{W}$.

\begin{theorem}[{Uniform Rates of Convergence for $\mathcal{W}$-Conditional Weighted Lasso}]\label{theorem:rateLASSOgraph}
Under Condition CI, we have that with probability at least $1-o(1)$
$$ \|\hat\gamma^j_u - \gamma^j_u\| \lesssim \frac{1}{\underline{f}_{u}} \sqrt{\frac{s(1+d_W)\log(p|V|n)}{n\Pr(\varpi)}} \ \ \ \mbox{and} \ \ \ \|\hat\gamma_u^j\|_0 \lesssim s, \ \ \ \mbox{uniformly over $u=(a,\tau,\varpi) \in\mathcal{U}$, $j\in[p]$.}$$
\end{theorem}

The following result establishes a uniform Bahadur representation for the final estimators.

\begin{theorem}[{Uniform Bahadur Representation for $\mathcal{W}$-Conditional CIQGM}] \label{theorem:semiparametric} Under  Condition CI, the estimator $(\check \beta_{uj})_{u \in \mathcal{U},j\in[p]}$ satisfies
$$ \sigma_{uj}^{-1}\sqrt{n}(\check\beta_{uj} - \beta_{uj}) =   \mathbb{U}_n(u,j)   + O_P(\delta_n) \text{ in } \ell^\infty(\mathcal{U}\times [p]),$$
where  $\sigma^2_{uj}= \tau(1-\tau)\Ep [K_\varpi(W)v_{uj}^2]^{-1}$ and
$$\mathbb{U}_n(u,j):=\frac{\{\tau(1-\tau)\Ep [K_\varpi(W)v_{uj}^2]\}^{-1/2}}{\sqrt{n}} \sum_{i=1}^n ( \tau - 1\{U_i(a,\varpi)\leq \tau\})K_\varpi(W_i)v_{i,uj},$$ where $U_1(a,\varpi), \ldots, U_n(a,\varpi)$ are i.i.d. uniform $(0,1)$ random variables, independent of $v_{1,uj},\ldots, v_{n,uj}$.
\end{theorem}

Theorem \ref{theorem:semiparametric} plays a key role. However, it is important to note that the marginal distribution of $\mathbb{U}_n(u,j)$ is pivotal. Nonetheless, there is a non-trivial correlation structure between $U(a,\varpi)$ and $U(\tilde a,\tilde \varpi)$. In order to construct confidence regions with non-conservative guarantees, we rely on a multiplier bootstrap method. We will approximate the process $\mathcal{N}=(\mathcal{N}_{uj})_{u\in\mathcal{U}, j\in[p]}$ by the Gaussian multiplier bootstrap based on estimates $\hat\psi_{uj}:= \{\tau(1-\tau)\En[K_{\varpi}(W)\hat{v}_{uj}^{2}]\}^{-1/2}(\tau-1\{ X_a \leq Z^a\hat\beta_u\})K_\varpi(W_i)\hat v_{uj}$ of $\bar\psi_{uj}(U,W)=\{\tau(1-\tau)\En[K_{\varpi}(W)v_{uj}^{2}]\}^{-1/2}(\tau - 1\{ U(a,\varpi) \leq \tau\})K_\varpi(W)v_{uj}$, namely
$$ \widehat{\mathcal{G}} = (\widehat{\mathcal{G}}_{uj})_{u\in\UU,j\in[p]} = \left\{ \frac{1}{\sqrt{n}}\sum_{i=1}^ng_i\hat\psi_{uj}(X_i,W_i)\right\}_{u\in\UU,j\in[p]}$$
where $(g_i)_{i=1}^n$ are independent standard normal random variables which are independent from the data $(W_i)_{i=1}^n$. Based on Theorem 5.2 of \cite{chernozhukov2013gaussian}, the following result shows that the multiplier bootstrap provides a valid approximation to the large sample probability law of $\sqrt{n}(\check \beta_{uj}- \beta_{uj})_{u \in \mathcal{U},j\in[p]}$ which is suitable for the construction of uniform confidence bands over the set of indices associated with $I_a(b)$ for all $a,b \in V$. We let $\mathcal{P}_n$  denote the collection of distributions $P$ for the data such that Condition CI is satisfied for given $n$. This is the collection of all approximately sparse models where the above sparsity conditions, moment conditions, and growth conditions are satisfied.

\begin{corollary}[{Gaussian Multiplier Bootstrap for $\mathcal{W}$-Conditional CIQGM}]\label{theorem: general bs}
Under  Condition CI with $\delta_n = o(\{ (1+d_W)\log(p|V|n)\}^{-1/2})$,  and $(1+d_W)\log(p|V|n) = o(\{ (n/L_n^2)^{1/7}\wedge (n^{1-2/q}/L_n^2)^{1/3}\})$, we have that
$$ \sup_{P \in  \mathcal{P}_n} \sup_{t,t' \in \RR, u\in\mathcal{U}, b\in V}  \left| \Pr_P\left( \max_{j\in I_a(b)}\frac{|\check\beta_{uj}-\beta_{uj}|}{n^{-1/2}\sigma_{uj}} \in [t,t'] \right) - \Pr_P\left( \max_{j\in I_a(b)}|\widehat{\mathcal{G}}_{uj}| \in [t,t']\mid (X_i,W_i)_{i=1}^n\right) \right| = o(1)  $$
\end{corollary}

Corollary  \ref{theorem: general bs} allows the construction of simultaneous confidence regions for the coefficients that are uniformly valid over the set of data generating processes induced by Condition CI. Based on the coefficients whose intervals do not overlap zero, we can construct a conditional independence graph process $\hat E^{I}(\tau,\varpi), \tau \in \mathcal{T}, \varpi\in\mathcal{W}$ that contains the true conditional independence quantile graph with a specified probability.

\subsection{$\mathcal{W}$-Conditional PQGM}

In this section, we derive theoretical guarantees for the $\mathcal{W}$-conditional predictive quantile estimators uniformly over $u=(a,\tau,\varpi)\in \mathcal{U}$. For each $u\in \mathcal{U}$ the estimand of interest is $\beta_u \in \RR^p $ that corresponds to the best linear predictor under asymmetric loss function, namely
\begin{equation}\label{Def:Model11} \beta_u \in \arg\min_\beta \Ep[ \rho_\tau(X_a - X_{-a}'\beta) \mid \varpi ]\end{equation}
where the event $\varpi \in \mathcal{W}$ is used for further conditioning. In the analysis below, the conditioning is implemented through the function $K_\varpi(W) = 1\{ W \in \varpi\}$.

In the analysis of this case, the main issue is to handle the inherent misspecification of the linear form $X_{-a}'\beta_u$ with respect to the true conditional quantile. The first consequence is to handle the identification condition. Given $X_{-a}$ and $\varpi\in\mathcal{W}$, we let  $f_{u}:=f_{X_a\vert X_{-a}, \varpi } (X_{-a}'\beta_{u} \vert  X_{-a}, \varpi )$ denote the value of the conditional density function evaluated at $X_{-a}'\beta_{u}$. In our analysis, we will consider
\begin{equation}\label{Def:fuForPrediction} \underline f_{u} = \inf_{\|\delta\|=1} \frac{\Ep[f_u \{X_{-a}'\delta\}^{2}\mid \varpi]}{\Ep[(X_{-a}'\delta)^{2}\mid \varpi ]} \ \ \mbox{and} \ \ \underline{f}_\mathcal{U} =\min_{u\in\mathcal{U}} \underline f_u.\end{equation}
We remark that $\underline f_{u}$ defined in (\ref{Def:fuForPrediction}) differs from (\ref{Def:fuQuantileTrue}) which is the standard conditional density at the true quantile value.  It turns out that Knight's identity can be used by exploiting the first order condition associated with the optimization problem (\ref{Def:Model11}) which yields zero mean condition similar to the conditional quantile condition.

A second consequence of the misspecification is the lack of pivotality of the score. Such pivotal property was convenient in the previous section to define penalty parameters and to conduct inference. We will exploit bounds on the VC-dimension of the relevant classes of sets formally stated below.

{\bf Condition P.} \textit{Let $\mathcal{U} = V \times \mathcal{T}\times \mathcal{W}$ and $(X_i,W_i)_{i=1}^n$ denote a sequence of independent and identically distributed random vectors generated accordingly to models (\ref{Def:Model11}):}

\textit{(i) Suppose that  $\sup_{u\in\mathcal{U}}\|\beta_u\|\leq C$ and $\mathcal{T}$ is a fixed compact set: (a) there exists $s=s_n$ and $\bar \beta_u$ such that $\sup_{u\in\mathcal{U}}\|\bar \beta_u\|_0 \leq s$,  $\sup_{u\in\mathcal{U}}\|\bar \beta_u-\beta_u\|+s^{-1/2}\|\bar \beta_u-\beta_u\|_1 \leq \sqrt{s/n}$; (b) the conditional distribution function of $X_a$ given $X_{-a}$ and $\varpi$ is absolutely continuous with continuously differentiable density $f_{X_a\mid  X_{ -a},\varpi}(t\mid X_{-a},\varpi)$ such that its values are bounded by $\bar f$ and its derivative is bounded by $\bar f'$ uniformly over $u\in\mathcal{U}$; (c) $|f_u-f_{u'}|\leq L_f\|u-u'\|$, $\|\beta_u-\beta_{u'}\|\leq L_\beta \|u-u'\|^\kappa$ with $\kappa \in [1/2,1]$,  and $\Ep[ |K_\varpi(W)-K_{\varpi'}(W)|] \leq  L_K\|\varpi-\varpi'\|$; (d) the VC dimension $d_W$ of the set $\mathcal{W}$ is fixed, $\{1\{ X_a \leq X_{-a}'\beta_u\} : (\tau,\varpi)\in \mathcal{T}\times  \mathcal{W}\}$ is a VC-class with VC-dimension $1+d_W$ for every $a\in V$;}

\textit{(ii) The following moment conditions hold uniformly over $u\in\UU$:  $\min_{a\in V}\inf_{\|\delta\|=1}\Ep[\{X_{-a}'\delta\}^2 \vert \varpi]\geq c$, $\max_{a\in V}\sup_{\|\delta\|=1}\Ep[ \{X_{-a}'\delta\}^4 \vert \varpi ]\leq C$;}

\textit{((iii) With probability $1-\Delta_n$, uniformly over $u\in \UU$ and $a\in V$: $\En[K_\varpi(W)\{|X_{-a}'(\bar\beta_u-\beta_u)|+|X_{-a}'(\bar\beta_u-\beta_u)|^2\}(Z^a\delta)^2] \leq \delta_n \En[K_\varpi(W)f_u(X_{-a}'\delta)^2]$;}

\textit{(iv) For a fixed $q \geq 4 \lor (1+d_W)$, we have that: $\mathrm{diam}(\mathcal{W}) \leq n^{1/2q}$, $ \Ep[\max_{i\leq n}\|X_{iV}\|_\infty^q]^{1/q}/\mu_\mathcal{W} \leq M_n$, $M_n^2 \log^7(n|V|) \leq \delta_n n \underf_\UU^2\mu_\mathcal{W}^2$, $M_n^{4}\log(n|V|) \log n\leq \delta_n n \mu_\mathcal{W}$, $(L_f+L_K)^2M_n^2\log^2(|V|n)\leq \delta_n n \mu_\mathcal{W}^3\underf_\mathcal{U}^6$, $M_n^2s\log^{3/2} (n|V|) \leq \delta_n \underline{f}_{\mathcal{U}}(n\mu_\mathcal{W})^{1/2}$, $M_n s \sqrt{\log(n |V|)} \leq \delta_n (n \mu_\mathcal{W})^{1/2}$, and $s^3\log^5(n|V|)\leq \delta_n n \underline{f}_{\mathcal{U}}^2\mu_\mathcal{W}^2$.}

Condition P is a high-level condition. It allows to cover conditioning events $\varpi \in \mathcal{W}$ whose probability can decrease to zero (although slower than $n^{-1/4}$).

Next we derive our main results regarding the proposed estimator for the best linear predictor. These results are also new $\ell_1$-penalized quantile regression methods as it holds under possible misspecification of the conditional quantile function and hold uniformly over the indices $u\in \mathcal{U}$. The following theorem summarizes the result.

\begin{theorem}[{Uniform Rates of Convergence for $\mathcal{W}$-Conditional Penalized Quantile Regression under Misspecification}]\label{theorem:rateQRgraphP}
Under Condition P, we have that with probability at least $1-o(1)$, uniformly over $u=(a,\tau,\varpi) \in\mathcal{U}$,
$$ \|\hat\beta_u - \beta_u\| \lesssim  \sqrt{\frac{s(1+d_W)\log(|V|n)}{n \underline{f}_u\Pr(\varpi)}}.$$
\end{theorem}

The data-driven choice of penalty parameter helps diminish the regularization bias and also allow to obtain sparse estimators with provably rates of convergence (through thresholding). Moreover, the $u$ specific penalty parameter combined with the new analysis yields an adaptive rate of convergence to each $u\in \mathcal{U}$ unlike previous works.

\begin{remark}[Simultaneous Confidence Bands for Coefficients in PQGMs]
We note that in some applications we might be interested in constructing (simultaneous) confidence bands for the coefficients in PQGMs. In particular, this would include the cases practitioners are using a misspecified linear specification in a quantile regression model. Provided the conditional density function at $X_{-a}'\beta_u$ can be estimated, a version of Algorithm \ref{Alg:1prime} using the penalty parameters in Algorithm \ref{Alg:2prime} for the initial step can deliver such confidence regions via a multiplier bootstrap.
\end{remark}

\section{Application: International Financial Contagion and Systemic Risk}\label{sec:Empirical}

There is widespread disagreement about what finacial contagion entails, e.g. \cite{Forbes2002, diebold2015trans}. The existing measures of contagion are mainly based on linear correlation and can only account for certain types of risk network structure or do not provide inference procedures for the large scale networks estimated. This paper defines contagion occurs whenever the quantile partial correlation from  one country to another country is nonzero, i.e. the presence of edges in PQGM. This definition takes into account network spillover effects when identifying contagion and measuring systemic risk.   Here, the weight of an edge indicates the strength of contagion effects. The estimated contagion network taking into account global interconnectedness is important for Eurozone financial regulators to identify globally systemically important EU countries, or for global financial portofolio diversification.

We revisit the analysis of of international
financial contagion, \cite{Claessens2001}. We provide an alternative approach to the literature by visualizing tail interdependence via PQGM. As shown in Section \ref{sub: Network-CoVaR}, our framework naturally extends the systemic risk measure CoVaR taking into account tail network spillover effects, hence after learning PQGM from data, we identify systemically important countries using our new systemic risk measures. We can also provide inference on networks estimated. To simplify the visualization, we provide graphical visualization for the confidence intervals of $\Delta CoVaR$s in Figure \ref{Figure:-GermanDeltaCoVaR} of Section \ref{subsection:systemic risk}.

We focus on examining financial contagion through the volatility spillovers perspective, i.e. recovering volatility interconnectedness. \cite{Engle1993} reported that international stock markets are related through their volatilities
instead of returns. \cite{DieboldYilmaz2009} studied the return and volatility
spillovers of 19 countries and found differences in return and volatility
spillovers. \footnote{Modelling the time dependence in volatility is an important issue, although not the focus of this work, there is no uniform agreement on  whether there is dependence once taking into account the heteroskedasticity of market returns or shifts in standard errors, see \cite{starica2005nonstationarities}. }

We use average two-day equity index returns, \footnote{This is to control for the fact that markets in different countries are not open during the same hours. Results are robust to whether using two-day returns or using daily returns as in older versions of our work. Daily returns are also adjusted for weekends and holidays. } September 2009 to September 2013, from Morgan Stanley Capital International (MSCI).
The returns are all translated into dollar-equivalents as of September
6th 2013. \footnote{We calculate returns based on U.S. dollars since these were most frequently used in past work on contagion. }  We use absolute returns as a proxy for volatility. \footnote{This has historically been used in the literature. While we do recognize that there are many different methods calculating volatility measures, volatility measuring itself is a large research study area and is outside the scope of the current work.} We have
a total of 45 countries in our sample, there are 21 developed markets
(Australia, Austria, Belgium, Canada, Denmark, France, Germany, Hong
Kong, Ireland, Italy, Japan, Netherlands, New Zealand, Norway, Portugal,
Singapore, Spain, Sweden, Switzerland, the United Kingdom, the United
States), 21 emerging markets (Brazil, Chile, Mexico, Greece, Israel,
China, Colombia, Czech Republic, Egypt, Hungary, India, Indonesia,
Korea, Malaysia, Peru, Philippines, Poland, Russia, Taiwan, Thailand,
Turkey), and 3 frontier markets (Argentina, Morocco, Jordan).

\begin{figure}
\begin{minipage}{\textwidth}
\centering
\begin{subfigure}[b]{0.475\textwidth}
\centering
    {{\small     Median ($\tau=0.5$)}}
            	\vskip -1.5cm
    \includegraphics[width=\textwidth]{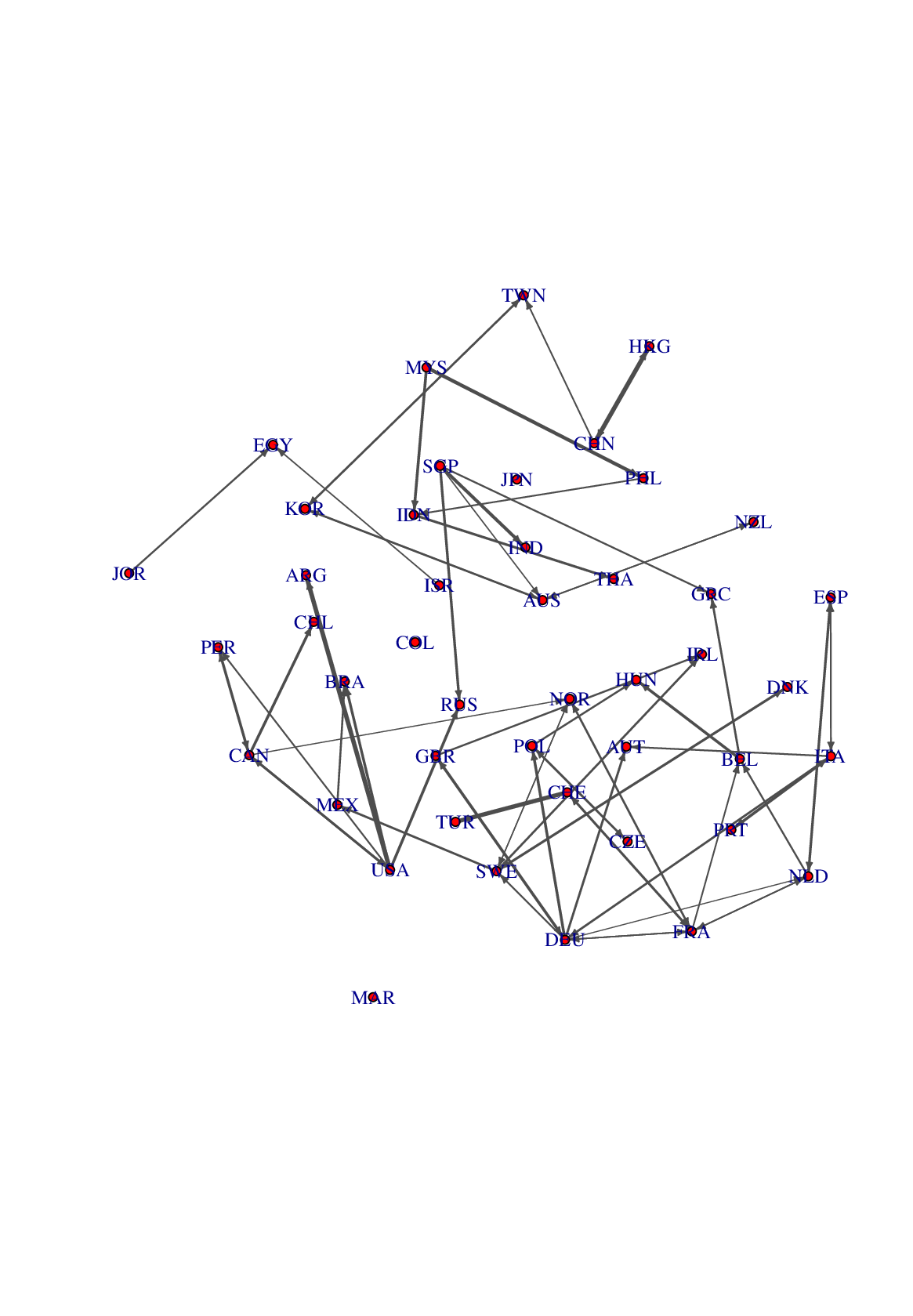}
 
\end{subfigure}
\hfill
 \begin{subfigure}[b]{0.475\textwidth}
 \centering
  	{{\small   Gaussian Graph}}
  	        	\vskip -1.5cm
	\includegraphics[width=\textwidth]{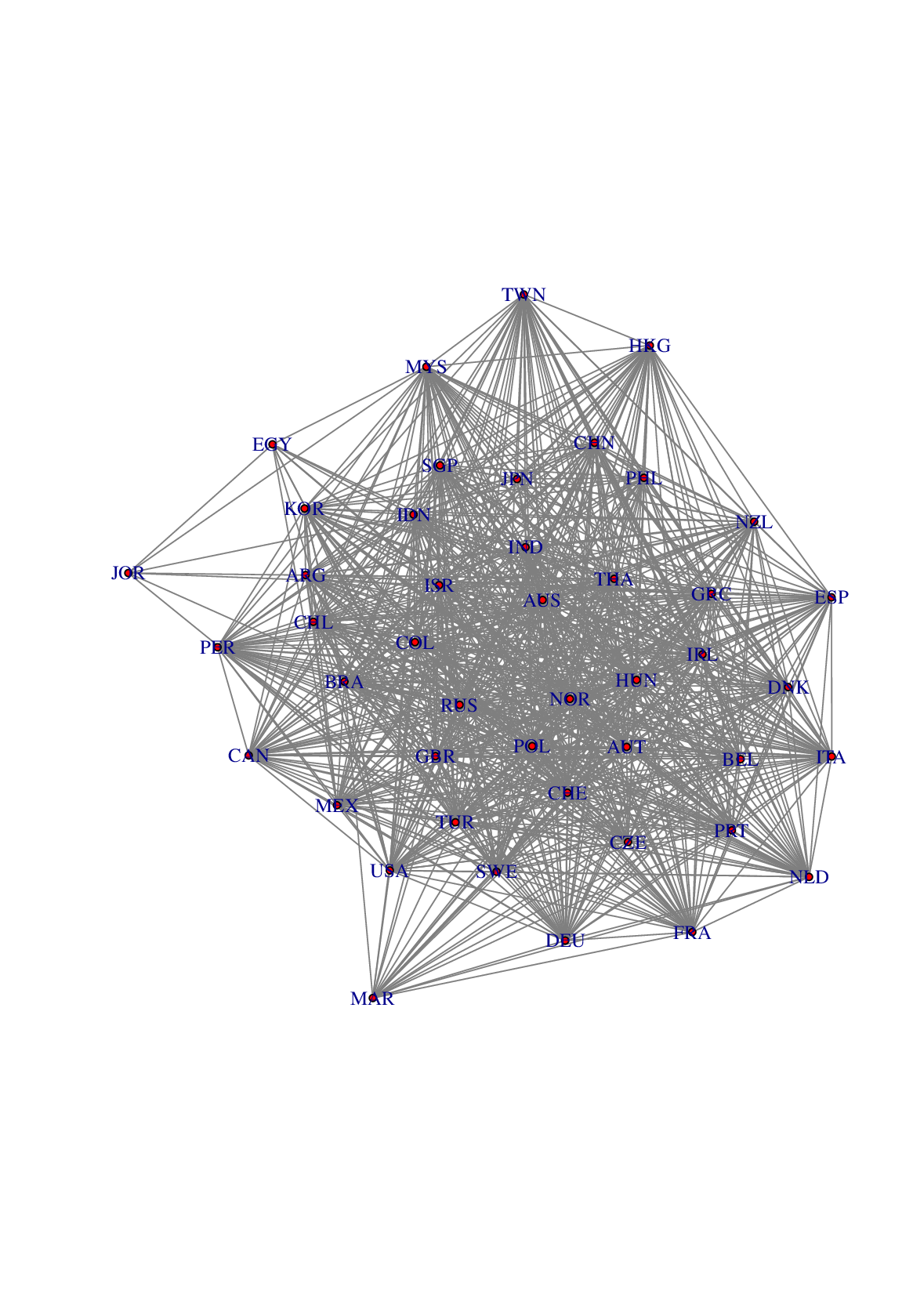}
 
\end{subfigure}
\vskip -1.5cm
 \begin{subfigure}[b]{0.475\textwidth}
 \centering
    	{{\small Low Tail ($\tau=0.1$)}}
    	        	\vskip -1.5cm
       \includegraphics[width=\textwidth]{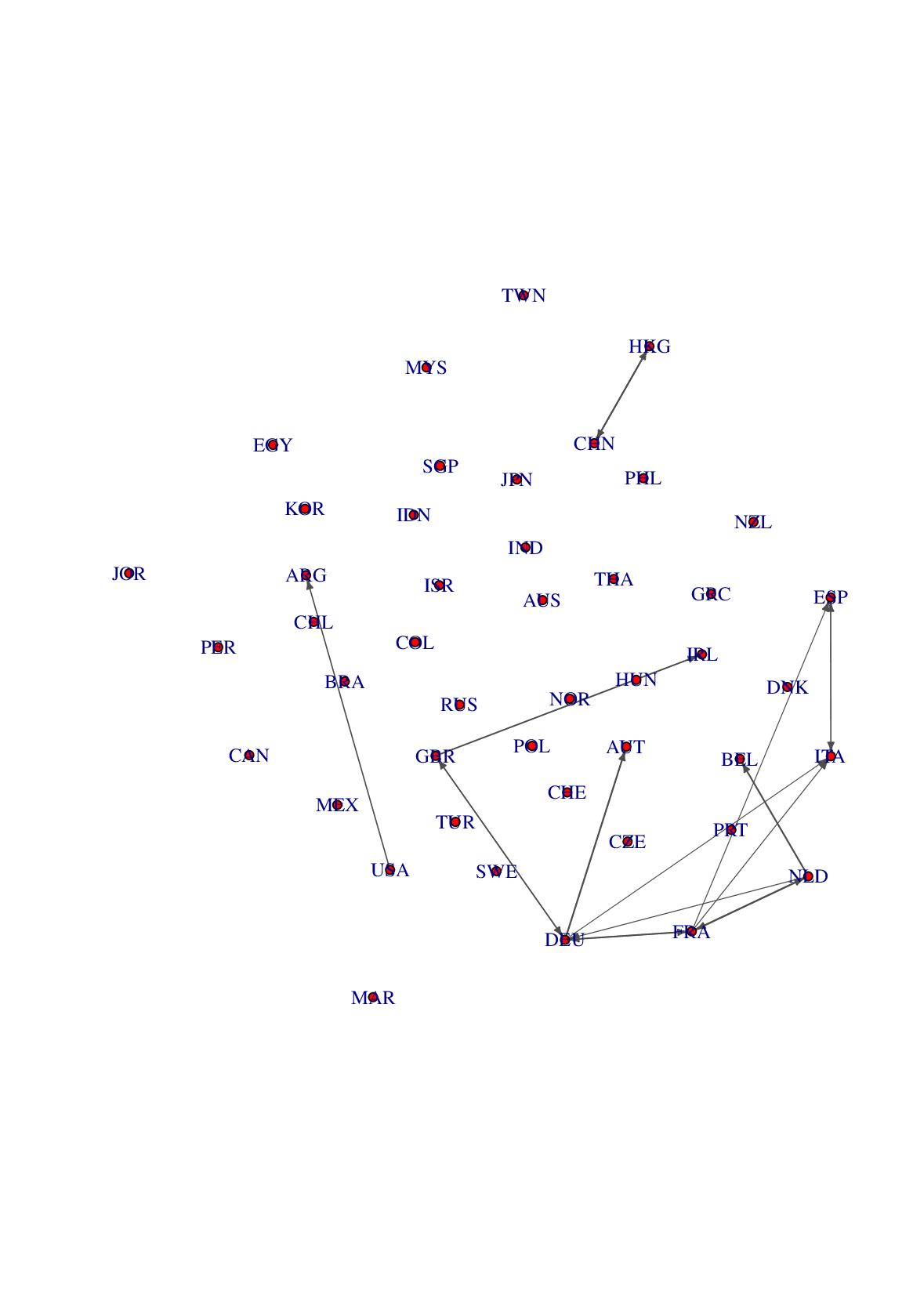}
 
\end{subfigure}
\quad
\begin{subfigure}[b]{0.475\textwidth}
\centering
	{{\small Up Tail ($\tau=0.9$) }}
	        	\vskip -1.5cm
	\includegraphics[width=\textwidth]{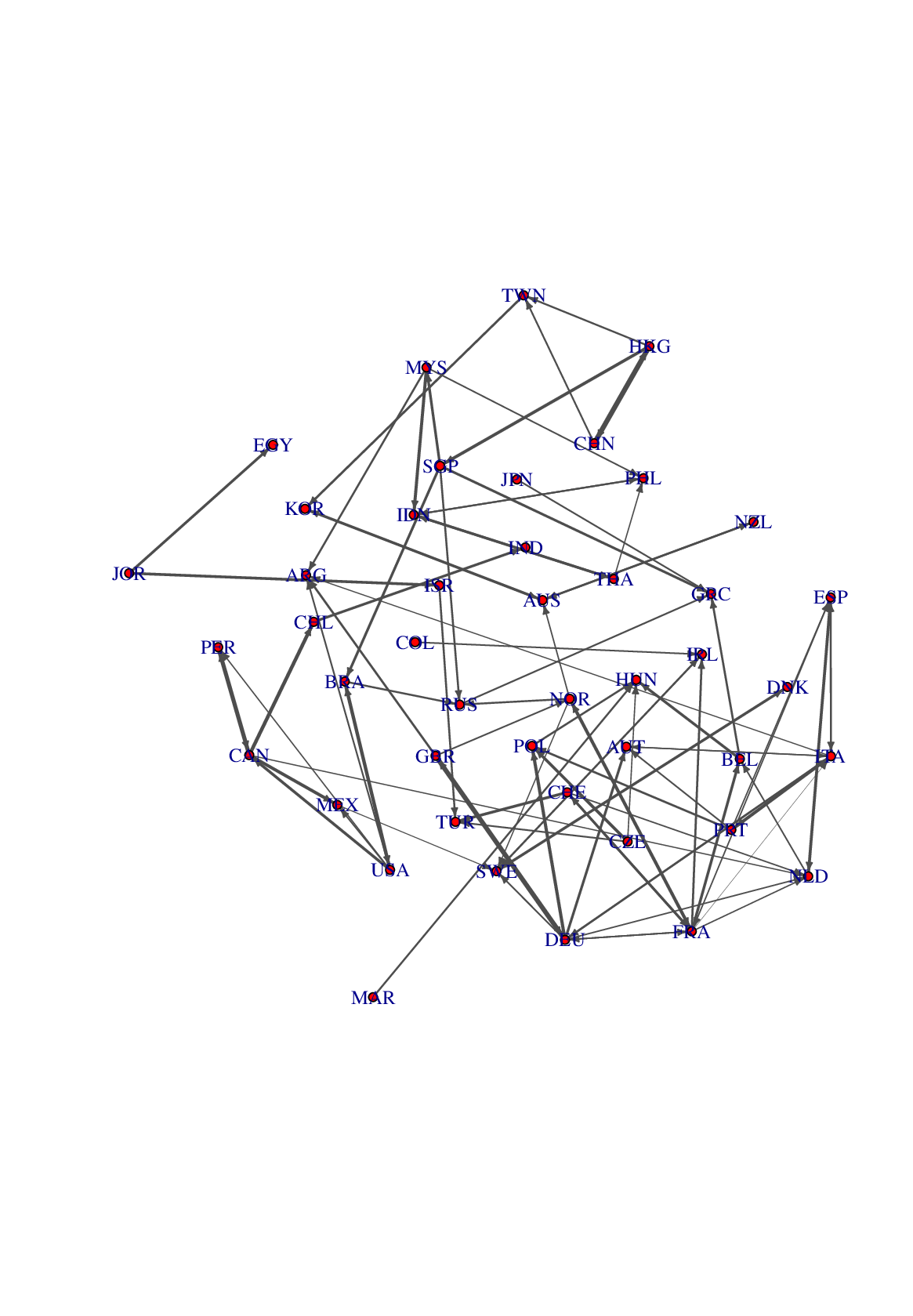}
      
\end{subfigure}
\vskip -1.8cm
     \caption[International Financial Contagion. Note: Lack of edges between nodes indicates that neither node's equity index volatility help predict the other node's volatility at $\tau$-th quantile. Arrows (or directed edges) mean that the source node helps explain the target node. These graphs show that the volatility transmission mechanism are asymmetric.]
    {\small International Financial Contagion.    Note: Lack of edges between nodes indicates that neither node's equity index volatility help predict the other node's volatility at $\tau$-th quantile. Arrows (or directed edges) mean that the source node's volatility helps predict the target node's volatility. These graphs show that the volatility transmission mechanism are asymmetric at different volatility quantiles (or tails).}
   \label{Figure-1:-International Contagion}
\end{minipage}
\end{figure}

\subsection{Contagion Networks.} Figure \ref{Figure-1:-International Contagion}   provides a full-sample analysis of global volatility network spillovers
at different tails. The networks are estimated via Algorithm \ref{Alg:2}.

We denote 10\% quantile as Low Tail, 50\% quantile
as Median, 90\% quantile as Up Tail. Results learnt from both PQGMs and GGM are presented. GGM or "Gaussian Graph" in Figure \ref{Figure-1:-International Contagion}  means the graph is estimated via graphical lasso (e.g., \cite{Friedman2008}), and the final graph is chosen by Extended Bayesian Information Criterion (ebic), see \cite{Foygel2010}.
Our purpose is to show the usefulness of PQGM in representing nonlinear
tail interdependence allowing for heteroscedasticity and to show that
PQGM can measure correlation asymmetry through looking at the tails
of the distribution (not specific to any model).

There are significant differences in the network structure in terms
of volatility spillovers when using PQGM and GGM. PQGM permits asymmetries in correlation dynamics, suited to investigate the presence of asymmetric responses. We find significant increase interdependence
at the up tail between the volatility series, that is we find downside correlations (high volatility) are much larger than
upside correlations (low volatility). This confirms findings in the finance
literature that financial markets become more interdependent during
high volatility periods.

We also find if two countries locate in the same geographic region,
with many similarities in terms of market structure and history, they
tend to be more closely connected (homophily effect as stated in network terminology),
while two economies locate in separate geographic regions are less
likely directly connected. In addition, we find among European Union member countries,
Germany appears to play a major role in the transmission of shocks
to others; while in Asia, Hong Kong, Thailand, and Singapore appear
to play major roles; and among all the north and south American countries,
Canada and US play major roles.

\subsection{Systemic Risk.}\label{subsection:systemic risk} With the estimated network, we can use different network statistics to measure the systemic risk contributions. Below we focus on the modified $\Delta CoVaR$ measure mentioned in Section \ref{sub: Network-CoVaR}.

Figure \ref{Figure:-GermanDeltaCoVaR} provides German's $\Delta CoVaR$s with $\tau = 0.9$ and their $90\%$ uniform confidence intervals obtained via Corollary \ref{theorem: general bs}. It reconfirms that France, Italy and UK contribute the most to German's $\Delta CoVaR$, means conditional on those countries being under distress relative to their median states, German would be affected the most. It is also interesting to find that other countries such as Netherlands can also have effects on German's $\Delta CoVaR$ although less statistically significant in terms of the magnitude of the effect.

\begin{figure}
\includegraphics[width=17cm, height=9cm ]{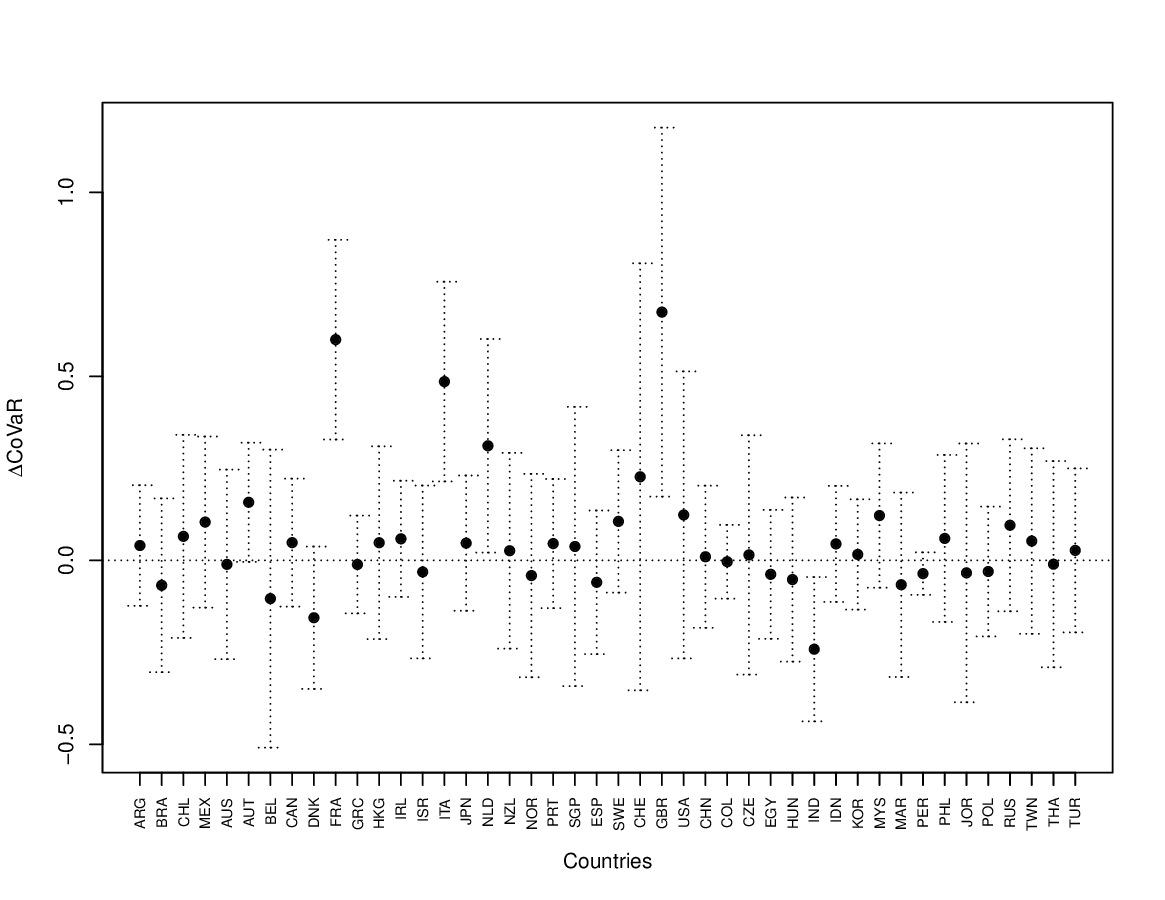}
\caption{German's $\Delta CoVaR$ and $90\%$ Confidence Intervals.}   \label{Figure:-GermanDeltaCoVaR}
\end{figure}

\begin{figure}
\includegraphics[width=17cm, height=16cm]{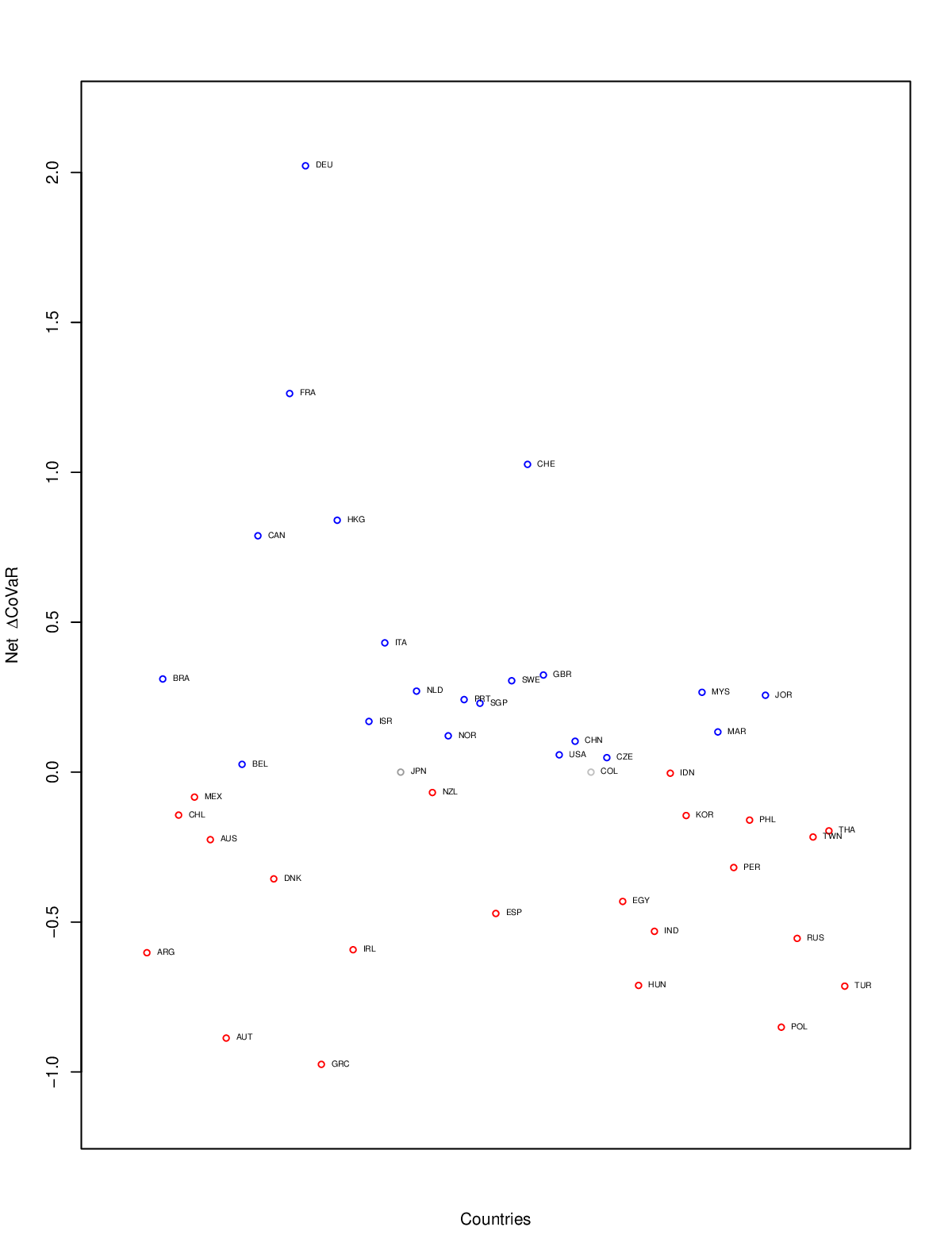}
\caption{Volatility Spillovers Net Contribution of Each Country} \label{Figure-2:-Net-DeltaCoVaR}
\end{figure}

 In addition, we present \textit{net}-$\Delta CoVaR$ discussed in Section \ref{sub: Network-CoVaR} with $\tau = 0.9$, i.e. the Up Tail, in Figure \ref{Figure-2:-Net-DeltaCoVaR} which shows that: globally, total
volatility spillovers from Germany and France to the
others are much larger than total volatility spillovers from the others
to them, and their \textit{net}-$\Delta CoVaR$ are positive. Both Greece
and Spain have negative \textit{net}-$\Delta CoVaR$.

\bibliographystyle{plain}
\bibliography{quantileGraphLiter}

\begin{appendix}
\newpage

\section{Implementation Details of Algorithms}\label{Appendix:Implementation}
This section provides details of the algorithms mentioned in Section \ref{sec:Estimator-and-Consistency}. Note, for the weighted-Lasso estimator, the choice of penalty level $\lambda := 1.1n^{-1/2}2\Phi^{-1}(1-\xi/N_n)$ and  penalty loading $\widehat \Gamma_\tau= \text{diag}[ \hat \Gamma_{\tau kk} , k \in [p]\backslash\{j\}]$
 is a diagonal matrix defined by the following procedure: (1) Compute the post-Lasso estimator $\widetilde \gamma_{a\tau}^{j}$ based on $\lambda$ and initial values $ \hat \Gamma_{\tau kk}  =   {\displaystyle \max_{i\leq n}} \|f_{ia\tau}Z^a_i\|_\infty \{\En[  |f_{a\tau}Z^a_k|^2]\}^{1/2}$. (2)  Compute the residuals $\widehat v_{i a\tau j} = f_{ia\tau}(Z_{ij}^a - Z_{i,-j}^a\widetilde \gamma_{a\tau}^{j})$ and update the loadings \begin{equation}\label{Parameter:Gamma} \hat \Gamma_{\tau kk}  =   \sqrt{\En[f_{a\tau}^2 |Z_{k}^a \widehat{v}_{a\tau j}|^2]}, \ k \in [p]\backslash\{j\} \end{equation}
and use them to recompute the post-Lasso estimator $\widetilde \gamma_{a\tau}^{j}$. In the case of Algorithm \ref{Alg:1} we can take $N_n=|V|p^3n^3$, in the case of Algorithm \ref{Alg:1prime} we take $N_n=|V|p^2\{pn^3\}^{1+d_W}$. Denote $\hat\sigma_{aj}^Z=\{\En[(Z_j^a)^2]\}^{1/2}$, $\hat\sigma_{aj}^X = \{\En[X_{-a,j}^2]\}^{1/2}$, $\hat\sigma_{a\varpi j}^Z= \{\En[K_\varpi(W)(Z^a_j)^2]\}^{1/2} $, and $\hat\sigma_{a\varpi j}^X = \{\En[K_\varpi(W)X_{-a,j}^2]\}^{1/2}$.

\begin{table}[ht]
\textbf{Detailed version of Algorithm \ref{Alg:1} (CIQGM)}\\
For each $a \in V$, $\tau \in \mathcal{T}$, and $j\in [p]$, perform the following:
\vspace{-0.1cm}\begin{enumerate}
\item Run Post-$\ell_1$-quantile regression of $X_a$ on   $Z^a$; keep fitted value $ Z_{-j}^a\widetilde \beta_{a\tau,-j}$,
$$\begin{array}{l}
\hat\beta_{a\tau}\in \arg\min_{\beta} \En[\rho_\tau(X_{a}-Z^a\beta)]+\lambda_{V\mathcal{T}} \sqrt{\tau(1-\tau)}\sum_{j=1}^p\hat\sigma_{aj}^Z|\beta_j|\\
\widetilde\beta_{a\tau}\in \arg\min_{\beta} \En[\rho_\tau(X_{a}-Z^a\beta)] \ : \ \beta_j=0 \ \ \mbox{if} \ |\hat\beta_{a\tau j}| \leq \lambda_{V\mathcal{T}}\sqrt{\tau(1-\tau)}/\hat \sigma_{aj}^Z.\\
\end{array}$$
\item Run  Post-Lasso of $f_{a\tau} Z_{j}^a$ on $f_{a\tau} Z_{-j}^a$; keep the residual $\widetilde v_i:=f_{ia\tau}\{Z_{ij}^a-Z_{i,-j}^a\widetilde\gamma_{a\tau}^j\}$,
$$\begin{array}{l}
\hat\gamma_{a\tau}^j\in \arg\min_\gamma\En[f_{a\tau}^2(Z_{j}^a-Z_{-j}^a\gamma)^2]+\lambda\|\widehat \Gamma_\tau\gamma\|_1\\
\widetilde\gamma_{a\tau}^j\in \arg\min_\gamma \En[f_{a\tau}^2(Z_{j}^a-Z_{-j}^a\gamma)^2]  \ : \ \supp(\gamma)\subseteq \supp(\hat\gamma_{a\tau}^j).\\
\end{array}$$
\item Run Instrumental Quantile Regression of $X_{a} - Z_{-j}^a\widetilde \beta_{a\tau,-j}$ on $Z_{j}^a$
using $\widetilde v$ as the instrument for $Z_{j}^a$,
$$\begin{array}{l}
\displaystyle \check \beta_{a\tau,j}\in \arg\min_{\alpha \in \mathcal{A}_{a \tau j}} \frac{\{\En[(1\{X_{a} \leq Z_{j}^a\alpha+Z_{-j}^a\widetilde\beta_{a\tau,-j}\}-\tau)\widetilde v]\}^2}{\En[(1\{X_{a} \leq Z_{j}^a\alpha+Z_{-j}^a\widetilde\beta_{a\tau,-j}\}-\tau)^2\widetilde v^2]},
\end{array}$$
with $\mathcal{A}_{a\tau j} = \{ \alpha \in \RR :   |\alpha - \widetilde \beta_{a\tau j}| \leq 10/\{\hat \sigma_{aj}^Z \log n\} \}$.
\end{enumerate}
\end{table}

\smallskip

\begin{table}[ht]
\textbf{Detailed version of Algorithm \ref{Alg:2} (PQGM)}\\
For each $a \in V$,  and $\tau \in \mathcal{T}$, perform the following:
\vspace{-0.1cm}\begin{enumerate}
\item Run Post-$\ell_1$-quantile regression of $X_a$ on $X_{-a}$,
$$\begin{array}{l}
\hat\beta_{a\tau}\in \arg\min_{\beta} \En[\rho_\tau(X_{a}-X_{-a}'\beta)]+\lambda_0\sum_{j\in [d]}\hat\sigma_{aj}^X|\beta_{j}| \\
\widetilde\beta_{a\tau}\in \arg\min_{\beta} \En[\rho_\tau(X_{a}-X_{-a}'\beta)] \ : \ \beta_j=0 \ \ \mbox{if} \ |\hat\beta_{a\tau j}| \leq \lambda_{0}/\hat \sigma_{aj}^X.\\
\end{array}
$$

\item Set $\hat \varepsilon_{ia\tau}= 1\{ X_{ia} \leq X_{i,-a}'\tilde{\beta}_{a\tau}\}-\tau$ for $i\in[n]$. Compute the penalty level $\bar{\lambda}_{V\mathcal{T}}$ via (\ref{barlambda}).

\item Run Post-$\ell_1$-quantile regression of $X_a$ on  $X_{-a}$,
$$\begin{array}{l}
\hat\beta_{a\tau} \in \arg\min_{\beta} \En[\rho_\tau(X_{a}-X_{-a}'\beta)]+\bar{\lambda}_{V\mathcal{T}}\sum_{j\in [d]}\{\En[\hat\varepsilon_{a\tau}^2X_{-a,j}^2]\}^{1/2}|\beta_j| \\
\check\beta_{a\tau}\in \arg\min_{\beta} \En[\rho_\tau(X_{a}-X_{-a}'\beta)] \ : \ \beta_j=0 \ \ \mbox{if} \ |\hat\beta_{a\tau j}| \leq \bar{\lambda}_{V\mathcal{T}}/\{\En[\hat\varepsilon_{a\tau}^2X_{-a,j}^2]\}^{1/2}.\\
\end{array}$$
\end{enumerate}
\end{table}

\begin{table}[ht]
\textbf{Detailed version of Algorithm \ref{Alg:1prime} ($\mathcal{W}$-Conditional CIQGM)}\\
For each $u = (a,\tau,\varpi) \in \UU =  V\times \mathcal{T}\times \mathcal{W}$, and $j\in [p]$, perform the following:
\vspace{-0.1cm}\begin{enumerate}
\item Run Post-$\ell_1$-quantile regression of $X_a$ on  $Z^a$; keep fitted value $ Z^a_{-j}\widetilde \beta_{u,-j}$,
$$\begin{array}{l}
\hat\beta_u\in \arg\min_{\beta} \En[K_\varpi(W)\rho_\tau(X_{a}-Z^a\beta)]+ \lambda_{u}\sum_{j=1}^{p}\hat{\sigma}_{a\varpi j}^{Z}|\beta_{j}|\\
\widetilde\beta_u\in \arg\min_{\beta} \En[K_\varpi(W)\rho_\tau(X_{a}-Z^a\beta)] \ : \ \beta_j=0 \ \ \mbox{if} \ |\hat\beta_{uj}| \leq \lambda_u/\hat\sigma_{a\varpi j}^Z.\\
\end{array}$$
\item Run  Post-Lasso of $f_u Z_j^a$ on $f_u Z_{-j}^a$; keep the residual $\widetilde v:=f_u(Z^a_j-Z^a_{-j}\widetilde\gamma^j_u)$,
$$\begin{array}{l}
\hat\gamma^j_u\in \arg\min_\theta \En[K_\varpi(W)f_u^2(Z^a_j-Z^a_{-j}\gamma)^2]+\lambda\|\widehat \Gamma_u\gamma\|_1\\
\widetilde\gamma^j_u\in \arg\min_\gamma \En[K_\varpi(W)f_u^2(Z^a_j-Z^a_{-j}\gamma)^2]  \ : \ \supp(\gamma)\subseteq \supp(\hat\gamma^j_u).\\
\end{array}$$
\item Run Instrumental Quantile Regression of $X_{a} - Z^a_{-j}\widetilde \beta_{u,-j}$ on $Z^a_{j}$
using $\widetilde v$ as the instrument,
$$\begin{array}{l}
\displaystyle \check \beta_{uj}\in \arg\min_{\alpha \in \mathcal{A}_{uj}}  \frac{\{\En[K_\varpi(W)(1\{X_{a} \leq Z^a_j\alpha+Z^a_{-j}\widetilde\beta_{u,-j}\}-\tau)\widetilde v]\}^2}{\En[K_\varpi(W)(1\{X_{a} \leq Z^a_j\alpha+Z^a_{-j}\widetilde\beta_{u,-j}\}-\tau)^2\widetilde v^2]}
\end{array}$$
where $\mathcal{A}_{uj}:= \{ \alpha \in \RR : |\alpha-\widetilde\beta_{uj}|\leq 10/\{\hat\sigma_{a\varpi j}^Z \log n\}  \}$.
\end{enumerate}
\end{table}

\begin{table}[ht]
\textbf{Detailed version of Algorithm \ref{Alg:2prime} ($\mathcal{W}$-Conditional PQGM)}\\
For each $u = (a,\tau,\varpi) \in \UU =  V\times \mathcal{T}\times \mathcal{W}$ perform the following:
\vspace{-0.1cm}\begin{enumerate}
\item Run Post-$\ell_1$-quantile regression of $X_a$ on $X_{-a}$,
$$\begin{array}{l}
\hat\beta_u\in \arg\min_{\beta} \En[K_\varpi(W)\rho_\tau(X_{a}-X_{-a}'\beta)]+\lambda_{0\mathcal{W}}\sum_{j\in[d]}\hat{\sigma}_{a\varpi j}^{X}|\beta_{j}|\\
\widetilde\beta_u\in \arg\min_{\beta} \En[K_\varpi(W)\rho_\tau(X_{a}-X_{-a}'\beta)] :  \ \beta_j=0 \ \ \mbox{if} \ |\hat\beta_{uj}| \leq \lambda_{0\mathcal{W}}/\hat\sigma_{a\varpi j}^X.  \\
\end{array}$$
\item Set $\hat\varepsilon_{iu}=1\{X_{ia} \leq X_{i,-a}'\widetilde\beta_{u}\}-\tau$ for $i\in[n]$, compute $\bar{\lambda}_{V\mathcal{T}\mathcal{W}}$ via (\ref{Def:NewLambda2}).\\
\item Run Post-$\ell_1$-quantile regression of $X_a$ on  $X_{-a}$,
$$\begin{array}{l}
\hat\beta_u\in \arg\min_{\beta} \En[K_\varpi(W)\rho_\tau(X_{a}-X_{-a}'\beta)]+\bar{\lambda}_{V\mathcal{T}\mathcal{W}} \sum_{j\in[d]}\{\En[K_\varpi(W)\hat\varepsilon_u^2X_{-a,j}^2]\}^{1/2}|\beta_j|.\\
\check\beta_{u}\in \arg\min_{\beta} \En[K_\varpi(W)\rho_\tau(X_{a}-X_{-a}'\beta)] \ : \ \beta_j=0 \ \ \mbox{if} \ |\hat\beta_{uj}| \leq \bar{\lambda}_{V\mathcal{T}\mathcal{W}}/\{\En[K_\varpi(W)\hat\varepsilon_u^2X_{-a,j}^2]\}^{1/2}.\\
\end{array}$$
\end{enumerate}
\end{table}

\section{Simulations of Quantile Graphical Models}\label{sec:Simulation}

In this section, we perform numerical examples to illustrate the
performance of the estimators proposed for QGMs. We will consider several different
designs. In order to compare with other proposals we will consider both
Gaussian and non-Gaussian examples.

\subsection{Isotropic Non-Gaussian Example}
In general, the equivalence between a zero in the inverse covariance matrix and
a pair of conditional independent variables will break down for non-gaussian
distributions. The nonparanormal graphical models extends Gaussian Graphical Models
to Semiparametric Gaussian Copula models by transforming the variables
with smooth functions. We illustrate the applicability of CIQGM in representing
the conditional independence structure of a set of variables when the random variables
are not even jointly nonparanormal.

Consider i.i.d. copies of an $d$-dimensional random vector $\tilde{X}_V = (\tilde{X}_{1},\ldots, \tilde{X}_{d-1},\tilde{X}_{d})$
from the following multivariate normal distribution,
 $\tilde{X}_V \sim N(0,I_{d\times d})$,
where $I_{d\times d}$ is the identity matrix. Further, we generate
\begin{equation}
X_d  =  -\mbox{\ensuremath{\sqrt{\frac{2}{3\pi-2}}}}+\mbox{\ensuremath{\sqrt{\frac{\pi}{3\pi-2}}}} \tilde{X}_{d-1}^{2}\vert \tilde{X}_{d}\vert.\label{eq: counter location scale}
\end{equation}
It follows that $\mathrm{E}[X_d]=\sqrt{\frac{\pi}{3\pi-2}}(\mathrm{E}[\vert \tilde{X}_{d}\vert]-\sqrt{2/\pi})=0$
and $\mathrm{Var}(X_d)=\frac{\pi}{3\pi-2}(\mathrm{E}[\tilde{X}_{d}^{2}\cdot \tilde{X}_{d-1}^{4}]-\frac{2}{\pi})=1$.
In addition, equation (\ref{eq: counter location scale}) is a location-scale-shift
model in which the conditional median of the response is zero while
quantile functions other than the median are nonzero. We define
vector $X_V$ as
$$
X_V=(X_d, \tilde{X}_{1},...,\tilde{X}_{d-1})'.
$$
In this new set of variables, only $X_d$ and $\tilde{X}_{d-1}$ (i.e., node $1$ and $15$, when $d=15$) are not conditionally independent. Nonetheless, the
covariance matrix of $X_V$ is still $I_{d\times d}$.%
 
Next we consider an example with $n=300$ and
$d=15$. We show graphs, in Figure \ref{Figure-QGM-GGM} and \ref{Figure-Tiger1-2}, estimated by both CIQGM(s) and GGMs in this non-Gaussian setting.

\begin{figure}[ht]
\begin{minipage}{0.9\textwidth}
\centering

\begin{subfigure}[b]{0.3\textwidth}
     \caption[]%
        {{\footnotesize CIQGM(0.2)}}
 	\centering
	\includegraphics[width=\textwidth]{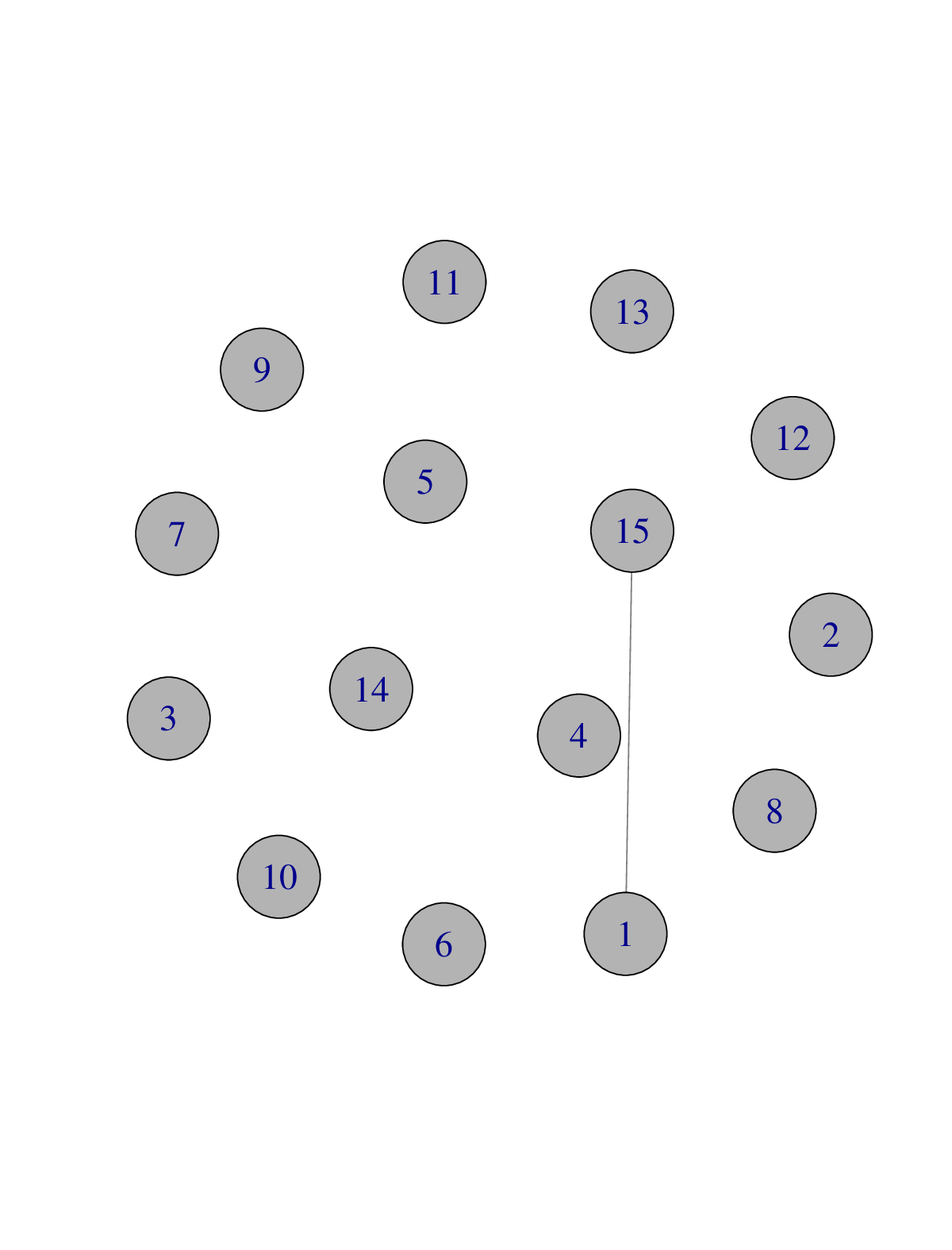}
\end{subfigure}
 \hfill
\begin{subfigure}[b]{0.3\textwidth}
   \caption[]%
        {{\footnotesize CIQGM(0.5)}}
 	\centering
	\includegraphics[width=\textwidth]{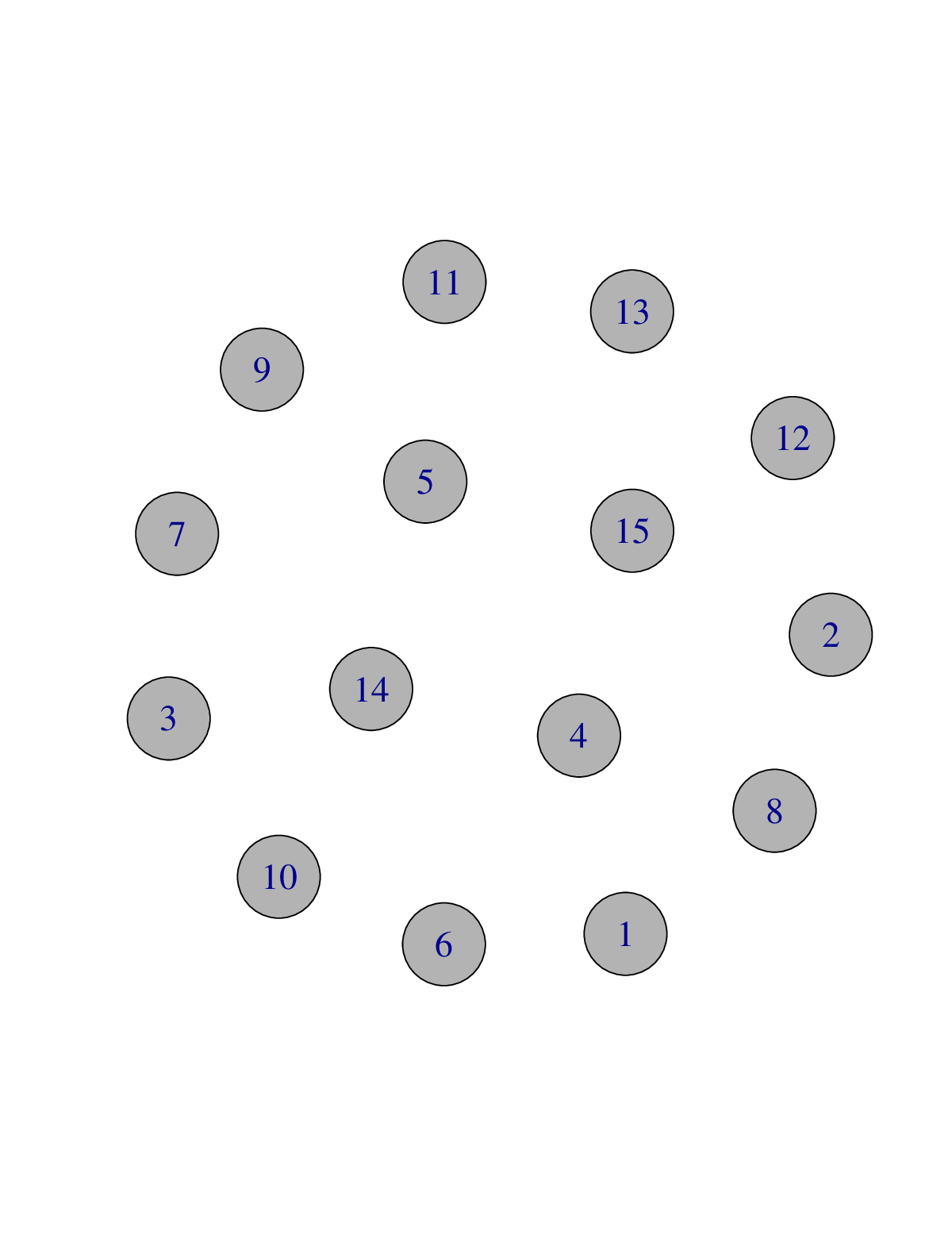}
\end{subfigure}
\hfill
\begin{subfigure}[b]{0.3\textwidth}
   \caption[]%
        {{\footnotesize CIQGM(0.8)}}
 	\centering
	\includegraphics[width=\textwidth]{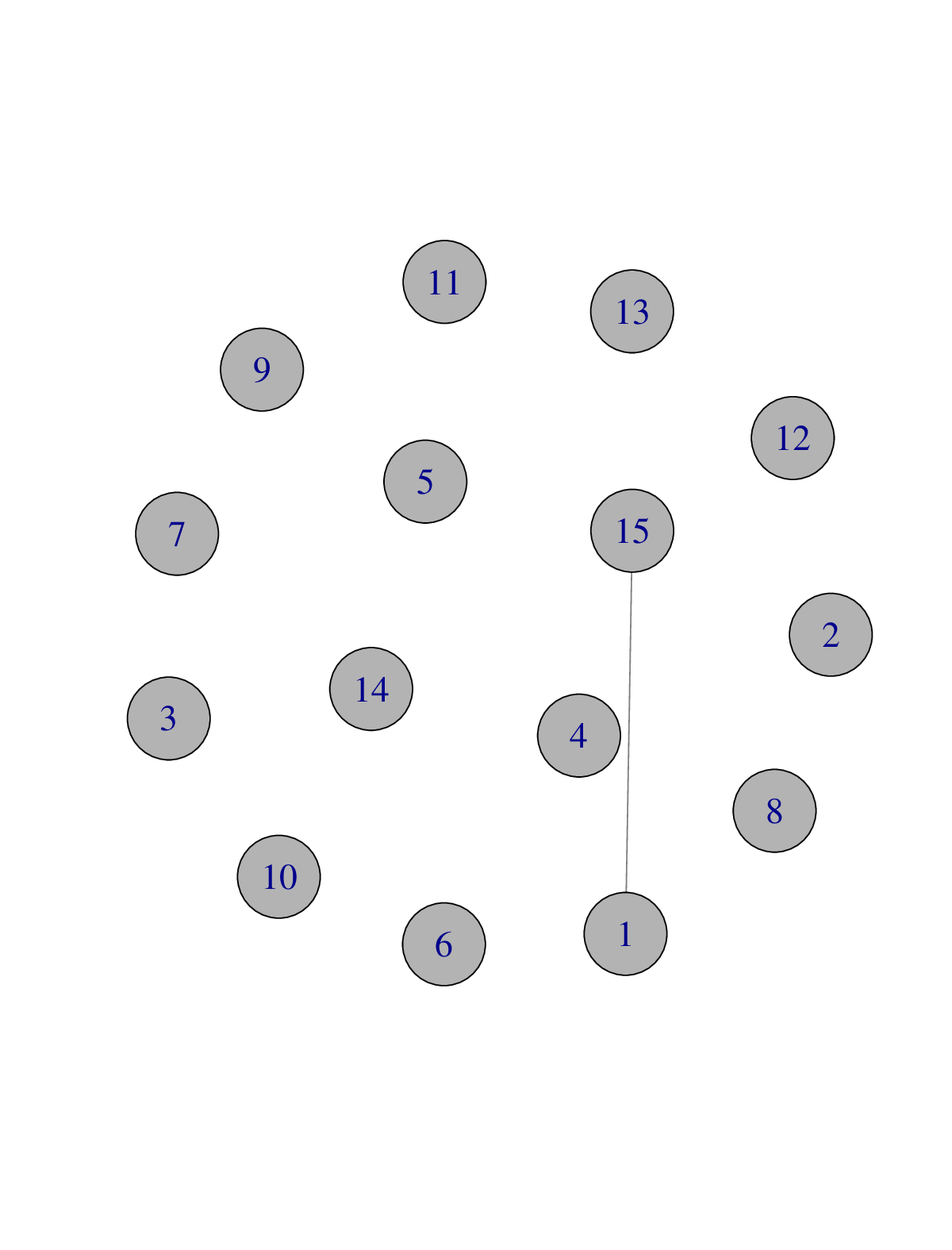}
\end{subfigure}

\vskip\baselineskip
\begin{subfigure}[b]{0.3\textwidth}
     \caption[]%
        {{\footnotesize CIQGM-Union}}
 	\centering
	\includegraphics[width=\textwidth]{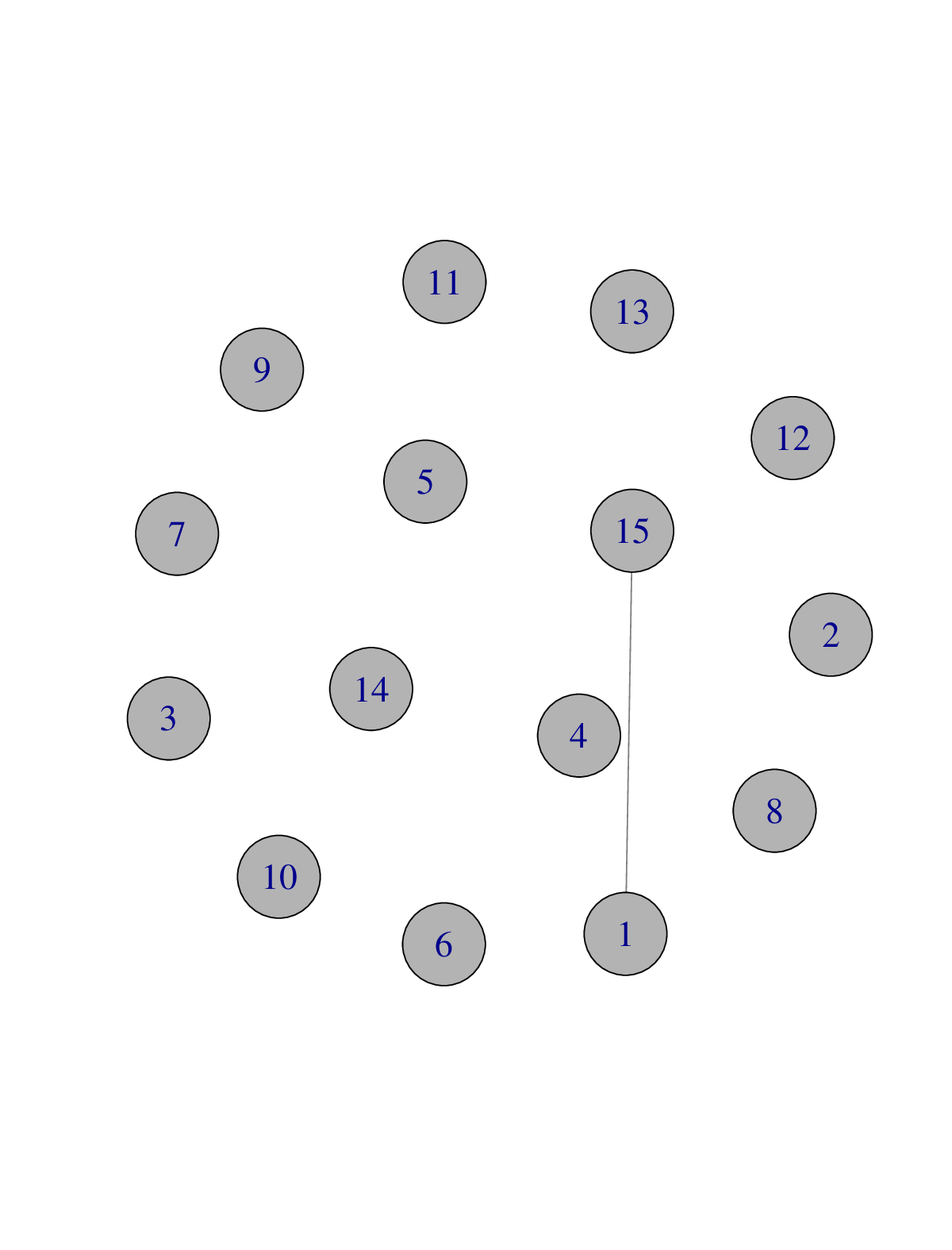}
\end{subfigure}
 \hfill
\begin{subfigure}[b]{0.3\textwidth}
   \caption[]%
        {{\footnotesize Gaussian}}
 	\centering
	\includegraphics[width=\textwidth]{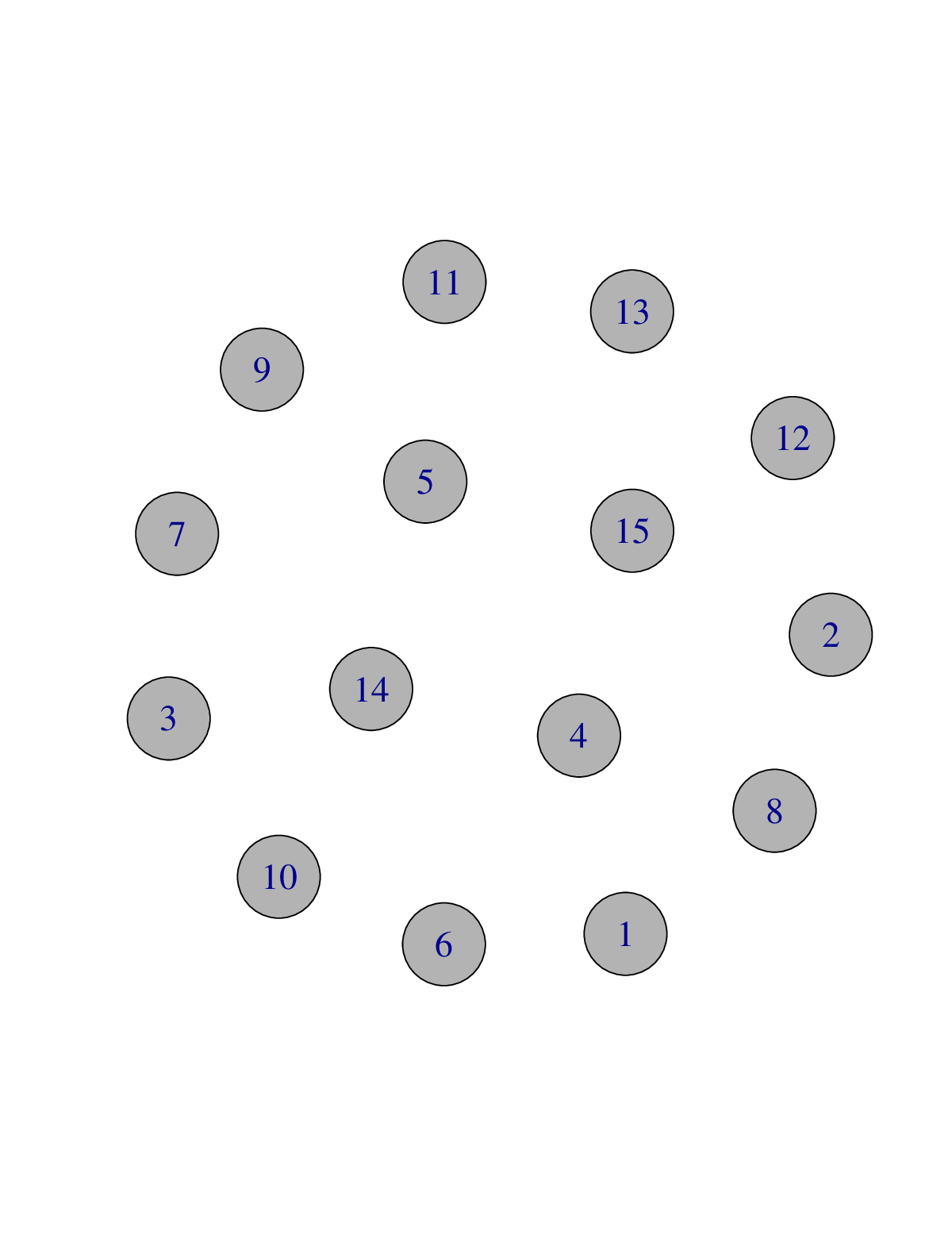}
\end{subfigure}
\hfill
\begin{subfigure}[b]{0.3\textwidth}
    \caption[]%
        {{\footnotesize Nonparanormal}}
 	\centering
	\includegraphics[width=\textwidth]{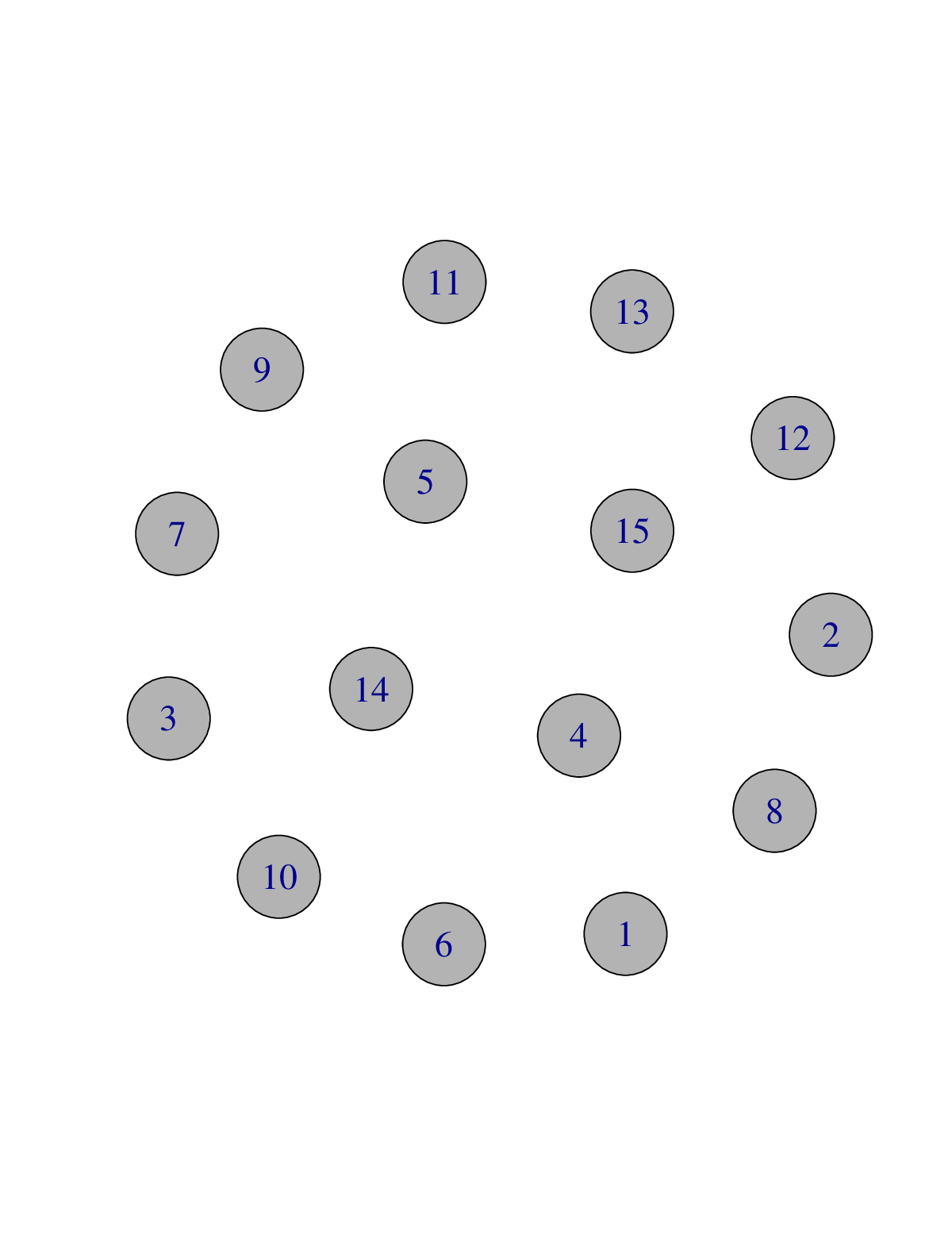}
\end{subfigure}

 \caption[ ]
        {\small QGM and GGM} \label{Figure-QGM-GGM}
\end{minipage}
\end{figure}

In Figure \ref{Figure-QGM-GGM}, Gaussian means the graph is estimated by using graphical lasso without any transformation
of $X_V$, and the final graph is chosen by Extended Bayesian
Information Criterion (ebic), see \cite{Foygel2010}. Nonparanormal means the graph
is estimated using graphical lasso (likelihood based approach) with nonparanormal transformation
of $X_V$, see \cite{Liu2009}, and again the final graph is chosen
by ebic. Both graphs are estimated using R-package \textbf{huge}.

In Figure \ref{Figure-Tiger1-2}, as a robustness check, we also compare results produced by CIQGM with those produced by neighborhood
selection methods (pseudo-likelihood approach), e.g. TIGER of \cite{Liu2012b} in R-package \textbf{flare}
the left graph is when choosing the turning parameter to be $\sqrt{\frac{\log d}{n}}$
while the right graph is when choosing the tuning parameter to be
$2\sqrt{\frac{\log d}{n}}$. Throughout, we use Tiger2
represent TIGER with penalty level $2\sqrt{\frac{\log d}{n}}$. As expected, GGM cannot detect the correct dependence structure when
the joint distribution is non-Gaussian while CIQGM can still represent
the right conditional independence structure.

\begin{figure}[ht]
\begin{minipage}{0.8\textwidth}
\centering

\begin{subfigure}[b]{0.4\textwidth}
     \caption[]%
        {{\footnotesize Tiger1}}
 	\centering
	\includegraphics[width=\textwidth]{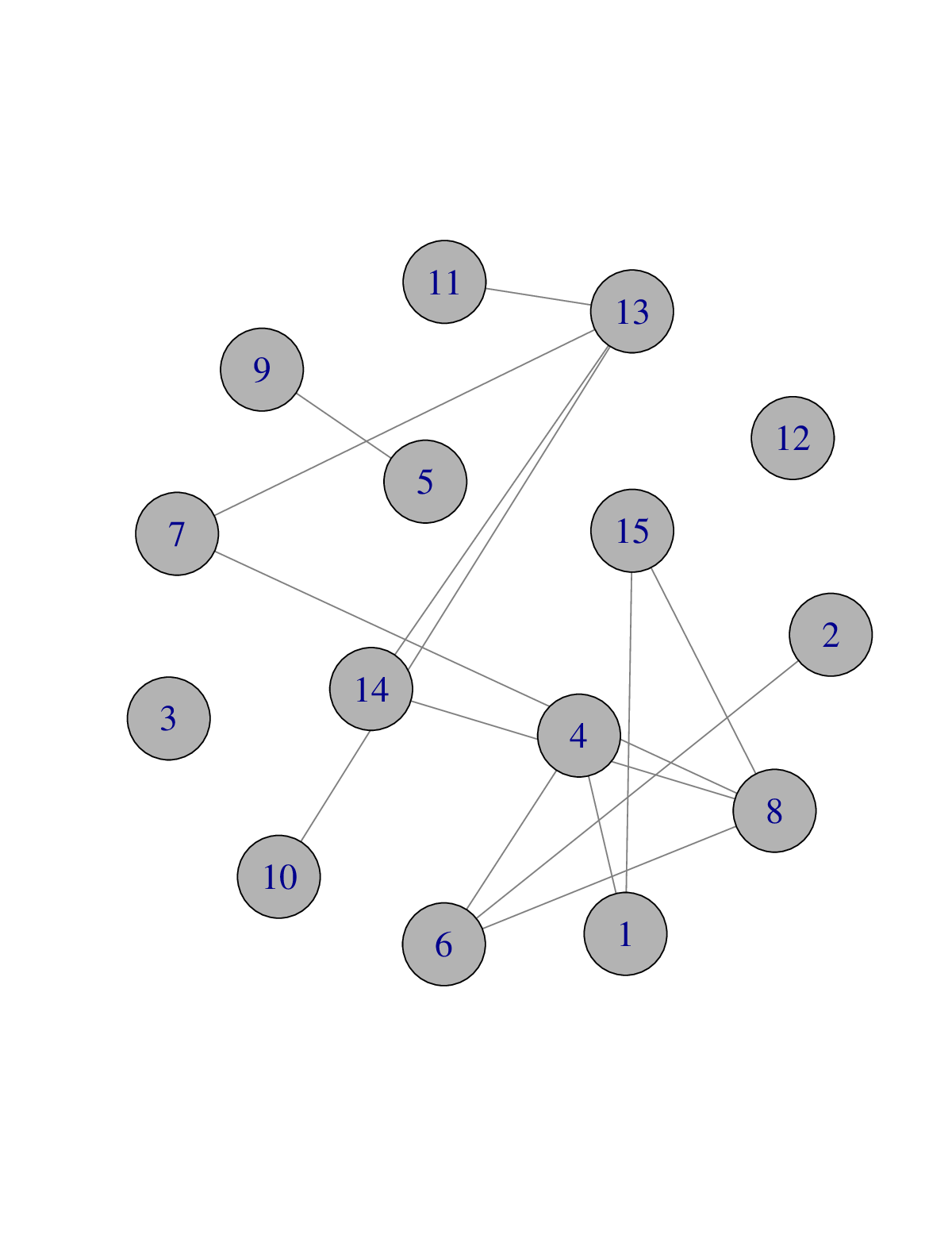}
\end{subfigure}
 \hfill
\begin{subfigure}[b]{0.4\textwidth}
   \caption[]%
        {{\footnotesize Tiger2}}
 	\centering
	\includegraphics[width=\textwidth]{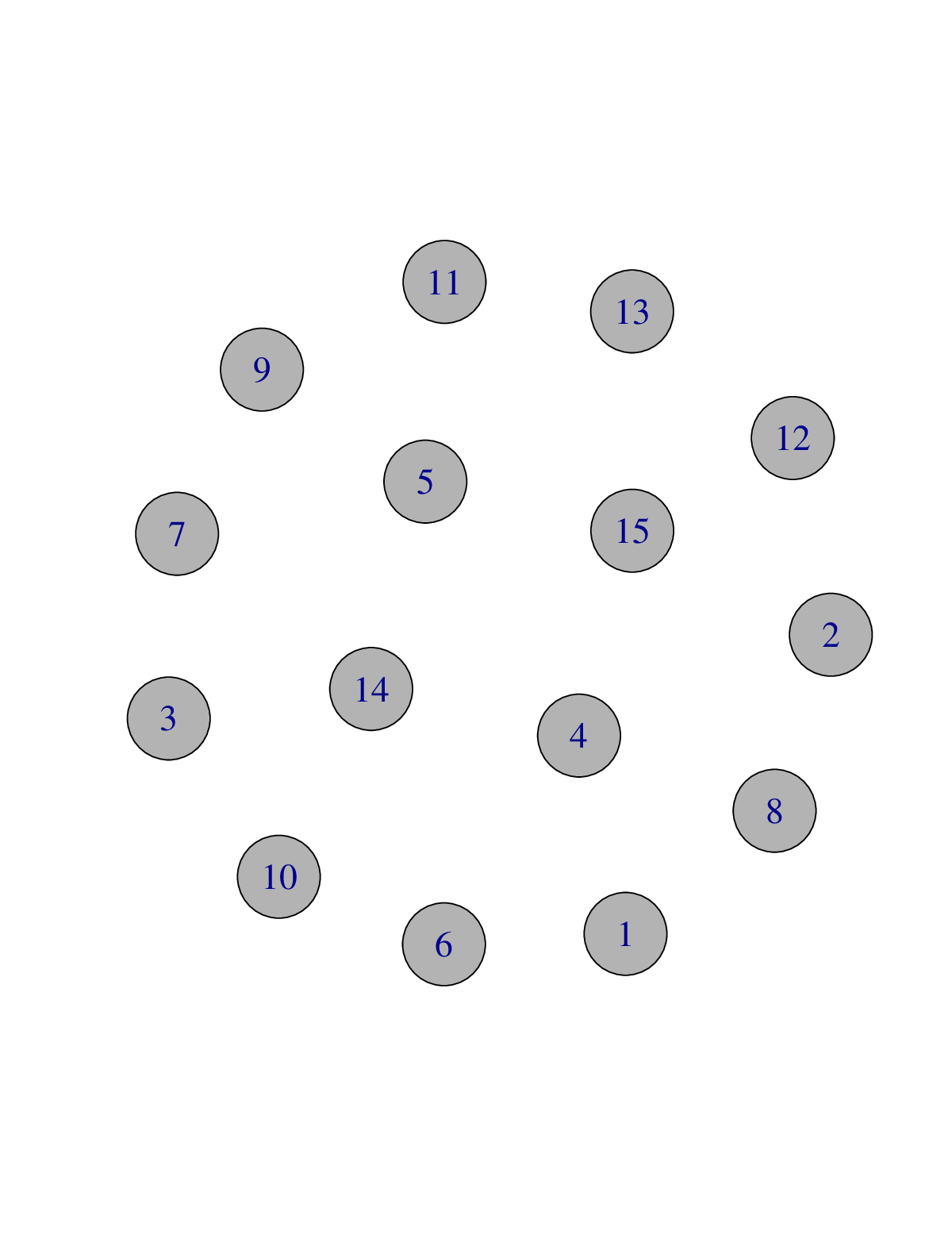}
\end{subfigure}

 \caption[ ]
        {\small TIGER} \label{Figure-Tiger1-2}
\end{minipage}
\end{figure}

\subsection{Gaussian Examples}

\subsubsection{Graph Recovery}

In this subsection, we start with comparing the numerical performance of QGM and other
methods, e.g. TIGER of \cite{Liu2012b} and graphical lasso algorithm
(Glasso) of \cite{Friedman2008}, in graph recovery using simulated datasets with different pairs of $(n,d)$. We start with one simulation for illustration purpose (the results are summarized in Figure \ref{Figure-3:-One-Simulation}), and then we show the performance of QGM through estimated degree distribution with 100 simulations (the results are summarized in Figure \ref{Figure-4:-Degree-of-PQGM}).

We mainly consider the Hub
graph, as mentioned in \cite{Liu2012b}, which also corresponds to
the star network mentioned in \cite{Acemoglu2010,Acemoglu2013}. In line with \cite{Liu2012b}, we generate a $d$-dimensional sparse
graph $G^I=(V,E^I)$ represents the conditional independence structure
between the variables. In our simulations, we consider 12 settings
to compare these methods: (A) $n=200$, $d=10$; (B) $n=200$,
$d=20$; (C) $n=200$, $d=40$; (D) $n=400$, $d=10$; (E)
$n=400$, $d=20$; (F) $n=400$, $d=40$;
(G) $n=200$, $d=100$; (H) $n=200$, $d=200$; (I) $n=200$,
$d=400$; (J) $n=400$, $d=100$; (K) $n=400$, $d=200$; (L)
$n=400$, $d=400$. We adopt the following model for generating undirected
graphs and precision matrices.

\textbf{Hub graph.} The $d$ nodes are evenly partitioned into $d/20$
(or $d/10$ when $d<20$) disjoint groups with each group contains
$20$ (or $10$) nodes. Within each group, one node is selected as
the hub and we add edges between the hub and the other $19$ (or $9$)
nodes in that group. For example, the resulting graph has $190$ edges
when $d=200$ and $380$ edges when $d=400$. Once the graph is obtained,
we generate an adjacency matrix $E^I$ by setting the nonzero
off-diagonal elements to be 0.3 and the diagonal elements to be 0.
We calculate its smallest eigenvalue $\Lambda_{\min}(E^I)$.
The precision matrix is constructed as
\begin{equation}
\Theta=\mathbf{D}[E^I+(\vert\Lambda_{\min}(E^I)\vert+0.2)\cdot I_{d\times d}]\mathbf{D}\label{eq:precision}
\end{equation}
where $\mathbf{D}\in\mathbb{R}^{d\times d}$ is a diagonal matrix
with $\mathbf{D}_{jj}=1$ for $j=1,...,d/2$ and $\mathbf{D}_{jj}=1.5$
for $j=d/2+1,...,d$. The covariance matrix $\Sigma:=\Theta^{-1}$is
then computed to generate the multivariate normal data: $X_{1},....,X_{d}\sim N(0,\Sigma)$.
Below we provide simulation results using different estimators: PQGM\footnote{Given the graphs are generated from multivariate Gaussian distribution we can use PQGM to simplify the computation.}, TIGER
and Glasso. We start with one simulation as an illustration:

\begin{figure}[ht]
\begin{minipage}{0.95\textwidth}
\centering
\begin{subfigure}[b]{0.3\textwidth}
   \caption[]%
        {{\footnotesize $n = 200, d = 10$}}
	\centering
   	 \includegraphics[width=\textwidth]{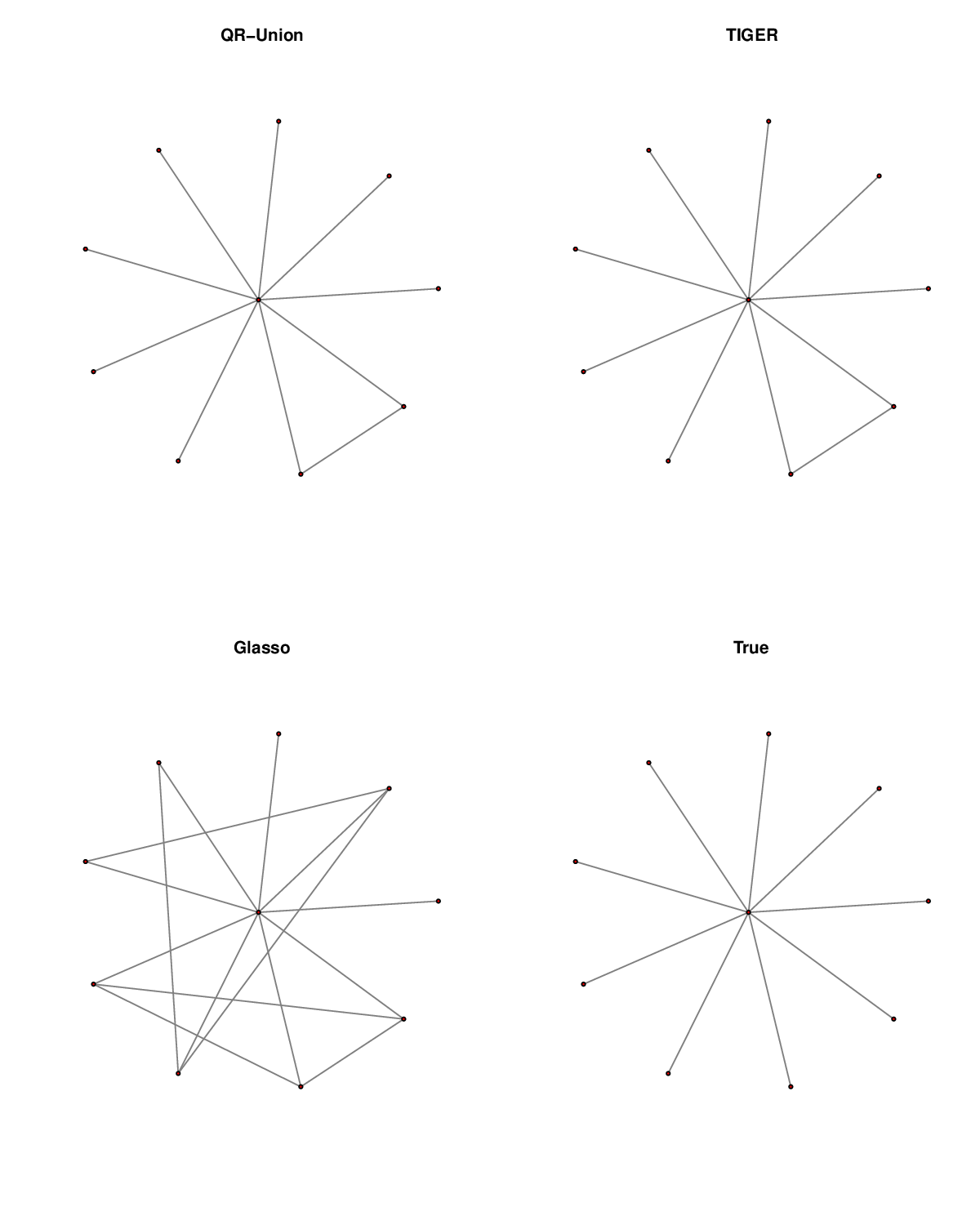}
\end{subfigure}
\hfill
\begin{subfigure}[b]{0.3\textwidth}
    \caption[]%
        {{\footnotesize $n = 200, d = 20$}}
 	\centering
	\includegraphics[width=\textwidth]{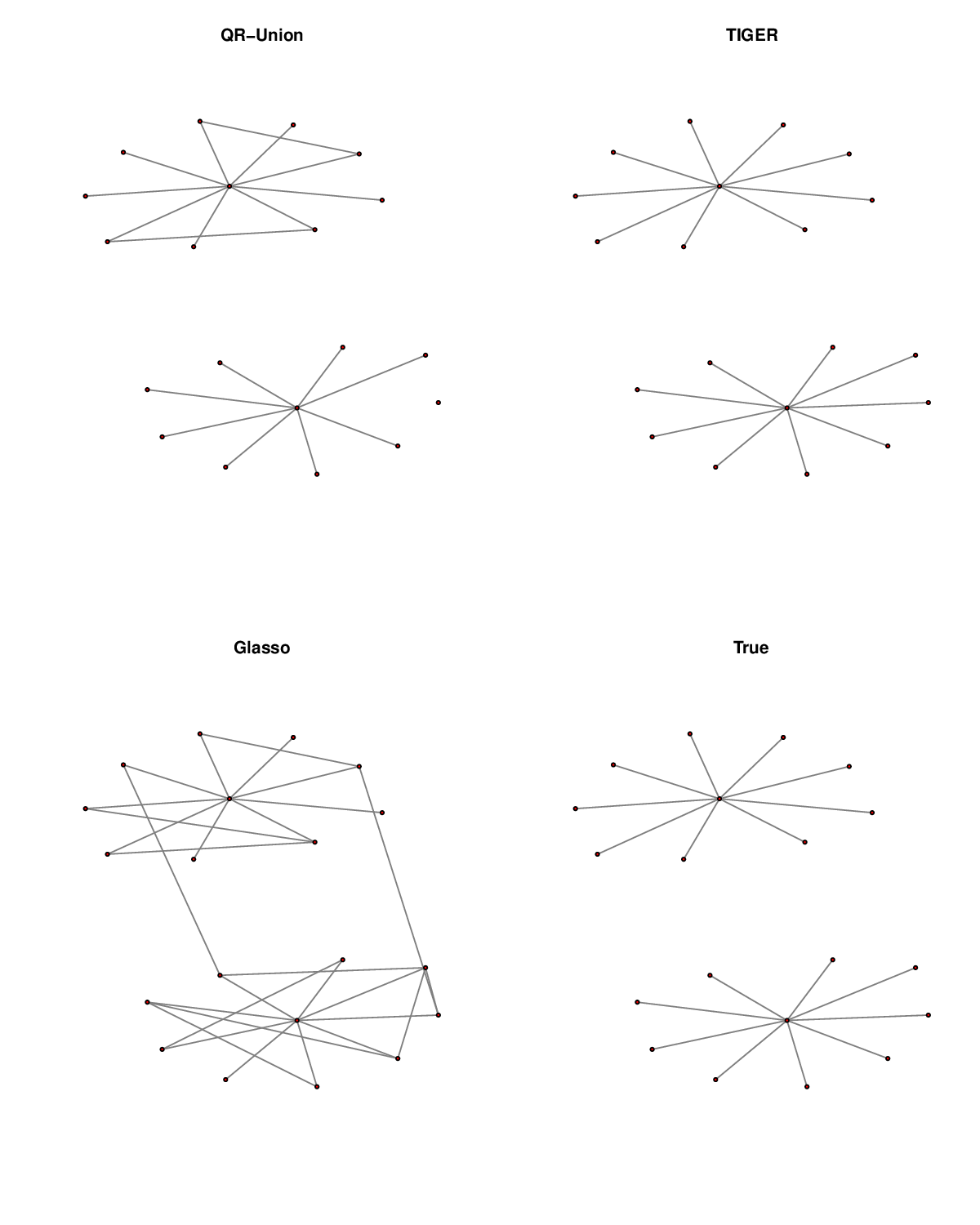}
\end{subfigure}
\hfill
\begin{subfigure}[b]{0.3\textwidth}
    \caption[]%
        {{\footnotesize $n = 200, d = 40$}}
 	\centering
	\includegraphics[width=\textwidth]{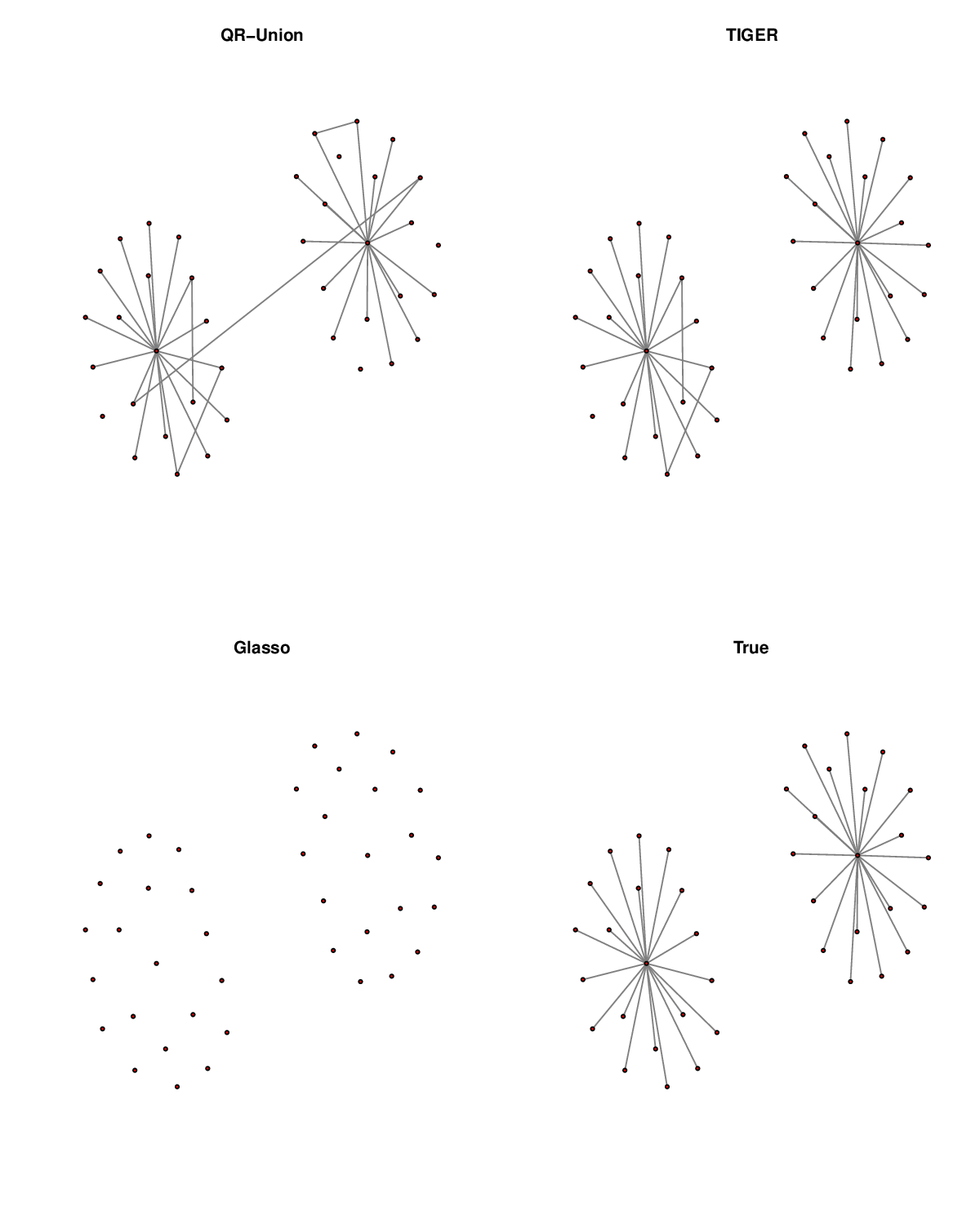}
\end{subfigure}

\vskip\baselineskip
\begin{subfigure}[b]{0.3\textwidth}
     \caption[]%
        {{\footnotesize $n = 400, d = 10$}}
 	\centering
	\includegraphics[width=\textwidth]{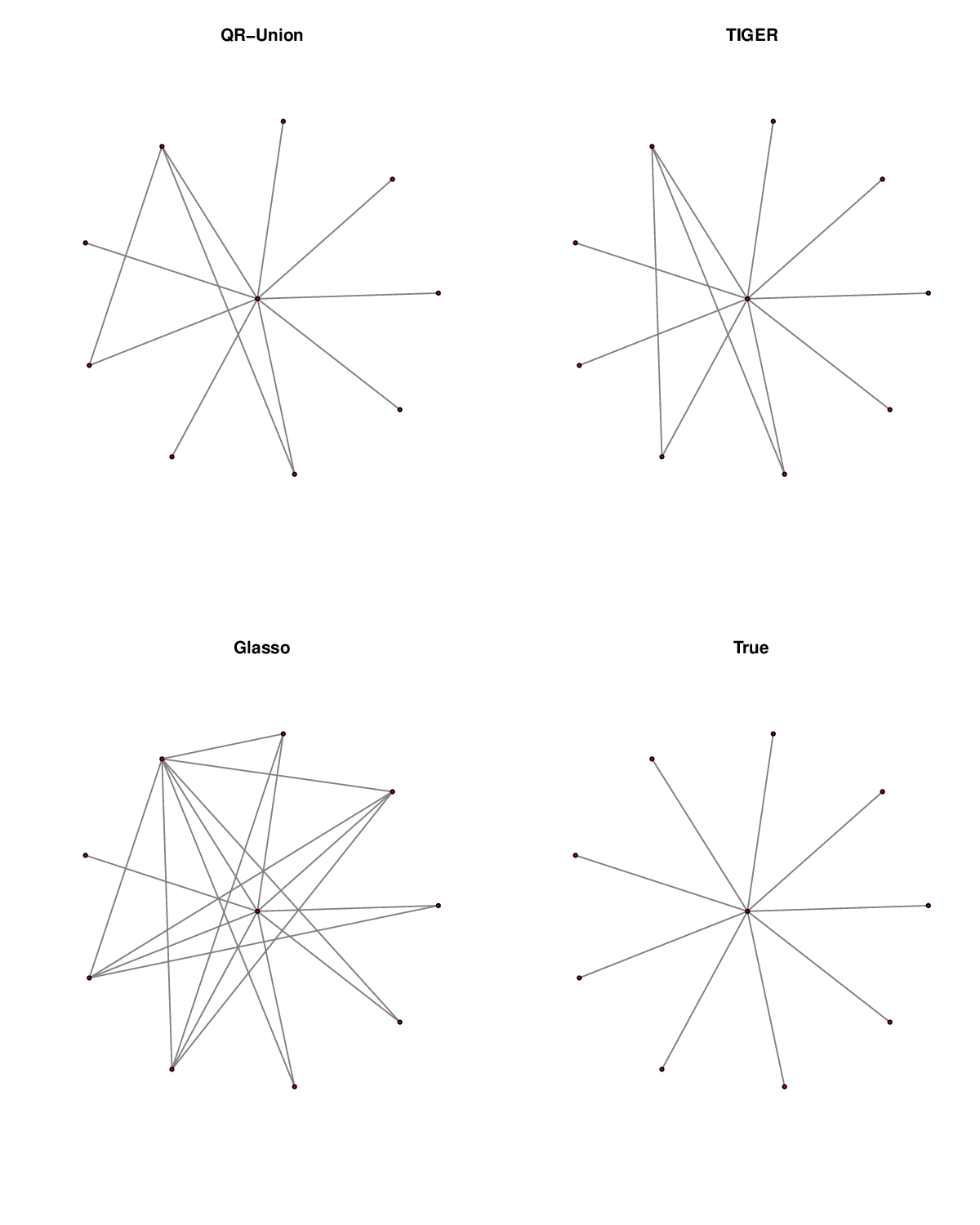}
\end{subfigure}
 \hfill
\begin{subfigure}[b]{0.3\textwidth}
   \caption[]%
        {{\footnotesize $n = 400, d = 20$}}
 	\centering
	\includegraphics[width=\textwidth]{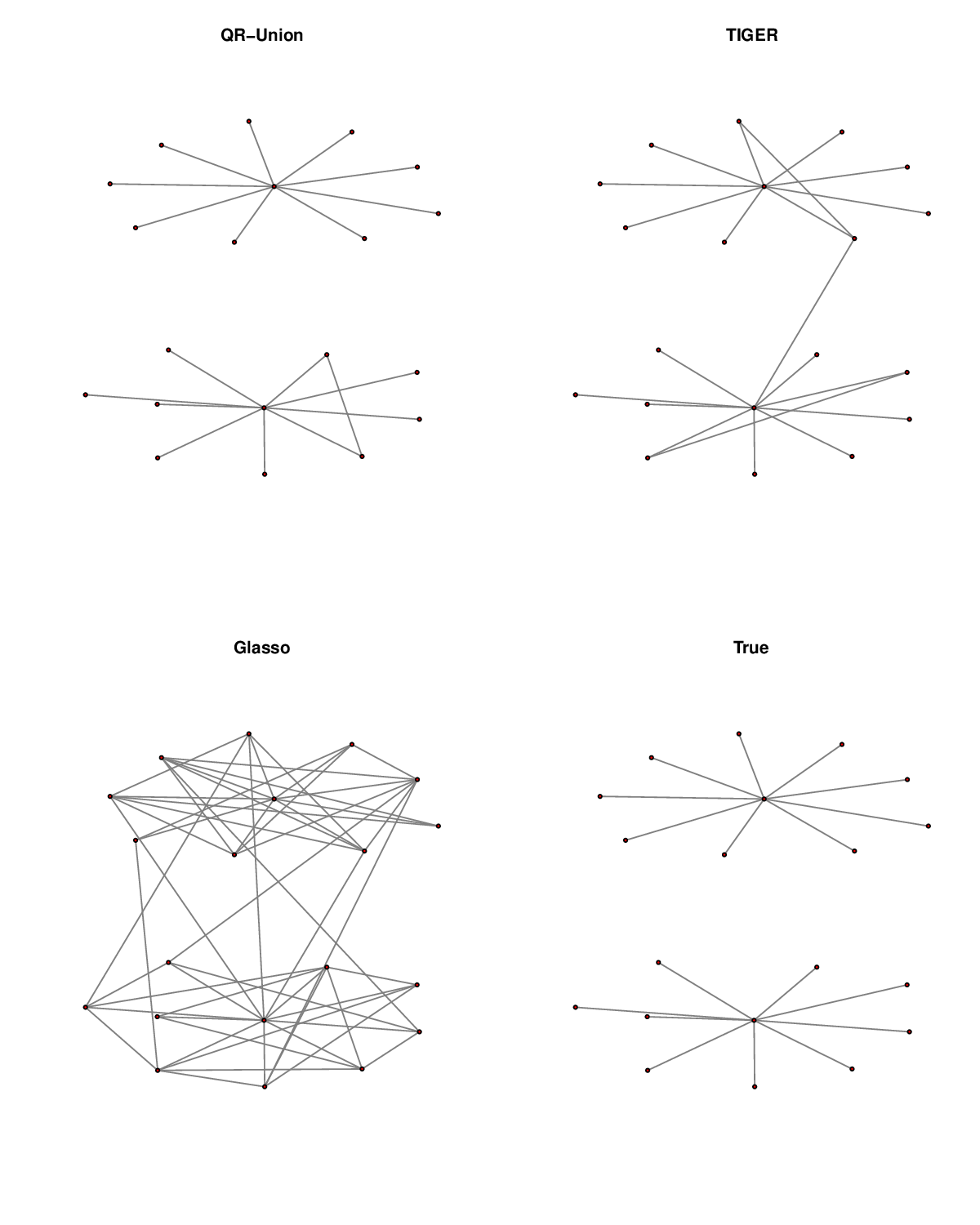}
\end{subfigure}
\hfill
\begin{subfigure}[b]{0.3\textwidth}
   \caption[]%
        {{\footnotesize $n = 400, d = 40$}}
 	\centering
	\includegraphics[width=\textwidth]{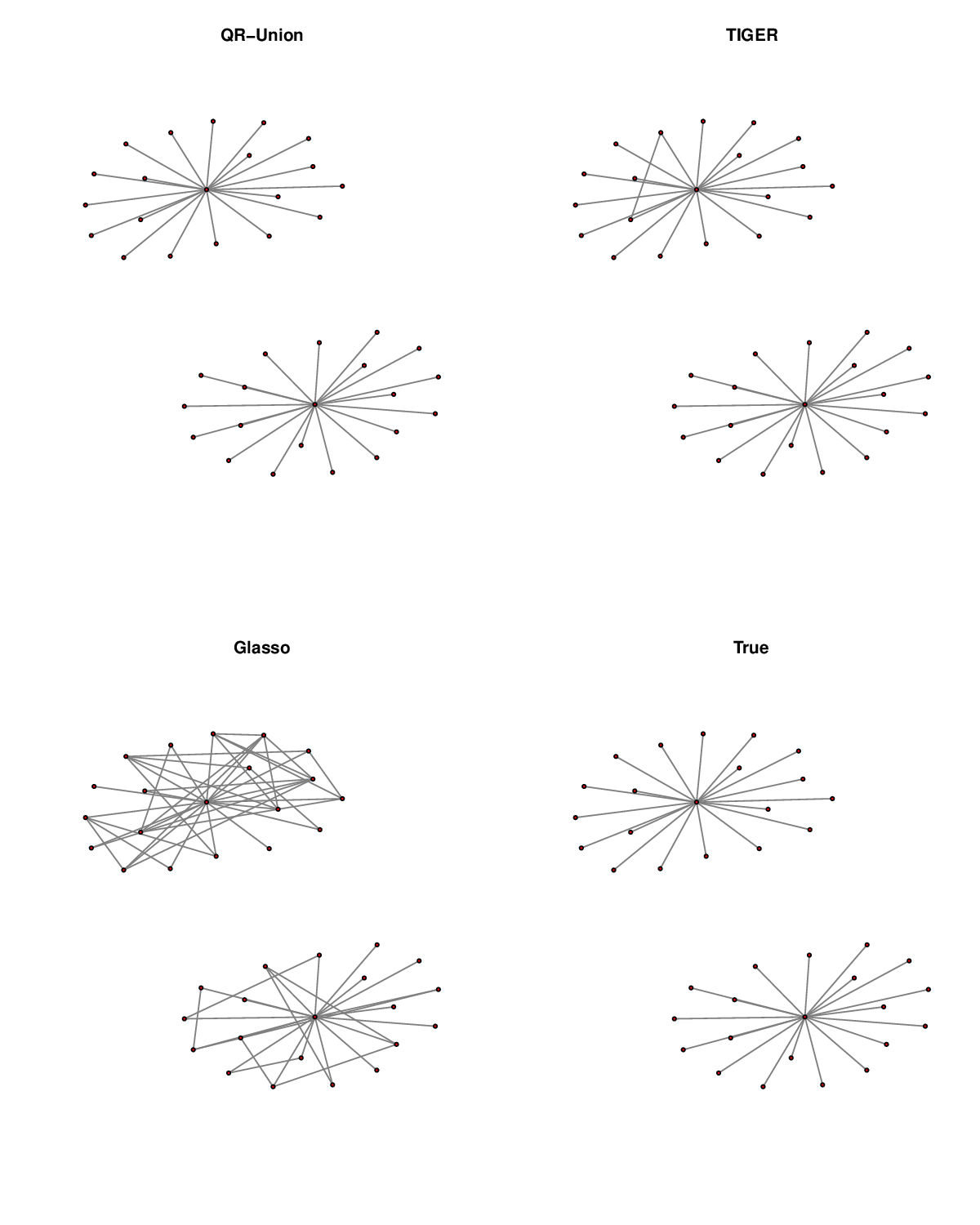}
\end{subfigure}

\caption[ ]
        {\small One simulation}
\end{minipage}
\end{figure}

\begin{figure}[ht]\ContinuedFloat
 \begin{minipage}{0.95\textwidth}
\centering

\begin{subfigure}[b]{0.3\textwidth}
     \caption[]%
        {{\footnotesize $n = 200, d = 100$}}
 	\centering
	\includegraphics[width=\textwidth]{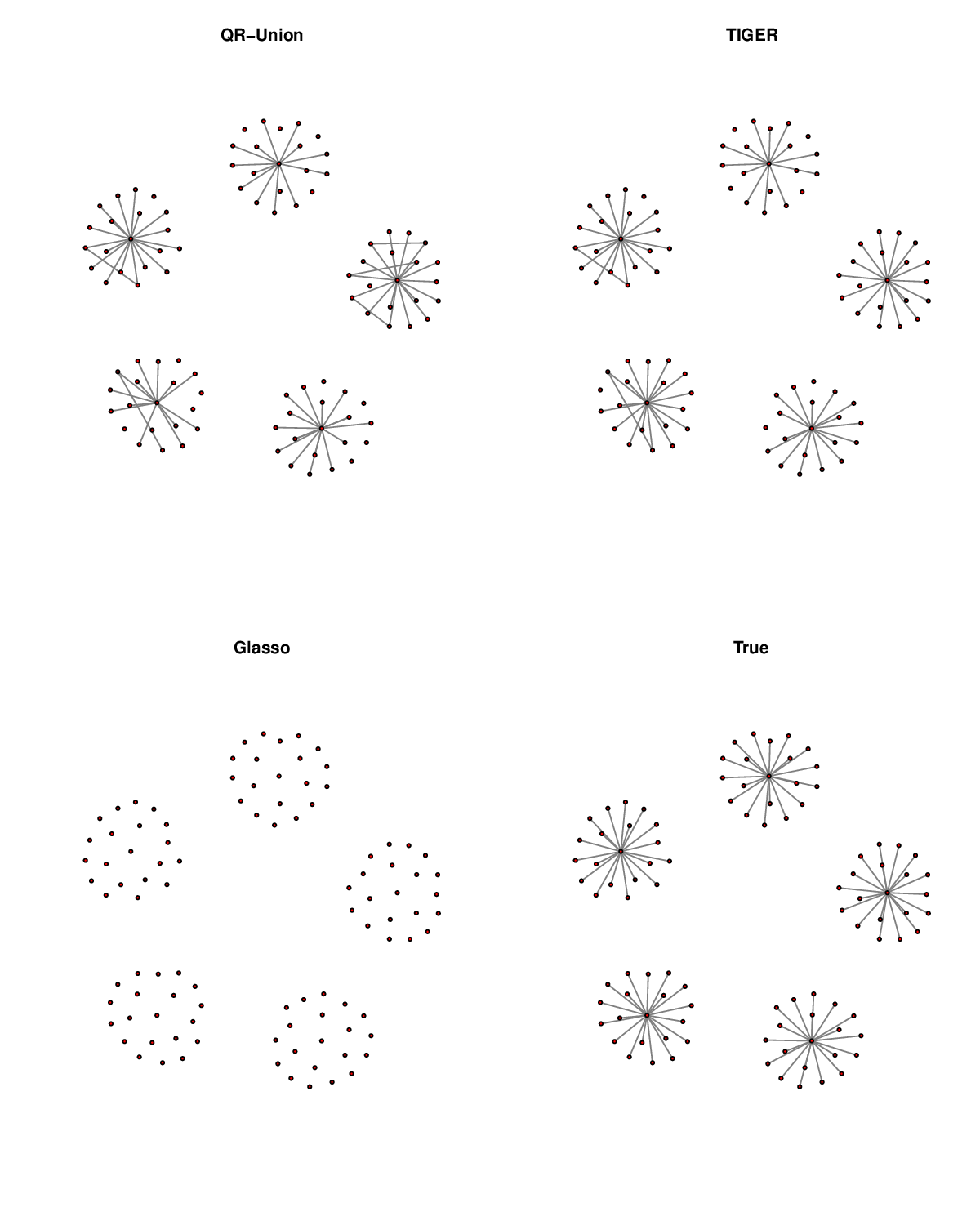}
\end{subfigure}
 \hfill
\begin{subfigure}[b]{0.3\textwidth}
   \caption[]%
        {{\footnotesize $n = 200, d = 200$}}
 	\centering
	\includegraphics[width=\textwidth]{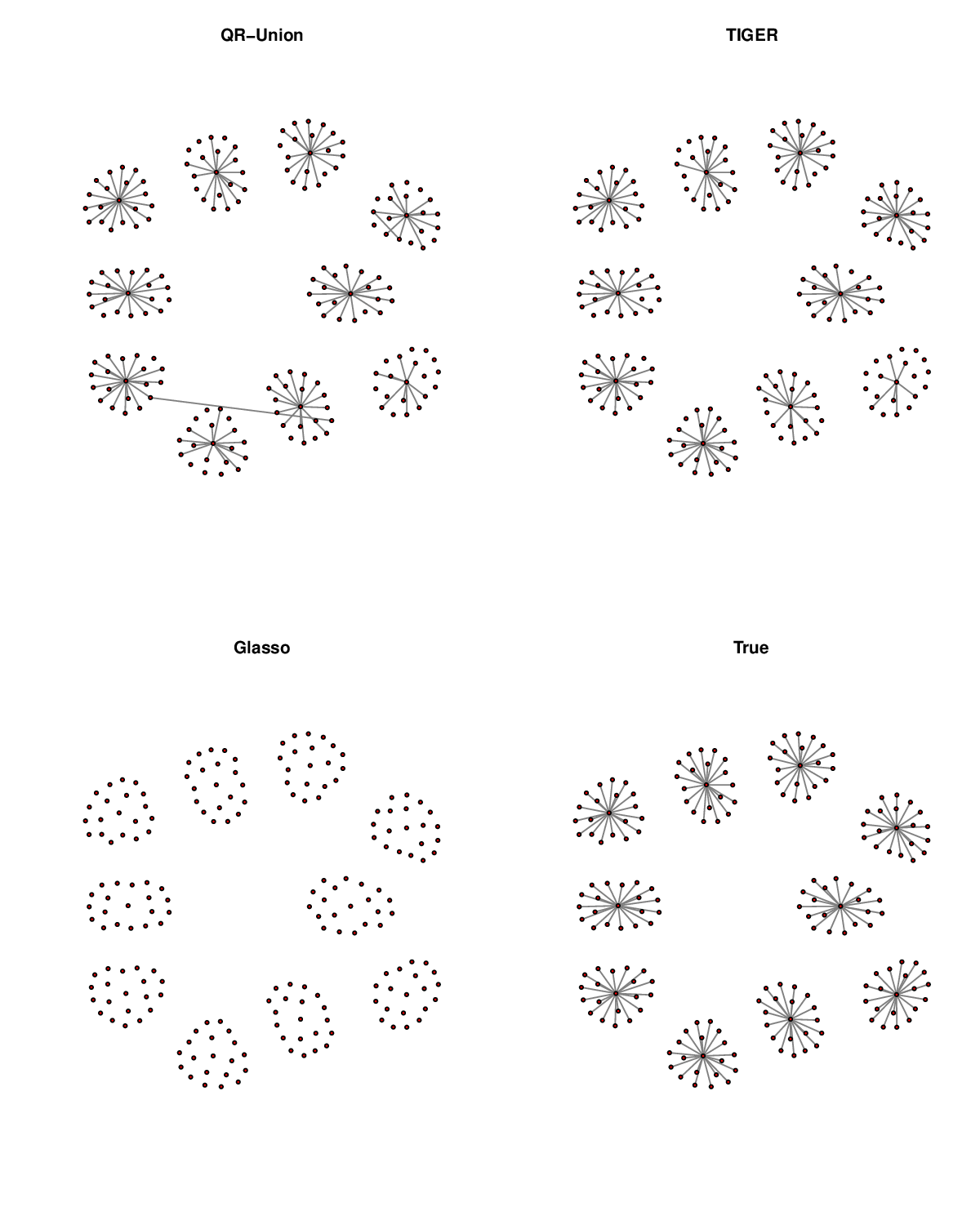}
\end{subfigure}
\hfill
\begin{subfigure}[b]{0.3\textwidth}
    \caption[]%
        {{\footnotesize $n = 200, d = 400$}}
 	\centering
	\includegraphics[width=\textwidth]{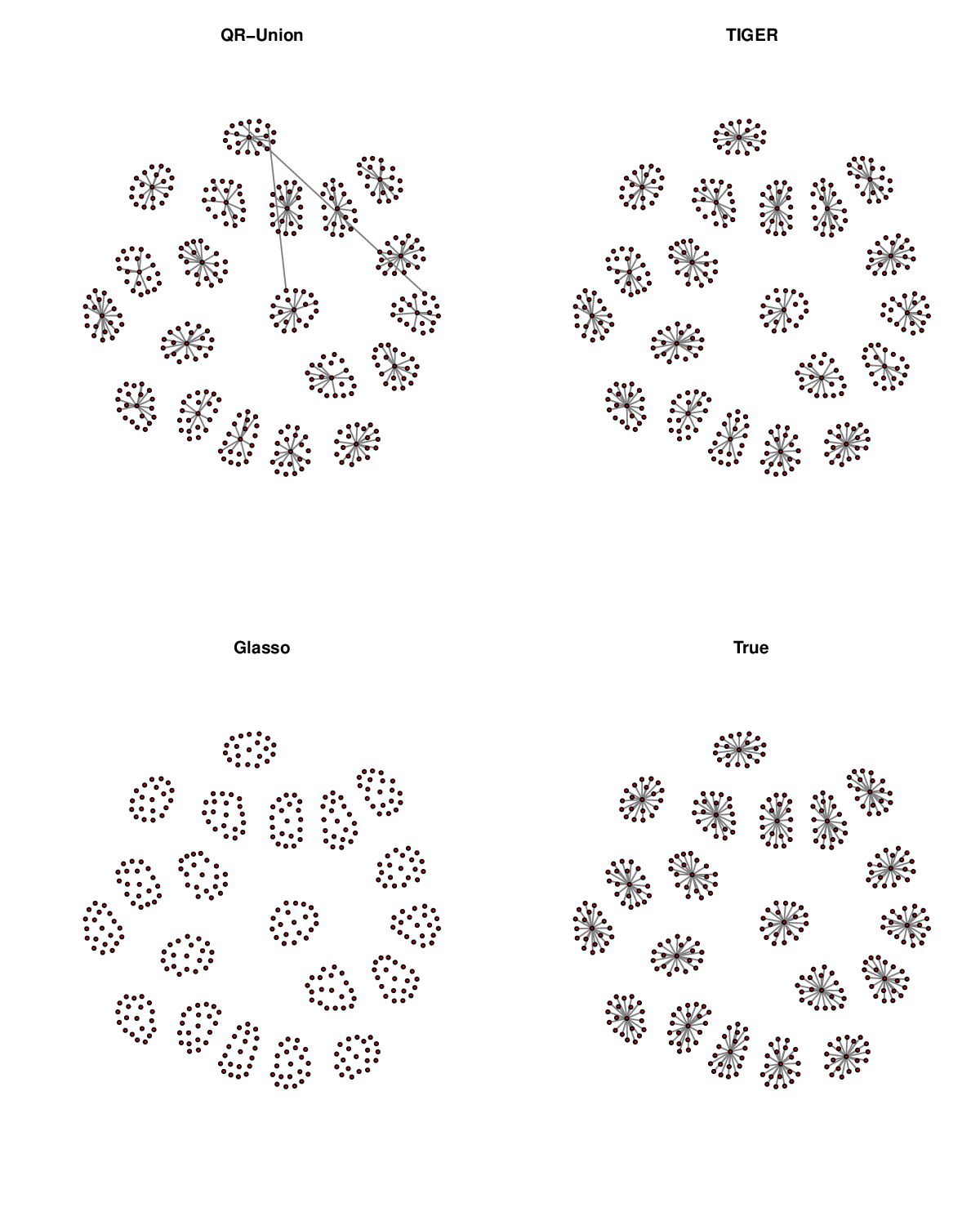}
\end{subfigure}

\begin{subfigure}[b]{0.3\textwidth}
   \caption[]%
        {{\footnotesize $n = 400, d = 100$}}
	\centering
   	 \includegraphics[width=\textwidth]{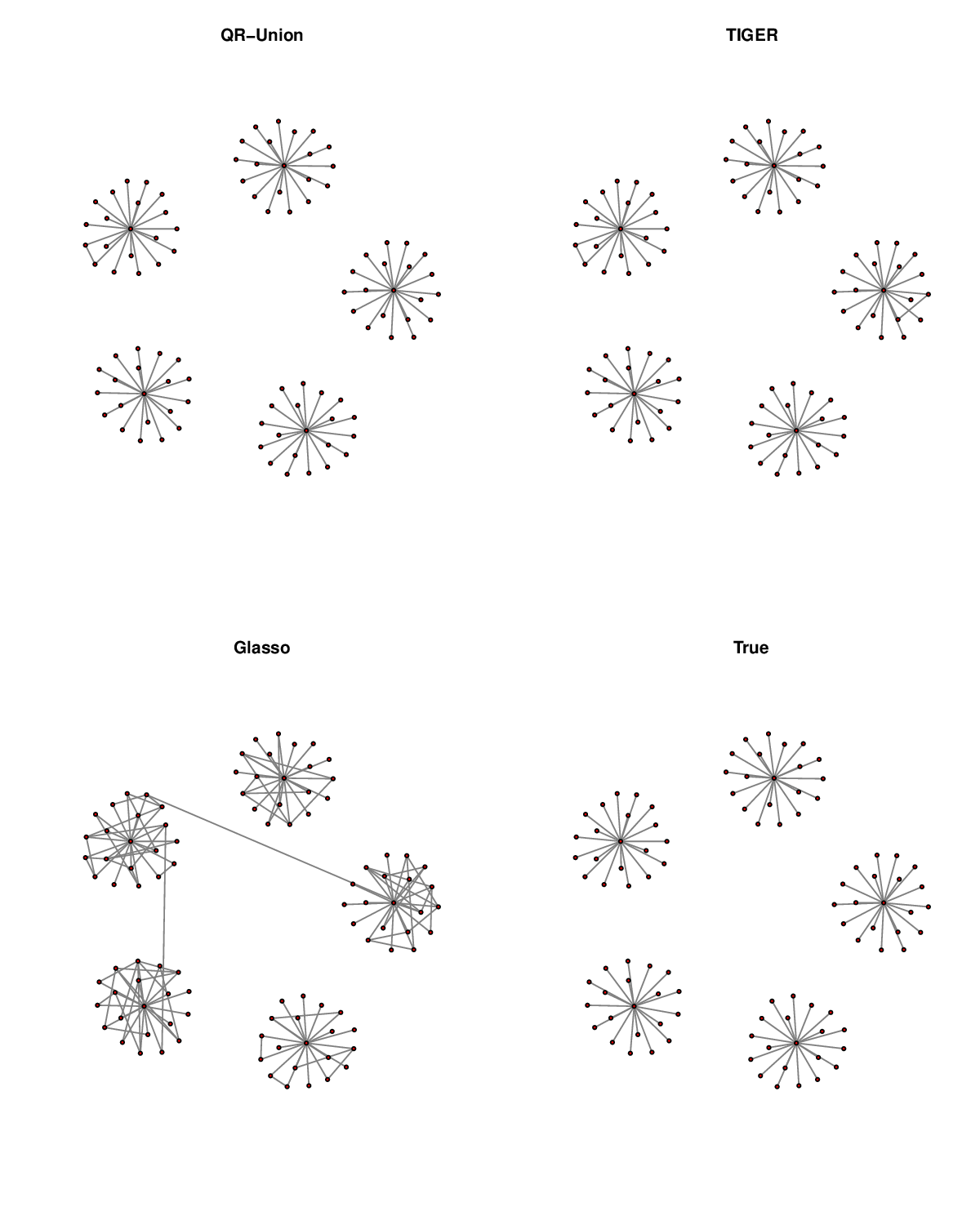}
\end{subfigure}
\hfill
\begin{subfigure}[b]{0.3\textwidth}
    \caption[]%
        {{\footnotesize $n = 400, d = 200$}}
 	\centering
	\includegraphics[width=\textwidth]{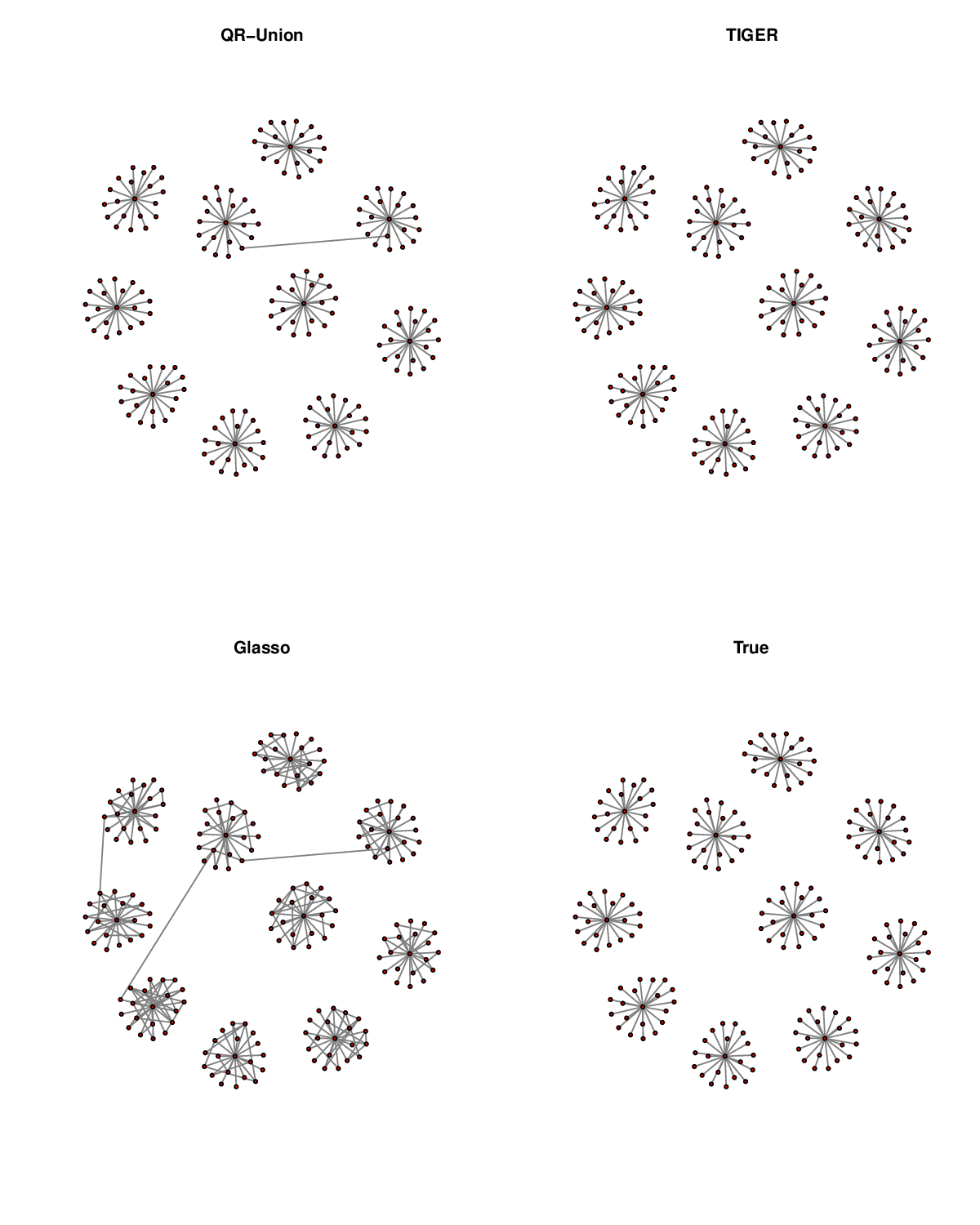}
\end{subfigure}
\hfill
\begin{subfigure}[b]{0.3\textwidth}
    \caption[]%
        {{\footnotesize $n = 400, d = 400$}}
 	\centering
	\includegraphics[width=\textwidth]{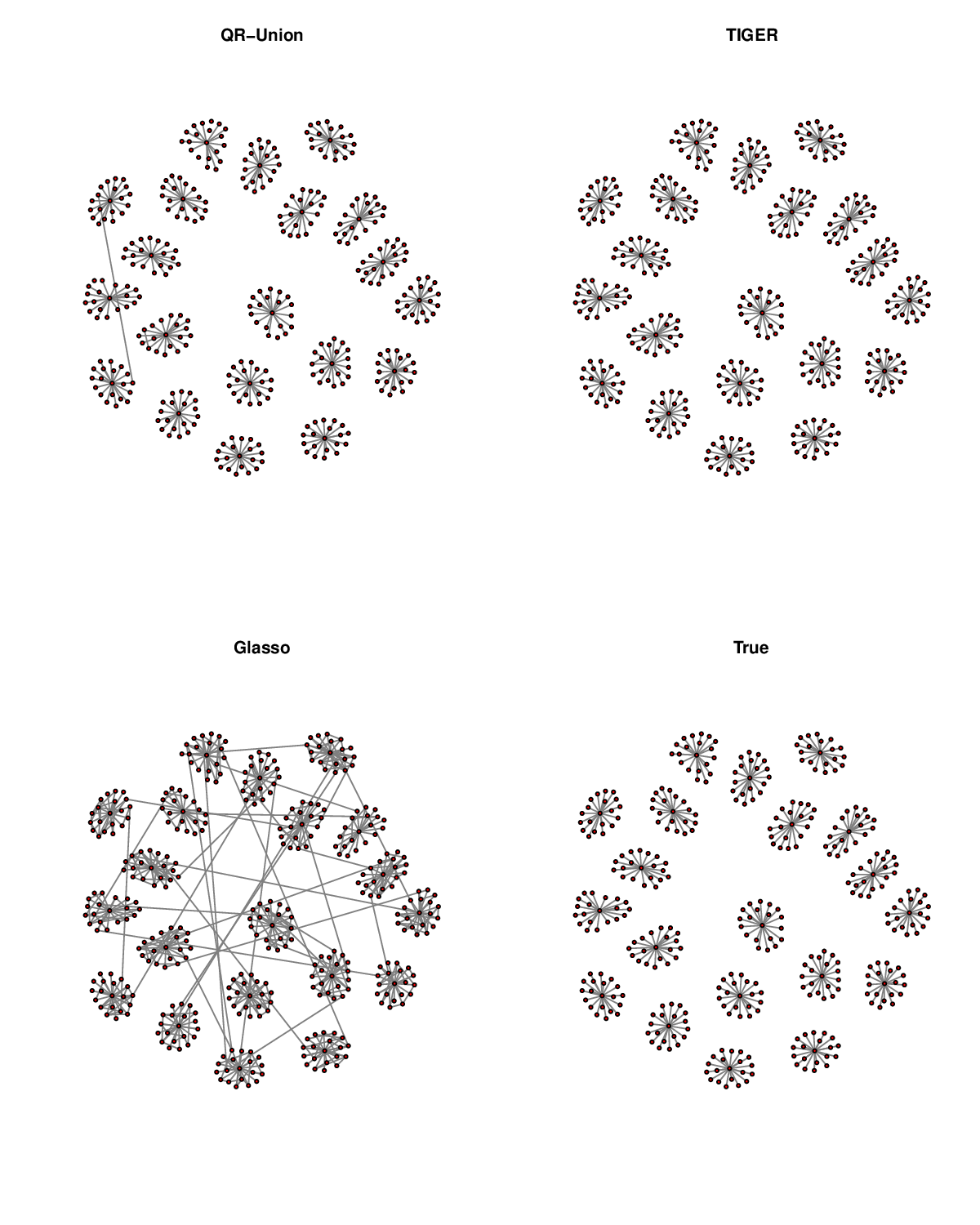}
\end{subfigure}

 \caption[ ]
        {\small One simulation (Cont.)} \label{Figure-3:-One-Simulation}
\end{minipage}
\end{figure}

Figure \ref{Figure-3:-One-Simulation} shows that: for the low dimensional cases, $d = {10, 20, 40}$, $n$ is large compared to $d$, CIQGM is comparable to TIGER and both are better than Glasso in terms of false positives; for the high dimensional cases, $d=\{100,200,300\}$, we can compare the performance of different graph estimators through looking at the “denseness” of the estimated graph (e.g., whether it is even or not), and again, both CIQGM and TIGER perform well in terms of graph recovery as compared to Glasso, and their performance are getting better when $n$ is increasing.

In what follows, Figure \ref{Figure-4:-Degree-of-PQGM} shows the degree distribution of true graph, the estimated ones, and the standard deviations of the degree difference (between the true graph and the estimated ones). It is based on simulations of Hub graph with $n=500$ and $d=40$. Simulated 100 times.

\begin{figure}[ht]
\begin{minipage}{0.9\textwidth}
\centering

\includegraphics[width=\linewidth, scale = 1.2]{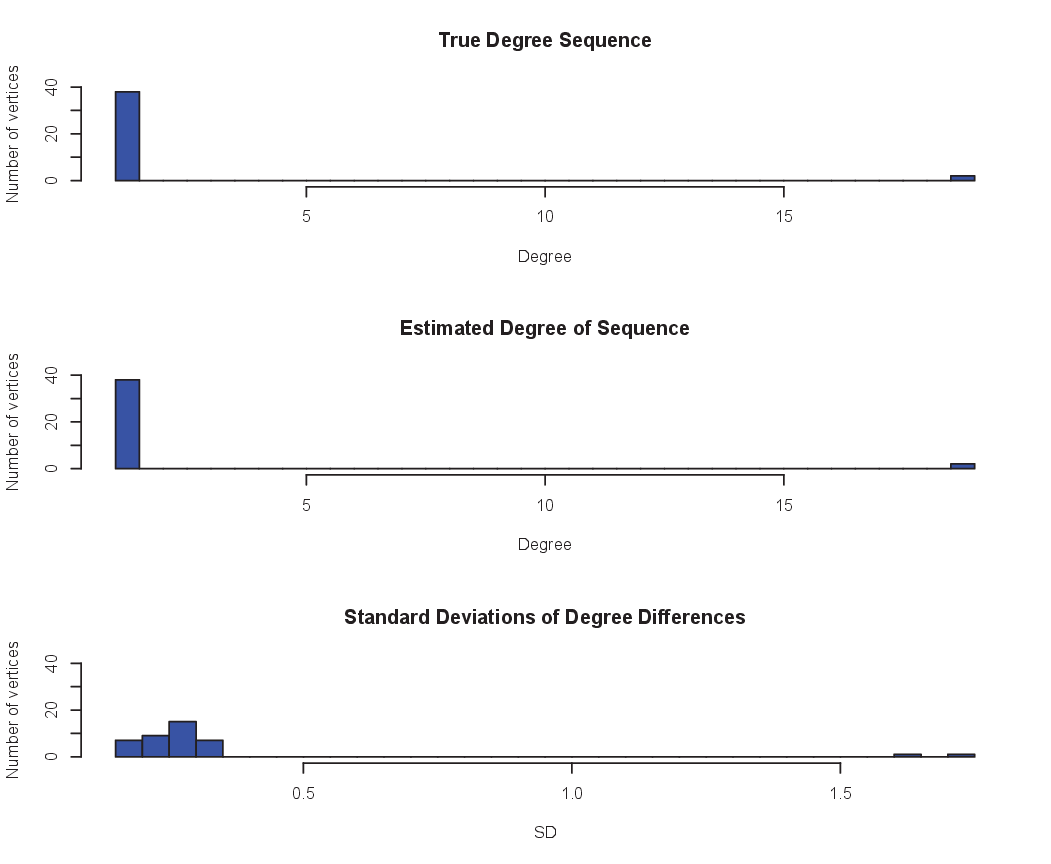}\tabularnewline
\caption{Upper panel shows the degree distribution of the true graphs. Middle panel shows the degree distribution of the estimated graphs. Bottom panel shows the standard deviations of the degree difference (between the true graph and the estimated ones). Hub graph with $n=500$ and $d=40$. Simulated 100 times. }\label{Figure-4:-Degree-of-PQGM}
 
\end{minipage}
\end{figure}

\newpage
\medskip
\subsubsection{Inference}
In this subsection Table \ref{Table2:-CIQGM-ER} shows the numerical performance of CIQGM, based on Algorithm \ref{Alg:1}, on estimating Erd\H{o}s-R{\'e}nyi random graphs. More precisely, we construct approximate $90\%$ confidence intervals for $\beta_{ab}$ with $\tau = 0.5$, and we report the coverage probabilities. Note, in the jointly Gaussian distributed case, we have closed form solution of $\beta_{ab}$ as shown in Example \ref{ExGauss}.

\textbf{Erd\H{o}s-R{\'e}nyi random graph.}
 We add an edge between each pair of nodes with probability $\sqrt{\log d/n}/d$ independently.
 
  Once the graph is obtained, we construct the adjacency matrix $E^I$ and generate the precision matrix $\Theta$ using (\ref{eq:precision}) but setting $\mathbf{D}_{jj}=1$ for $j=1,...,d/2$ and $\mathbf{D}_{jj}=1.5$ for $j=d/2+1,...,d$. We then invert $\Theta$ to get the covariance matrices $\Sigma:=\Theta^{-1}$ and generate the multivariate Gaussian data: $X_{1},....,X_{d}\sim N(0,\Sigma)$.

\begin{table}[ht]
\begin{minipage}{\textwidth}
\centering
\renewcommand{\arraystretch}{1.5}
  \begin{threeparttable}
\begin{center}
\caption{ Erd\H{o}s-R{\'e}nyi Random Graph}  \label{Table2:-CIQGM-ER}
\begin{tabular}{ccccccc}
\hline

\hline
 & $(a,b)$ & $n=200$ &  & $n=500$ &  & $n=1000$\tabularnewline
\hline
d =20 & (1,20) & 84.0 &  & 87.5 &  & 91.5\tabularnewline
 & (10,11) & 84.0 &  & 88.0 &  & 92.5\tabularnewline
 & (19, 20) & 86.5 &  & 86.0 &  & 90.0\tabularnewline
 & ACP & 86.3 &  & 89.4 &  & 89.8\tabularnewline
 &  &  &  &  &  & \tabularnewline
d=50 & (1,50) & 86.5 &  & 93.0 &  & 90.5\tabularnewline
 & (25,26) & 88.0 &  & 87.0 &  & 91.0\tabularnewline
 & (49, 50) & 87.5 &  & 90.5 &  & 87.5\tabularnewline
 & ACP & 86.7 &  & 89.4 &  & 90.1\tabularnewline
 &  &  &  &  &  & \tabularnewline
d=100 & (1,100) & 82.5 &  & 81 &  & 89.0\tabularnewline
 & (50,51) & 85.0 &  & 86 &  & 92.0\tabularnewline
 & (99, 100) & 78.5 &  & 84 &  & 87.0\tabularnewline
 & ACP & 86.5 &  & 86.8 &  & 90\tabularnewline
\hline
\end{tabular}
\begin{tablenotes}
      \small
      \item $(a,b)$, coverage probability for $\beta_{ab}$; ACP, average coverage
probability for $\beta_{ab}$, with $a\in V$, $b\in V\backslash\{a\}$.  Simulated 200 times.
    \end{tablenotes}
\par\end{center}
 \end{threeparttable}
\end{minipage}
\end{table}

\newpage

\section{Proofs of Section 4}

\begin{proof}[Proof of Theorem \ref{theorem:rateQRgraph}]
By Lemma \ref{Lemma:Matrices}, under Condition CI, for any $\theta$ such that $\|\theta\|_0\leq Cs\ell_n$, $\ell_n \to\infty$ slowly,  we have that $$\|\sqrt{f_u}Z^a\theta\|_{n,\varpi}/\{\Ep[K_\varpi(W)f_u(Z^a\theta)^2]\}^{1/2} = 1+o_P(1).$$ Moreover, $\Ep[K_\varpi(W)f_u(Z^a\theta)^2]\geq \underf_u \Ep[K_\varpi(W)(Z^a\theta)^2]$,  $\Ep[K_\varpi(W)(Z^a\theta)^2]=\Ep[(Z^a\theta)^2| \varpi]\Pr(\varpi)$, and $\Ep[(Z^a\theta)^2\mid \varpi] \geq c\|\theta\|^2$ by Condition CI. Lemma \ref{Lemma:Matrices} further implies that the ratio of the minimal and maximal eigenvalues of order $s\ell_n$ are bounded away from zero and from above uniformly over $\varpi\in \mathcal{W}$ and $a\in V$ with probability $1-o(1)$. Therefore, since $c\{\Pr(\varpi)\}^{1/2}\|\delta\|_1 \leq \|\delta\|_{1,\varpi} \leq C\{\Pr(\varpi)\}^{1/2}\|\delta\|_1$, we have $\kappa_{u,2\cc} \geq c$ uniformly over $u\in\mathcal{U}$ with the same probability for $n$ large enough, see for instance \cite{BickelRitovTsybakov2009}.

To establish rates of convergence of the estimator obtained in Step 1 we will apply Lemma \ref{Theorem:L1QRnp}. Consider the events $\Omega_1, \Omega_2,$ and $\Omega_3$ as defined in (\ref{def:Omega1}), (\ref{def:Omega2}) and (\ref{def:Omega3}). By the choice of $\lambda_u$ we have $\Pr(\Omega_1) \geq 1-o(1)$. By Condition CI with $\bar{R}_{u\xi} \leq Cs \log(p|V|n)/n$ and Lemma \ref{Lemma:ControlR} we have $\Pr(\Omega_2) \geq 1-o(1)$. Moreover,  $\Pr(\Omega_3) \geq 1-o(1)$ by Lemma \ref{LemmaOmega3} with $t_3 \leq C n^{-1/2} \sqrt{ (1+d_W)\log(p|V|nL_f) }$.

Using the same argument (with $Z^a$ replacing $X_{-a}$) as in (\ref{qA:1}), (\ref{qA:2}), and (\ref{qA:3}), for
 $$\delta\in A_u := \Delta_{\varpi,2\cc}\cup\{v:\|v\|_{1,\varpi}\leq 2\cc\bar R_{u\xi}/\lambda_u, \|\sqrt{f_u}Z^av\|_{n,\varpi} \geq C\sqrt{s(1+d_W)\log(p|V|n)/n}/\kappa_{u,2\cc}  \},$$
with the restricted set defined as  $\Delta_{u, 2\tilde\cc} = \{ \delta : \|\delta_{T^{c}_u}\|_1\leq 2\tilde\cc \|\delta_{T_u}\|_1\}$ for $u\in \mathcal{U}$. we have $ \bar q_{A_u} \geq c (\underf_\UU^{3/2}/\bar f')\mu_\mathcal{W}^{1/2}/ \{\sqrt{s} \max_{a\in V, i\leq n}\|Z_i^a\|_\infty\}$ where $\max_{a\in V, i\leq n}\|Z_i^a\|_\infty \lesssim_P M_n$. Thus the conditions on $\bar q_{A_u}$ are satisfied since Condition CI assumes $M_n^2s^2\log(p|V|n) \leq n \mu_\mathcal{W}\underf_\UU^3$. The conditions on the approximation error are assumed in Condition CI.

Therefore, setting $\xi=1/\log n$, by Lemma \ref{Theorem:L1QRnp} we have uniformly over $u=(a,\tau,\varpi)\in \mathcal{U}$
 \begin{equation}\label{InitialRatesLemma}\begin{array}{rl}
 \|\sqrt{f_u}Z^a(\hat \beta_u -\beta_u)\|_{n,\varpi} & \lesssim \sqrt{(1+(t_3/\lambda_u)\bar R_{u\xi}} + (\lambda_u + t_3)\sqrt{s} \lesssim  \sqrt{\frac{s(1+d_W)\log(p|V|n)}{n\tau(1-\tau)}} \\
 \|\hat \beta_u -\beta_u\|_{1,\varpi} & \lesssim s\sqrt{\frac{(1+d_W)\log(p|V|n)}{n}} \end{array}\end{equation}
here we used that $\lambda_u \leq C\sqrt{\frac{(1+d_W)\log(p|V|n)}{n}}$. Indeed by Lemma \ref{lem:Lambda} with $\tilde x_{ij} = K_\varpi(W_i)Z^a_{ij}$ and $\hat\sigma_j \geq c\Pr(\varpi)^{1/2}$, we can bound $\Lambda_{a\tau \varpi}(1-\xi/\{|V|n^{1+2d_W}\}| X_{-a},W)$ under $M_n^2 \log(p|V|n/\{\tau(1-\tau)\}) = o(n \tau(1-\tau)\mu_\mathcal{W})$ for all $\tau\in\mathcal{T}$ and $\varpi \in \mathcal{W}$, and the bound on $\lambda_u$ follows from the union bound.

Let $\delta_u = \hat \beta_u -\beta_u$. By triangle inequality it follows that
\begin{equation}\label{TriInit} \{\Ep[K_\varpi(W)f_u(Z^a\delta_u )^2]\}^{1/2} \leq \|\sqrt{f_u}Z^a\delta_u \|_{n,\varpi} + \|\delta_u\|_{1}\{ |(\En-\Ep)[K_\varpi(W)f_u(Z^a\delta_u)^2]/\|\delta_u\|_{1}^2]| \}^{1/2} \end{equation}
and the last term can be bounded by
$$\begin{array}{rl}{\displaystyle \sup_{\|\delta\|_{1} \leq 1}} |(\En-\Ep)[K_\varpi(W)f_u(Z^a\delta)^2/\|\delta\|_{1}^2]| & \leq   \max_{k,j}|(\En-\Ep)[K_\varpi(W)f_uZ^a_{k}Z^a_{j}]| \\
&\lesssim \sqrt{\frac{(1+d_W)\log(p|V|n)}{n}}\end{array}$$
with probability $1-o(1)$ by Lemma \ref{lemma:CCK} under our conditions.

Combining the relations above with (\ref{TriInit}), under $(1+d_W)s^2\log(p|V|n) = o(n)$ we have uniformly over $u\in\mathcal{U}$
$$\begin{array}{rl}
\|\delta_u\| & \lesssim \{\Ep[(Z^a\delta_u)^2\mid\varpi]\}^{1/2} \lesssim \{\Pr(\varpi)\}^{-1/2} \{\Ep[K_\varpi(W)(Z^a\delta_u)^2]\}^{1/2} \\
& \lesssim \{\Pr(\varpi)\underf_u\}^{-1/2}\{\Ep[K_\varpi(W)f_u(Z^a\delta_u)^2]\}^{1/2} \\
 & \leq \{\Pr(\varpi)\underf_u\}^{-1/2}\|\sqrt{f_u}Z^a\delta_u\|_{n,\varpi} + \{\Pr(\varpi)\underf_u\}^{-1/2}\sqrt[4]{\frac{(1+d_W)\log(p|V|n)}{n}}\|\delta_u\|_{1}\\
& \leq C\sqrt{\frac{s(1+d_W)\log(p|V|n)}{n\underf_u \Pr(\varpi) }}.\end{array}$$
given $\|\delta_u\|_{1}\leq \|\delta_u\|_{1,\varpi}/ \Pr(\varpi)^{1/2}$, (\ref{InitialRatesLemma}), and $s^2(1+d_W)\log(p|V|n) = o(n \mu_\mathcal{W}^4\underf_\UU^2)$ assumed in Condition CI.

Finally, let $\hat\beta^{\bar\lambda}_u$ obtained by thresholding the estimator $\hat\beta_u$ with $\bar\lambda:= \sqrt{(1+d_W)\log(p|V|n)/n}$ (note that each component is weighted by $\En[K_\varpi(W)(Z_j^a)^2]^{1/2}$). By Lemma \ref{Lemma:Bound2nNormSecond}, we have with probability $1-o(1)$
$$
\begin{array}{rl}
\|Z^a(\hat\beta^{\bar\lambda}_u - \beta_u)\|_{n,\varpi} \lesssim \sqrt{s(1+d_W)\log(p|V|n)/n} \\
\|\hat\beta^{\bar\lambda}_u - \beta_u\|_{1,\varpi} \lesssim s\sqrt{(1+d_W)\log(p|V|n)/n}\\
|\supp(\hat\beta^{\bar\lambda}_u)| \lesssim s
\end{array}
$$ by the choice of $\bar\lambda$ and the rates in (\ref{InitialRatesLemma})

\end{proof}

\begin{proof}[Proof of Theorem \ref{theorem:rateLASSOgraph}]
 
We verify Assumption \ref{ass: M} and Condition WL for the weighted Lasso model with index set $\UU \times [p]$ where $Y_u = K_\varpi(W)Z^a_{j}$, $X_u=K_\varpi(W)Z^a_{-j}$, $\theta_u=\bar \gamma_u^j$, $a_u = (f_u,\bar r_{uj})$, $\bar r_{uj}=K_\varpi(W)Z^a_{-j}(\gamma_u^j-\bar\gamma_u^j)$, $S_{uj} = K_\varpi(W)f_u^2(Z^a_{j}-Z^a_{-j}\gamma_u^j)Z^a_{-j}=K_\varpi(W)f_uv_{uj}Z^a_{-j}$, and $w_u =K_\varpi(W)f_u^2$. We will take $N_n = |V|p^2\{pn^3\}^{1+d_W}$ in the definition of $\lambda$.

We first verify Condition WL. We have $\Ep[S_{ujk}^2] \leq \bar f^2\Ep[ |v_{uj}Z^a_{-j,k}|^2] \leq \bar f^2\{\Ep[ |v_{uj}|^4|Z^a_{-j,k}|^4]\}^{1/2} \leq C$ by the bounded fourth moment condition.
We have that $$\frac{\Ep[|S_{ujk}|^3]^{1/3}}{\Ep[|S_{ujk}|^2]^{1/2}} =\frac{\Ep[|S_{ujk}|^3\mid \varpi]^{1/3}}{\Ep[|S_{ujk}|^2 \mid \varpi]^{1/2}} \{  \Pr(\varpi)\}^{-1/6} =  \frac{\Ep[|f_uv_{uj}Z_{-jk}^a|^3\mid \varpi]^{1/3}}{\Ep[|f_uv_{uj}Z^a_{-jk}|^2 \mid \varpi]^{1/2}} \{  \Pr(\varpi)\}^{-1/6} =:M_{uk} $$
By the choice of $N_n$ and $\Phi^{-1}(1-t)\leq C\sqrt{\log(1/t)}$, we have $M_{uk}\Phi^{-1}(1-\xi/\{2pN_n\}) \leq M_{uk}C(1+d_W)\log^{1/2}(pn |V|) \leq C\delta_n n^{1/6}$ where the last inequality holds by Condition CI so  Condition WL(i) holds.

To verify Condition WL(ii) we will establish the validity of the choice of $N_n$. We will consider $u=(a,\tau,\varpi)\in \UU$ and $u'=(a,\tau',\varpi')\in \UU$. By Condition CI we have that  \begin{equation}\label{Lipsf}|f_u-f_{u'}| \leq L_f\|u-u'\| \ \ \mbox{and} \ \  \Ep[|K_\varpi(W)-K_{\varpi'}(W)|] \leq L_K\|\varpi-\varpi'\|.\end{equation}
Further, by Lemma \ref{Lemma:LipsGamma} we have
\begin{equation}\label{LipsGamma} \|\gamma_u^j-\gamma_{u'}^j\|\leq  L_\gamma \{\|u-u'\|+\|\varpi-\varpi'\|^{1/2}\}.\end{equation}
By definition we have $$S_{ujk}-S_{u'jk}=\{K_\varpi(W)f_u^2-K_{\varpi'}(W)f_{u'}^2\}\{Z_j^a-Z^a_{-j}\gamma_u^j\}Z_k^a-K_{\varpi'}(W)f_{u'}^2\{Z^a_{-j}(\gamma_u^j-\gamma_{u'}^j)\}Z_k^a$$ and note that  $f_u+f_{u'} \leq 2f_u+L\|u-u'\|$, $|Z_j-Z_{-j}^a\gamma_u^j|\cdot|Z_k^a| \leq  |Z_j|^2+ 2|Z_k^a|^2 + |Z_{-j}^a\gamma_u^j|^2$. Moreover, $$\begin{array}{rl}
|K_\varpi(W)f_u^2-K_{\varpi'}(W)f_{u'}^2| & \leq K_\varpi(W)K_{\varpi'}(W)|f_u^2-f_{u'}^2| + (f_u+f_{u'})^2|K_\varpi(W)-K_{\varpi'}(W)|\\
& \leq 2 \bar f K_\varpi(W)K_{\varpi'}(W)|f_u-f_{u'}| + 4\bar f^2|K_\varpi(W)-K_{\varpi'}(W)| \\
\end{array}$$
Using these relations we have $|\En[ S_{ujk}-S_{u'jk} ]|  \leq  (I) + (II)$ where
$$\begin{array}{rl}
(I)  & = \En[|K_\varpi(W)f_u^2-K_{\varpi'}(W)f_{u'}^2|\cdot |\{Z_j^a-Z^a_{-j}\gamma_u^j\}Z_k^a|] \\
& \leq \max_{i\leq n}\|Z_i^a\|_\infty^2(1+\|\gamma_u^j\|_1) \En[ 2\bar f K_\varpi(W)K_{\varpi'}(W)|f_u-f_{u'}| + 4\barf^2|K_\varpi(W)-K_{\varpi'}(W)|]\\
& \leq (\bar f+\bar f^2) C\sqrt{s}\max_{i\leq n}\|Z_i^a\|_\infty^2\{ L_f \|u-u'\|+ \En[|K_\varpi(W)-K_{\varpi'}(W)|]\}\\
(II) & = \En[K_{\varpi'}(W)f_{u'}^2|Z^a_{-j}(\gamma_u^j-\gamma_{u'}^j)Z_k^a|] \\
& \leq \bar f^2 \En[ \|Z^a\|_\infty^2 ] \|\gamma_u^j-\gamma_{u'}^j\|_1 \\
&  \leq \bar f^2 \En[ \|Z^a\|_\infty^2 ] \sqrt{p} L_\gamma\{ \|u-u'\|+\|\varpi-\varpi'\|^{1/2}\} \\
 \end{array}$$

Moreover, we have that $\max_{i\leq n} \|Z_i^a\|_\infty^2 \lesssim_P M_n^2$. For $d_\mathcal{U}=\|\cdot\|$, an uniform $\epsilon$-cover of $\mathcal{U}$ satisfies $(6\diam(\mathcal{U})/\epsilon)^{1+d_W} \geq N(\epsilon,\UU,\|\cdot\|)$. We will set $1/\epsilon = (1+\bar f^2)\{L_\gamma+L_f\}^2pnM_n^2\log^2(p|V|n)/\{\mu_\mathcal{W}\underf_\UU^2\} \leq pn^3$ so that with probability $1-o(1)$, for any pair $u,u'\in\UU$, $\|u-u'\|\leq \epsilon$, we have
$$\begin{array}{rl}
 |\En[ S_{ujk}-S_{u'jk} ]| & \lesssim (\bar f+\bar f^2) \sqrt{s}\max_{i\leq n}\|Z_i^a\|_\infty^2\{ L_f \epsilon+ \En[|K_\varpi(W)-K_{\varpi'}(W)|]\} \\
 & +\bar f^2 \En[ \|Z_i^a\|_\infty^2 ] \sqrt{p} L_\gamma\{ \epsilon+ \epsilon^{1/2}\} \\
& \lesssim  \delta_n n^{-1/2}\{\mu_\mathcal{W}\underline f_\mathcal{U}^2\}^{1/2} + \sqrt{s}M_n\log(n)\En[|K_\varpi(W)-K_{\varpi'}(W)|]\\
\end{array}$$
by the choice of $\epsilon$.  
To control the last term, note that $\mathcal{W}$ is a VC-class of events with VC dimension $d_W$. Thus by Lemma \ref{lemma:CCK}, with probability $1-o(1)$
$$ \begin{array}{rl}
\En[|K_\varpi(W)-K_{\varpi'}(W)|] & \leq |(\En-\Ep)[|K_\varpi(W)-K_{\varpi'}(W)|]|+\Ep[|K_\varpi(W)-K_{\varpi'}(W)|] \\
&\displaystyle \leq \sup_{\varpi,\varpi'\in\mathcal{W},\|\varpi-\varpi'\|\leq \epsilon}|(\En-\Ep)[|K_\varpi(W)-K_{\varpi'}(W)|]|+ L_K\epsilon \\
& \lesssim \sqrt{\frac{d_W \log(n/\epsilon)}{n}}\epsilon^{1/2} + \frac{d_W\log(n/\epsilon)}{n} + L_K\epsilon
\end{array}
$$ which yields uniformly over $u\in \mathcal{U}$ and $j\in[p]$
\begin{equation}\label{eq:Shelp1} |\En[ S_{ujk}-S_{u'jk} ]| \lesssim \delta_n n^{-1/2} \mu_\mathcal{W}^{1/2}\underf_\mathcal{U}\end{equation}
under $\sqrt{\epsilon d_W \log(n/\epsilon)}M_n\log n = o(\mu_\mathcal{W}^{1/2}\underf_\mathcal{U})$ and $d_W\log(n/\epsilon)M_n\log n=o(n^{1/2}\mu_\mathcal{W}^{1/2}\underf_\mathcal{U})$ assumed in Condition CI. In turn this implies
$$ \begin{array}{c}
\displaystyle \sup_{|u-u'|\leq \epsilon} \underset{j,k\in[p],j\neq k}{\max}\frac{|\En[S_{ujk}-S_{u'jk}]|}{\Ep[S_{ujk}^2]^{1/2}} \leq \delta_n n^{-1/2}  \end{array}$$
since $\Ep[S_{ujk}^2] \geq c \mu_\mathcal{W}\underf_\mathcal{U}^2$. Using the same choice of $\epsilon$, similar arguments also imply
\begin{equation}\label{eq:Shelp2} \begin{array}{c}
\ \displaystyle \sup_{|u-u'|\leq \epsilon}\underset{j,k\in[p],j\neq k}{\max}\frac{|\Ep[S_{ujk}^2-S_{u'jk}^2]|}{\Ep[S_{ujk}^2]} \leq \delta_n \end{array}\end{equation}

To establish the last requirement of Condition WL(ii), note that
\begin{equation}\label{eq:load} \begin{array}{rl}
\displaystyle \sup_{u\in \UU} \max_{j,k\in [p], j\neq k} |(\En-\Ep)[S_{ujk}^2]| &\displaystyle \leq  \sup_{u\in \UU^\epsilon} \max_{j,k\in [p], j\neq k} |(\En-\Ep)[S_{ujk}^2]| +  \Delta_n \\
\end{array}  \end{equation}
where $ \Delta_n := \sup_{u,u'\in \UU, \|u-u'\|\leq \epsilon} \max_{j,k \in [p], j\neq k} |(\En-\Ep)[S_{ujk}^2]-(\En-\Ep)[S_{u'jk}^2]|.$

To bound the first term, we will apply Corollary \ref{thm:RV34} with $k=1$, $\hat \UU := \UU^\epsilon\times[p]$ and the vector $\{(\bar X)_{uj}=S_{uj}, (u,j)\in \hat \UU\}$. In this case note that
 $$K^2=\Ep[ \max_{i\leq n} \sup_{u\in \UU} \max_{j,k\in [p], j\neq k} S_{ujk}^2 ] \leq \Ep[ \max_{i\leq n} \sup_{u\in\UU,j\in[p]} |v_{iuj}|^2\|f_{iu}Z_i^a\|_\infty^2] \leq \bar f^2 M_n^2L_n^2.$$
Therefore, by Corollary \ref{thm:RV34} and Markov inequality, we have with probability $1-o(1)$ that
$$ \begin{array}{rl}
\displaystyle \sup_{u\in \UU} \max_{j,k\in [p], j\neq k} |(\En-\Ep)[S_{ujk}^2]|
& \leq C n^{-1/2}M_nL_n \log^{1/2}(p|V|n)\leq C\delta_n \mu_\mathcal{W}\underf_\mathcal{U} + \Delta_n \end{array}  $$
under $M_n^2L_n^2\log(p|V|n) \leq \delta_n n \mu_\mathcal{W}^2\underf_\mathcal{U}^2$.

To control $\Delta_n$, note that
$$\begin{array}{rl}
 |(\En-\Ep)[S_{ujk}^2]-(\En-\Ep)[S_{u'jk}^2]| & \leq |\En[S_{ujk}^2-S_{u'jk}^2]|+|\Ep[S_{ujk}^2-S_{u'jk}^2]|\\
& \leq \En[|S_{ujk}-S_{u'jk}|]\sup_{u\in\mathcal{U}}\max_{i\leq n}|2S_{iu'jk}|+|\Ep[S_{ujk}^2-S_{u'jk}^2]|\\
& \lesssim \delta_n n^{-1/2} \mu_\mathcal{W}^{1/2}\underf_\mathcal{U} \bar f \sup_{u\in\mathcal{U}}\max_{i\leq n}|v_{iuj}|\|Z_i^a\|_\infty+\delta_n \mu_\mathcal{W}\underf_\mathcal{U}^2\\
& \lesssim \delta_n n^{-1/2} \mu_\mathcal{W}^{1/2}\underf_\mathcal{U} M_n L_n \log n+\delta_n \mu_\mathcal{W}\underf_\mathcal{U}^2\\
\end{array}$$
with probability $1-o(1)$ where we used (\ref{eq:Shelp1}) and (\ref{eq:Shelp2}). Therefore $\Delta_n \lesssim \delta_n\mu_\mathcal{W}\underf_\mathcal{U}^2$ with probability $1-o(1)$ as required.

To verify Assumption \ref{ass: M}(a), note that $[\partial_\theta M_u(Y_u,X,\theta_u)-\partial_\theta M_u(Y_u,X,\theta_u,a_u)]'\delta=  -f_u^2 K_\varpi(W)\bar r_{uj}Z_{-j}^a\delta$, so that by Cauchy-Schwartz,  we have
$$\begin{array}{rl}
 \En[\partial_\theta M_u(Y_u,X,\theta_u)-\partial_\theta M_u(Y_u,X,\theta_u,a_u)]'\delta & \leq \| f_u\bar r_{uj}\|_{n,\varpi}\| f_u Z_{-j}^a\delta\|_{n,\varpi}  \leq C_{un} \| f_u Z_{-j}^a\delta\|_{n,\varpi} \end{array} $$
where we choose $C_{un}$ so that $\{ C_{un} \geq \max_{j\in[p]} \| f_u\bar r_{uj}\|_{n,\varpi}: u\in \mathcal{U}\}$ with probability $1-o(1)$. To bound $C_{un}$, by Lemma \ref{ControlRUJ}, uniformly over $u\in\UU, j\in[p]$ we have with probability $1-o(1)$  $$\|f_u\bar r_{uj}\|_{n,\varpi} = \|f_uZ_{-j}^a(\gamma_u^j-\bar\gamma_u^j)\|_{n,\varpi}\lesssim \underf_u\{\Pr(\varpi)\}^{1/2}\{n^{-1} s\log(p|V|n)\}^{1/2}$$
so that setting $C_{un}=\underf_u\{\Pr(\varpi)\}^{1/2}\{n^{-1} s\log(p|V|n)\}^{1/2}$ suffices.

Next we show that Assumption \ref{ass: M}(b) holds. First, by (\ref{eq:load}) and the corresponding bounds, note the uniform convergence of the loadings
$$ \sup_{u\in \mathcal{U}, j, k\in[p], j\neq k} (|\En[ S_{ujk}^2] - \Ep[S_{ujk}^2]  | + |(\En-\Ep)[ K_\varpi(W)f_u^2|Z_j^aZ_{-jk}^a|^2]|) \leq \delta_n \mu_\mathcal{W}\underf_\mathcal{U}^2   $$ so that $\En[ S_{ujk}^2]/\Ep[S_{ujk}^2]= 1+o_P(1)$. It follows that  $\tilde c$ is bounded above by a constant for $n$ large enough. Indeed, uniformly over $u\in\mathcal{U}$, $j\in[p]$, since $c\underf_u\leq \Ep[|f_uv_{uj}Z^a_k|^2\mid \varpi]^{1/2}\leq C\underf_u$, with probability $1-o(1)$ we have
$c \underf_u\Pr(\varpi)^{1/2} \leq \widehat\Psi_{u0jj} \leq C \underf_u\Pr(\varpi)^{1/2}$ so that $c/C \leq \|\widehat\Psi_{u0}\|_\infty\|\widehat\Psi_{u0}^{-1}\|_\infty \leq C/c$.

Assumption \ref{ass: M}(c) follows directly from the choice of $M_u(Y_u,X_u,\theta)=K_\varpi(W)f_u^2(Z_j^a-Z^a_{-j}\theta)^2$ with $\bar q_{A_u} = \infty$.

The result for the rate of convergence then follows from Lemma \ref{Lemma:PostLassoMRateRaw}, namely
\begin{equation}\label{RateInt} \|f_u X_u'(\hat\gamma_u^j-\gamma_u^j) \|_{n,\varpi} \lesssim \frac{\|\widehat \Psi_{u0}\|_\infty}{\bar\kappa_{u,2\cc}}\sqrt{\frac{s\log(p|V|n)}{n}}+ C_{un} \lesssim \frac{\underf_u\Pr(\varpi)^{1/2}}{\bar\kappa_{u,2\cc}}\sqrt{\frac{s\log(p|V|n)}{n}} \end{equation}

By Lemma \ref{Lemma:Matrices} we have that for sparse vectors, $\|\theta\|_0\leq \ell_n s$ satisfies
$$ \|f_uZ^a_{-j}\theta\|_{n,\varpi}^2/\Ep[K_\varpi(W)f_u^2(Z^a_{-j}\theta)^2] = 1+o_P(1)$$
so that $\semax{\ell_n s,uj} \leq C\underf_u^2\Pr(\varpi)$ and $\hat s_{uj} \leq \min_{m\in \mathcal{M}_u}\semax{m,uj}L_u^2 \leq Cs$ provided $L_u^2 \lesssim s\{\underf_u^2\Pr(\varpi)\}^{-1}$. Indeed, with probability $1-o(1)$, we have $\|\widehat\Psi_{u0}^{-1}\|_\infty \leq C\underf_u^{-1}\Pr(\varpi)^{-1/2}$, so that $L_u\lesssim \underf_u^{-1}\Pr(\varpi)^{-1/2} \frac{n}{\lambda} \{ C_{un} + L_{un} \}$. Moreover, we can take $C_{un} \lesssim \underf_u\{\Pr(\varpi) n^{-1}s\log(p|V|n)\}^{1/2}$, and $L_{un} \lesssim \{n^{-1}s\log(p|V|n)\}^{1/2}$ in Assumption \ref{ass: M}  because
$$\begin{array}{rl}
 & |\{\En[\partial_\gamma M_u(Y_u,X_u,\hat\gamma_u^j)-\partial_\gamma M_u(Y_u,X_u,\gamma_u^j)]\}'\delta| \\
& =  2|\En[K_\varpi(W)f_u^2\{X_u'(\hat\gamma_u^j-\gamma_u^j)\}X_u'\delta]|\\
&\leq 2\|f_u X_u'(\hat\gamma_u^j-\gamma_u^j) \|_{n,\varpi} \|f_u X_u'\delta\|_{n,\varpi}=:L_{un}\|f_u X_u'\delta\|_{n,\varpi},\\
 \end{array}$$
where the last inequality hold by (\ref{RateInt}) since $\bar\kappa_{u,2\cc} \geq c\underf_u\{\Pr(\varpi)\}^{1/2}$. The bound on the restricted eigenvalue $\bar \kappa_{u,2\cc}$ holds\footnote{Note that there are two restricted eigenvalues definitions, one used for the quantile regression ($\kappa_{u,2\cc}$), and another used here for the weighted lasso ($\bar\kappa_{u,2\cc}$). It is a consequence of the use of different norms.} by arguments similar to (\ref{qA:1}) and using that $\|\delta\|_1\leq C\sqrt{s}\|\delta\|$ for any $\delta \in \Delta_{u, 2\cc}$, and since for any $\|\delta\|=1$, we have
$$\begin{array}{rl}
  c\underf_u\Pr(\varpi) &\leq \Ep[K_\varpi(W)f_u(Z^a\delta)^2] \\
  & \leq \{\Ep[K_\varpi(W)f_u^2(Z^a\delta)^2]\}^{1/2} \{\Ep[K_\varpi(W)(Z^a\delta)^2]\}^{1/2} \\
  & \leq \{\Ep[K_\varpi(W)f_u^2(Z^a\delta)^2]\}^{1/2} C\{\Pr(\varpi)\}^{1/2} \end{array}$$
where the first inequality follows from the definition of $\underf_u$, $\|\delta\|=1$,  and Condition CI, so that we have $\{\Ep[K_\varpi(W)f_u^2(Z^a\delta)^2]\}^{1/2} \geq c'\underf_u\{\Pr(\varpi)\}^{1/2}$.

Return to the rate of convergence we have by (\ref{RateInt}) and $\bar\kappa_{u,2\cc} \geq c\underf_u\{\Pr(\varpi)\}^{1/2}$ that
\begin{equation}\label{RateInt2} \|f_u X_u'(\hat\gamma_u^j-\gamma_u^j) \|_{n,\varpi} \lesssim \frac{\underf_u\Pr(\varpi)^{1/2}}{\bar\kappa_{u,2\cc}}\sqrt{\frac{s\log(p|V|n)}{n}} \lesssim \sqrt{\frac{s\log(p|V|n)}{n}}  \end{equation}
and the result follows by noting that $\|f_u X_u'(\hat\gamma_u^j-\gamma_u^j) \|_{n,\varpi} \geq c \underf_u\Pr(\varpi)^{1/2}\|\hat\gamma_u^j-\gamma_u^j\|$ with probability $1-o(1)$ by arguments similar to (\ref{qA:1}) under Condition CI.

The sparsity result follows from Lemma \ref{Lemma:LassoMSparsity}. The result for Post Lasso follows from Lemma \ref{Lemma:LassoMRateRaw} under the growth requirements in Condition CI.

\end{proof}

\begin{proof}[Proof of Theorem \ref{theorem:semiparametric}]
We will verify Assumptions \ref{ass: S1} and \ref{ass: AS}, and the result follows from Theorem \ref{theorem:semiparametricMain}.
The estimate of the nuisance parameter is constructed from the estimators in Steps 1 and 2 of the Algorithm.

For each $u=(a,\tau,\varpi)\in \UU$ and $j\in [p]$, let $W_{uj}=(W,X_a,Z^a,v_{uj},r_u)$, where $v_{uj} = f_u(Z_j^a - Z_{-j}^a\gamma_u^j)$ and let $\theta_{uj} \in \Theta_{uj} = \{ \theta \in \RR : |\theta-\beta_{uj}|\leq c/\log n \}$ (Assumption \ref{ass: S1}(i) holds). The score function is $$ \psi_{uj}(W_{uj},\theta,\eta_{uj}) = K_\varpi(W)\{ \tau - 1\{X_a \leq Z_j^a\theta+Z^a_{-j}\beta_{u,-j}+r_u\}\}f_u(Z_j^a - Z_{-j}^a\gamma_u^j)$$ where the nuisance parameter is $\eta_{uj}=(\beta_{u, -j},\gamma_u^j,r_u)$ and the last component is a function $r_{u}=r_{u}(X)$. Recall that $K_\varpi(W) \in \{0,1\}$ and let $
a_n = \max(n,p,|V|)$. Define the nuisance parameter set
$\mT_{uj} = \{ \eta =(\eta^{(1)},\eta^{(2)}, \eta^{(3)}): \|\eta - \eta_{uj}\|_e \leq \tau_n\}$ where
 $\|\eta-\eta_{uj}\|_e = \|(\delta_\eta^{(1)},\delta_\eta^{(2)},\delta_\eta^{(3)})\|_e = \max \{ \|\delta_\eta^{(1)}\|, \|\delta_\eta^{(2)}\|, \Ep[|\delta_\eta^{(3)}|^2]^{1/2}\}$, and
 $$ \tau_n := C\sup_{u\in\mathcal{U}} \frac{1}{1\wedge \underf_\mathcal{U}}\sqrt{\frac{s \log a_n}{n \mu_\mathcal{W}}}$$

The differentiability of the mapping $(\theta,\eta) \in \Theta_{uj}\times \mT_{uj}\mapsto \Ep\psi_{uj}(W_{uj},\theta,\eta)$ follows from the differentiability of the conditional probability distribution of $X_a$ given $X_{V\backslash \{a\}}$ and $\varpi$. Let $\eta=(\eta^{(1)},\eta^{(2)},\eta^{(3)})$, $\delta_\eta = (\delta_\eta^{(1)},\delta_\eta^{(2)},\delta_\eta^{(3)})$, and $\theta_{\bar r} = \theta + \bar{r}\delta_\theta$, $\eta_{\bar{r}} = \eta + \bar{r} \delta_\eta$.

To verify Assumption \ref{ass: S1}(v)(a) with $\alpha = 2$, for any $(\theta,\eta),(\bar \theta,\bar\eta)\in \Theta_{uj}\times \mT_{uj}$ note that $f_{X_a\mid X_{-a},\varpi}$ is uniformly bounded from above by $\bar {f}$, therefore
$$\begin{array}{rl}
\Ep[ \{ \psi_{uj}(W_{uj},\theta,\eta)-\psi_{uj}(W_{uj},\bar \theta,\bar\eta)\}^2]^{1/2} \\
\leq \bar{f}\Ep[|Z^a_{-j}(\eta^{(2)}-\bar \eta^{(2)})|^2]^{1/2} + \bar{f}^2\Ep[ (Z^a_j-Z^a_{-j}\bar\eta^{(2)})^2\{ |\eta^{(3)}-\bar \eta^{(3)}|+ |Z^a_{-j}(\eta^{(1)}-\bar \eta^{(1)})|+ |Z_j^a(\theta-\bar\theta)|\}]^{1/2}\\
\leq C\|\eta^{(2)}-\bar \eta^{(2)}\| + \bar f\Ep[(Z^a_j-Z^a_{-j}\bar\eta^{(2)})^4]^{1/4}\{ \Ep[|\eta^{(3)}-\bar \eta^{(3)}|^2]^{1/4} + C\|\eta^{(1)}-\bar \eta^{(1)}\|+|\theta-\bar\theta| \}^{1/2}\\
\leq C'|\theta-\bar\theta|^{1/2}\vee \|\eta-\bar \eta\|_e^{1/2}
\end{array}$$
for some constance $C'<\infty$ since by Condition CI we have $\Ep[|Z^a\bar{\xi}|^2]^{1/2} \leq C\|\bar{\xi}\|$ for all vectors $\bar{\xi}$, and the conditions $\sup_{u\in\UU,j\in[p]}\|\gamma_u^j\|\leq C$, $\sup_{\theta \in \Theta_{uj}}|\theta|\leq C$,  and $\sqrt{s \log(a_n)} \leq \delta_n \sqrt{n}$. This implies that $\|\eta^{(2)}-\bar\eta^{(2)}\|\leq \|\eta^{(2)}-\eta_{uj}^{(2)}\|+ \|\eta_{uj}^{(2)}-\bar\eta^{(2)}\| \leq 1$ so that $\|\eta^{(2)}-\bar\eta^{(2)}\|\leq \|\eta^{(2)}-\bar\eta^{(2)}\|^{1/2}$.

To verify Assumption \ref{ass: S1}(v)(b), let $t_{\bar r}= Z_j^a\theta_{\bar{r}}+Z^a_{-j}\eta_{\bar{r}}^{(1)}+\eta_{\bar r}^{(3)}$. We have
$$\begin{array}{rl}
\left.  \partial_r\Ep(\psi_{uj}(W_{uj},\theta+r\delta_\theta,\eta+r\delta_\eta)) \right|_{r=\bar r}  = \\ -\Ep[K_\varpi(W)f_{X_a\mid X_{-a},\varpi}(t_{\bar r})(Z_j^a - Z^a_{-j}\eta_{\bar{r}}^{(2)})\{Z_j^a\delta_\theta+Z_{-j}^a\delta_\eta^{(1)}+\delta_\eta^{(3)}\}] \\
-\Ep[K_\varpi(W)\{\tau-F_{X_a\mid X_{-a},\varpi}(t_{\bar r})\}Z_{-j}^a\delta_\eta^{(2)}] \end{array}$$
Applying Cauchy-Schwartz  we have that
   $$\begin{array}{lr}
\left|\left.  \partial_r\Ep(\psi_{uj}(W_{uj},\theta+r\delta_\theta,\eta+r\delta_\eta)) \right|_{r=\bar r}\right|
\\
\leq \bar f \Ep[(Z_j^a - Z^a_{-j}\eta_{\bar{r}}^{(2)})^2]^{1/2}\{\Ep[ (Z_j^a)^2]^{1/2}|\delta_\theta|+\Ep[(Z_{-j}^a\delta_\eta^{(1)})^2]^{1/2}+\Ep[|\delta_\eta^{(3)}|^2]^{1/2}\} +\bar{f}\Ep[(Z_{-j}^a\delta_\eta^{(2)})^2]^{1/2}\\
\leq \bar B_{1n} (|\delta_\theta| \vee \|\eta-\eta_{uj}\|_e)\end{array}$$
where $\bar B_{1n} \leq C$ by the same arguments of bounded (second) moments of linear combinations.

Assumption \ref{ass: S1}(v)(c) follows similarly as
 $$\begin{array}{rl}
\left.  \partial_r^2\Ep(\psi_{uj}(W_{uj},\theta+r\delta_\theta,\eta+r\delta_\eta)) \right|_{r=\bar r}  =\\
-\Ep[K_\varpi(W)f_{X_a\mid X_{-a},\varpi}'(t_{\bar r})(Z_j^a - Z^a_{-j}\eta_{\bar{r}}^{(2)})\{Z_j^a\delta_\theta+Z_{-j}^a\delta_\eta^{(1)}+\delta_\eta^{(3)}\}^2] \\
+2\Ep[K_\varpi(W)f_{X_a\mid X_{-a},\varpi}(t_{\bar r})(Z^a_{-j}\delta_\eta^{(2)})\{Z_j^a\delta_\theta+Z_{-j}^a\delta_\eta^{(1)}+\delta_\eta^{(3)}\}] \\
 \end{array}$$
and under $|f'_{X_a\mid X_{-a},\varpi}|\leq \bar f'$, from Cauchy-Schwartz inequality we have
  $$\begin{array}{ll}
\left|\left.  \partial_r^2\Ep(\psi_{uj}(W_{uj},\theta+r\delta_\theta,\eta+r\delta_\eta)) \right|_{r=\bar r} \right| \\
 \leq
|\bar f'_n\Ep[(Z_j^a - Z^a_{-j}\eta_{\bar{r}}^{(2)})^2]^{1/2}\{ \Ep[(Z_j^a)^4]|\delta_\theta|^2+\Ep[(Z_{-j}^a\delta_\eta^{(1)})^4]^{1/2}\}+C\Ep[\{\delta_\eta^{(3)}\}^2] \\
+2\bar f\Ep[(Z^a_{-j}\delta_\eta^{(2)})^2]^{1/2}\{\Ep[(Z_j^a)^2]^{1/2}|\delta_\theta|+\Ep[(Z_{-j}^a\delta_\eta^{(1)})^2]^{1/2}+\Ep[\{\delta_\eta^{(3)}\}^2]^{1/2}\} \\
\leq \bar B_{2n}  (\delta_\theta^2\vee \|\eta-\eta_{uj}\|_e^2) \end{array}$$
where  $\bar B_{2n} \leq C(1 + \bar f'_n)$ by the same arguments of bounded (fourth) moments as before and using that $|\Ep[(Z_j^a - Z^a_{-j}\eta_{\bar{r}}^{(2)})(\delta_\eta^{(3)})^2]|\leq \{ \Ep[(Z_j^a - Z^a_{-j}\eta_{\bar{r}}^{(2)})^2(\delta_\eta^{(3)})^2]\}^{1/2}\Ep[(\delta_\eta^{(3)})^2]^{1/2} \leq C\Ep[(\delta_\eta^{(3)})^2]$.

To verify the near orthogonality condition, note that for all $u\in\UU$ and $j\in[p]$, since by definition $f_u = f_{X_a\mid X_{-a},\varpi}(Z^a\beta_u+r_u)$  we have
$$ |\mathrm{D}_{u,j,0}[\tilde \eta_{uj} - \eta_{uj}]| = |-\Ep[K_\varpi(W)f_u\{Z_{-j}^a(\tilde\eta^{(2)} - \eta_{uj}^{(2)})+ r_{u}\}v_{uj} ]| \leq \delta_n n^{-1/2} $$
by the relations $\Ep[K_\varpi(W)(\tau - F_{X_a\mid X_{-a},\varpi}(Z^a\beta_u+r_u))Z_{-j}^a]=0$ and $\Ep[K_\varpi(W)f_uZ_{-j}^av_{uj}]=0$ implied by the model, and $|\Ep[ K_\varpi(W) f_u r_{u}v_{uj} ]|\leq \delta_n n^{-1/2}$ by Condition CI. Thus, condition (\ref{eq:cont}) holds.

Furthermore, since $\Theta_{uj} \subset \theta_{uj} \pm C/\log n$, for $J_{uj}=\partial_\theta\Ep[\psi_{uj}(W_{uj},\theta_{uj},\eta_{uj})]=\Ep[K_\varpi(W)f_uZ_j^av_{uj}]=\Ep[K_\varpi(W)v_{uj}^2] =\Ep[v_{uj}^2| \varpi]\Pr(\varpi)$ as $\Ep[K_\varpi(W)f_uZ_{-j}^av_{uj}]=0$, we have that for all $\theta\in\Theta_{uj}$
$$ \Ep[\psi_{uj}(W_{uj},\theta,\eta_{uj})] = J_{uj}(\theta-\theta_{uj})+\frac{1}{2}\partial_\theta^2\Ep[\psi_{uj}(W_{uj},\bar\theta,\eta_{uj})](\theta-\theta_{uj})^2 $$ where $|\partial_\theta^2\Ep[\psi_{uj}(W_{uj},\bar\theta,\eta_{uj})]|\leq \bar f'\Ep[|Z_j^a|^2|v_{uj}|\mid \varpi] \Pr(\varpi) \leq \bar f'\Ep[|Z_j^a|^4| \varpi]^{1/2}\Ep[|v_{uj}|^2| \varpi]^{1/2} \Pr(\varpi) \leq C\bar f' \Pr(\varpi)$ so that for all $\theta\in\Theta_{uj}$ $$|\Ep[\psi_{uj}(W_{uj},\theta,\eta_{uj})]| \geq \{|\Ep[v_{uj}^2\mid \varpi]|- (C^2\bar f')/\log n\}\Pr(\varpi)|\theta-\theta_{uj}|$$
and we can take $j_n \geq c\inf_{\varpi\in\mathcal{W}}\Pr(\varpi) = c\mu_\mathcal{W}$.

Next we verify Assumption \ref{ass: AS} with $\mT_{ujn} = \{ \eta =(\beta,\gamma,0) : \|\beta\|_0\leq Cs, \|\gamma\|_0\leq Cs,  \|\beta-\beta_{u,-j}\| \leq C\tau_n, \|\gamma - \gamma_u^j\| \leq C\tau_n, \|\gamma-\gamma_u^j\|_1 \leq C\sqrt{s}\tau_n \}$.
We will show that $\hat \eta_{uj} = (\widetilde\beta_{u,-j},\widetilde \gamma_u^j,0) \in \mT_{ujn}$ with probability $1-o(1)$, uniformly over $u\in \UU$ and $j\in [p]$.

Under Condition CI and the choice of penalty parameters, by Theorems \ref{theorem:rateQRgraph} and \ref{theorem:rateLASSOgraph}, with probability $1-o(1)$, uniformly over $u\in\mathcal{U}$ we have
$$ \|\widetilde \beta_u - \beta_u\| \leq C\tau_n,  \ \  \max_{j\in[p]}\sup_{u\in\UU} \|\widetilde \gamma_{u}^j - \gamma_{u}^j\| \leq C\tau_n, \ \ \mbox{and} \ \ \max_{j\in[p]} \sup_{u\in\UU} \|\widetilde \gamma_{u}^j\|_0 \leq \bar C s, $$
further by thresholding we can achieve $\sup_{u\in\UU} \|\widetilde \beta_u\|_0 \leq \bar C s$ using Lemma \ref{Lemma:Bound2nNormSecond}.

Next we establish the entropy bounds. For $\eta \in \mT_{ujn}$ we have that $$\psi_{uj}(W_{uj},\theta,\eta)= K_\varpi(W)( \tau - 1\{ X_a \leq Z_j^a\theta+Z^a_{-j}\beta_{-j}\} )f_u\{Z_j^a - Z_{-j}^a\gamma\} $$
It follows that  $\mF_1\subset \mathcal{W}\mG_1\mG_2\mG_3 \cup \bar{\mF_0}$ where $\bar{\mF_0}= \{ \psi_{uj}(W_{uj},\theta,\eta_{uj}) : u\in \UU, j\in[p], \theta \in \Theta_{uj}\}$, $\mG_1 =\{ \tau - 1\{ X_a\leq Z^a\beta\} : \|\beta\|_0 \leq Cs,  \tau \in \mathcal{T}, a \in V\}$, $\mG_2 =\{ Z^a \to Z^a(1,-\gamma), \|\gamma\|_0\leq Cs, \|\gamma\|\leq C, a\in V\}$, $\mG_3 = \{ f_u : u \in \UU\}$.
Under Condition CI, $\mathcal{W}$ is a VC class of sets with VC index $d_W$ (fixed). It follows that $\mG_1$ and $\mG_2$ are $p$ choose $O(s)$ VC-subgraph classes with VC indices at most
$O(s)$. Therefore, ${\rm ent}(\mG_1)\vee {\rm ent}(\mG_2)\vee {\rm ent}(\mathcal{W})\leq C s \log(a_n/\epsilon)+Cd_W\log(e/\epsilon)$ by  Theorem 2.6.7 in \cite{vdV-W} and by standard arguments. Also, since $f_u$ is Lipschitz in $u$ by Condition CI, we have ${\rm ent}(\mG_3)\leq (1+d_W)\log( a_nL_f/\epsilon)$. Moreover, an envelope $F_G$ for $\mF_1$ satisfies $$\begin{array}{rl}
 \Ep[F_G^q] & = \Ep[ \sup_{u\in\UU,j\in[p],\|\gamma-\gamma_u^j\|_1\leq C\sqrt{s}\tau_n} |v_{uj} - f_uZ^a_{-j}(\gamma-\gamma_u^j)|^q ]\\
 & \leq 2^{q-1}\Ep[ \sup_{ u\in\UU, j\in [p]} |v_{uj}|^q]  + 2^{q-1}\bar f\Ep[\max_{a\in V}\|Z^a\|_\infty^q]\{C\sqrt{s}\tau_n\}^q \\
 & \leq 2^{q-1}L_n^q + 2^{q-1}\bar f\{M_nC\sqrt{s}\tau_n\}^q \leq 2^{q}L_n^q \\
 \end{array}$$ since $M_nC\sqrt{s}\tau_n \leq \delta_n L_n/\bar f$ and $\delta_n \leq 1$ for $n$ large.

Next we bound the entropy in $\bar{\mF_0}$. Note that for any $\psi_{uj}(W_{uj},\theta,\eta_{uj})\in \bar{\mF_0}$, there is some $\delta \in [-C,C]$ such that
$$\psi_{uj}(W_{uj},\theta,\eta_{uj})= K_\varpi(W)\{ \tau - 1\{ X_a \leq Z_j^a\delta+Q_{X_a}(\tau\mid X_{-a},\varpi)\} \}v_{uj}$$
and therefore $\bar{\mF_0} \subset \mathcal{W}\{\mathcal{T}- \phi(\mathcal{V})\}\mathcal{L}$ where $\phi(t)=1\{t\leq 0\}$, $\mathcal{V}=\cup_{a\in V,j\in[p]}\mathcal{V}_{aj}$ with $$\mathcal{V}_{aj}:=\{ X_a - Z_j^a\delta - Q_{X_a}(\tau\mid X_{-a},\varpi): \tau\in \mathcal{T}, \varpi\in\mathcal{W}, |\delta| \leq C\},$$
and $\mathcal{L} = \cup_{a\in V, j\in[p]}(\mathcal{L}_{aj}+\{v_{\bar uj}\})$ where $\mathcal{L}_{aj} = \{ (X,W)\mapsto v_{uj} - v_{\bar uj} = f_uZ_{-j}^a(\gamma_u^j-\gamma_{\bar u}^j) : u\in\UU\}$.  Note that each $\mathcal{V}_{aj}$ is a VC subgraph class of functions with index $1+Cd_W$ as $\{Q_{X_a}(\tau| X_{-a},\varpi) : (\tau,\varpi)\in \mathcal{W}\times\mathcal{T}\}$ is a VC-subgraph with VC-dimension $Cd_W$ for every $a\in V$. Since $\phi$ is monotone, $\phi(\mathcal{V})$ is also the union of VC-dimension of order $1+Cd_W$.

Letting $F_1=1$ be an envelope for $\mathcal{W}$ and $\mathcal{T}-\phi(\mathcal{V})$. By Lemma \ref{Lemma:LipsGamma}, it follows that $\|\gamma_u^j-\gamma_{u'}^j\|\leq  L_\gamma \{\|u-u'\|+\|\varpi-\varpi'\|^{1/2}\}$ for some $L_\gamma$ satisfying $\log(L_\gamma) \leq C\log(p|V|n)$ under Condition CI. Therefore, $|v_{uj}-v_{\bar uj}| = |Z_{-j}^a(\gamma_u^j-\gamma_{\bar u}^j)|\leq \|Z^a\|_\infty \sqrt{p}\|\gamma_u^j-\gamma_{\bar u}^j\|$. For a choice of envelope $F_a = M_n^{-1}\|Z^a\|_\infty +2\sup_{u\in \UU} |v_{uj}|$ which satisfies $\|F_a\|_{P,q} \lesssim L_n$, we have
$$\begin{array}{rl}
\log N(\epsilon \|F_a\|_{Q,2},\mathcal{L}_{aj},\|\cdot\|_{Q,2}) & \leq \log N(\frac{\epsilon}{M_n} \| \ \|Z^a\|_\infty\|_{Q,2},\mathcal{L}_{aj},\|\cdot\|_{Q,2}) \\
&\leq  \log N(\epsilon /\{ M_n\sqrt{p}L_\gamma\},\mathcal{U},|\cdot|)  \leq Cd_u\log(M_npL_\gamma/\epsilon)\end{array}$$

Since $\mathcal{L} = \cup_{a\in V,j\in[p]}(\mathcal{L}_{aj}+\{v_{\bar uj}\})$, taking $F_L = \max_{a\in V} F_a$, we have that
 $$\begin{array}{rl}
\log N(\epsilon \|F_LF_1\|_{Q,2},\bar\mF_0,\|\cdot\|_{Q,2}) & \leq \log N(\frac{\epsilon}{4} \|F_1\|_{Q,2},\mathcal{W},\|\cdot\|_{Q,2})+\log N(\frac{\epsilon}{4} \|F_1\|_{Q,2},\mathcal{T}-\phi(\mathcal{V}),\|\cdot\|_{Q,2})\\
 & + \log \sum_{a\in V, j\in [p]}N(\frac{\epsilon}{2} \|F_a\|_{Q,2},\mathcal{L}_{aj},\|\cdot\|_{Q,2})\\
&\leq \log (p|V|) + 1+C'\{d_W+d_u\}\log(4eM_n|V|pL_\gamma/\epsilon)\\ \end{array}$$
where the last line follows from the previous bounds.

Next we verify the growth conditions in Assumption \ref{ass: AS} with the proposed $\mF_1$ and $K_n \lesssim CL_n$. We take $s_{n(\UU,p)}= (1+ d_W)s$ and $a_n = \max\{ n,p,|V| \}$. Recall that $\bar B_{1n} \leq C$, $\bar B_{2n} \leq C$, $j_n\geq c \mu_\mathcal{W}$. Thus, we have $\sqrt{n}(\tau_n/j_n)^2 \lesssim \sqrt{n}\frac{s\log(p|V|n)}{n(1\wedge \underf_\mathcal{U}^2)\mu_\mathcal{W}^3}\leq \delta_n$ under $s^2\log^2(p|V|n) \leq n(1\wedge \underf_\mathcal{U}^4)\mu_\mathcal{W}^6$. Moreover, $(\tau_n/j_n)^{\alpha/2}\sqrt{s_{n(\mathcal{U},p)}\log(a_n)} \lesssim \sqrt[4]{\frac{(1+ d_W)^3s^3\log^3(p|V|n)}{n(1\wedge \underf_\mathcal{U}^2)\mu_\mathcal{W}^3}}\lesssim \delta_n$ under $d_W$ fixed and $s^3\log^3(p|V|n) \leq \delta_n^4 n(1\wedge \underf_\mathcal{U}^2)\mu_\mathcal{W}^3$ and $s_{n(\mathcal{U},p)}n^{-\frac{1}{2}}K_n\log(a_n)\log n \lesssim  (1+ d_W)sn^{\frac{1}{q}-\frac{1}{2}}M_n\log(p|V|n)\log n \leq \delta_n$ under our conditions. Finally, the conditions of Corollary \ref{theorem: general bsMain} hold with $\rho_n = (1+d_W)$ since the score is the product of VC-subgraph classes of function with VC index bounded by $C(1+d_W)$.
\end{proof}

\begin{proof}[Proof of Theorem \ref{theorem:rateQRgraphP}]

We will invoke Lemma \ref{Theorem:L1QRnpPrediction} with $\bar \beta_u$ as the estimand and $r_{iu}=X_{i,-a}'(\beta_u-\bar\beta_u)$, therefore $\Ep[ K_\varpi(W)(\tau - 1\{ X_a \leq X_{-a}'\bar\beta_u + r_{u} \})X_{-a}]=0$. To invoke the lemma we verify that the events $\Omega_1$, $\Omega_2$, $\Omega_3$ and $\Omega_4$ hold with probability $1-o(1)$
$$\begin{array}{rl}
\Omega_1 & :=\{ \lambda_u\geq c |S_{uj}|/\hat\sigma_{a\varpi j}^{X}, \ \mbox{for all} \ \ u\in\mathcal{U}, j\in V\}, \\
\Omega_2 &:=\{\hat R_u(\bar\beta_u) \leq \bar R_{u\xi} : u\in\mathcal{U}\}\\
\Omega_3 & :=\left\{ \sup_{u\in\mathcal{U}, 1/\sqrt{n} \leq \|\delta\|_{1,\varpi} \leq \sqrt{n}} |\En[g_u(\delta,X,W)-\Ep[g_u(\delta,X,W)\mid X_{-a},W]]|/\|\delta\|_{1,\varpi} \leq t_3 \right\}\\
\Omega_4 & :=\{ K_{u} \hat\sigma_{a\varpi j}^{X} \geq |\En[h_{uj}(X_{-a},W)]|,  \mbox{for all} \ u\in\UU, j\in V\backslash\{a\} \}
\end{array}$$
where $g_u(\delta,X,W)= K_\varpi(W)\{\rho_\tau(X_a-X_{-a}'(\bar\beta_u+\delta))-\rho_\tau(X_a-X_{-a}'\bar\beta_u)\}$, $h_{uj}(X_{-a},W) =\Ep[ K_\varpi(W)\{\tau-F_{X_a\mid X_{-a},W }(X_{-a}'\bar\beta_u+r_u)\}X_{j}\mid X_{-a},W  ]$.

By Lemma \ref{Lemma:PenaltyMisspecification} with $\xi=1/n$, by setting $\lambda_u=\lambda_0=c2(1+1/16) \sqrt{2\log(8|V|^2\{ne/d_W\}^{2d_W}n)/n}$, we have  $\Pr(\Omega_1)=1-o(1)$. By Lemma \ref{Lemma:PQGMControlR}, setting $\bar R_{u\xi} = C s(1+d_W)\log(|V|n)/n$ we have $\Pr(\Omega_2)=1-o(1)$ for some $\xi = o(1)$.
By Lemma \ref{Lemma:PQGMOmega3} we have $\Pr(\Omega_3)=1-o(1)$ by setting $t_3 :=  C\sqrt{(1+d_W)\log\left(|V| n M_n/\xi \right)} $.
Finally, by Lemma \ref{Lemma:PredictionOmega4} with  $K_u =C \sqrt{\frac{(1+d_W) \log(|V|n)}{n}}$ we have $\Pr(\Omega_4)=1-o(1)$.

Moreover, we have that $\|\bar \beta_u\|_{1,\varpi} \leq \sqrt{s}\|\bar \beta_u\|_{2,\varpi} \leq C\sqrt{s}= o(\sqrt{n})$ and $\frac{1}{\lambda_u(1-1/c)}\bar R_{u\xi} = o(\sqrt{n})$ for all $u\in\UU$. Finally, we verify condition (\ref{ConditionIdenfitication}) holds for all $$\delta\in A_u := \Delta_{u,2\cc}\cup\{v:\|v\|_{1,\varpi}\leq 2\cc\bar R_{u\xi}/\lambda_u, \|\sqrt{f_u}X_{-a}'v\|_{n,\varpi} \geq C\sqrt{s(1+d_W)\log(n|V|)/n}/\kappa_{u,2\cc}  \},$$ $\bar q_{A_u}/4 \geq (\sqrt{\bar f}+1) \|r_{u}\|_{n,\varpi} + \left[\lambda_u+t_3+K_u\right]\frac{3\cc\sqrt{s}}{\kappa_{u,2\cc}}$ and $\bar q_{A_u} \geq \{ 2\cc\left(1+ \frac{t_3+K_u}{\lambda_u}\right)\bar R_{u\gamma} \}^{1/2}$.

Consider the matrices $\En[K_\varpi(W)f_uX_{-a}X_{-a}']$ and $\Ep[K_\varpi(W)f_uX_{-a}X_{-a}']$. By Lemma \ref{Lemma:Matrices}, with probability $1-o(1)$, it follows that we can take $\eta = \eta_n = C M_n\sqrt{s(1+d_W)\log(|V|n)}\log(1+s)\{\log n\}/\sqrt{n}$ and $D_{kk} = 2\eta$ in Lemma \ref{lem:transfer}. (Note that we increase $\delta_n$ by a factor of $\sqrt{\log n}$.) Therefore, with at least the same probability we have (taking $s\geq 2$)
\begin{equation}\label{qA:1} \delta'\En[K_\varpi(W)f_uX_{-a}X_{-a}']\delta  \geq \delta'\Ep[K_\varpi(W)f_uX_{-a}X_{-a}']\delta - 4\eta\|\delta\|_1^2/s\end{equation}
and by definition of $\underf_u$ we have
$$ \En[K_\varpi(W)f_u|X_{-a}'\delta|^2]  \geq \underf_u \Ep[K_\varpi(W)|X_{-a}'\delta|^2] - 4\eta\|\delta\|_1^2/s \geq c\underf_u \Pr(\varpi)\|\delta\|^2 - 4\eta\|\delta\|_1^2/s .$$
For $\delta \in \Delta_{u,2\cc}$ we have $\|\delta\|_1 \leq C\|\delta\|_{1,\varpi}/\{\Pr(\varpi)\}^{1/2} \leq C'  \|\delta_{T_u}\|_{1,\varpi}/\{\Pr(\varpi)\}^{1/2} \leq C' \sqrt{s} \|\delta_{T_u}\|_{2}$. Note that we can assume $\|\delta\| \geq c\sqrt{s(1+d_W)\log(n|V|)/n}$ otherwise we are done. So that for $\delta \in A_u \backslash \Delta_{u,2\cc}$ we have that $\|\delta\|_1 / \|\delta\|_2 \leq  C s\sqrt{\log(|V|n)/n} / \sqrt{s(1+d_W)\log(n|V|)/n}  \leq C'\sqrt{s}$.

Similarly we have
\begin{equation}\label{qA:2} \En[K_\varpi(W)|X_{-a}'\delta|^2] \leq \Ep[K_\varpi(W)|X_{-a}'\delta|^2] + 4\eta\|\delta\|_1^2/s \leq C  \Pr(\varpi)\|\delta\|^2 - 4\eta\|\delta\|_1^2/s.\end{equation}

Under the condition that $\eta = o(\underf_\UU \mu_\mathcal{W})$, which holds by Condition P, for $n$ sufficiently large we have with probability $1-o(1)$ that
\begin{equation}\label{qA:3} \begin{array}{rl}
 \bar q_{A_u} & \displaystyle \geq \frac{c}{\bar f'} \inf_{\delta\in A_u} \frac{\En[K_\varpi(W)f_u|X_{-a}'\delta|^2]^{3/2}}{\En[K_\varpi(W)|X_{-a}'\delta|^3]} \geq \frac{c}{\bar f'} \inf_{\delta\in A_u} \frac{\En[K_\varpi(W)f_u|X_{-a}'\delta|^2]^{3/2}}{\En[K_\varpi(W)|X_{-a}'\delta|^2]\max_{i\leq n}\|X_i\|_\infty\|\delta\|_1} \\
 &  \displaystyle \geq  \frac{c}{\bar f'} \inf_{\delta\in A_u} \frac{ \{ c'\underf_u \Pr(\varpi) \|\delta\|^2 \}^{3/2}}{C'\Pr(\varpi)\|\delta\|^2\max_{i\leq n}\|X_i\|_\infty\|\delta\|_1}   \displaystyle \geq  \frac{c}{\bar f'} \inf_{\delta\in A_u} \frac{ c'\underf_u^{3/2} \Pr(\varpi)^{1/2} \|\delta\|}{C'\max_{i\leq n}\|X_i\|_\infty\|\delta\|_1}   \\
& \displaystyle  \geq  C'' \frac{\underf_\UU^{3/2}}{\bar f'} \frac{\mu_\mathcal{W}^{1/2}}{\sqrt{s}\max_{i\leq n}\|X_i\|_\infty}\\
 \end{array}\end{equation}
\noindent where $\max_{i\leq n}\|X_i\|_\infty \leq \ell_n M_n$ with probability $1-o(1)$ for any $\ell_n \to \infty$.
Therefore, under the condition $M_n s\sqrt{\log(p|V|n)} = o( \sqrt{n \mu_\mathcal{W}})$ assumed in Condition P, the conditions on $\bar q_{A_u}$ are satisfied.

By Lemma \ref{Theorem:L1QRnpPrediction}, we have uniformly over all $u=(a,\tau,\varpi)\in \UU := V\times \mathcal{T}\times \mathcal{W}$
 {\small $$
\|\sqrt{f_u}X_{-a}'(\hat\beta_u - \beta_u)\|_{n,\varpi}  \leq C \sqrt{\frac{(1+d_W)\log(n|V|)}{n} }\frac{\sqrt{s}}{\kappa_{u,2\cc}} \ \ \mbox{and} \ \ \|\hat\beta_u - \beta_u\|_{1,\varpi} \leq C \sqrt{\frac{(1+d_W)\log(n|V|)}{n} }\frac{s}{\kappa_{u,2\cc}}  $$}
\noindent where $\kappa_{u,2\cc}$ is bounded away from zero with probability $1-o(1)$ for $n$ sufficiently large.
Consider the thresholded estimators  $\hat\beta_u^{\bar \lambda}$ for $\bar \lambda =\{ (1+d_W)\log(n|V|)/n\}^{1/2}$. By Lemma \ref{Lemma:Bound2nNormSecond} we have $\|\hat\beta_u^{\bar \lambda}\|_0 \leq C s$ and the same rates of convergence as $\hat\beta_u$. Therefore, by refitting over the support of $\hat\beta_u^{\bar \lambda}$ we have by Lemma \ref{Thm:RefittedGeneric}, the estimator $\widetilde \beta_u$ has the same rate of convergence where we used that $\widehat Q_u \leq \lambda_u\|\hat\beta_u^{\bar \lambda}- \beta_u\|_{1,\varpi} \lesssim C s(1+d_W)\log(|V|n)/n$ (the other conditions of Lemma \ref{Thm:RefittedGeneric} hold as for the conditions in Lemma \ref{Theorem:L1QRnpPrediction}).

Next we will invoke Lemma \ref{Theorem:L1QRnpPrediction} for the new penalty choice and penalty loadings. (We note that minor modifications cover the new penalty loadings.)
$$\begin{array}{rl}
\Omega_1 & :=\{ \lambda_u\geq c |S_{uj}|/\{\En[K_w(W)\varepsilon_u^2X_{-a,j}^2]\}^{1/2}, \ \mbox{for all} \ \ u\in\mathcal{U}, j\in V\}, \\
\Omega_2 &:=\{\hat R_u(\bar\beta_u) \leq \bar R_{u\xi} : u\in\mathcal{U}\}\\
\Omega_3 & :=\left\{ \sup_{u\in\mathcal{U}, 1/\sqrt{n} \leq \|\delta\|_{1,\varpi} \leq \sqrt{n}} |\En[g_u(\delta,X,W)-\Ep[g_u(\delta,X,W)| X_{-a},W]]|/\{\theta_u\|\delta\|_{1,\varpi}\} \leq t_3 \right\}\\
\Omega_4 & :=\{ K_{u}\theta_u \hat\sigma_{a\varpi j}^{X}  \geq |\En[h_{uj}(X_{-a},W)]|,  \mbox{for all} \ u\in\UU, j\in V\backslash\{a\} \}\\
\Omega_5 & :=\{ \theta_u \geq \max_{j\in V}\hat\sigma_{a\varpi j}^{X} / \{\En[K_\varpi(W)\varepsilon_u^2X_j^2]\}^{1/2} \}
\end{array}$$
where event $\Omega_5$ simply makes the relevant norms equivalent, $\|\cdot\|_{1,u}\leq \|\cdot\|_{1,\varpi}\leq \theta_u\|\cdot\|_{1,u}$. Note that we can always take $\theta_u \leq 1/\{\tau(1-\tau)\}\leq C$ since $\mathcal{T}$ is a fixed compact set. 

Next we show that the bootstrap approximation of the score provides a valid choice of penalty parameter. Let $\hat \varepsilon_u := 1\{ X_a \leq X_{-a}'\widetilde\beta_u\}-\tau$. For notational convenience for $u\in \mathcal{U}$, $j\in V\backslash\{a\}$ define $$\begin{array}{rl}
  \displaystyle \bar\psi_{iuj}=\frac{K_\varpi(W_i)\varepsilon_{iu} X_{ij}}{\Ep[K_\varpi(W)\varepsilon_u^2X_j^2]^{1/2}}, \ \
\psi_{iuj}=\frac{K_\varpi(W_i)\varepsilon_{iu} X_{ij}}{\En[K_\varpi(W)\varepsilon_{u}^2X_j^2]^{1/2}}, \  \ \displaystyle \hat \psi_{iuj} =\frac{K_\varpi(W_i)\hat\varepsilon_{iu} X_{ij}}{\En[K_\varpi(W)\hat\varepsilon_{u}^2X_j^2]^{1/2}}.
 \end{array}$$
We will consider the following processes:
$$\begin{array}{rl}
\bar S_{uj} = \frac{1}{\sqrt{n}}\sum_{i=1}^n \bar  \psi_{iuj}, \ \ S_{uj} = \frac{1}{\sqrt{n}}\sum_{i=1}^n \psi_{iuj}, \ \  \overline{\mathcal{G}}_{uj}= \frac{1}{\sqrt{n}}\sum_{i=1}^ng_i \bar \psi_{iuj}, \ \
\displaystyle \widehat{\mathcal{G}}_{uj}  = \frac{1}{\sqrt{n}}\sum_{i=1}^ng_i \hat \psi_{iuj}, \ \  \end{array}
$$ and $\mathcal{N}$ is a tight zero-mean Gaussian process with covariance operator given by $\Ep[\bar\psi_{uj}\bar\psi_{u'j'}]$. Their supremum are denoted by $\bar Z_{S}:=\sup_{u\in \mathcal{U}, j\in V\backslash\{a\}} | \bar S_{uj}|$, $Z_S:=\sup_{u\in \mathcal{U}, j\in V\backslash\{a\}} | S_{uj}|$, $\bar Z_G^*:=\sup_{u\in \mathcal{U}, j\in V\backslash\{a\}} | \overline{\mathcal{G}}_{uj}|$, $\hat Z_G^*:=\sup_{u\in \mathcal{U}, j\in V\backslash\{a\}} | \widehat{\mathcal{G}}_{uj}|$,   and $Z_N:=\sup_{u\in \mathcal{U}, j\in V\backslash\{a\}} | \mathcal{N}_{uj}|$.

The penalty choice should majorate $Z_S$ and we simulate via $\hat Z_G^*$. We have that
$$ \begin{array}{rl}
| \Pr( Z_S \leq t ) - \Pr( \hat Z_G^* \leq t )| & \leq | \Pr( Z_S \leq t ) - \Pr( \bar Z_S \leq t )|+| \Pr( \bar Z_S \leq t ) - \Pr(  Z_N \leq t )|\\
& +| \Pr( Z_N \leq t ) - \Pr( \bar Z_G^* \leq t )|+| \Pr( \bar Z_G^* \leq t ) - \Pr( \hat Z_G^* \leq t )|\\
\end{array}
$$
We proceed to bound each term. We have that
$$ \begin{array}{rl}
 | Z_S - \bar Z_S| &\displaystyle  \leq \bar Z_S  \sup_{u\in\mathcal{U},j\in V\backslash\{a\}}\left| \frac{\Ep[K_\varpi(W)\varepsilon_{u}^2X_j^2]^{1/2}}{\En[K_\varpi(W)\varepsilon_{u}^2X_j^2]^{1/2}} - 1  \right| \\
 & \displaystyle  \leq \bar Z_S  \sup_{u\in\mathcal{U},j\in V\backslash\{a\}}\left| \frac{(\En-\Ep)[K_\varpi(W)\varepsilon_{u}^2X_j^2]}{\En[K_\varpi(W)\varepsilon_{u}^2X_j^2]^{1/2}\{\En[K_\varpi(W)\varepsilon_{u}^2X_j^2]^{1/2}+\Ep[K_\varpi(W)\varepsilon_{u}^2X_j^2]^{1/2}\}}  \right|\end{array} $$
Therefore, since $\{ 1\{ X_{a} \leq X_{-a}'\beta_u\} : u\in \mathcal{U}\}$ is a VC-subgraph of VC dimension $1+d_W$, and $\mathcal{W}$ is a VC class of sets of dimension $d_W$, we apply Lemma \ref{lemma:CCK} with envelope $F=\|X\|_\infty^2$ and $\sigma^2 \leq \max_{j\in V}\Ep[ X_j^4] \leq C$ to obtain with probability $1-o(1)$
$$ \sup_{u\in \mathcal{U}, j \in  V\backslash\{a\} } |(\En-\Ep)[K_\varpi(W)\varepsilon_{u}^2X_j^2]| \lesssim \delta_{1n}':= \sqrt{\frac{(1+d_W)\log(|V|n)}{n}} + \frac{M_n^2(1+d_W)\log(|V|n)}{n} $$
where $\delta_{1n}=o(\mu_\mathcal{W}^2)$ under Condition P. Note that this implies that the denominator above is bounded away from zero by $c\mu_\mathcal{W}$. Therefore,
$$ | Z_S - \bar Z_S| \lesssim_P \delta_{1n} := \bar Z_S \delta_{1n}'/\mu_\mathcal{W}.$$
where $\bar Z_S \lesssim_P \{(1+d_W)\log(n|V|)\}^{1/2}$.
By Theorem 2.1 in \cite{chernozhukov2015noncenteredprocesses}, since $\Ep[ \bar \psi_{uj}^4 ] \leq C$, there is a version of $Z_N$ such that
$$ | \bar Z_S - Z_N| \lesssim_P \delta_{2n}:=\left( \frac{M_n(1+d_W)\log(n|V|)}{n^{1/2}}+ \frac{M_n^{1/3}((1+d_W)\log(n|V|))^{2/3}}{n^{1/6}}\right) $$
and by Theorem 2.2 in  \cite{chernozhukov2015noncenteredprocesses}, there is also a version of
$$ | Z_N - \bar Z_G^*| \lesssim_P \left( \frac{M_n(1+d_W)\log(n|V|)}{n^{1/2}}+ \frac{M^{1/2}_n((1+d_W)\log(n|V|))^{3/4}}{n^{1/4}}\right) $$
Finally, we have that
$$ |\bar Z_G^* -\hat Z_G^*|\leq \sup_{u\in\mathcal{U},j} \left| \frac{1}{\sqrt{n}}\sum_{i=1}^n g_i (\hat \psi_{iuj} -\bar \psi_{iuj} )  \right| $$
where conditional on $(X_i,W_i), i=1,\ldots,n$, $\frac{1}{\sqrt{n}}\sum_{i=1}^n g_i (\hat \psi_{iuj} -\bar \psi_{iuj} )  $ is a zero-mean Gaussian with variance $\En[(\hat \psi_{iuj} -\bar \psi_{iuj} )^2] \leq \bar \delta_n^2$. Next we bound $\bar\delta_n$. We have
$$ \begin{array}{rl}
\bar \delta_n & \leq \En[(\hat \psi_{uj} -  \psi_{uj} )^2]^{1/2} + \En[( \psi_{uj} - \bar \psi_{uj} )^2]^{1/2} \leq \En[(\hat \psi_{uj} -  \psi_{uj} )^2]^{1/2} + \delta_{1n}/\mu_{\mathcal{W}},\\
\end{array}
$$
and
$$\begin{array}{rl}
\En[(\hat \psi_{uj} -  \psi_{uj} )^2]^{1/2} &  \leq \frac{\En[(K_\varpi(W)X_{ij}|\hat\varepsilon_u - \varepsilon_u|)^2]^{1/2}}{\En[K_\varpi(W)X_{ij}^2\hat\varepsilon_u^2]^{1/2}} \\
 & + \frac{\En[K_\varpi(W)X_{ij}^2\varepsilon_u^2]^{1/2}}{c\Pr(\varpi)} |\En[K_\varpi(W)X_{ij}^2\hat\varepsilon_u^2]^{1/2}-\En[K_\varpi(W)X_{ij}^2\varepsilon_u^2]^{1/2}|\\
 & \leq \En[K_\varpi(W)X_{ij}^2|\hat\varepsilon_u - \varepsilon_u|^2]^{1/2} \left\{\frac{1}{\En[K_\varpi(W)X_{ij}^2\hat\varepsilon_u^2]^{1/2}} + \frac{\En[K_\varpi(W)X_{ij}^2\varepsilon_u^2]^{1/2}}{c\Pr(\varpi)} \right\},\\
\end{array}$$
note that the term in the curly brackets is bounded by $C/\Pr(\varpi)^{1/2}$ with probability $1-o(1)$. To bound the other term note that $|\hat\varepsilon_u - \varepsilon_u|^2 = | 1\{ X_a \leq X_{-a}'\widetilde\beta_u\} - 1\{ X_a \leq X_{-a}'\beta_u\}|$. Note that $\hat \varepsilon_u = 1\{ X_a \leq X_{-a}'\widetilde\beta_u\} - \tau$ where $\|\widetilde\beta_u\|_0 \leq Cs$. Therefore, we have $\{ 1\{ X_a \leq X_{-a}'\widetilde\beta_u\} : u\in \mathcal{U}\} \subset \{ 1\{ X_a \leq X_{-a}'\beta\} : \|\beta\|_0\leq Cs\}$ which is the union of $\binom{|V|}{Cs}$ VC subgraph classes of functions with VC dimension $C's$. Moreover, we have
$$\begin{array}{rl}\Ep[K_\varpi(W)X_{ij}^2|\hat\varepsilon_u - \varepsilon_u|^2] & =  \Ep[K_\varpi(W)X_{ij}^2| 1\{ X_a \leq X_{-a}'\widetilde\beta_u\} - 1\{ X_a \leq X_{-a}'\beta_u\}|] \\
 & \leq \bar f  \Ep[K_\varpi(W)X_{ij}^2| X_{-a}'(\widetilde\beta_u -\beta_u)|] \\
 & \leq \bar f  \Ep[K_\varpi(W)X_{ij}^4]^{1/2}\Ep[K_\varpi(W)| X_{-a}'(\widetilde\beta_u-\beta_u)|^2]^{1/2} \\
 & \leq C(\bar f/\underf_\UU^{1/2}) \Pr(\varpi)^{1/2} \sqrt{s(1+d_W)\log(n|V|)/n}\\
 \end{array}$$
Therefore, by Lemma \ref{lemma:CCK},  with probability $1-o(1)$ we have
$$ \left|(\En-\Ep)[K_\varpi(W)X_{ij}^2|\hat\varepsilon_u - \varepsilon_u|^2]\right| \lesssim \sqrt{\frac{s(1+d_W)\log(n|V|)}{n}} C(\bar f/\underf_\UU^{1/2})  \sqrt{s(1+d_W)\log(n|V|)/n} $$
Under $\sqrt{s(1+d_W)\log(n|V|)/n} = o( \underf_{\UU}\mu_{\mathcal{W}})$ we have that with probability $1-o(1)$ that
$$ \bar\delta_n \leq C\{s(1+d_W)\log(n|V|)/n\}^{1/4}.$$
Therefore, using again the sparsity of $\widetilde \beta_u$ in the definition of $\hat \psi_{iuj}$
$$\begin{array}{rl}
 \displaystyle \sup_{u\in \mathcal{U}, j\in V\backslash\{a\}}\left| \frac{1}{\sqrt{n}}\sum_{i=1}^ng_i (\hat \psi_{iuj} -\bar \psi_{iuj} )  \right| & \lesssim_P \bar \delta_n\sqrt{s(1+d_W)\log(|V|n)} \\
 & \lesssim_P \delta_{3n}:=\{ s\log(|V|n)/n\}^{1/4}\sqrt{s(1+d_W)\log(|V|n)} \end{array}$$

The rest of the proof follows similarly to Corollary 2.2 in \cite{BCCW-ManyProcesses} since under Condition P (and the bounds above) we have that $r_n:=\delta_{1n}+\delta_{2n}+\delta_{3n}=o( \{\Ep[Z_N]\}^{-1} ))$ where $\Ep[Z_N]\lesssim \{(1+d_W)\log(|V|n)\}^{1/2}$. Then we have $\sup_t | \Pr( Z_S \leq t ) - \Pr( \hat Z_G^* \leq t )| = o_P(1)$ which in turn implies that
$$\begin{array}{rl}
  \Pr( \Omega_1 ) & = \Pr( Z_S \leq \hat c_G^*(\xi) ) \\
  & \geq \Pr( \hat Z_G^* \leq \hat c_G^*(\xi) ) - | \Pr( Z_S \leq \hat c_G^*(\xi) ) - \Pr( \hat Z_G^* \leq \hat c_G^*(\xi) )|  \\
  & \geq 1-\xi + o_P(1)\end{array}$$
Note that the occurrence of the events $\Omega_2$, $\Omega_3$ and $\Omega_4$  follows by similar arguments. The result follows by Lemma \ref{Theorem:L1QRnpPrediction}, thresholding and applying Lemma \ref{Lemma:Bound2nNormSecond} and  Lemma \ref{Thm:RefittedGeneric} similarly to before.

\end{proof}

\section{Technical Lemmas for Conditional Independence Quantile Graphical Model}\label{Sec:ProofsTechnicalLemmaL1QR}

Let $u=(a,\tau,\varpi)\in \UU:= V\times \mathcal{T}\times \mathcal{W}$, and $T_u = \supp(\beta_u)$ where $|T_u|\leq s$ for all $u\in\UU$.

Define the pseudo-norms
$$ \|v\|_{n,\varpi}^2 := \frac{1}{n}\sum_{i=1}^n K_\varpi(W_i)(v_i)^2, \ \ \ \|\delta\|_{2,\varpi}:= \left\{\sum_{j=1}^p \{\hat\sigma_{a\varpi j}^{Z}\}^2 |\delta_j|^2\right\}^{1/2}, \ \ \ \mbox{and} \ \ \ \|\delta\|_{1,\varpi}:= \sum_{j=1}^p \hat\sigma_{a\varpi j}^Z|\delta_j|, $$ where $\hat\sigma_{a\varpi j}^{Z} =\{\En[\{K_\varpi(W)Z_j^a\}^2]\}^{1/2}$. These pseudo-norms induce the following restricted eigenvalue as
$$ \kappa_{u,\cc} = \min_{\|\delta_{T^c_u}\|_{1,\varpi} \leq \cc\|\delta_{T_u}\|_{1,\varpi} }\frac{\|\sqrt{f_{u}}Z^a\delta\|_{n,\varpi}}{\|\delta\|_{1,\varpi}/\sqrt{s}}.$$
The restricted eigenvalue $\kappa_{u,\cc}$ is an counterpart of the restricted eigenvalue proposed in \cite{BickelRitovTsybakov2009} for our setting. We note that $ \kappa_{u,\cc}$ typically will vary with the events $\varpi\in\mathcal{W}$.

We will consider three key events in our analysis. Let
 \begin{equation}\label{def:Omega1}\Omega_1:=\{ \lambda_u\geq c |S_{uj}|/\hat\sigma_{a\varpi j}^Z, \ \mbox{for all} \ \ u\in\mathcal{U}, j\in [p]\}\end{equation} which occurs with probability at least $1-\xi$ by the choice of $\lambda_u$. For CIQGMs, we have $S_{uj}:=\En[K_\varpi(W)(\tau-1\{X_a\leq Z^a\beta_u+r_{u}\}) Z^a_j]$, and $\lambda_u=\lambda_{V\mathcal{T}\mathcal{W}}\sqrt{\tau(1-\tau)}$. (In the case of PQGMs, we have $S_{uj}:= \En[K_\varpi(W)(\tau-1\{X_a\leq X_{-a}'\beta_u\}) X_{-a}]$, $\hat\sigma_{a\varpi j}^X=\{\En[K_\varpi(W)X_{-a,j}^2]\}^{1/2}$ and $\lambda_u = \lambda_0$.)

To define the next event, for each $u\in \mathcal{U}$, consider the function defined as $$\hat R_u(\beta_u) = \En[K_\varpi(W)\{\rho_u(X_a-Z^a\beta)-\rho_u(X_a-Z^a\beta_u-r_{u})-(\tau-1\{X_a\leq Z^a\beta_u+r_{u}\})(Z^a\beta-Z^a\beta_u-r_{u})\}]$$ in the case of CIQGMs. (In the case of PQGMs, we replace $Z^a$ with $X_{-a}$.)  By convexity we have $\hat R_u(\beta_u) \geq 0$. The event  \begin{equation}\label{def:Omega2}\Omega_2 :=\{\hat R_u(\beta_u) \leq \bar R_{u\xi} : u\in\mathcal{U}\}\end{equation} where $\bar R_{u\xi}$ are chosen so that $\Omega_2$ occurs with probability at least $1-\xi$. Note that by Lemma \ref{Lemma:ControlR}, we have $\En[\Ep[\hat R_u(\beta_u)\vert X_{-a},W]] \leq \bar f \|r_{u}\|_{n,\varpi}^2/2$ and with probability at least $1-\xi$, $\hat R_u(\beta_u) \leq \bar R_{u\xi}:= 4\max\{\bar f \|r_{u}\|_{n,\varpi}^2, \ \|r_{u}\|_{n,\varpi}C\sqrt{\log(n^{1+d_W}p/\xi)/n}\} \leq C' s \log(n^{1+d_W}p/\xi) / n$.

Define $g_u(\delta,X,W)= K_\varpi(W)\{\rho_\tau(X_a-Z^a(\beta_u+\delta))-\rho_\tau(X_a-Z^a\beta_u)\}$ so that event $\Omega_3$ is defined as
 \begin{equation}\label{def:Omega3} \Omega_3 := \left\{ \sup_{u\in\mathcal{U}, 1/\sqrt{n} \leq \|\delta\|_{1,\varpi} \leq \sqrt{n}} \frac{|\En[g_u(\delta,X,W)-\Ep[g_u(\delta,X,W)\vert X_{-a},W]]|}{\|\delta\|_{1,\varpi}} \leq t_3 \right\}\end{equation}
where $t_3$ is given in Lemma \ref{LemmaOmega3} so that $\Omega_3$ holds with probability at least $1-\xi$.

\begin{lemma}\label{Theorem:L1QRnp}
Suppose that $\Omega_1$, $\Omega_2$ and $\Omega_3$ holds. Further assume $2\frac{1+1/c}{1-1/c}\|\beta_u\|_{1,\varpi} + \frac{1}{\lambda_u(1-1/c)}\bar R_{u\xi} \leq \sqrt{n}$ for all $u\in\UU$, and (\ref{ConditionIdenfitication}) holds for all $\delta\in A_u := \Delta_{u,2\cc}\cup\{v:\|v\|_{1,\varpi}\leq 2\cc\bar R_{u\xi}/\lambda_u\}$, $\bar q_{A_u}/4 \geq (\sqrt{\bar f}+1) \|r_{u}\|_{n,\varpi} + \left[\lambda_u+t_3\right]\frac{3\cc\sqrt{s}}{\kappa_{u,2\cc}}$ and $\bar q_{A_u} \geq \{ 2\cc\left(1+ \frac{t_3}{\lambda_u}\right)\bar R_{u\xi} \}^{1/2}$.
Then uniformly over all $u\in \UU$ we have
{\small $$\begin{array}{rl}
\|\sqrt{f_u}Z^a(\hat\beta_u - \beta_u)\|_{n,\varpi} & \leq \sqrt{8\cc\left(1+ \frac{t_3}{\lambda_u}\right)\bar R_{u\xi}} + (\bar f^{1/2}+1) \|r_{u}\|_{n,\varpi}+ \frac{3\cc\lambda_u\sqrt{s}}{\kappa_{u,2\cc}}+t_3\frac{(1+\cc)\sqrt{s}}{\kappa_{u,2\cc}}\\
\|\hat\beta_u - \beta_u\|_{1,\varpi} & \leq  (1+2\cc)\sqrt{s}\|\sqrt{f_u}Z^a\delta_u\|_{n,\varpi}/\kappa_{u,2\cc} + \frac{2\cc}{\lambda_u}\bar R_{u\xi}  \\
\end{array}$$}\end{lemma}
\begin{proof}[Proof of Lemma \ref{Theorem:L1QRnp}]
Let $u=(a,\tau,\varpi)\in \UU$ and $\delta_u = \hat\beta_u - \beta_u$.  
By convexity and definition of $\hat \beta_u$ we have
 \begin{equation}\label{ImpDefbeta}\begin{array}{rl}
 & \hat R_u(\hat\beta_u)-\hat R_u(\beta_u)+S_u'\delta_u \\
 & = \En[K_\varpi(W)\rho_u(X_a-Z^a\hat\beta_u)] - \En[K_\varpi(W)\rho_u(X_a-Z^a\beta_u)] \\
  & \leq \lambda_u\|\beta_u\|_{1,\varpi}-\lambda_u\|\hat\beta_u\|_{1,\varpi}\end{array}\end{equation}
where $S_u$ is defined as in (\ref{def:Omega1}) so that under $\Omega_1$ we have $\lambda_u\geq c|S_{uj}|/\hat\sigma_{a\varpi j}^{Z}$.

Under $\Omega_1 \cap \Omega_2$, and since $\hat R_u(\hat{\beta}_u)\geq 0$, we have
\begin{equation}\label{ImpDefbetaTilde}
\begin{array}{rl}
-\hat R_u(\beta_u)-\frac{\lambda_u}{c}\|\delta_u\|_{1,\varpi} & \leq \hat R_u(\beta_u+\delta_u)-\hat R_u(\beta_u)+\En[K_\varpi(W)(\tau-1\{X_a\leq Z^a\beta_u+r_{u}\})Z^a\delta_u] \\
  & = \En[K_\varpi(W)\rho_u(X_a-Z^a(\delta_u +\beta_u))] - \En[K_\varpi(W)\rho_u(X_a-Z^a\beta_u)] \\
 & \leq \lambda_u\|\beta_u\|_{1,\varpi}-\lambda_u\|\delta_u +\beta_u\|_{1,\varpi}\end{array}\end{equation}
so that for $\cc = (c+1)/(c-1)$
$$ \|\delta_{T^c_u}\|_{1,\varpi} \leq \cc \|\delta_{T_u}\|_{1,\varpi}  + \frac{c}{\lambda_u(c-1)}\hat R_u(\beta_u).$$
To establish that $\delta_u \in A_u:=\Delta_{u,2\cc}\cup \{ v  :  \|v\|_{1,\varpi} \leq 2\cc \bar R_{u\xi}/\lambda_u\}$  we consider two cases. If $\|\delta_{T^c_u}\|_{1,\varpi} \geq 2\cc\| \delta_{T_u}\|_{1,\varpi}$ we have
$$ \frac{1}{2}\|\delta_{T^c_u}\|_{1,\varpi} \leq  \frac{c}{\lambda_u(c-1)}\hat R_u(\beta_u)$$
and consequentially $$ \|\delta_u\|_{1,\varpi}\leq \{1+1/(2c)\}\|\delta_{T^c_u}\|_{1,\varpi} \leq  \frac{2\cc}{\lambda_u
}\hat R_u(\beta_u).$$
Otherwise, we have $\|\delta_{T^c_u}\|_{1,\varpi} \leq 2\cc\|\delta_{T_u}\|_{1,\varpi}$ which implies
$$\|\delta_u\|_{1,\varpi} \leq (1+2\cc)\|\delta_{T_u}\|_{1,\varpi} \leq (1+2\cc)\sqrt{s}\|\sqrt{f_u}Z^a\delta_u\|_{n,\varpi}/\kappa_{u,2\cc}$$
by definition of $\kappa_{u,2\cc}$. Thus we have  $\delta_u \in A_u$ under $\Omega_1 \cap \Omega_2$.

Furthermore, (\ref{ImpDefbetaTilde}) also implies that
$$\begin{array}{rl}\|\delta_u +\beta_u\|_{1,\varpi} & \leq  \|\beta_u\|_{1,\varpi}+\frac{1}{c}\|\delta_u\|_{1,\varpi}+\hat R_u(\beta_u)/\lambda_u \\
&\leq (1+1/c)\|\beta_u\|_{1,\varpi}+(1/c)\|\delta_u +\beta_u\|_{1,\varpi}+\hat R_u(\beta_u)/\lambda_u.\end{array}$$
which in turn establishes
$$\|\delta_u\|_{1,\varpi} \leq 2\frac{1+1/c}{1-1/c}\|\beta_u\|_{1,\varpi} + \frac{1}{\lambda_u(1-1/c)}\hat R_u(\beta_u)\leq 2\frac{1+1/c}{1-1/c}\|\beta_u\|_{1,\varpi} + \frac{1}{\lambda_u(1-1/c)}\bar R_{u\xi} $$
where the last inequality holds under $\Omega_2$. Thus, $\|\delta_u\|_{1,\varpi}\leq \sqrt{n}$ under our condition.
In turn, $\delta_u$ is considered in the supremum that defines $\Omega_3$.

Under $\Omega_1 \cap \Omega_2 \cap \Omega_3$ we have
{\small \begin{equation}\label{AuxUpper1}\begin{array}{rl}
& \En[\Ep[K_\varpi(W)\{\rho_u(X_a-Z^a(\beta_u+ \delta_u))-\rho_u(X_a-Z^a\beta_u)\} \mid X_{-a}, W]\\
& \leq \En[K_\varpi(W)\{\rho_u(X_a-Z^a(\beta_u+ \delta_u))-\rho_u(X_a-Z^a\beta_u)\}] + t_3\|\delta_u\|_{1,\varpi}\\
 & \leq  \lambda_u\| \delta_u\|_{1,\varpi} + t_3\|\delta_u\|_{1,\varpi}\\
 & \leq  2\cc\left(1+ \frac{1}{\lambda_u}t_3\right)\bar R_{u\xi} + \|\sqrt{f_u}Z^a \delta_u\|_{n,\varpi}\left[\lambda_u+t_3\right] \frac{3\cc\sqrt{s}}{\kappa_{u,2\cc}}\\
 \end{array}\end{equation}}
here we used the bound $\|\delta_u\|_{1,\varpi} \leq (1+2\cc)\sqrt{s}\|\sqrt{f_u}Z^a\delta_u\|_{n,\varpi}/\kappa_{u,2\cc} + \frac{2\cc}{\lambda_u}\bar R_{u\xi}$ under $\Omega_1\cap \Omega_2$.

Using Lemma \ref{Lemma:IdentificationSparseQRNP}, since (\ref{ConditionIdenfitication}) holds, we have for each $u\in \mathcal{U}$
$$\begin{array}{rl}
 \En[\Ep[K_\varpi(W)\{\rho_u(X_a-Z^a(\beta_u+ \delta_u))-\rho_u(X_a-Z^a\beta_u)\} \mid X_{-a},W] \\
 \geq - (\sqrt{\bar f}+1) \|r_{u}\|_{n,\varpi} \|\sqrt{f_u}Z^a \delta_u\|_{n,\varpi} - \sup_{u\in \mathcal{U}, j\in[p]} |\En[\Ep[ S_{uj} \vert X_{-a},W]/\hat\sigma^{Z}_{a\varpi j} ]|  \ \|\delta_u\|_{1,\varpi} \\
 + \frac{\|\sqrt{f_u}Z^a\delta_u\|_{n,\varpi}^2}{4} \wedge \{\bar q_{A_u}\|\sqrt{f_u}Z^a\delta_u\|_{n,\varpi}\}\end{array}$$
here we have $\Ep[ S_{iuj}\vert X_{i,-a},W_i]=0$ since $\tau = \Pr(X_a \leq Z^a\beta_u+r_u\vert X_{-a},W)$ by the definition of conditional quantile.

Note that for positive numbers $(t^2/4) \wedge q t \leq A + B t$ implies $t^2/4 \leq A+Bt$ provided $q/2 > B$ and $2q^2>A$. (Indeed, otherwise $(t^2/4) \geq q t$ so that $t\geq 4q$ which in turn implies that $2q^2 + qt/2 \leq (t^2/4) \wedge q t \leq A + Bt$.) Since $\bar q_{A_u}/4 \geq (\sqrt{\bar f}+1) \|r_{u}\|_{n,\varpi} + \left[ \{\lambda_u+t_3\}\frac{3\cc\sqrt{s}}{\kappa_{u,2\cc}}\right]$ and $\bar q_{A_u} \geq \{ 2\cc\left(1+ \frac{t_3}{\lambda_u}\right)\bar R_{u\xi} \}^{1/2}$, the minimum on the right hand side is achieved by the quadratic part for all $u\in\mathcal{U}$. Therefore we have uniformly over $u\in\mathcal{U}$
{\small $$ \frac{\|\sqrt{f_u}Z^a\delta_u\|_{n,\varpi}^2}{4} \leq  2\cc\left(1+ \frac{t_3}{\lambda_u}\right)\bar R_{u\xi} + \|\sqrt{f_u}Z^a \delta_u\|_{n,\varpi}\left[(\sqrt{\bar f}+1) \|r_{u}\|_{n,\varpi}+ \{\lambda_u+t_3\}\frac{3\cc\sqrt{s}}{\kappa_{u,2\cc}}\right]
$$}
which implies that
{\small $$\begin{array}{rl}
\|\sqrt{f_u}Z^a\delta_u\|_{n,\varpi} & \leq \sqrt{8\cc\left(1+ \frac{t_3}{\lambda_u}\right)\bar R_{u\xi}} + \left[(\sqrt{\bar f}+1) \|r_{u}\|_{n,\varpi}+ \{\lambda_u+t_3\}\frac{3\cc\sqrt{s}}{\kappa_{u,2\cc}}\right].
\end{array}$$}
\end{proof}

\begin{lemma}[CIQGM, Event $\Omega_2$]\label{Lemma:ControlR}
Under Condition CI we have $\En[\Ep[\hat R_u(\beta_u)\vert X_{-a},\varpi]] \leq \bar f \|r_{u}\|_{n,\varpi}^2/2$, $\hat R_u(\beta_u)\geq 0$ and
$$\Pr\left( \sup_{u\in\mathcal{U}}\hat R_u(\beta_u) \leq C\{1+ \bar f\} \{n^{-1}s(1+d_W)\log(p|V|n)\} \right) = 1-o(1).$$
\end{lemma}
\begin{proof}[Proof of Lemma \ref{Lemma:ControlR}]
We have that $\hat R_u(\beta_u)\geq 0$ by convexity of $\rho_\tau$. Let $\varepsilon_{iu}=X_{ia} - Z_{i}^a \beta_u - r_{iu}$ where $\|\beta_u\|_0\leq s$ and $r_{iu}=Q_{X_a}(\tau\vert X_{-a},\varpi) - Z^a\beta_u$.

By Knight's identity (\ref{Eq:TrickRho}), $\hat R_u(\beta_u) =  -\En[ K_\varpi(W)r_{u} \int_0^1  1\{\varepsilon_{u} \leq -t r_{u}\} - 1\{\varepsilon_{u}\leq 0\} \ dt ]\geq 0$.
 $$ \begin{array}{rl}
 \En[\Ep[\hat R_u(\beta_u)\vert X_{-a},\varpi] & = \En[  K_\varpi(W)r_{u} \int_0^1  F_{X_a\mid X_{-a},\varpi}(Z^a\beta_u+(1-t)r_u) -  F_{X_a\mid X_{-a},\varpi}(Z^a\beta_u+r_u) \ dt]\\
  & \leq \En[  K_\varpi(W)r_{u} \int_0^1  \bar f t r_{u} dt] \leq \bar f \|r_{u}\|_{n,\varpi}^2/2\leq C\bar f s/n.\end{array}$$
Since Condition CI assumes $\Ep[\|r_{u}\|_{n,\varpi}^2]\leq \Pr(\varpi)s/n$, by Markov's inequality we have $\Pr( \hat R_u( \beta_u) \leq  C\bar f s/n ) \geq 1/2$. Define $z_{iu} := -\int_0^1  1\{\varepsilon_{iu} \leq -t r_{iu}\} - 1\{\varepsilon_{iu}\leq 0\} \ dt$, so that $\hat R_u(\beta_u) = \En[K_\varpi(W)r_{u}z_{u}]$ where $|z_{iu}|\leq 1$. By Lemma 2.3.7 in \cite{vdVaartWellner2007} (note that the Lemma does not require zero mean stochastic processes), for $t \geq  2 C\bar f s/n$ we have
$$ \frac{1}{2}\Pr\left( \sup_{u\in\mathcal{U}}|\En[K_\varpi(W)r_{u}z_{u}]| \geq t \right) \leq 2\Pr\left(\sup_{u\in\mathcal{U}}|\En[\varepsilon K_\varpi(W)r_{u}z_{u}]|>t/4\right) $$
where $\varepsilon_i, i=1,\ldots,n$ are Rademacher random variables independent of the data.

Next consider the class of functions $\mathcal{F}= \{  - K_\varpi(W)r_{u} (1\{\varepsilon_{iu} \leq -B_i r_{iu}\} - 1\{\varepsilon_{iu}\leq 0\} ) : u\in\mathcal{U}\}$ where $B_i\sim {\rm Uniform}(0,1)$ independent of $(X_i,W_i)_{i=1}^n$. It follows that $K_\varpi(W)r_{u}z_{u} =  \Ep[  - K_\varpi(W)r_{u} (1\{\varepsilon_{iu} \leq -B_i r_{iu}\} - 1\{\varepsilon_{iu}\leq 0\} ) \vert X_i,W_i]$ where the expectation is taken over $B_i$ only. Thus we will bound the entropy of $\overline{\mathcal{F}}= \{ \Ep[ f\vert X,W]: f \in \mathcal{F} \}$ via Lemma \ref{Lemma:PartialOutCovering}. Note that  $\mathcal{R}:=\{ r_u = Q_{X_a}(\tau\vert X_{-a},\varpi) - Z^a\beta_u: u\in\mathcal{U}\}$ where $\mathcal{G}:=\{ Z^a\beta_u: u\in\mathcal{U}\}$ is contained in the union of at most $|V|\binom{p}{s}$ VC-classes of dimension $Cs$ and  $\mathcal{H}:=\{Q_{X_a}(\tau\vert X_{-a},\varpi): u\in\mathcal{U}\}\}$ is the union of $|V|$ VC-class of functions of dimension $(1+d_W)$ by Condition CI. Finally note that $\mathcal{E}:=\{ \varepsilon_{iu}: u\in \mathcal{U}\} \subset \{ X_{ia}: a\in V\} - \mathcal{G} - \mathcal{R}$.

Therefore, we have
$$
\begin{array}{rl}
\sup_Q \log N(\epsilon \|\bar F\|_{Q,2}, \overline{\mathcal{F}}, \|\cdot\|_{Q,2}) & \leq  \sup_Q \log N( (\epsilon/4)^2 \| F\|_{Q,2}, \mathcal{F}, \|\cdot\|_{Q,2})\\
& \leq  \sup_Q \log N( \mbox{$\frac{1}{8}$}(\epsilon^2/16), \mathcal{W}, \|\cdot\|_{Q,2}) \\
&+  \sup_Q\log N( \mbox{$\frac{1}{8}$} (\epsilon^2/16) \| F\|_{Q,2}, \mathcal{R}, \|\cdot\|_{Q,2}) \\
&+  \sup_Q \log N( \mbox{$\frac{1}{8}$}(\epsilon^2/16), 1\{ \mathcal{E}+\{B\}\mathcal{R} \leq 0\} - 1\{ \mathcal{E} \leq 0\}, \|\cdot\|_{Q,2})\\
\end{array}
$$
We will apply Lemma \ref{lemma:CCK} with envelope $\bar F = \sup_{u\in\mathcal{U}} |K_\varpi(W)r_u|$, so that $ \Ep[\max_{i\leq n}\bar F_i^2]\leq C$, and $\sup_{u\in \mathcal{U}} \Ep[K_\varpi(W)r_{u}^2 ] \leq Cs/n =: \sigma^2$ by Condition CI. Thus, we have that with probability $1-o(1)$
$$ \sup_{u\in\mathcal{U}}|\En[\varepsilon K_\varpi(W)r_{u}z_{u}]| \lesssim \sqrt{\frac{s(1+d_W)\log(p|V|n)}{n}}\sqrt{\frac{s}{n}}+ \frac{s(1+d_W)\log(p|V|n)}{n}\lesssim \frac{s(1+d_W)\log(p|V|n)}{n}$$
under  $M_n\sqrt{s^2/n} \leq C$.
\end{proof}

\begin{lemma}[CIQGM, Event $\Omega_3$]\label{LemmaOmega3}
For $u=(a,\tau,\varpi)\in \mathcal{U} := V\times \mathcal{T}\times \mathcal{W}$, define the function $g_u(\delta,X,W)= K_\varpi(W)\{\rho_\tau(X_a-Z^a(\beta_u+\delta))-\rho_\tau(X_a-Z^a\beta_u)\}$, and the event
$$ \Omega_3 := \left\{ \sup_{u\in\mathcal{U}, 1/\sqrt{n} \leq \|\delta\|_{1,\varpi} \leq \sqrt{n}} \frac{|\En[g_u(\delta,X,W)-\Ep[g_u(\delta,X,W)\mid X_{-a}, W]]|}{\|\delta\|_{1,\varpi}} < t_3 \right\}.$$
Then, under Condition CI we have $P(\Omega_3) \geq 1-\xi$ for any $t_3$ satisfying
$$ t_3\sqrt{n} \geq   12 + 16\sqrt{2\log(64|V|p^2 n^{3+2d_W}\log(n)L^{1+d_W/\kappa}_\beta M_n/\xi )}  $$
\end{lemma}
\begin{proof}
We have that $\Omega_3^c := \{ \max_{a\in V} A_a \geq t_3\sqrt{n}\}$ for
$$A_a:= \sup_{(\tau,\varpi)\in\mathcal{T}\times\mathcal{W}, \underline{N} \leq \|\delta\|_{1,\varpi} \leq \bar{N}} \sqrt{n}\left|\frac{\En[g_u(\delta,X,W)-\Ep[g_u(\delta,X,W)\mid X_{-a},W]]}{\|\delta\|_{1,\varpi}}\right|.  $$
Therefore, for $\underline{N} = 1/\sqrt{n}$ and $\bar{N} = \sqrt{n}$ we have by Lemma  \ref{Lemma:EmpProc:Normalized} with $\rho = \kappa$, $L_{\eta} = L_{\beta}$, $\tilde{x} = Z^a$
$$\begin{array}{rl}
\Pr(\Omega_3^c) & = \Pr(\max_{a\in V} A_a \geq t_3\sqrt{n}) \\
& \leq |V|\max_{a\in V} \Pr(A_a \geq t_3\sqrt{n}) \\
& = |V|\max_{a\in V} \Ep_{X_{-a},W}\left\{ \Pr(A_a \geq t_3\sqrt{n} \mid X_{-a},W )\right\} \\
& \leq |V|  \max_{a\in V} \Ep_{X_{-a},W}\left\{8p |\widehat{\mathcal{N}}|\cdot |\widehat{\mathcal{W}}|\cdot |\widehat{\mathcal{T}}| \exp(-(t_3\sqrt{n}/4-3)^2/32)\right\}\\
& \leq \exp(-(t_3\sqrt{n}/4-3)^2/32) |V|64p n^{1+d_W}\log(n)L_\beta \Ep_{X_{-a}}\left\{\frac{\max_{i\leq n}\|Z^a_i\|_\infty^{1+d_W/\kappa}}{\underline{N}^{1+d_W}}\right\}\\
& \leq \xi \end{array} $$
by the choice of $t_3$ and noting that $M_n^{(1+d_W/\kappa)/q} \geq  \Ep_{X_{-a}}[\max_{i\leq n}\|Z^a_i\|_\infty^{1+d_W/\kappa}]$, $ 1+d_W/\kappa \leq  q$ and $M_n \geq 1$.
\end{proof}

\begin{lemma}[CIQGM, Uniform Control of Approximation Error in Auxiliary Equation]\label{ControlRUJ}
Under Condition CI, with probability $1-o(1)$ uniformly over $u\in\mathcal{U}$ and $j\in[p]$ we have
$$
\En[K_\varpi(W)f_u^2\{Z_{-j}^a(\gamma_u^j-\bar\gamma_u^j)\}^2]  \lesssim \underf_u^2\Pr(\varpi)  \{n^{-1}s\log (p|V| n)\}^{1/2}.
$$
\end{lemma}
\begin{proof}
Define the class of functions $\mG = \cup_{a\in V, j\in [p]}\mG_{aj}$ with $\mG_{aj}:=\{ Z_{-j}^a(\gamma_u^j-\bar\gamma_u^j) : \tau \in \mathcal{T}, \varpi\in \mathcal{W}\}$.
Under Condition CI we have $\sup_{u\in\UU}\|\bar\gamma_u^j\|_0 \leq Cs$, $\sup_{u\in\UU, j\in [p]}\|\bar \gamma_u^j-\gamma_u^j\| \vee \frac{\|\bar\gamma_u^j-\gamma_u^j\|_1}{\sqrt{s}}\leq \{n^{-1}s\log (p|V| n)\}^{1/2}.$
Without loss of generality we can set $\bar\gamma_{uk}^j = \gamma_{uk}^j$ for $k \in \supp(\bar\gamma_u^j)$. Letting $\mG_{aj,T}:=\{ Z_{-j}^a(\gamma_u^j-\gamma_{uT}^j) : \tau \in \mathcal{T}, \varpi\in \mathcal{W}\}$ for $T\subset \{1,\ldots,p\}$, it follows that $ \mG \subset \cup_{a\in V, j\in [p]}\cup_{|T|\leq Cs}\mG_{aj,T}$.

By Lemma \ref{Lemma:LipsGamma}, we have $\|\gamma_u^j-\gamma_{u'}^j\|\leq L_\gamma (\|u-u'\|+\|u-u'\|^{1/2})$ for each $a\in V$, $j\in[p]$. (Note that although $\bar\gamma^j_u$ might not be Lipschitz in $u$, however, for each $T$, $\gamma_{uT}^j$ satisfies the same Lipschitz relation as $\gamma_u^j$, in fact $\|\bar\gamma_{uT}^j-\gamma_{u'T}^j\|\leq \|\gamma_u^j-\gamma_{u'}^j\|$ by construction.) Therefore, for each $T$ we have
 $$\begin{array}{rl}
 & \|\{Z_{-j}^a(\bar\gamma_{uT}^j-\gamma_u^j)\}^2-\{Z_{-j}^a(\bar\gamma_{u'T}^j-\gamma_{u'}^j)\}^2\|_{Q,2}  \\
 & \leq \|Z_{-j}^a(\bar\gamma_{uT}^j-\bar\gamma_{u'T}^j+\gamma_{u'}^j-\gamma_u^j)Z_{-j}^a(\bar\gamma_{uT}^j-\gamma_u^j+\bar\gamma_{u'T}^j-\gamma_{u'}^j)\|_{Q,2}\\
& \leq \|\|Z_{-j}^a\|_\infty^2\|_{Q,2} \|\bar\gamma_{uT}^j-\bar\gamma_{u'T}^j+\gamma_{u'}^j-\gamma_u^j\|_1\|\bar\gamma_{uT}^j-\gamma_u^j+\bar\gamma_{u'T}^j-\gamma_{u'}^j\|_1\\
& \leq  4\|\|Z_{-j}^a\|_\infty^2\|_{Q,2}\sup_{u\in\UU}\|\bar\gamma_{uT}^j-\gamma_u^j\|_1 \sqrt{2p}\|\gamma_u^j-\gamma_{u'}^j\|\\
&\leq \|\|Z_{-j}^a\|_\infty^2\|_{Q,2}L_\gamma'(\|u-u'\|+\|u-u'\|^{1/2}).\end{array}$$
where $L_\gamma'=4\{n^{-1}s^2\log(p|V|n)\}^{1/2}\sqrt{2p}L_\gamma$. Thus, for the envelope $G = \max_{a\in V}\|Z^a\|_\infty^2\sup_{u\in\UU} \|\bar\gamma_u^j-\gamma_u^j\|_1^2$ that
$$ \begin{array}{rl}\log N(\epsilon \|G\|_{Q,2}, \mG, \|\cdot\|_{Q,2}) \leq Cs\log(|V|p)+\log N\left(\epsilon\frac{\sup_{u\in\UU} \|\bar\gamma_u^j-\gamma_u^j\|_1^2}{L_\gamma'}, \UU, d_\UU\right) \leq Cs(1+d_W)^2 \log (L_\gamma'n/\epsilon). \end{array}$$
Next define the functions $\mathcal{W}_0=\{ K_\varpi(W)f_u^2 : u\in\mathcal{U} \}$, $\mathcal{W}_1=\{ \Pr( \varpi)^{-1} : \varpi \in \mathcal{W} \}$ and $\mathcal{W}_2=\{ K_\varpi(W) : \varpi \in \mathcal{W} \}$. We have that $\mathcal{W}_2$ is VC class with VC index $Cd_W$ and $\mathcal{W}_1$ is bounded by $\mu_\mathcal{W}^{-1}$ and covering number bounded by $(Cd_W/\{\mu_\mathcal{W}\epsilon\})^{1+d_W}$. Finally, since $|K_\varpi(W)f_u^2-K_{\varpi'}(W)f_{u'}^2|\leq K_\varpi(W)K_{\varpi'}(W)|f_u^2-f_{u'}^2| + \bar f^2 |K_\varpi(W)-K_{\varpi'}(W)|\leq 2\bar fL_f\|u-u'\|+\bar f^2 |K_\varpi(W)-K_{\varpi'}(W)|$, we have $N(\epsilon,\mathcal{U},|\cdot|) \leq (C(1+d_W)/\epsilon)^{1+d_W}$. Therefore, using standard bounds we have
$$ \log N(\epsilon \|\mu_\mathcal{W}^{-1}G\bar f\|_{Q,2}, \mathcal{W}_0\mathcal{W}_1\mathcal{W}_2\mG, \|\cdot\|_{Q,2}) \lesssim  s(1+d_W)^2 \log (L_\gamma' L_f n/\epsilon)$$
By Lemma \ref{lemma:CCK} we have that with probability $1-o(1)$ that
$$ \begin{array}{l} \sup_{u\in \UU, j\in[p]} |(\En-\Ep)[f_u^2\{Z_{-j}^a(\gamma_u^j-\bar\gamma_u^j)\}^2/\Pr(\varpi)]| \\
 \lesssim \sqrt{ \frac{s(1+d_W)^2 \log(p|V| n) \sup_{u\in\UU}\Ep[K_\varpi(W)f_u^4\{Z_{-j}^a(\gamma_u^j-\bar\gamma_u^j)\}^4]/\Pr(\varpi)^2}{n}} + \frac{s(1+d_W)^2M_n^2\mu_\mathcal{W}^{-1}\sup_{u\in\UU} \|\bar\gamma_u^j-\gamma_u^j\|_1^2 \log(p|V|n)}{n}\\
 \lesssim \sqrt{\frac{s(1+d_W)^2\log(p|V|n)}{\mu_\mathcal{W}n}}\frac{s\log(p|V|n)}{n} + \frac{(1+d_W)^2M_n^2s^2\log(p|V|n)}{n\mu_\mathcal{W}}\frac{s\log(p|V|n)}{n}\\
 \lesssim \frac{s\log(p|V|n)}{n}\mu_\mathcal{W}\underf_\mathcal{U} \left\{ \sqrt{\frac{s(1+d_W)^2\log(p|V|n)}{\mu_\mathcal{W}^3\underf_\mathcal{U}^2 n}} + \frac{(1+d_W)^2M_n^2s^2\log(p|V|n)}{n\mu_\mathcal{W}^2\underf_\mathcal{U}}  \right\}\\
 \lesssim \frac{s\log(p|V|n)}{n}\mu_\mathcal{W}\underf_\mathcal{U} \{  \delta_n^{1/2} + \delta_n^2 \}
\end{array}$$
here we used that $\Ep[f_u^4\{Z^a\delta\}^4\vert \varpi] \leq \bar f^4 \Ep[\{Z^a\delta\}^4\vert \varpi] \leq C\|\delta\|^4$, $\|\bar\gamma_u^j-\gamma_u^j\|+s^{-1/2}\|\bar\gamma_u^j-\gamma_u^j\|_1 \leq \{n^{-1}s\log(p|V|n)\}^{1/2}$, $ s(1+d_W)^2\log(p|V|n)\leq \delta_n n \underf_\mathcal{U}^2\mu_\mathcal{W}^3$ and $(1+d_W)M_n s\log^{1/2}(p|V|n)\leq \delta_n n^{1/2}\mu_\mathcal{W}\underf_\mathcal{U}$ by Condition CI. Furthermore, by Condition CI, the result follows from $\Ep[f_u^2\{Z_{-j}^a(\bar\gamma_u^j-\gamma_u^j)\}^2\vert \varpi] \leq C\underf^2_u\|\bar\gamma_u^j-\gamma_u^j\|^2 \leq C \underf^2_u n^{-1} s\log(p|V|n)$.
\end{proof}

\begin{lemma}\label{Lemma:LipsGamma}
Under Condition CI, for $u=(a,\tau,\varpi) \in \UU$ and $u'=(a,\tau',\varpi')\in \UU$ we have that $$\|\gamma_u^j-\gamma_{u'}^j\| \leq \frac{C'}{\underline{f}_{u'}^2\Pr(\varpi')}\{ \bar f^2\Ep[\{K_{\varpi'}(W)-K_{\varpi}(W)\}^2]^{1/2}+ \Ep[K_\varpi(W)K_{\varpi'}(W)\{f_{u'}^2-f_u^2\}^2]^{1/2}\}. $$ In particular, we have $\|\gamma_u^j-\gamma_{u'}^j\| \leq L_\gamma \{ \|\varpi-\varpi'\|^{1/2}+\|u-u'\| \}$ for $L_\gamma = C\{L_f+L_K\}/\{\underf_{\mathcal{U}}^2\mu_\mathcal{W}\}$ under $\Ep[|K_\varpi(W)-K_{\varpi'}(W)|] \leq L_K\|\varpi-\varpi'\|$,  $K_\varpi(W)K_{\varpi'}(W)|f_{u'}-f_u|\leq L_f\|u'-u\|$, and $f_u \leq \bar f\leq C$.
\end{lemma}
\begin{proof} Let $u=(a,\tau,\varpi)$ and $u'=(a,\tau',\varpi')$. By Condition CI we have $$\begin{array}{rl}
\|\gamma_u^j-\gamma_{u'}^j\|^2 & \leq C \Ep[\{Z^a_{-j}(\gamma_u^j-\gamma_{u'}^j)\}^2\vert \varpi' ]   \leq \{C/\Pr(\varpi')\} \Ep[K_{\varpi'}(W)\{Z^a_{-j}(\gamma_u^j-\gamma_{u'}^j)\}^2]\end{array}$$
To bound the last term of the right hand side above, by the definition of $\underline{f}_{u'}$ and  Cauchy-Schwarz's inequality we have
$$\begin{array}{rl}
\underline{f}_{u'}\Ep[K_{\varpi'}(W)\{Z^a_{-j}(\gamma_u^j-\gamma_{u'}^j)\}^2] & \leq \Ep[K_{\varpi'}(W)f_{u'}\{Z^a_{-j}(\gamma_u^j-\gamma_{u'}^j)\}^2]\\
& \leq \{\Ep[K_{\varpi'}(W)f_{u'}^2\{Z^a_{-j}(\gamma_u^j-\gamma_{u'}^j)\}^2] \ \Ep[K_{\varpi'}(W)\{Z^a_{-j}(\gamma_u^j-\gamma_{u'}^j)\}^2]\}^{1/2}\\
\end{array}
$$
so that $\Ep[K_{\varpi'}(W)\{Z^a_{-j}(\gamma_u^j-\gamma_{u'}^j)\}^2]^{1/2}\leq \{\Ep[K_{\varpi'}(W)f_{u'}^2\{Z^a_{-j}(\gamma_u^j-\gamma_{u'}^j)\}^2]\}^{1/2}/\underline{f}_{u'}$. Therefore
  \begin{equation}\label{Eq:AuxAux1}\|\gamma_u^j-\gamma_{u'}^j\|^2 \leq \{1/\underline{f}_{u'}\}^2\{C/\Pr(\varpi)\} \Ep[K_{\varpi'}(W)f_{u'}^2\{Z^a_{-j}(\gamma_u^j-\gamma_{u'}^j)\}^2].\end{equation}
We proceed to bound the last term. The optimality of $\gamma_u^j$ and $\gamma_{u'}^j$ yields
$$ \Ep[K_\varpi(W)f_u^2Z^a_{-j}(Z_j^a - Z^a_{-j}\gamma_u^j)]=0 \ \ \ \mbox{and} \ \ \ \Ep[K_{\varpi'}(W)f_{u'}^2Z^a_{-j}(Z_j^a - Z^a_{-j}\gamma_{u'}^j)]=0 $$
Therefore, we have \begin{equation}\label{AuxStep1}\begin{array}{rl}
\Ep[K_{\varpi'}(W)f_{u'}^2\{Z^a_{-j}(\gamma_u^j-\gamma_{u'}^j)\}Z^a_{-j}]& = -\Ep[K_{\varpi'}(W)f_{u'}^2\{Z^a_j-Z^a_{-j}\gamma_u^j\}Z^a_{-j}] \\
& = -\Ep[\{K_{\varpi'}(W)f_{u'}^2-K_{\varpi}(W)f_u^2\}\{Z^a_j-Z^a_{-j}\gamma_u^j\}Z^a_{-j}] \\
\end{array}
\end{equation}
Multiplying by $(\gamma_u^j-\gamma_{u'}^j)$ both sides of (\ref{AuxStep1}),  we have $$\begin{array}{rl}
 &\Ep[K_{\varpi'}(W)f_{u'}^2\{Z^a_{-j}(\gamma_u^j-\gamma_{u'}^j)\}^2]  \\
 & \leq \Ep[\{K_{\varpi'}(W)f_{u'}^2-K_{\varpi}(W)f_u^2\}^2]^{1/2}\{\Ep[\{Z^a_j-Z^a_{-j}\gamma_u^j\}^2\{Z^a_{-j}(\gamma_u^j-\gamma_{u'}^j)\}^2]\}^{1/2}\\
 & \leq \Ep[\{K_{\varpi'}(W)f_{u'}^2-K_{\varpi}(W)f_u^2\}^2]^{1/2}C \|\gamma_u^j-\gamma_{u'}^j\|\\
 \end{array}$$
by the fourth moment assumption in Condition CI. By Condition CI, $f_u,f_{u'}\leq \bar f$, and it follows that
\begin{equation}|K_\varpi(W)f_u^2-K_{\varpi'}(W)f_{u'}^2| \leq  K_\varpi(W)K_{\varpi'}(W)|f_u^2-f_{u'}^2|+\bar f^2 |K_{\varpi}(W)-K_{\varpi'}(W)|\end{equation}

From (\ref{Eq:AuxAux1}) we obtain
$$\|\gamma_u^j-\gamma_{u'}^j\| \leq \frac{C'}{\underline{f}_{u'}^2\Pr(\varpi')}\{\bar f^2 \Ep[\{K_{\varpi'}(W)-K_{\varpi}(W)\}^2]^{1/2}+ \Ep[K_\varpi(W)K_{\varpi'}(W)\{f_{u'}^2-f_u^2\}^2]^{1/2}\}. $$

\end{proof}

\begin{lemma}\label{Lemma:Matrices}
Let $\mathcal{U} = V \times \mathcal{T} \times \mathcal{W}$. Under Condition CI, for $m=1,2$,  we have
$$ \Ep\left[ \sup_{u \in \mathcal{U}, \|\theta\|_0 \leq k, \|\theta\|=1 } | (\En-\Ep)[K_\varpi(W)f_u^m(Z^a\theta)^2] | \right] \lesssim C \delta_n \sup_{u \in \mathcal{U}, \|\theta\|_0 \leq k, \|\theta\|=1 } \{ \Ep[K_\varpi(W)f_u^m(Z^a\theta)^2]\}^{1/2} $$
where $\delta_n = M_n\sqrt{k(1+d_W)C\log(p|V|n)}\log (1+k) \sqrt{\log n/n}$. Moreover, under Condition CI, $\delta_n = o(\mu_{\mathcal{W}})$.
\end{lemma}
\begin{proof}
By symmetrization we have
$$\Ep\left[ \sup_{u \in \mathcal{U}, \|\theta\|_0 \leq k, \|\theta\|=1 } | (\En-\Ep)[K_\varpi(W)f_u^m(Z^a\theta)^2] | \right] \leq 2 \Ep\left[\sup_{u \in \mathcal{U}, \|\theta\|_0 \leq k, \|\theta\|=1 } | \En[\varepsilon K_\varpi(W)f_u^m(Z^a\theta)^2] | \right] $$
where $\varepsilon_i$ are i.i.d. Rademacher random variables. We have that
$|K_\varpi(W) f_u^m-K_{\varpi'}(W)f_{u'}^m| \leq K_\varpi(W) K_{\varpi'}(W)|f_u - f_{u'} | (1 + 2\bar f)+\bar f^m|K_\varpi(W)- K_{\varpi'}(W)|$ for $m=1,2$ where $u$ and $u'$ have the same $a\in V$. However, conditional on $\{(W_i,X_i), i=1,\ldots,n\}$, $\{K_\varpi(W_i):i=1,\ldots,n, \varpi \in \mathcal{W}\}$ induces at most $n^{d_W}$ different sequences by Corollary 2.6.3 in \cite{vdV-W}. This induces (at most) $n^{d_W}$ partitions of $\mathcal{W}$ such that $K_\varpi(W) = K_{\varpi'}(W)$ for any $\varpi,\varpi'$ in the same partition given the conditioning. Thus, for such suitable $\varpi'$ we have $|K_\varpi(W) f_u^m-K_{\varpi'}(W)f_{u'}^m| \leq K_\varpi(W) |f_{u} - f_{u'} |(1 + 2\bar f)$ for $m=1,2$. (Thus it suffices to create a net for each partition.) We can take a cover $\widehat{\mathcal{U}}$ of $V\times \mathcal{T}\times \mathcal{W}$ such that $\|u-u'\|\leq \{L_f(1+2\bar f)nk\max_{i\leq n}\|Z^a_i\|_\infty^2 \}^{-1}$ so that $|f_u-f_{u'}|(Z^a\theta)^2 \leq |f_u-f_{u'}|\|Z^a\|_\infty^2\|\theta\|_1^2 \leq|f_u-f_{u'}|\|Z^a\|_\infty^2k\|\theta\|^2$ which implies  $$\begin{array}{rl}
\displaystyle \left|\sup_{u \in \mathcal{U}, \|\theta\|_0 \leq k, \|\theta\|=1 } | \En[\varepsilon K_\varpi(W)f_u^m(Z^a\theta)^2] | - \sup_{u \in \widehat{\mathcal{U}}, \|\theta\|_0 \leq k, \|\theta\|=1 } | \En[\varepsilon K_\varpi(W)f_u^m(Z^a\theta)^2] | \right| \leq n^{-1}\end{array}$$
Consequentially
$$ \Ep\left[\sup_{u \in \mathcal{U}, \|\theta\|_0 \leq k, \|\theta\|=1 } | \En[\varepsilon K_\varpi(W)f_u^m(Z^a\theta)^2] | \right] \leq \Ep\left[\sup_{u \in \widehat{\mathcal{U}}, \|\theta\|_0 \leq k, \|\theta\|=1 } | \En[\varepsilon K_\varpi(W)f_u^m(Z^a\theta)^2] | \right] + \frac{1}{n}$$
where $|\widehat{\mathcal{U}}| \leq |V| n^{d_W} \{L_f(1+2\bar f)n k\max_{i\leq n}\|Z^a_i\|_\infty^2\}^{(1+d_W)}$.

By Lemma \ref{lemma:RV34} with $K=K(W,X)= (1+\bar f^2)\sup_{a\in V}\max_{i\leq n}\|Z_i^a\|_\infty$ and
$$\begin{array}{rl}
\delta_n(W,X) & := \bar C K(W,X) \sqrt{k}\left(\sqrt{\log |\hat\UU|} + \sqrt{1+\log p} + \log k \sqrt{\log (p\vee n)} \sqrt{\log n} \right)/\sqrt{n} \\
& \lesssim  K(W,X) \sqrt{k(1+d_W)C\log(p|V|nK(W,X))}\log (1+k) \sqrt{\log n}/\sqrt{n} \\
\end{array}$$
so that conditional on $(W,X)$ we have
$$\Ep\left[\sup_{u \in \widehat{\mathcal{U}}, \|\theta\|_0 \leq k, \|\theta\|=1 } | \En[\varepsilon K_\varpi(W)f_u^m(Z^a\theta)^2] | \right]\lesssim \delta_n(W,X) \sup_{u \in \widehat{\mathcal{U}}, \|\theta\|_0 \leq k, \|\theta\|=1 }\sqrt{\En[K_\varpi(W)f_u^m(Z^a\theta)^2]} $$
Therefore,
$$\begin{array}{l}
\displaystyle  \Ep\left[\sup_{u \in \mathcal{U}, \|\theta\|_0 \leq k, \|\theta\|=1 } | \En[\varepsilon K_\varpi(W)f_u^m(Z^a\theta)^2] | \right] \\
 \displaystyle  \leq \Ep_{W,X}\left[\delta_n(W,X) \sup_{u \in \widehat{\mathcal{U}}, \|\theta\|_0 \leq k, \|\theta\|=1 }\sqrt{\En[K_\varpi(W)f_u^m(Z^a\theta)^2]}\right] + \frac{1}{n} \\
\displaystyle  \leq \Ep_{W,X}[\delta_n^2(W,X)] + \Ep_{W,X}[\delta_n^2(W,X)]^{1/2} \sup_{u \in \widehat{\mathcal{U}}, \|\theta\|_0 \leq k, \|\theta\|=1 }\Ep[K_\varpi(W)f_u^m(Z^a\theta)^2]^{1/2} + \frac{1}{n}
 \end{array}
 $$
Note that for a random variable $A\geq 1$, we have that $\Ep[A^2\sqrt{\log(CA)}]\leq \Ep[A^2]\sqrt{\log(C)} + \Ep[A^2\sqrt{\log(A)}] \leq \Ep[A^2]\sqrt{\log(C)} + \Ep[A^{2+1/4}]$.
Therefore, under Condition CI, since $q\geq 2+1/4$ in the definition of $M_n$, we have
$$ \Ep_{W,X}[\delta_n^2(W,X)]^{1/2} \lesssim M_n\sqrt{k(1+d_W)C\log(p|V|n)}\log (1+k) \sqrt{\log n}/\sqrt{n}.$$
The results follows by setting $\delta_n=\Ep_{W,X}[\delta_n^2(W,X)]^{1/2}$.
\end{proof}

\section{Results for Prediction Quantile Graphical Models}

In the analysis of PQGM we also use the following event for some sequence $(K_u)_{u\in\UU}$
\begin{equation}\label{def:Omega4} \begin{array}{c}\Omega_4 =\{ K_{u} \hat{\sigma}^{X}_{a\varpi j} \geq \En[\Ep[ K_\varpi(W)(\tau-1\{X_a\leq X_{-a}'\beta_u+r_u\})X_{-a}\vert X_j,W  ]], \ u\in\UU, j\in V\backslash\{a\} \}.\end{array} \end{equation}

\begin{lemma}[Rate for PQGM]\label{Theorem:L1QRnpPrediction}
Suppose that $\Omega_1$, $\Omega_2$, $\Omega_3$ and $\Omega_4$ hold. Further assume $2\frac{1+1/c}{1-1/c}\|\beta_u\|_{1,\varpi} + \frac{1}{\lambda_u(1-1/c)}\bar R_{u\xi} \leq \sqrt{n}$ for all $u\in\UU$, and (\ref{ConditionIdenfitication}) holds for all $\delta\in A_u := \Delta_{\varpi,2\cc}\cup\{v:\|v\|_{1,\varpi}\leq 2\cc\bar R_{u\xi}/\lambda_u\}$, $\bar q_{A_u}/4 \geq (\sqrt{\bar f}+1) \|r_{u}\|_{n,\varpi} + \left[\lambda_u+t_3+K_u\right]\frac{3\cc\sqrt{s}}{\kappa_{u,2\cc}}$ and $\bar q_{A_u} \geq \{ 2\cc\left(1+ \frac{t_3+K_u}{\lambda_u}\right)\bar R_{u\xi} \}^{1/2}$.
Then uniformly over all $u=(a,\tau,\varpi)\in \UU := V\times \mathcal{T}\times \mathcal{W}$, the $\|\cdot\|_{1,\varpi}$-penalized estimator $\hat\beta_u$ satisfies
{\small $$\begin{array}{rl}
\|\sqrt{f_u}X_{-a}'(\hat\beta_u - \beta_u)\|_{n,\varpi} & \leq \sqrt{8\cc\left(1+ \frac{t_3}{\lambda_u}\right)\bar R_{u\xi}} + (\bar f^{1/2}+1) \|r_{u}\|_{n,\varpi}+ [\lambda_u+t_3+K_u]\frac{3\cc\sqrt{s}}{\kappa_{u,2\cc}}\\
\|\hat\beta_u - \beta_u\|_{1,\varpi} & \leq  (1+2\cc)\sqrt{s}\|\sqrt{f_u}X_{-a}'\delta_u\|_{n,\varpi}/\kappa_{u,2\cc} + \frac{2\cc}{\lambda_u}\bar R_{u\xi}  \\
\end{array}$$}\end{lemma}
\begin{proof}[Proof of Lemma \ref{Theorem:L1QRnpPrediction}]
The proof proceeds similarly to the proof of Lemma \ref{Theorem:L1QRnp} by defining
$$\begin{array}{rl}
\hat R_u(\beta) & = \En[K_\varpi(W)\{\rho_u(X_a-X_{-a}'\beta)-\rho_u(X_a-X_{-a}'\beta_u-r_{u})\}]\\
& -\En[K_\varpi(W)\{(\tau-1\{X_a\leq X_{-a}'\beta_u+r_{u}\})(X_{-a}'\beta-X_{-a}'\beta_u-r_{u})\}].\end{array}$$
The same argument yields $\delta_u=\hat\beta_u-\beta_u \in A_u := \Delta_{\varpi,2\cc}\cup \{ v : \|v\|_{1,\varpi}\leq 2\cc \bar R_{u\xi}/\lambda_u\}$ under $\Omega_1 \cap \Omega_2$. (Similarly we also have $\|\delta_u\|_{1,\varpi}\leq \sqrt{n}$.) Furthermore, under $\Omega_1\cap \Omega_2\cap\Omega_3$ we have that (\ref{AuxUpper1}) also holds which implies
$$\En[\Ep[K_\varpi(W)\{\rho_u(X_a-X_{-a}'(\beta_u+ \delta_u))-\rho_u(X_a-X_{-a}'\beta_u)\} \vert X_{-a},W]  \leq  (\lambda_u+t_3)\|\delta_u\|_{1,\varpi}$$
Since the conditions of Lemma \ref{Lemma:IdentificationSparseQRNP} hold we have
$$\begin{array}{rl}
 \En[\Ep[K_\varpi(W)\{\rho_\tau(X_a-X_{-a}'(\beta_u+ \delta_u))-\rho_\tau(X_a-X_{-a}'\beta_u)\} \vert X_{-a},W] \\
 \geq - (\sqrt{\bar f}+1) \|r_{u}\|_{n,\varpi} \|\sqrt{f_u}X_{-a}' \delta\|_{n,\varpi} - K_u\|\delta_u\|_{1,\varpi} \\
 + \frac{\|\sqrt{f_u}X_{-a}'\delta_u\|_{n,\varpi}^2}{4} \wedge \bar q_{A_u}\|\sqrt{f_u}X_{-a}'\delta_u\|_{n,\varpi}
 \end{array}$$
where $K_u$ is given in $\Omega_4$ which accounts for the misspecification the conditional quantile condition. Therefore,  we have
$$\begin{array}{rl}
\frac{\|\sqrt{f_u}X_{-a}'\delta_u\|_{n,\varpi}^2}{4} \wedge \bar q_{A_u}\|\sqrt{f_u}X_{-a}'\delta_u\|_{n,\varpi} & \leq  (\bar f^{1/2}+1) \|r_{u}\|_{n,\varpi} \|\sqrt{f_u}X_{-a}' \delta_u\|_{n,\varpi} + (\lambda_u + t_3 + K_u)\|\delta_u\|_{1,\varpi} \\
& \leq \{ (\bar f^{1/2}+1) \|r_{u}\|_{n,\varpi} +\frac{3\cc\sqrt{s}}{\kappa_{u,2\cc}}(\lambda_u + t_3 + K_u)\} \|\sqrt{f_u}X_{-a}' \delta_u\|_{n,\varpi}\\
&  + (\lambda_u + t_3 + K_u)\frac{2\cc}{\lambda_u}\bar R_{u\xi}
 \end{array}$$
The result then follows with the same argument under the current assumptions that account for $K_u$.
\end{proof}

\begin{lemma}[PQGM, Event $\Omega_1$]\label{Lemma:PenaltyMisspecification}
Under Condition P, we have
$$ \Pr\left( \sup_{u\in \mathcal{U},j\in [d]} \frac{|\En[K_\varpi(W)(\tau -1\{X_a\leq X_{-a}'\beta_u+r_u\})X_{-a, j}]|}{\hat{\sigma}^{X}_{a\varpi j}} > t \right) \leq 8 |V|(\frac{ne}{d_W})^{2d_W} \exp\left( - \left\{\frac{t/(1+\bar{\delta}_n)}{2(1+1/16)}\right\}^2\right)$$
where $t\geq 4 \sup_{u\in \mathcal{U}} \{\mathrm{E}[K_\varpi(W)(\tau-1\{X_a\leq X_{-a}'\beta_u+r_u\})^2X_{-a, j}^2]\}^{1/2}$ and $\bar{\delta}_n = o(1)$. In particular, the RHS is less than $\xi$ if $t\geq  2(1+\bar{\delta}_n)(1+1/16) \sqrt{\log(8|V|\{ne/d_W\}^{2d_W}/\xi)}$.
\end{lemma}
\begin{proof}
Set $\sigma_{a\varpi j}^X := \Ep[K_\varpi(W)X_{-a,j}^2]^{1/2}$. We have that for any $\bar \delta_n \to 0$
\begin{equation}\label{interOmega1bound0} \begin{array}{rl}
 \Pr( \lambda_0 \leq {\displaystyle \sup_{u\in \mathcal{U}, j\in [d]} }|\En[K_\varpi(W)(\tau -1\{X_a\leq X_{-a}'\beta_u\})X_{-a,j}]|/\hat{\sigma}^{X}_{a\varpi j}) \\
 \leq \Pr( \lambda_0 \leq (1+\bar \delta_n){\displaystyle \sup_{u\in \mathcal{U}, j\in [d]} }|\En[K_\varpi(W)(\tau -1\{X_a\leq X_{-a}'\beta_u\})X_{-a,j}]|/ {\sigma}^{X}_{a\varpi j} )\\
 + \Pr( {\displaystyle \sup_{u \in \mathcal{U},j\in [d]} } {\sigma}^{X}_{a\varpi j}/\hat{\sigma}^{X}_{a\varpi j}  \geq (1+\bar \delta_n) )  \end{array} \end{equation}

To bound the last term in (\ref{interOmega1bound0}), note that under Condition P,  $ c\mu_\mathcal{W} \leq  ({\sigma}^{X}_{a\varpi j})^2 \leq C$ and $\mathcal{W}$ is a VC class of set with VC dimension $d_W$. Therefore, by Lemma \ref{lemma:CCK} we have that with probability $1-o(1)$
\begin{equation}\label{interOmega1bound}  \sup_{u \in \mathcal{U}, j\in [d]} (\En-\Ep)[K_\varpi(W)X_{-a,j}^2] \lesssim \sqrt{\frac{(1+d_W)\log(|V|M_n/\sigma_1)}{n}} + \frac{(1+d_W)M_n^2\log(|V|M_n/\sigma_1)}{n} \end{equation}
for ${\sigma}^2_{1} = \max_{u \in \mathcal{U}, j\in [d]}\Ep[K_\varpi(W)X_{-a,j}^2] \leq \max_{j\in V} \Ep[X_{-a,j}^2] \leq C$ and envelope $F=\|X\|_\infty^2$ so that $\|F\|_{P,2}\leq \|\max_{i\leq n}F_i\|_{P,2}\leq M_n^2$. Thus for $\bar \delta_n \to 0$, provided $(1+d_W)M_n^2\log(|V|n) = o(n^{1/2})$ and $ (1+d_W)\log(|V|n)=o(n\bar \delta_n^2 \mu_\mathcal{W}^2)$, so that the RHS of (\ref{interOmega1bound}) is $o(\bar\delta_n\mu_\mathcal{W})$, we have \begin{equation}\label{AuxProbEvents}\Pr\left( \frac{1}{1+\bar\delta_n} \leq \frac{\sigma_{a\varpi j}^X}{\hat{\sigma}_{a\varpi j}^X } \leq (1+\bar\delta_n), \ \mbox{for all} \ u\in \UU, j\in [d]\right) = 1-o(1). \end{equation}

Now we bound the first term of the RHS of (\ref{interOmega1bound0}). and let $\sigma_2^2 = \sup_{u\in\UU,j\in [d]} {\rm Var}(K_\varpi(W)(\tau -1\{X_a\leq X_{-a}'\beta_u\})X_{-a,j}/\sigma_{a\varpi j}^X)\leq 1$. By symmetrization (adapting Lemma 2.3.7 in \cite{vdVaartWellner2007} to replace the ``arbitrary" factor $2$ with $1+\bar \delta_n$), for $\delta := 1/(2(1+n\sigma_2^2/t^2))<1/2$ we have
$$ \begin{array}{rl}
(*):= \Pr(  \sup_{u\in \mathcal{U}, j\in [d]} |\sum_{i=1}^n K_\varpi(W_i)(\tau -1\{X_{ia}\leq  X_{i,-a}'\beta_u\})X_{i,-aj}/\sigma^{X}_{a\varpi j}| \geq t) \\ \leq 2 \Pr(  \sup_{u\in \UU, j\in [d]} |\sum_{i=1}^n\varepsilon_i K_\varpi(W_i)(\tau -1\{X_{ia}\leq X_{i,-a}'\beta_u\})X_{i,-aj}/\sigma^{X}_{a\varpi j}|\geq t\delta  )\\
\leq 2 \Pr\left( {\displaystyle \sup_{u\in \UU, j\in [d]}} \frac{|\sum_{i=1}^n\varepsilon_i K_\varpi(W_i)(\tau -1\{X_{ia}\leq X_{i,-a}'\beta_u\})X_{i,-aj}|}{\hat{\sigma}^{X}_{a\varpi j} }\geq t\delta/(1+\bar\delta_n)  \right) + o(1)
\end{array}$$
where $\varepsilon_i, i=1,\ldots,n$, are Rademacher random variables independent of the data, and the last inequality follows from (\ref{AuxProbEvents}).

Therefore, by the union bound and symmetry, and iterated expectations we have
$$ (*)\leq 4 |V| \max_{j\in [d]} \Ep_{W,X}\left[ \Pr_{\varepsilon}\left( {\displaystyle \sup_{u\in \UU}} \frac{\left|\sum_{i=1}^n\varepsilon_i K_\varpi(W_i)(\tau -1\{X_{ia}\leq X_{i,-a}'\beta_u\})X_{i,-aj}\right|}{\hat{\sigma}^{X}_{a\varpi j} }\geq t\delta/(1+\bar{\delta}_n) \mid W,X \right) \right] $$

Next we use that $\{ \varpi \in \mathcal{W}\}$ is a VC class of sets with VC dimension bounded by $d_W$ and $\{ 1\{X_{a}\leq X_{-a}'\beta_u\} : (\tau,\varpi) \in \mathcal{T}\times \mathcal{W}\}$ is a VC class of sets with VC dimension bounded by $1+d_W$. By Corollary 2.6.3 in \cite{vdV-W}, we have that conditionally on $(W_i,X_i)_{i=1}^n$, the set of (binary) sequences $\{ (K_\varpi(W_i))_{i=1,\ldots,n} : \varpi \in \mathcal{W}\}$ has at most $\sum_{j=0}^{d_W-1}\binom{n}{j}$ different values. Similarly,  $\{ (1\{X_{ia}\leq X_{i,-a}'\beta_u\})_{i=1,\ldots,n} : u \in \UU\}$ assumes at most $\sum_{j=0}^{d_W}\binom{n}{j}$ different values. Assuming that $n \geq d_W$, we have  $\sum_{j=0}^{k}\binom{n}{j} \leq \{ne/k\}^{k}$ and
$$\begin{array}{l}
\Pr_\varepsilon\left(  \sup_{u\in \mathcal{U}}\frac{\left|\sum_{i=1}^n\varepsilon_i K_\varpi(W_i)(\tau -1\{X_{ia}\leq X_{i,-a}'\beta_u\})X_{i,-aj}\right|}{\hat{\sigma}^{X}_{a\varpi j} }\geq t\delta/(1+\bar{\delta}_n) \mid W,X \right) \\
\leq \{\frac{ne}{d_W-1}\}^{d_W-1}\{\frac{ne}{d_W}\}^{d_W}\sup_{u \in \mathcal{U}} \Pr_\varepsilon\left( \sup_{\tilde \tau \in \mathcal{T}} \frac{\left|\sum_{i=1}^n\varepsilon_i K_\varpi(W_i)(\tilde \tau -1\{X_{ia}\leq X_{i,-a}'\beta_u\})X_{i,-aj}\right|}{\hat{\sigma}^{X}_{a\varpi j} }\geq t\delta/(1+\bar{\delta}_n) \mid W,X \right) \\
\leq  \{ne/d_W\}^{2d_W}\sup_{u\in \mathcal{U}, \tilde \tau \in [\underline{\tau},\bar\tau]} \Pr_\varepsilon\left( \frac{\left|\sum_{i=1}^n\varepsilon_i K_\varpi(W_i)(\tilde \tau -1\{X_{ia}\leq X_{i,-a}'\beta_u\})X_{i,-aj}\right|}{\hat{\sigma}^{X}_{a\varpi j} }\geq t\delta/(1+\bar{\delta}_n) \mid W,X \right) \\
\leq 2\{ne/d_W\}^{2d_W} \exp( - \{t\delta/[1+\bar{\delta}_n]\}^2)
\end{array}
$$
here we used that the expression is linear in $\tau$ and so it is maximized at the extremes. Combining the bounds in the last two displayed equations we have
$$ (*) \leq 8|V| \{ne/d_W\}^{2d_W} \exp( - \{t\delta/[1+\bar\delta_n]\}^2).$$
Therefore, setting $\lambda_0 = ct/n$ where $t\geq 4 \sqrt{n}\sigma_2$ and $t\geq  2(1+\bar\delta_n)(1+1/16) \sqrt{2\log(8p|V|\{ne/d_W\}^{2d_W}/\xi)}$. (Note that $t\geq 4 \sqrt{n}\sigma_2$ implies that $\delta \geq 1/\{2(1+1/16)\}$.)
\end{proof}

\begin{lemma}[PQGM, Event $\Omega_2$]\label{Lemma:PQGMControlR}
Under Condition P we have  
$$\Pr\left( \sup_{u\in\mathcal{U}}\hat R_u(\bar \beta_u) \leq C\{1+ \bar f\} \{n^{-1}s(1+d_W)\log(|V|n)\} \right) = 1-o(1).$$
\end{lemma}
\begin{proof}[Proof of Lemma \ref{Lemma:PQGMControlR}]
We have that $\hat R_u(\bar \beta_u)\geq 0$ by convexity of $\rho_\tau$. Let $\varepsilon_{iu}=X_{ia} - X_{i,-a}\bar \beta_u - r_{iu}$ where $\|\bar\beta_u\|_0\leq s$ and $r_{iu}=X_{-a}'(\beta_u-\bar\beta_u)$. By Knight's identity (\ref{Eq:TrickRho}), $\hat R_u(\bar\beta_u) =  -\En[ K_\varpi(W)r_{u} \int_0^1  1\{\varepsilon_{u} \leq -t r_{u}\} - 1\{\varepsilon_{u}\leq 0\} \ dt ]\geq 0$.
 $$ \begin{array}{rl}
 \En[\Ep[\hat R_u(\bar\beta_u)] & = \En[  K_\varpi(W)r_{u} \int_0^1  F_{X_a\mid X_{-a},\varpi}(X_{-a}'\bar\beta_u+(1-t)r_u) -  F_{X_a\mid X_{-a},\varpi}(X_{-a}'\bar\beta_u+r_u) \ dt]\\
  & \leq \En[K_\varpi(W)r_{u} \int_0^1  \bar f t r_{u} dt] \leq \bar f [\|r_{u}\|_{n,\varpi}^2]/2 \leq C\bar f s/n.\end{array}$$
Thus, by Markov's inequality we have $\inf_{u\in \mathcal{U}} \mathrm{P}( \hat R_u(\bar \beta_u) \leq  C\bar f s/n ) \geq 1/2$.

Define $z_{iu} := -\int_0^1  1\{\varepsilon_{iu} \leq -t r_{iu}\} - 1\{\varepsilon_{iu}\leq 0\} \ dt$, so that $\hat R_u(\bar\beta_u) = \En[K_\varpi(W)r_{u}z_{u}]$ with $|z_{iu}|\leq 1$. By Lemma 2.3.7 in \cite{vdVaartWellner2007} (note that the Lemma does not require zero mean stochastic processes), for $t \geq  2 C\bar f s/n$ we have
$$ \frac{1}{2}\mathrm{P}\left( \sup_{u\in\mathcal{U}}|\En[K_\varpi(W)r_{u}z_{u}]| \geq t \right) \leq 2\mathrm{P}\left(\sup_{u\in\mathcal{U}}|\En[\varepsilon K_\varpi(W)r_{u}z_{u}]|>t/4\right) $$
where $\varepsilon_i, i=1,\ldots,n,$ are Rademacher random variables independent of the data.

Consider the class of functions $\mathcal{F}= \{  - K_\varpi(W)r_{u} (1\{\varepsilon_{iu} \leq -B_i r_{iu}\} - 1\{\varepsilon_{iu}\leq 0\} ) : u\in\mathcal{U}\}$ where $B_i\sim {\rm Uniform}(0,1)$ independent of $(X_i,W_i)_{i=1}^n$. It follows that $K_\varpi(W)r_{u}z_{u} =  \Ep[  - K_\varpi(W)r_{u} (1\{\varepsilon_{iu} \leq -B_i r_{iu}\} - 1\{\varepsilon_{iu}\leq 0\} ) \vert X_i,W_i]$ where the expectation is taken over $B_i$ only. Thus we will bound the entropy of $\overline{\mathcal{F}}= \{ \Ep[ f\vert X,W] : f \in \mathcal{F} \}$ via Lemma \ref{Lemma:PartialOutCovering}. Note that  $\mathcal{R}:=\{ r_u = X_{-a}'\beta_u - X_{-a}'\bar\beta_u : u\in\mathcal{U}\}$ where $\mathcal{G}:=\{ X_{-a}'\bar\beta_u : u\in\mathcal{U}\}$ is contained in the union of at most $|V|\binom{p}{s}$ VC-classes of dimension $Cs$ and  $\mathcal{H}:=\{X_{-a}'\beta_u : u\in\mathcal{U}\}\}$ is a VC-class of functions of dimension $(1+d_W)$ by Condition P. Finally note that $\mathcal{E}:=\{ \varepsilon_{iu} : u\in \mathcal{U}\} \subset \{ X_{ia} : a\in V\} - \mathcal{G} - \mathcal{R}$.

Therefore, we have
$$
\begin{array}{rl}
\sup_Q \log N(\epsilon \|\bar F\|_{Q,2}, \overline{\mathcal{F}}, \|\cdot\|_{Q,2}) & \leq  \sup_Q \log N( (\epsilon/4)^2 \| F\|_{Q,2}, \mathcal{F}, \|\cdot\|_{Q,2})\\
& \leq  \sup_Q \log N( \mbox{$\frac{1}{8}$}(\epsilon^2/16), \mathcal{W}, \|\cdot\|_{Q,2}) \\
&+  \sup_Q\log N( \mbox{$\frac{1}{8}$} (\epsilon^2/16) \| F\|_{Q,2}, \mathcal{R}, \|\cdot\|_{Q,2}) \\
&+  \sup_Q \log N( \mbox{$\frac{1}{8}$}(\epsilon^2/16), 1\{ \mathcal{E}+\{B\}\mathcal{R} \leq 0\} - 1\{ \mathcal{E} \leq 0\}, \|\cdot\|_{Q,2})\\
\end{array}
$$
By Lemma \ref{lemma:CCK} with envelope $\bar F = \|X\|_\infty\sup_{u\in\mathcal{U}}\|\beta_u-\bar\beta_u\|_1$, and $(\sigma^{r}_{\max})^2 = \sup_{u\in \mathcal{U}} \Ep[K_\varpi(W)r_{u}^2 ] \lesssim  s/n$ by Condition P, we have that with probability $1-o(1)$
$$ \sup_{u\in\mathcal{U}}|\En[\varepsilon K_\varpi(W)r_{u}z_{u}]| \lesssim \sqrt{\frac{s(1+d_W)\log(|V|n)}{n}}\sqrt{\frac{s}{n}}+ \frac{M_n\sqrt{s^2/n}\log(|V|n)}{n}\lesssim \frac{s(1+d_W)\log(|V|n)}{n}$$
under  $M_n\sqrt{s^2/n} \leq C$.
\end{proof}

\begin{lemma}[PQGM, Event $\Omega_3$]\label{Lemma:PQGMOmega3}
Under Condition P, for $u=(a,\tau,\varpi)\in V\times \mathcal{T}\times \mathcal{W}$, define $g_u(\delta,X,W)= K_\varpi(W)\{\rho_\tau(X_a-X_{-a}'(\beta_u+\delta))-\rho_\tau(X_a-X_{-a}'\beta_u)\}$, and
$$ \Omega_3 := \left\{ \sup_{u\in\mathcal{U}, 1/\sqrt{n} \leq \|\delta\|_{1,\varpi} \leq \sqrt{n}} \frac{|\En[g_u(\delta,X,W)-\Ep[g_u(\delta,X,W)\vert X_{-a}, W]]|}{\|\delta\|_{1,\varpi}} < t_3 \right\}.$$
Then, under Condition CI we have $\mathrm{P}(\Omega_3) \geq 1-\xi$ for any
$$ t_3\sqrt{n} \geq   12 + 16\sqrt{2\log\left(64|V|^2 n^{1+d_W}\log(n)(L_{\beta}M_n\sqrt{n})^{1+d_W/\kappa}  /\xi \right)}  $$
\end{lemma}
\begin{proof}
We have that $\Omega_3^c := \{ \max_{a\in V} A_a \geq t_3\sqrt{n}\}$ for
$$A_a:= \sup_{(\tau,\varpi)\in\mathcal{T}\times\mathcal{W}, \underline{N} \leq \|\delta\|_{1,\varpi} \leq \bar{N}} \sqrt{n}\left|\frac{\En[g_u(\delta,X,W)-\Ep[g_u(\delta,X,W)\vert X_{-a},W]]}{\|\delta\|_{1,\varpi}}\right|.  $$
We will apply Lemma \ref{Lemma:EmpProc:Normalized} with $\rho=\kappa$, $L_\eta=L_\beta$, $\tilde x = X_{-a}$ (so we take $p=|V|$), $\underline{N} = 1/\sqrt{n}$ and $\bar{N} = \sqrt{n}$. Therefore, we have by Lemma \ref{Lemma:EmpProc:Normalized} and the union bound
$$\begin{array}{rl}
\mathrm{P}(\Omega_3^c) & = \mathrm{P}(\max_{a\in V} A_a \geq t_3\sqrt{n}) \\
& \leq |V|\max_{a\in V} \mathrm{P}(A_a \geq t_3\sqrt{n}) \\
& = |V|\max_{a\in V} \Ep_{X_{-a},W}\left\{ \mathrm{P}(A_a \geq t_3\sqrt{n} \mid X_{-a},W )\right\} \\
& \leq |V|  \max_{a\in V} \Ep_{X_{-a},W}\left\{8|V| \ |\widehat{\mathcal{N}}|\cdot |\widehat{\mathcal{W}}|\cdot |\widehat{\mathcal{T}}| \exp(-(t_3\sqrt{n}/4-3)^2/32)\right\}\\
& \leq C |V|^2 n^{1+d_W}\log(n)L_\beta^{1+d_W/\kappa} \Ep_{X_{-a}}\left\{\frac{\max_{i\leq n}\|X_{i,-a}\|_\infty^{1+d_W/\kappa}}{\underline{N}^{1+d_W/\kappa}}\right\}\exp(-(t_3\sqrt{n}/4-3)^2/32) \\
& \leq \xi \end{array} $$
where the last step follows by the choice of $t_3$.
\end{proof}

\begin{lemma}[PQGM, Event $\Omega_4$]\label{Lemma:PredictionOmega4}
Under Condition P, and setting  $K_u =C \sqrt{\frac{(1+d_W) \log(|V|n)}{n}}$, we have that  $ \Pr( \Omega_4 ) = 1-o(1)$.
\end{lemma}
\begin{proof}
First note that by Lemma \ref{lemma:CCK} we have that with probability $1-o(1)$
$$  \sup_{\varpi \in \mathcal{W},j\in V} (\En-\Ep)[K_\varpi(W)X_{-a,j}^2] \lesssim \sqrt{\frac{(1+d_W)\log(|V|M_n/\sigma)}{n}} + \frac{(1+d_W)M_n^2\log(|V|M_n/\sigma)}{n} $$
and $\sigma^{X}_{a\varpi j}\geq c \Pr(\varpi)$. Under $(1+d_W)\log(|V|M_n/\sigma)\leq \delta_n^2 \mu^2_{\mathcal{W}}$ and $(1+d_W)M_n^2\log(|V|M_n/\sigma)\leq \delta_n n \mu_{\mathcal{W}}$, we have that
$|(\En-\Ep)[K_\varpi(W)X_{-a,j}^2]| = o( \sigma^X_{a\varpi j})$ for all $u\in \mathcal{U}$. Therefore, we have
$$
\begin{array}{rl}
& \Pr( \sup_{u \in \mathcal{U}} |\En[h_{uj}(X_{-a},W)]|/\{K_u \hat{\sigma}^{X}_{a\varpi j}\} > 1 ) \\
& \leq \Pr( \sup_{u \in \mathcal{U}} |\En[h_{uj}(X_{-a},W)]|/\{K_u \sigma^{X}_{a\varpi j}\} > 1+ O(\delta_n)) + o(1)\\
\end{array}
$$
Applying Lemma  \ref{lemma:CCK} to $\mathcal{F}=\{ h_{uj}(X_{-a},W)/\sigma^{X}_{a\varpi j}: u \in \mathcal{U}\}$. For convenience define $\bar{\mathcal{H}}_j=  \{ h_{uj}(X_{-a},W) : u \in \mathcal{U}\}$ and $\bar{\mathcal{K}}_j:=\{ \Ep[K_\varpi(W)X_j^2] : \varpi \in \mathcal{W}\}$. Note that $\bar{\mathcal{K}}_j$ has covering numbers bounded by the covering number of $\mathcal{K}_j:=\{ K_\varpi(W)X_j^2 : \varpi \in \mathcal{W}\}$ hence $\sup_Q \log N( \epsilon \|\bar K_j\|_{Q,2}, \bar{\mathcal{K}}_j, \|\cdot\|_{Q,2}) \leq \log \sup_{\tilde Q} N((\epsilon/4)^2\|F\|_{\tilde Q,2}, \mathcal{K}_j,\|\cdot\|_{\tilde Q,2})$ by Lemma \ref{Lemma:PartialOutCovering}. Similarly, Lemma \ref{Lemma:PartialOutCovering} also allows us to bound covering numbers of $\bar{\mathcal{H}}_j$ via covering numbers of $\mathcal{H}_j=  \{ K_\varpi(W)(\tau-1\{X_a\leq X_{-a}'\beta_u\})X_j : u \in \mathcal{U}\}$.
$$
\begin{array}{l}
\sup_Q \log N( \epsilon \|F\|_{Q,2}, \mathcal{F}, \|\cdot\|_{Q,2}) \leq p \max_{j\in[p]}\sup_Q \log N( \epsilon \|F_j\|_{Q,2}, \mathcal{F}_j, \|\cdot\|_{Q,2})\\
 \leq  p \max_{j\in[p]}\sup_Q \{ \log N( (1/2)\epsilon \|\bar H_j\|_{Q,2}, \bar{\mathcal{H}}_j, \|\cdot\|_{Q,2})  + \sup_Q \log N( (1/2)\epsilon c \mu_{\mathcal{W}}^{1/2}, 1/\bar{\mathcal{K}}_j^{1/2}, \|\cdot\|_{Q,2})\}\\
 \leq  p \max_{j\in[p]}\sup_Q \{ \log N( (1/2)\epsilon \|\bar H_j\|_{Q,2}, \bar{\mathcal{H}}_j, \|\cdot\|_{Q,2})  + \sup_Q \log N( (1/2)\epsilon c \mu_{\mathcal{W}}^{3/2}, \bar{\mathcal{K}}_j^{1/2}, \|\cdot\|_{Q,2})\}\\
 \leq  p \max_{j\in[p]}\sup_Q \{ \log N( (1/2)\epsilon \|\bar H_j\|_{Q,2}, \bar{\mathcal{H}}_j, \|\cdot\|_{Q,2}) + \sup_Q \log N( C(1/2)\epsilon c\mu_{\mathcal{W}}^{3/2}, \bar{\mathcal{K}}_j, \|\cdot\|_{Q,2})\}\\
 \leq  p \max_{j\in[p]}\sup_Q \{ \log N( (1/2)\epsilon \|\bar H_j\|_{Q,2}, \bar {\mathcal{H}}_j, \|\cdot\|_{Q,2})  + \sup_Q \log N( 1/(2C)\epsilon c\mu_{\mathcal{W}}^{3/2}, \bar{\mathcal{K}}_j, \|\cdot\|_{Q,2})\}\\
 \leq  p \max_{j\in[p]}\sup_{\tilde Q}   \log N( (1/4)\epsilon^2 \|H_j\|_{Q,2}, \mathcal{H}_j, \|\cdot\|_{Q,2}) \\
  + p \max_{j\in[p]}\sup_{\tilde Q} \log N( (1/(4C^2)\epsilon^2 c^2\mu_{\mathcal{W}}^{3}, \mathcal{K}_j, \|\cdot\|_{\tilde Q,2}) \\ \end{array}
$$
where $F_j = c\|X\|_\infty/\mu_{\mathcal{W}}^{1/2}$, $H_j=\|X\|_\infty$. Since $\mathcal{K}_j$ is the product of a VC subgraph of dimension $d_W$ with a single function, and $\mathcal{H}_j$ is the product of two VC subgraph of dimension $1+d_W$ and a single function, by Lemma  \ref{lemma:CCK} with $\sigma^2 = 1$, we have with probability $1-o(1)$
$$\sup_{u\in \mathcal{U}} \left| \frac{(\En-\Ep)[h_{uj}(X_{-a},W)]}{\Ep[K_\varpi(W)X_j^2]^{1/2}}\right| \leq C \sqrt{\frac{(1+d_W) \log(|V|n)}{n}} + C \frac{M_n (1+d_W)\log(|V|n)}{n\mu_{\mathcal{W}}^{1/2}}. $$
Thus, under $M_n(1+d_W)\log(|V|n) \leq n^{1/2}\mu_{\mathcal{W}}^{1/2}$ we have that we can take $K_u = C \sqrt{\frac{(1+d_W) \log(|V|n)}{n}}$.
\end{proof}

\section{Technical Results for High-Dimensional Quantile Regression}

In this section we provide technical results for high-dimensional quantile regression. It is based on a sample $(\tilde y_i, \tilde x_i, W_i)_{i=1}^n$, independent across $i$, $\rho_\tau(t) = (\tau-1\{t \leq 0\})t$, $\tau \in \mathcal{T} \subset (0,1)$ a compact interval, and a family of indicator functions $K_w(W)=1$ if $W \in \Omega_\varpi$, $K_w(W)=0$ otherwise, here $\Omega_\varpi\in\mathcal{W}$. For convenience we index the sets $\Omega_\varpi$ by $\varpi \in B_W \subset \RR^{d_W}$ where we normalize the diameter of $B_W$ to be less or equal than 1/6. Let $f_{\tilde y|\tilde x,r_u,\varpi}(\cdot)$ denote the conditional density function, $f_{\tilde y|\tilde x,r_u,\varpi}(\cdot)\leq \bar f$, $|f_{\tilde y|\tilde x,r_u,\varpi}'(\cdot)|\leq \bar f'$ and $f_u:=f_{\tilde y|\tilde x,r_u,\varpi}(\tilde x'\eta_u)$. Moreover, we assume that
\begin{equation}\label{LforEta}
\|\eta_{u}-\eta_{\tilde u}\|_1 \leq L_\eta\{|\tau-\tilde \tau|+ \|\varpi-\tilde\varpi\|^\rho \}.\end{equation}

Although the results can be applied more generally, these results will be used for $(\eta_u,r_u), u=(\tilde y, \tau,\varpi) \in \mathcal{U} :=\{\tilde y\}\times \mathcal{T}\times \mathcal{W}$ satisfying
$$\Ep[ K_\varpi(W)(\tau-1\{\tilde y \leq \tilde x'\eta_u + r_u \})\tilde x] = 0.$$
Note that this generality is flexible enough to allow us to cover the case that the $\tau$-conditional quantile function $Q_{\tilde y}(\tau\vert \tilde x, \varpi ) = \tilde x'\tilde \eta_u+\tilde r_{u}$ by setting $\eta_u = \tilde\eta_u$ and $r_u=\tilde r_u$ in which case $\Ep[(\tau-1\{\tilde y \leq \tilde x'\eta_u + r_u \})\vert \tilde x, \varpi] = 0$. It also covers the case that
$$ \tilde \eta_u \in \arg\min_\beta \Ep[ K_\varpi(W) \rho_\tau(\tilde y - \tilde x'\beta)] $$
so that $\Ep[ K_\varpi(W)(\tau-1\{\tilde y \leq \tilde x'\tilde \eta_u \})\tilde x]=0$ holds by the first order condition by setting $\eta_u=\tilde \eta_u$ and $r_u=0$. Moreover, it also covers the case that we work with a sparse approximation $\bar \eta_u$ of $\tilde \eta_u$ by setting $\eta_u=\bar\eta_u$ and $r_u=\tilde x'(\tilde\eta_u-\bar\eta_u)$.

\begin{lemma}[Identification Lemma]\label{Lemma:IdentificationSparseQRNP}
For $u=(a,\tau,\varpi)\in\UU$, and a subset $A_u\subset \RR^p$ let
\begin{equation}\label{def:qAu} \bar q_{A_u} =  1/(2\bar{f'}) \cdot \inf_{ \delta \in A_u} \En\[K_\varpi(W)f_u|\tilde x'\delta|^2\]^{3/2}/\En\[K_\varpi(W)|\tilde x'\delta|^3\]\end{equation}
and assume that for all $\delta \in A_u$
\begin{equation}\label{ConditionIdenfitication}\En\[ K_\varpi(W)|r_{u}|\cdot |\tilde x'\delta|^2\]  + \En\[ K_\varpi(W)r_{u}^2\cdot |\tilde x'\delta|^2\] \leq 1/(4\bar f')\En[K_\varpi(W)f_u|\tilde x'\delta|^2].\end{equation}
Then  we have
$$\begin{array}{l} \En[\Ep[K_\varpi(W)\rho_\tau(\tilde y-\tilde x'(\eta_u + \delta))\mid \tilde x,r_u,W ]] - \En[\Ep[K_\varpi(W)\rho_\tau(\tilde y-\tilde x'\eta_u)\mid \tilde x,r_u,W]] \\
\geq  \frac{\|\sqrt{f_u}\tilde x'\delta\|_{n,\varpi}^2}{4} \wedge \left\{ \bar q_{A_u}\|\sqrt{f_u}\tilde x'\delta\|_{n,\varpi}\right\} - K_{n2}\|\sqrt{f_u}\tilde x'\delta\|_{n,\varpi}-K_{n1}\|\delta\|_{1,\varpi}.\end{array}$$
where $K_{n2}:=(\bar f^{1/2}+1)\|r_{u}\|_{n,\varpi}$ and $K_{n1} := \sup_{u\in\mathcal{U}, j\in[p]} \frac{|\En[\Ep[K_\varpi(W)(\tau-1\{\tilde y \leq \tilde x'\eta_u+r_u\})\tilde x_j \mid \tilde x, W]]|}{\{\En[K_\varpi(W)\tilde x_j^2]\}^{1/2}}$.\end{lemma}
\begin{proof}[Proof of Lemma \ref{Lemma:IdentificationSparseQRNP}] Let $T_u = \supp(\eta_u)$, and $Q_u(\eta):=\En\Ep[K_\varpi(W)\rho_\tau(\tilde y-\tilde x'\eta)\mid \tilde x,r_u,W]$. The proof proceeds in steps.

 Step 1. (Minoration)   Define the maximal radius over which the criterion function can be minorated by a quadratic function
$$ r_{A_u} = \sup_{r} \left\{\begin{array}{rl}
 r \  : &  Q_u(\eta_u+ \delta) - Q_u(\eta_u) + K_{n2}\|\sqrt{f_u}\tilde x'\delta\|_{n,\varpi}+ K_{n1}\|\delta\|_{1,\varpi} \geq \frac{1}{4} \|\sqrt{f_u}\tilde x' \delta\|^{2}_{n,\varpi}, \\
   & \forall  \delta\in A_u, \ \|\sqrt{f_u}\tilde x' \delta\|_{n,\varpi} \leq r \end{array}\right\}.$$
Step 2 below shows that  $r_{A_u} \geq \bar q_{A_u}$. By construction of $r_{A_u}$ and the convexity of $Q_u(\cdot)$, $\|\cdot \|_{1,\varpi}$ and $\|\cdot \|_{n,\varpi}$,
 $$ \begin{array}{lll}
 && Q_u(\eta_u + \delta) - Q_u(\eta_u)  +K_{n2}\|\sqrt{f_u}\tilde x'\delta\|_{n,\varpi} + K_{n1}\|\delta\|_{1,\varpi}  \geq  \\
&&     \geq \frac{\|\sqrt{f_u}\tilde x'\delta\|^2_{n,\varpi}}{4} \wedge \left\{ \frac{\|\sqrt{f_u}\tilde x'\delta\|_{n,\varpi}}{r_{A_u}} \cdot \begin{array}{rl}
 \inf_{\tilde\delta} &  Q_u(\eta_u+\tilde \delta) - Q_u(\eta_u) +K_{n2}\|\sqrt{f_u}\tilde x'\tilde\delta\|_{n,\varpi} + K_{n1}\|\tilde\delta\|_{1,\varpi}\\
 & {\tilde \delta\in A_u, \|\sqrt{f_u}\tilde x' \tilde \delta\|_u \geq r_{A_u}}  \end{array} \right\}\\
&&  \geq   \frac{\|\sqrt{f_u}\tilde x'\delta\|^2_{n,\varpi}}{4} \wedge \left\{ \frac{\|\sqrt{f_u}\tilde x'\delta\|_{n,\varpi}}{r_{A_u}} \frac{r_{A_u}^2}{4}\right\}
\\
 && \geq    \frac{\|\sqrt{f_u}\tilde x'\delta\|_{n,\varpi}^2}{4} \wedge \left\{ \bar q_{A_u}\|\sqrt{f_u}\tilde x'\delta\|_{n,\varpi}\right\}.
\end{array}$$

Step 2. ($r_{A_u} \geq \bar q_{A_u}$) Let $F_{\tilde y \mid \tilde x,r_u,\varpi}$ denote the conditional distribution of
$\tilde y$ given $\tilde x,r_u,\varpi$. From \cite{Knight1998},
for any two scalars $w$ and $v$ the Knight's identity is
\begin{equation}\label{Eq:TrickRho}
\rho_\tau(w-v) - \rho_\tau(w) = -v (\tau - 1\{w\leq 0\}) + \int_0^v(
1\{w\leq z\} - 1\{w\leq 0\})dz.
\end{equation}
Using (\ref{Eq:TrickRho}) with $w=\tilde y_i - \tilde x_i'\eta_u$ and $v =
\tilde x_i'\delta$ and taking expectations with respect to $\tilde y$, we have
$$\begin{array}{rl}
 Q_u(\eta_u + \delta) - Q_u(\eta_u)  = & -\En[\Ep\[K_\varpi(W)(\tau-1\{\tilde y\leq \tilde x'\eta_u\})\tilde x_i'\delta \mid \tilde x, r_u, W ]\] \\
& + \En\[ \int_0^{K_\varpi(W)\tilde x'\delta} F_{\tilde y|\tilde x,r_u,\varpi}(\tilde x'\eta_u + t) - F_{\tilde y|\tilde x,r_u,\varpi}(\tilde x'\eta_u) dt \].\end{array}$$
Using the law of iterated expectations and mean value
expansion, the relation
$$\begin{array}{rl}
& |\En[\Ep\[K_\varpi(W)(\tau-1\{\tilde y\leq \tilde x'\eta_u\})\tilde x'\delta \mid \tilde x, r_u, W \]] | \\
& =  |\En\[K_\varpi(W)\{F_{\tilde y\mid \tilde x,r_u,\varpi}(\tilde x'\eta_u+r_u)-F_{\tilde y\mid \tilde x,r_u,\varpi}(\tilde x'\eta_u)\}\tilde x'\delta \] \\
& + \En\[K_\varpi(W)\{\tau -F_{\tilde y\mid \tilde x,r_u,\varpi}(\tilde x'\eta_u+r_u)\}\tilde x'\delta \mid \tilde x, r_u, W \] |\\
& \leq \En[K_\varpi(W)f_u|r_u| \ |\tilde x'\delta|]+\bar f'\En[K_\varpi(W)|r_u|^2 |\tilde x'\delta|] + K_{n1}\|\delta\|_{1,\varpi}\\
& \leq \|\sqrt{f_u}r_u\|_{n,\varpi}\|\sqrt{f_u}\tilde x'\delta\|_{n,\varpi} + \bar f'\|r_u\|_{n,\varpi}\|r_u \tilde x'\delta\|_{n,\varpi}+ K_{n1}\|\delta\|_{1,\varpi}\\
& \leq  (\bar f^{1/2}+1)\|r_u\|_{n,\varpi}\|\sqrt{f_u}\tilde x'\delta\|_{n,\varpi}+ K_{n1}\|\delta\|_{1,\varpi} \end{array}$$ where we used our assumption on the approximation error and we have $K_{n2}=(\bar f^{1/2}+1)\|r_u\|_{n,\varpi}$. With that and similar arguments we obtain for $\tilde t_{\tilde x_i,t} \in [0,t]$
\begin{equation}\label{Eq:Piece1}
\begin{array}{rcl}
&& Q_u(\eta_u + \delta) - Q_u(\eta_u) + K_{n2}\|\sqrt{f_u}\tilde x'\delta\|_{n,\varpi}+K_{n1}\|\delta\|_{1,\varpi} \geq \\
&& Q_u(\eta_u + \delta) - Q_u(\eta_u) +\En[\Ep\[K_\varpi(W)(\tau-1\{\tilde y\leq \tilde x'\eta_u\})\tilde x'\delta \mid \tilde x, r_u, W ]\] = \\
&& = \En\[ \int_0^{K_\varpi(W)\tilde x'\delta} F_{\tilde y|\tilde x,r_u,\varpi}(\tilde x'\eta_u + t) - F_{\tilde y|\tilde x,r_u,\varpi}(\tilde x'\eta_u) dt \] \\
&& =  \En\[ \int_0^{K_\varpi(W)\tilde x'\delta} tf_{\tilde y|\tilde x,r_u,\varpi}(\tilde x'\eta_u) + \frac{t^2}{2}f'_{\tilde y|\tilde x,r_u,\varpi}(\tilde x'\eta_u+\tilde t_{\tilde x,t}) dt \] \\
&& \geq    \frac{1}{2}\|\sqrt{f_u}\tilde x'\delta\|_{n,\varpi}^2  - \frac{1}{6}\bar f ' \En[K_\varpi(W)|\tilde x'\delta|^3] -  \En\[ \int_0^{K_\varpi(W)\tilde x_i'\delta} t[f_{\tilde y\mid \tilde x,r_u,\varpi}(\tilde x'\eta_u)-f_{\tilde y\mid \tilde x,r_u,\varpi}(\tilde x'\eta_u + r_u)]dt\]\\
&&\geq \frac{1}{4}\|\sqrt{f_u}\tilde x'\delta\|_{n,\varpi}^2  + \frac{1}{4}\|\sqrt{f_u}\tilde x'\delta\|_{n,\varpi}^2 - \frac{1}{6} \bar f' \En[K_\varpi(W)|\tilde x'\delta|^3]-(\bar f'/2) \En\[K_\varpi(W) |\tilde r_u|\cdot |\tilde x'\delta|^2\].\\
\end{array}
\end{equation}
Moreover, by assumption we have
\begin{equation}\label{Eq:AUXlemma}\begin{array}{rl}
 \En\[ K_\varpi(W) |r_{u}|\cdot |\tilde x'\delta|^2\] \leq \frac{1}{4\bar f'}\En[K_\varpi(W)f_u|\tilde x'\delta|^2] \\
\end{array} \end{equation}

Note that for any $\delta$ such that $ \|\sqrt{f_u}\tilde x'\delta\|_{n,\varpi} \leq  \bar q_{A_u}$ we have $$ \|\sqrt{f_u}\tilde x'\delta\|_{n,\varpi}\leq  \bar q_{A_u} \leq
 1/(2\bar{f'}) \cdot \En\[K_\varpi(W)f_u|\tilde x'\delta|^2\]^{3/2}/\En\[K_\varpi(W)|\tilde x'\delta|^3\].$$
It follows that  $(1/6)\bar f'\En[K_\varpi(W)|\tilde x'\delta|^3] \leq
(1/8) \En[K_\varpi(W)f_u|\tilde x'\delta|^2]$. Combining this with (\ref{Eq:AUXlemma}) we have
\begin{equation}\label{Eq:Piece2}\frac{1}{4}\En[K_\varpi(W)f_u|\tilde x'\delta|^2] - \frac{\bar f'}{6}  \En[K_\varpi(W)|\tilde x'\delta|^3]-\frac{\bar f'}{2} \En\[K_\varpi(W) |r_u|\cdot |\tilde x'\delta|^2\] \geq 0.\end{equation}

Combining (\ref{Eq:Piece1}) and (\ref{Eq:Piece2}) we have $r_{A_u} \geq \bar q_{A_u}$.
\end{proof}

\begin{lemma}\label{Lemma:EmpProc:Normalized} Let $\mathcal{W}$ be a VC-class of sets with VC-index $d_W$. Conditional on $\{(W_i,\tilde x_i), i=1,\ldots,n\}$ we have
{\tiny \begin{eqnarray*}
P_{\tilde y}\left( \sup_{\footnotesize {\tiny \begin{array}{c}\tau\in\mathcal{T},\varpi\in\mathcal{W}, \\ \underline{N}\leq\|\delta\|_{1,\varpi}\leq \bar N\end{array}}} \left| \Gn\left( K_\varpi(W)\frac{\rho_\tau(\tilde y-\tilde x'(\eta_u+\delta))-\rho_\tau(\tilde y-\tilde x'\eta_u)}{\|\delta\|_{1,\varpi}}\right) \right| \geq M \mid (W_i,\tilde x_i)_{i=1}^n \right) \leq S_n\exp(-(M/4-3)^2/32)
\end{eqnarray*}}
where $ S_n \leq 8p |\widehat{\mathcal{N}}|\cdot |\widehat{\mathcal{W}}|\cdot |\widehat{\mathcal{T}}|$, with
$$ |\widehat{\mathcal{N}}|\leq 1+\floor{3\sqrt{n}\log(\bar N/ \underline{N})}, \ \ |\widehat{\mathcal{T}}|\leq 2\sqrt{n}\frac{\max_{i\leq n}\|\tilde x_i\|_\infty}{\underline{N}}L_\eta, \ \ |\widehat{\mathcal{W}}|\leq n^{d_W}+\left\{2\sqrt{n}\frac{\max_{i\leq n}\|\tilde x_i\|_\infty}{\underline{N}}L_\eta\right\}^{d_W/\rho}.$$
 
\end{lemma}
\begin{proof}[Proof of Lemma \ref{Lemma:EmpProc:Normalized}]
Let $g_{i\tau\varpi}(b) = K_\varpi(W_i)\{\rho_\tau(\tilde y_i-\tilde x_i'\eta_{\tau\varpi} +b)-\rho_\tau(\tilde y_i-\tilde x_i'\eta_{\tau\varpi})\}\leq K_\varpi(W_i)|b|$ since $K_\varpi(W_i)\in\{0,1\}$. Note that   $|g_{i\tau\varpi}(b)- g_{i\tau\varpi}(a)|\leq K_\varpi(W_i)|b-a|$. To easy the notation we omit the conditioning on $(\tilde x_i,W_i)$ from the probabilities.

For any $\delta\in\RR^p$, since $\rho_\tau$ is $1$-Lipschitz, we have
$$ \begin{array}{rl}
{\rm Var}\left( \Gn\left( \frac{g_{\tau\varpi}(\tilde x'\delta)}{\|\delta\|_{1,\varpi}}\right) \right) & \leq \frac{\En[\{g_{\tau\varpi}(\tilde x'\delta)\}^2] }{\|\tilde x'\delta\|_{n,\varpi}^2}  \leq  \frac{\En[|K_\varpi(W)\tilde x'\delta|^2] }{\|\tilde x'\delta\|_{n,\varpi}^2} = 1 \end{array}$$
since by definition $\|\delta\|_{1,\varpi}=\sum_j\|\delta_j\|_{1,\varpi}=\sum_j\|\tilde x'_j\delta_j\|_{n,\varpi}\geq \|\tilde x'\delta\|_{n,\varpi}$.

Since we are conditioning on $(W_i,\tilde x_i)_{i=1}^n$ the process is independent across $i$. Then, by Lemma 2.3.7 in \cite{vdV-W} (Symmetrization for Probabilities) we have for any $M>1$
{\small $$\begin{array}{rl}\displaystyle \Pr\left( \sup_{ \tau \in\mathcal{T}, \varpi\in\mathcal{W},  \underline{N}\leq \|\delta\|_{1,\varpi}\leq \bar{N}} \left| \Gn\left( \frac{g_{\tau\varpi}(\tilde x'\delta)}{\|\delta\|_{1,\varpi}}\right) \right| \geq M \right) \\
\displaystyle  \leq \frac{2}{1-M^{-2}}P\left( \sup_{ \tau \in\mathcal{T}, \varpi\in\mathcal{W}, \underline{N}\leq \|\delta\|_{1,\varpi}\leq \bar{N}} \left| \Gn^o\left( \frac{g_{\tau\varpi}(\tilde x'\delta)}{\|\delta\|_{1,\varpi}}\right) \right| \geq M/4 \right)\end{array}$$} where $\Gn^o$ is the symmetrized process.

Consider $ \mathcal{F}_{t,\tau,\varpi} = \{ \delta : \|\delta\|_{1,\varpi} = t\}$. We will consider the families of $\mathcal{F}_{t,\tau,\varpi}$ for $t \in [\underline{N},\bar N]$, $\tau \in \mathcal{T}$ and $\varpi\in\mathcal{W}$.

We will construct a finite net $\widehat{\mathcal{T}}\times \widehat{\mathcal{W}}\times \widehat{\mathcal{N}}$ of $\mathcal{T}\times \mathcal{W}\times[\underline{N},\bar{N}]$ such that
$$\begin{array}{ll}
 \displaystyle\sup_{\tau \in \mathcal{T},\varpi\in\mathcal{W},t\in[\underline{N},\bar{N}],\delta \in \mathcal{F}_{t,\tau,\varpi}} \left| \Gn^o\left( \frac{g_{\tau\varpi}(\tilde x'\delta)}{\|\delta\|_{1,\varpi}}\right) \right| &
  \leq  \displaystyle 3 + \sup_{\tau\in\widehat{\mathcal{T}}, \varpi\in\widehat{\mathcal{W}},t\in \widehat{\mathcal{N}}}\sup_{\delta \in \mathcal{F}_{t,\tau,\varpi}} \left| \Gn^o\left( \frac{g_{\tau\varpi}(\tilde x'\delta)}{t}\right) \right|
 \displaystyle =: 3 + \mathcal{A}^o.\end{array}$$

By triangle inequality we have
\begin{equation}\label{MainMain} \begin{array}{rl}
\left| \Gn^o\left( \frac{g_{\tau\varpi}(\tilde x'\delta)}{t} - \frac{g_{\tilde \tau\tilde\varpi}(\tilde x'\tilde\delta)}{\tilde t} \right) \right|& \leq  \left| \Gn^o\left( \frac{g_{\tau\varpi}(\tilde x'\delta)}{t} - \frac{g_{\tilde \tau\varpi}(\tilde x'\delta)}{t} \right) \right| +\left| \Gn^o\left( \frac{g_{\tilde\tau\varpi}(\tilde x'\delta)}{t} - \frac{g_{\tilde\tau\tilde\varpi}(\tilde x'\delta)}{t} \right) \right| \\
& +\left| \Gn^o\left( \frac{g_{\tilde\tau\tilde\varpi}(\tilde x'\delta)}{t} - \frac{g_{\tilde\tau\tilde\varpi}(\tilde x'\tilde\delta)}{\tilde t} \right) \right|\end{array}  \end{equation}
The first term in (\ref{MainMain}) is such that
\begin{equation}\label{auxFirst}\begin{array}{rl}
\left| \Gn^o\left( \frac{g_{\tau\varpi}(\tilde x'\delta)}{t} - \frac{g_{\tilde\tau\varpi}(\tilde x'\delta)}{t} \right) \right| & \leq \frac{2\sqrt{n}}{t}\En[ K_\varpi(W) |\tilde x'(\eta_{\tau\varpi}-\eta_{\tilde\tau\varpi})|] \\ & \leq \frac{2\sqrt{n}}{\underline{N}}\max_{i\leq n}\|\tilde x_i\|_\infty\En[K_\varpi(W)]\|\eta_{\tau\varpi} - \eta_{\tilde\tau\varpi}\|_1 \\
& \leq \frac{2\sqrt{n}}{\underline{N}}\max_{i\leq n}\|\tilde x_i\|_\infty\En[K_\varpi(W)] L_\eta|\tau - \tau'|.
\end{array}\end{equation}
Define a net $\widehat{\mathcal{T}}=\{\tau_1,\ldots,\tau_T\}$ such that
$$ |\tau_{k+1} - \tau_k | \leq \left\{2\sqrt{n}\frac{\max_{i\leq n}\|\tilde x_i\|_\infty}{\underline{N}}L_\eta\right\}^{-1}.  $$

To bound the second term in (\ref{MainMain}), note that $\mathcal{W}$ is a VC-class. Therefore, by Corollary 2.6.3 in \cite{vdV-W} we have that conditional on $(W_i)_{i=1}^n$, there are at most $n^{d_W}$ different sets $\varpi \in \mathcal{W}$ that induce a different sequence $\{K_\varpi(W_1),\ldots,K_\varpi(W_n)\}$.  Thus we can choose a (data-dependent) cover $\widehat{\mathcal{W}}$ with at most $n^{d_W}$ values of $\varpi$. Further, similarly to (\ref{auxFirst}) we have $\|\eta_{\tilde\tau\varpi}-\eta_{\tilde\tau\tilde\varpi}\|_1 \leq L_\eta\|\varpi-\tilde \varpi\|^\rho$ and
\begin{equation}\label{auxFirst}\begin{array}{rl}
\left| \Gn^o\left( \frac{g_{\tilde\tau\varpi}(\tilde x'\delta)}{t} - \frac{g_{\tilde\tau\tilde\varpi}(\tilde x'\delta)}{t} \right) \right| & \leq \frac{2\sqrt{n}}{t}\En[ K_\varpi(W) |\tilde x'(\eta_{\tilde\tau\varpi}-\eta_{\tilde\tau\tilde\varpi})|] + \frac{2\sqrt{n}}{t}\En[ |K_\varpi(W)-K_{\tilde \varpi}(W) | |\tilde x'\delta|] \\ & \leq \frac{2\sqrt{n}}{\underline{N}}\max_{i\leq n}\|\tilde x_i\|_\infty\En[K_\varpi(W)]\|\eta_{\tilde \tau\varpi} - \eta_{\tilde\tau\tilde \varpi}\|_1\\
& \leq \frac{2\sqrt{n}}{\underline{N}}\max_{i\leq n}\|\tilde x_i\|_\infty\En[K_\varpi(W)] L_\eta\|\varpi - \tilde\varpi\|^\rho.
\end{array}\end{equation}
We define a net $\widehat{\mathcal{W}}$ such that $|\widehat{\mathcal{W}}|\leq n^{d_W}+\left\{2\sqrt{n}\frac{\max_{i\leq n}\|\tilde x_i\|_\infty}{\underline{N}}L_\eta\right\}^{d_W/\rho}$

To bound the third term in (\ref{MainMain}), note that for any $\delta \in \mathcal{F}_{t,\tau,\varpi}$, $t \leq \tilde t$, by considering $\tilde\delta :=\delta(\tilde t/t) \in \mathcal{F}_{\tilde t,\tau,\varpi}$ we have
$$\begin{array}{rl}
\left| \Gn^o\left( \frac{g_{\tau\varpi}(\tilde x'\delta)}{t} - \frac{g_{\tau\varpi}(\tilde x'\delta(\tilde t/t))}{\tilde t} \right) \right| & \leq \left| \Gn^o\left( \frac{g_{\tau\varpi}(\tilde x'\delta)}{t} - \frac{g_{\tau\varpi}(\tilde x'\delta(\tilde t/t))}{t} \right) \right| + \left| \Gn^o\left( \frac{g_{\tau\varpi}(\tilde x'\delta(\tilde t/t))}{t} - \frac{g_{\tau\varpi}(\tilde x'\delta(\tilde t/t))}{\tilde t} \right) \right|\\
& = \frac{1}{t}\left| \Gn^o\left( g_{\tau\varpi}(\tilde x'\delta) - g_{\tau\varpi}(\tilde x'\delta[\tilde t/t]) \right) \right| + \left| \Gn^o\left( g_{\tau\varpi}(\tilde x'\delta(\tilde t/t)) \right) \right|\cdot \left| \frac{1}{t} - \frac{1}{\tilde t}\right| \\
& \leq \sqrt{n}\En\left(\frac{|K_\varpi(W)\tilde x'\delta|}{t}\right) \frac{|t-\tilde t|}{t} +    \sqrt{n}\En\left( |K_\varpi(W)\tilde x'\delta|\right)  \frac{\tilde t}{t}\left| \frac{1}{t} - \frac{1}{\tilde t}\right|\\
& = 2 \sqrt{n}\En\left(\frac{|K_\varpi(W)\tilde x'\delta|}{t}\right) \left| \frac{t-\tilde t}{ t} \right| \leq 2\sqrt{n} \left| \frac{t-\tilde t}{t} \right|.
\end{array}
$$
We let $\widehat{\mathcal{N}}$ be a $\varepsilon$-net $\{\underline{N}=:t_1,t_2,\ldots,t_K:= \bar N\}$ of $[\underline{N},\bar N]$ such that $|t_k-t_{k+1}|/t_k \leq 1/(2\sqrt{n})$. Note that we can achieve that with $|\widehat{\mathcal{N}}|\leq 1 + \floor{3\sqrt{n}\log(\bar N/ \underline{N})}$.

By Markov bound, we have
$$ \begin{array}{rl}
\mathrm{P}( \mathcal{A}^o \geq K ) & \leq \min_{\psi\geq 0} \exp(-\psi K)\Ep[ \exp(\psi\mathcal{A}^o)]\\
& \leq 8p|\widehat{\mathcal{T}}|\cdot|\widehat{\mathcal{W}}|\cdot|\widehat{\mathcal{N}}|\min_{\psi\geq 0} \exp(-\psi K)\exp\left( 8\psi^2\right)\\
&\leq 8p|\widehat{\mathcal{T}}|\cdot|\widehat{\mathcal{W}}|\cdot|\widehat{\mathcal{N}}|\exp (-K^2/32 )
\end{array}
$$ here we set $\psi = K/16  $ and bound $\Ep[ \exp(\psi\mathcal{A}^o)]$ as follows
{\small $$\begin{array}{rl}
\displaystyle \Ep\left[ \exp\left( \psi \mathcal{A}^o \right)\right]&
\displaystyle  \leq_{(1)} 2|\widehat{\mathcal{T}}|\cdot|\widehat{\mathcal{W}}|\cdot|\widehat{\mathcal{N}}|\sup_{(\tau,\varpi,t)\in \widehat{\mathcal{T}}\times\widehat{\mathcal{W}}\times\widehat{\mathcal{N}}} \Ep\left[ \exp\left( \psi \sup_{\|\delta\|_{1,\varpi}=t} \Gn^o\left( \frac{g_{\tau\varpi}(\tilde x'\delta)}{t}\right)  \right)\right] \\
& \displaystyle \leq_{(2)} 2|\widehat{\mathcal{T}}|\cdot|\widehat{\mathcal{W}}|\cdot|\widehat{\mathcal{N}}|\sup_{(\tau,\varpi,t)\in \widehat{\mathcal{T}}\times\widehat{\mathcal{W}}\times\widehat{\mathcal{N}}}  \Ep\left[ \exp\left( 2\psi \sup_{\|\delta\|_{1,\varpi}=t} \Gn^o\left(\frac{K_\varpi(W)\tilde x'\delta}{t}\right)  \right)\right] \\
& \displaystyle \leq_{(3)} 2|\widehat{\mathcal{T}}|\cdot|\widehat{\mathcal{W}}|\cdot|\widehat{\mathcal{N}}|\sup_{(\tau,\varpi,t)\in \widehat{\mathcal{T}}\times\widehat{\mathcal{W}}\times\widehat{\mathcal{N}}} \Ep\left[ \exp\left( 2\psi \left[\sup_{\|\delta\|_{1,\varpi}=t}\frac{\|\delta\|_{1,\varpi}}{t}\max_{j\leq p}\frac{|\Gn^o(K_\varpi(W)\tilde x_{j})|}{\{\En[K_\varpi(W)\tilde x_j^2]\}^{1/2}} \right]\right)\right] \\
& \displaystyle =_{(4)} 2|\widehat{\mathcal{T}}|\cdot|\widehat{\mathcal{W}}|\cdot|\widehat{\mathcal{N}}|\sup_{(\tau,\varpi,t)\in \widehat{\mathcal{T}}\times\widehat{\mathcal{W}}\times\widehat{\mathcal{N}}} \Ep\left[ \exp\left( 2\psi \left[\max_{j\leq p}\frac{|\Gn^o(K_\varpi(W)\tilde x_{j})|}{\{\En[K_\varpi(W)\tilde x_j^2]\}^{1/2}} \right]\right)\right] \\
& \displaystyle \leq_{(5)} 4p|\widehat{\mathcal{T}}|\cdot|\widehat{\mathcal{W}}|\cdot|\widehat{\mathcal{N}}|\max_{j\leq p}\sup_{\varpi\in\widehat{\mathcal{W}}}\Ep\left[ \exp\left( 4\psi \frac{\Gn^o(K_\varpi(W)\tilde x_{j})}{\{\En[K_\varpi(W)\tilde x_j^2]\}^{1/2}}  \right)\right] \\
& \displaystyle \leq_{(6)} 8p|\widehat{\mathcal{T}}|\cdot|\widehat{\mathcal{W}}|\cdot|\widehat{\mathcal{N}}|\exp\left( 8\psi^2  \right) \\
\end{array} $$}
\noindent here (1) follows by $\exp(\max_{i\in I} |z_i|) \leq 2|I|\max_{i\in I} \exp(z_i)$, (2) by contraction principle (apply Theorem 4.12 \cite{LedouxTalagrandBook} with $t_i=K_\varpi(W_i)\tilde x_i'\delta$, and $\phi_i(t_i)=\rho_\tau(K_\varpi(W_i)\tilde y_i-K_\varpi(W_i)\tilde x_i'\eta_\tau + t_i)-\rho_\tau(K_\varpi(W_i)\tilde y_i-K_\varpi(W_i)\tilde x_i'\eta_\tau)$ so that $|\phi_i(s)-\phi_i(t)|\leq |s-t|$ and $\phi_i(0)=0$, 
 (3) follows by  $$|\Gn^o(K_\varpi(W)\tilde x'\delta)| \leq \|\delta\|_{1,\varpi}\max_{j\leq p}|\Gn^o(K_\varpi(W)\tilde x_j)/\{\En[K_\varpi(W)\tilde x_j^2]\}^{1/2}|,$$ (4) by the definition of suprema, (5) we again use $\exp(\max_{i\in I} |z_i|) \leq 2|I|\max_{i\in I} \exp(z_i)$, and   (6) $\exp(z)+\exp(-z)\leq 2\exp(z^2/2)$.

\end{proof}

\begin{lemma}[Estimation Error of Refitted Quantile Regression]\label{Thm:RefittedGeneric}  Consider an arbitrary vector $\hat\eta_u$ and suppose $\|\eta_u\|_0\leq s$. Let $\|r_{iu} \leq \bar r_u\|_{n,\varpi}$, $|\supp(\hat\eta_u)| \leq \widehat s_u$ and $\En[K_\varpi(W)\{\rho_\tau(\tilde y_i - \tilde x_i'\hat\eta_u)-\rho_\tau(\tilde y_i-\tilde x_i'\eta_u)\}] \leq \hat Q_u$ for all $u\in\mathcal{U}$ hold. Furthermore, suppose that
{\small $$\sup_{\footnotesize {\tiny u=(\tau,\varpi)\in\mathcal{U}}} \left| \En\left( K_\varpi(W)\frac{\rho_\tau(\tilde y-\tilde x'\widetilde\eta_u)-\rho_\tau(\tilde y-\tilde x'\eta_u)}{\|\widetilde\eta_u-\eta_u\|_{1,\varpi}} - \Ep\left[ K_\varpi(W)\frac{\rho_\tau(\tilde y-\tilde x'\widetilde\eta_u)-\rho_\tau(\tilde y-\tilde x'\eta_u)}{\|\widetilde\eta_u-\eta_u\|_{1,\varpi}} \mid W, \tilde x\right]\right) \right| \leq \frac{t_3}{\sqrt{n}}. $$}
Under these events, we have for $n$ large enough,
$$\|\sqrt{f_u}\tilde x_i'(\widetilde \eta_u - \eta_u)\|_{n,\varpi} \lesssim \widetilde N_u:= \sqrt{ \frac{ (\hat s_u + s)}{\semin{u,\hat s_u+s}}}(K_{n1}+t_3/\sqrt{n}) +K_{n2} + \bar f \bar r_u + \hat Q^{1/2}_u$$
where $\semin{u,k}=\inf_{\|\delta\|_0=k} \|\sqrt{f_u}\tilde x'\delta\|_{n,\varpi}^2/\|\delta\|^2$,  provided that
\begin{equation}\label{extracondLemma} \sup_{u\in \mathcal{U}, \|\bar\delta\|_0\leq \hat s_u+s} \frac{\bar f'\En[K_\varpi(W)(|r_u|+|r_u|^2)|\tilde x'\bar\delta|^2]}{\En[K_\varpi(W)f_u|\tilde x'\bar \delta|^2]} +\widetilde N_u/\bar q_{A_u} \to 0.\end{equation}
where $A_u = \{ \delta \in \RR^p : \|\delta\|_0 \leq \hat s_u + s\}$.\end{lemma}
\begin{proof}[Proof of Lemma \ref{Thm:RefittedGeneric}]
Let $ \ \hat \delta_u = \hat \eta_u - \eta_u$ which satisfies $\|\hat \delta_u\|_0\leq \hat s_u + s$. By optimality of $\widetilde
\eta_u$ in the refitted quantile regression we have
\begin{equation}\label{2step:Rel1a} \begin{array}{rl}
  \En[K_\varpi(W)\rho_\tau( \tilde y_i-\tilde x_i'\widetilde \eta_u )] - \En[K_\varpi(W)\rho_\tau(\tilde y_i-\tilde x_i'\eta_u )] &\\
   \leq  \En[K_\varpi(W)\rho_\tau(\tilde y_i-\tilde x_i'\hat \eta_u )] - \En[K_\varpi(W)\rho_\tau(\tilde y_i-\tilde x_i'\eta_u )] \leq  \hat Q_u\end{array}\end{equation}
where the second inequality holds by assumption.

Moreover, by assumption, uniformly over $u\in \mathcal{U}$, we have conditional on $(W_i,\tilde x_i, r_{iu})_{i=1}^n$ that
\begin{equation}\label{2stepRel3a}
\left| \Gn\left( K_\varpi(W)\frac{\rho_\tau(\tilde y-\tilde x'(\eta_u+\widetilde\delta_u))-\rho_\tau(\tilde y-\tilde x'\eta_u)}{\|\widetilde\delta_u\|_{1,\varpi}}\right) \right| \leq  t_3.\end{equation}
Thus combining relations
(\ref{2step:Rel1a}) and (\ref{2stepRel3a}),  we have  $$
\En[\Ep[K_\varpi(W)\{\rho_u(\tilde y-\tilde x'(\eta_u+\widetilde \delta_u) )-\rho_u(\tilde y-\tilde x'\eta_u )\}\vert \tilde x, \tilde r, \varpi]] \leq
 \|\widetilde \delta_u\|_{1,\varpi}t_3/\sqrt{n} +  \hat Q_u .
$$
Invoking the sparse identifiability relation  of Lemma \ref{Lemma:IdentificationSparseQRNP},  since the required condition on the approximation errors $r_u$'s holds by assumption (\ref{extracondLemma}), for $n$ large enough  $$   \displaystyle
\frac{\|\sqrt{f_u}\tilde x'\widetilde\delta_u\|_{n,\varpi}^2}{4} \wedge \left\{ \bar q_{A_u}\|\sqrt{f_u}\tilde x'\widetilde\delta_u\|_{n,\varpi}\right\}  \leq  \displaystyle K_{n2}\|\sqrt{f_u}\tilde x'\widetilde\delta_u\|_{n,\varpi} + \|\widetilde \delta_u\|_{1,\varpi} (K_{n1}+t_3/\sqrt{n}) +  \hat Q_u,
 $$ where $\bar q_{A_u}$ is defined with $A_u:=\{ \delta : \|\delta\|_0 \leq \hat s_u+s\}$. Moreover, by the sparsity of $\tilde\delta_u$ we have $\|\widetilde \delta_u\|_{1,\varpi}\leq \sqrt{(\hat s_u+s)/\semin{u,\hat s_u+s}}\|\sqrt{f_u}\tilde x'\widetilde\delta_u\|_{n,\varpi}$ so that we have for $t=\|\sqrt{f_u}\tilde x'\widetilde\delta_u\|_{n,\varpi}$,
$$ \begin{array}{rcl} \displaystyle
\frac{t^2}{4} \wedge \left\{ \bar q_{A_u}t\right\} & \leq & t(\displaystyle K_{n2} + \sqrt{(\hat s_u+s)/\semin{u,\hat s_u+s}} \{K_{n1}+t_3/\sqrt{n}\}) +  \hat Q_u.
 \end{array}$$
Note that for positive numbers $(t^2/4) \wedge (\bar q_{A_u} t) \leq A + B t$ implies $t^2/4 \leq A+Bt$ provided $\bar q_{A_u}/2 > B$ and $2\bar q_{A_u}^2>A$. (Indeed, otherwise $(t^2/4) \geq q t$ so that $t\geq 4q$, which in turn implies that $2\bar q_{A_u} ^2 + \bar q_{A_u} t/2 \leq (t^2/4) \wedge \bar q_{A_u}  t \leq A + Bt$.) Note that $\bar q_{A_u}/2 > B$ and $2\bar q_{A_u}^2>A$ is implied by condition (\ref{extracondLemma}) when we set $A = \hat Q_u$ and $B = (\displaystyle K_{n2} + \sqrt{(\hat s_u+s)/\semin{u,\hat s_u+s_u}} \{K_{1n}+t_3/\sqrt{n}\})$.
Thus the minimum is achieved in the quadratic part. Therefore, for $n$ sufficiently large, we have
$$\|\sqrt{f_u}\tilde x'\widetilde\delta_u\|_{n,\varpi} \leq \hat Q_u^{1/2} + K_{n2} + (K_{n1}+t_3/\sqrt{n})\sqrt{(\hat s_u+s)/\semin{u,\hat s_u+s_u}}.$$
\end{proof}
Under the condition $\max_{i\leq n} \|\tilde x_i\|_\infty^2\log(n\vee p) = o(n \min_{\tau\in\mathcal{T}}\tau(1-\tau))$, the next result provides new bounds for the data driven penalty choice parameter when the quantile indices in $\mathcal{T}$ can approach the extremes.

\begin{lemma}[Pivotal Penalty Parameter Bound]\label{lem:Lambda}
Let $\underline{\tau}=\min_{\tau\in\mathcal{T}}\tau(1-\tau)$ and $K_n=\max_{i\leq n, j\in[p]}|\tilde x_{ij}/\hat\sigma_j|$, $\hat\sigma_j = \En[\tilde x_j^2]^{1/2}$. Under $K_n^2\log(p/\underline{\tau})=o(n\underline{\tau})$, for $n$ large enough we have that for some constant $\bar C$
$$\Lambda(1-\xi\vert \tilde x_1,\ldots,\tilde x_n) \leq \bar{C}\sqrt{\frac{\log(16p/(\underline{\tau}\xi))}{n}}$$ where $\Lambda(1-\xi\vert \tilde x_1,\ldots,\tilde x_n)$ is the $1-\xi$ quantile of $ \max_{j\in [p]}\sup_{\tau \in \mathcal{T}} \left|\frac{\sum_{i=1}^n\tilde  x_{ij}(\tau - 1\{U_i \leq \tau\})}{\hat\sigma_j\sqrt{\tau(1-\tau)}}\right|$ conditional on $\tilde x_1,\ldots,\tilde x_n$, and $U_i$ are independent uniform$(0,1)$ random variables.
\end{lemma}
\begin{proof}
Conditional on $\tilde x_1,\ldots, \tilde x_n$, letting $\hat \sigma_j^2 = \En[x_{j}^2]$, we have that $$n\Lambda = \max_{j\in [p]} \sup_{\tau \in \mathcal{T}} \left|\frac{\sum_{i=1}^n\tilde  x_{j}(\tau - 1\{U \leq \tau\})}{\hat\sigma_j\sqrt{\tau(1-\tau)}}\right|.$$
Step 1. (Entropy Calculation) Let $\mathcal{F} = \{ \tilde x_{ij}(\tau - 1\{U_i \leq \tau\})/\hat\sigma_j : \tau \in \mathcal{T}, j\in[p] \}$, $h_\tau = \sqrt{\tau(1-\tau)}$, and $\mathcal{G}=\{ f_\tau/ h_\tau : \tau \in \mathcal{T}\}$.
We have that
$$\begin{array}{rl}
d(f_\tau/h_\tau, f_{\bar \tau}/h_{\bar\tau})& \leq d(f_\tau, f_{\bar \tau})/h_\tau + d(f_{\bar \tau}/h_\tau, f_{\bar \tau}/h_{\bar\tau})\\
& \leq d(f_\tau, f_{\bar \tau})/h_\tau + d(0,f_{\bar \tau}/h_{\bar\tau})|h_\tau - h_{\bar \tau}|/h_\tau\end{array}$$
Therefore, since $\|F\|_Q\leq \|G\|_Q$ by $h_\tau \leq 1$, and $d(0,f_{\bar \tau}/h_{\bar\tau})\leq 1/h_{\bar\tau}$ we have
$$N(\epsilon\|G\|_Q,\mathcal{G},Q)\leq N(\epsilon\|F\|_Q/\{2 \min_{\tau\in \mathcal{T}} h_\tau\},\mathcal{F},Q)N(\epsilon/\{2\min_{\tau\in \mathcal{T}} h_\tau^2\}, \mathcal{T},|\cdot|).$$
Thus we have for some constants $K$ and $v$ that $$N(\epsilon\|G\|_Q,\mathcal{G},Q) \leq p( K / \{\epsilon \min_{\tau\in \mathcal{T}} h_\tau^2\} )^v.$$

Step 2.(Symmetrization) Since we have $\Ep[g^2] = 1$ for all $g\in\mathcal{G}$, by Lemma 2.3.7 in \cite{vdV-W} we have
$$
\Pr( \Lambda \geq t\sqrt{n} ) \leq 4 \Pr( \max_{j\leq p} \sup_{\tau\in\mathcal{T}}\left|\Gn^o(g)\right| \geq t/4 )
$$ here $\Gn^o:\mathcal{G}\to \mathbb{R}$ is the symmetrized process generated by Rademacher variables. Conditional on $(x_1,u_1),\ldots,(x_n,u_n)$, we have that $\{\Gn^o(g):g\in\mathcal{G}\}$ is sub-Gaussian with respect to the $L_2(\mathbb{P}_n)$-norm by the Hoeffding inequality. Thus, by Lemma 16 in \cite{BC-SparseQR}, for $\delta_n^2 = \sup_{g\in\mathcal{G}}\En[g^2]$ and $\bar\delta_n = \delta_n/\|G\|_{\mathbb{P}_n}$, we have
$$ \Pr( \sup_{g\in\mathcal{G}}|\Gn^o(g)|> C K\delta_n\sqrt{\log(pK/\underline{\tau})}\mid \{\tilde x_i,U_i\}_{i=1}^n) \leq \int_0^{\bar\delta_n/2}\epsilon^{-1} \{ p( K / \{\epsilon \min_{\tau\in \mathcal{T}} h_\tau^2\} )^v \}^{-C^2+1}d\epsilon$$
for some universal constant $K$.

In order to control $\delta_n$, note that
$ \delta_n^2 = \sup_{g\in\mathcal{G}}\frac{1}{\sqrt{n}}\Gn(g^2)+\mathrm{E}[g^2].$ In turn, since $\sup_{g\in\mathcal{G}}\En[g^4]\leq \delta_n^2\max_{i\leq n}G_i^2$, we have
$$ \Pr( \sup_{g\in\mathcal{G}}|\Gn^o(g^2)|> C\bar K \delta_n\max_{i\leq n}G_i\sqrt{\log(pK/\underline{\tau})}\mid \{\tilde x_i,U_i\}_{i=1}^n) \leq \int_0^{\bar\delta_n/2}\epsilon^{-1} \{p( K / \{\epsilon \underline{\tau}\} )^v \}^{-C^2+1}d\epsilon.$$
Thus with probability $1-\int_0^{1/2}\epsilon^{-1} \{ p( K/\epsilon \underline{\tau} )^v \}^{-C^2+1}d\epsilon$, since $\Ep[g^2]=1$ and $\max_{i\leq n}G_i\leq K_n/\sqrt{\underline{\tau}}$, we have $$\delta_n \leq 1 + \frac{C'K_n\sqrt{\log(pK/\underline{\tau})}}{\sqrt{n}\sqrt{\underline{\tau}}}.$$

Therefore, under $K_n\sqrt{\log(pK/\underline{\tau})}= o(\sqrt{n}\sqrt{\underline{\tau}})$, conditionally on $\{\tilde x_i\}_{i=1}^n$ and $n$ sufficiently large, with probability $1-2\int_0^{1/2}\epsilon^{-1} \{ p( K / \{\epsilon \underline{\tau}\} )^v \}^{-C^2+1}d\epsilon$ we have  that
$$\sup_{g\in\mathcal{G}}|\Gn^o(g)| \leq 2C K\sqrt{\log(pK/\underline{\tau})}$$

The stated bound follows since for $C>2$ $$2\int_0^{1/2}\epsilon^{-1} \{ p( K / \{\epsilon \underline{\tau}\} )^v \}^{-C^2+1}d\epsilon \leq \{p/\underline{\tau}\}^{-C^2+1}2\int_0^{1/2}\epsilon^{-2+C^2}d\epsilon\leq \{p/\underline{\tau}\}^{-C^2+1}.$$

\end{proof}

\section{Inequalities}

\begin{lemma}[Transfer principle, \cite{oliveira2013lower}]\label{lem:transfer}
Let $\hat\Sigma$ and $\Sigma$ be $p\times p$ matrices with non-negative diagonal entries, and assume that for some $\eta \in (0,1)$ and $s \leq p$  we have
 $$ \forall v\in \RR^p, \|v\|_0 \leq s, v'\hat\Sigma v \geq (1-\eta)v'\Sigma v $$
Let $D$ be a diagonal matrix such that $D_{kk} \geq \hat \Sigma_{kk} - (1-\eta)\Sigma_{kk}$. Then for all $\delta \in \RR^p$ we have
$$ \delta'\hat\Sigma \delta \geq (1-\eta)\delta'\Sigma \delta - \|D^{1/2}\delta\|_1^2/(s-1).$$\end{lemma}

\begin{lemma}\label{Lemma:Bound2nNormSecond}
Consider $\hat\beta_u$ and $\beta_u$ with $\|\beta_u\|_0\leq s$. Denote by $\hat \beta^{\lambda}_{u}$ the vector with $\hat\beta^\lambda_{uj}=\hat \beta_{uj}1\{\hat\sigma^Z_{a\varpi j}|\hat \beta_{uj}|\geq \lambda\}$ where $\hat\sigma^Z_{a\varpi j}=\{\En[K_\varpi(W)(Z_j^a)^2]\}^{1/2}$. We have that
$$\begin{array}{rl}
\|\hat \beta^\lambda_u - \beta_u\|_{1,\varpi} & \leq \|\hat \beta_u - \beta_u \|_{1,\varpi}+\lambda s \\
|\supp(\hat\beta_u^\lambda)| & \leq s + \|\hat\beta_u-\beta_u\|_{1,\varpi}/\lambda\\
\|Z^a(\hat \beta_u^\lambda-\beta_u)\|_{n,\varpi} & \leq  \|Z^a(\hat \beta_u-\beta_u)\|_{n,\varpi}  + \sqrt{\semaxtilde{s,\varpi}} \{ 2\sqrt{s}\lambda + \|\hat\beta_u-\beta_u\|_{1,\varpi}/\sqrt{s}\}\end{array}$$
here $\semaxtilde{m,\varpi}= \sup_{1\leq \|\theta\|_0\leq m}\|\tilde Z^a\theta\|_{n,\varpi}/\|\theta\|$ and $\tilde Z^a_{ij} = Z^a_{ij}/\{\En[K_\varpi(W)(Z^a_j)^2]\}^{1/2}$. \end{lemma}
\begin{proof}
Let $T_u = \supp(\beta_u)$. The first relation follows from the triangle inequality
$$\begin{array}{rl}
 \|\hat \beta^\lambda_u - \beta_u\|_{1,\varpi} &= \|(\hat \beta^\lambda_u - \beta_u)_{T_u}\|_{1,\varpi} + \|(\hat \beta^\lambda_u)_{T_u^c}\|_{1,\varpi} \\
&\leq  \|(\hat \beta^\lambda_u - \hat\beta_u)_{T_u}\|_{1,\varpi} + \|(\hat \beta_u - \beta_u)_{T_u}\|_{1,\varpi}+ \|(\hat \beta^\lambda_u)_{T_u^c}\|_{1,\varpi} \\
& \leq \lambda s + \|(\hat \beta_u - \beta_u)_{T_u}\|_{1,\varpi}+ \|(\hat \beta_u)_{T_u^c}\|_{1,\varpi}\\
& = \lambda s + \|\hat \beta_u - \beta_u\|_{1,\varpi} \end{array}$$

To show the second result note that  $\|\hat\beta_u-\beta_u\|_{1,\varpi} \geq \{|\supp(\hat\beta_u^\lambda)|-s\}\lambda$. Therefore, $$\begin{array}{rl}
|\supp(\hat\beta_u^\lambda)| \leq s + \|\hat\beta_u-\beta_u\|_{1,\varpi}/\lambda
\end{array}$$
which yields the result.

To show the third bound, we start using the triangle inequality
$$ \|Z^a(\hat\beta^\lambda_u-\beta_u)\|_{n,\varpi} \leq \|Z^a(\hat\beta^\lambda_u-\hat\beta_u)\|_{n,\varpi} + \|Z^a(\hat\beta_u-\beta_u)\|_{n,\varpi}.$$
Without loss of generality, assume that the components are ordered so that $|(\hat\beta^\lambda_u-\hat\beta_u)_j|\hat\sigma_{uj}$ is decreasing. Let $T_1$ be the set of $s$ indices corresponding to the largest values of $|(\hat\beta^\lambda_u-\hat\beta_u)_j|\hat\sigma_{uj}$. Similarly define $T_k$ as the set of $s$ indices corresponding to the largest values of $|(\hat\beta^\lambda_u-\hat\beta_u)_j|\hat\sigma_{uj}$ outside $\cup_{m=1}^{k-1}T_m$. Therefore, $\hat\beta^\lambda_u-\hat\beta_u=\sum_{k=1}^{\ceil{p/s}} (\hat\beta^\lambda_u-\hat\beta_u)_{T_k}$. Moreover,
given the monotonicity of the components, $\|(\hat\beta^\lambda_u-\hat\beta_u)_{T_k}\|_{2,\varpi} \leq \|(\hat\beta^\lambda_u-\hat\beta_u)_{T_{k-1}}\|_1/\sqrt{s}$. Then, we have
$$ \begin{array}{rl}
 \|Z^a(\hat\beta^\lambda_u-\hat\beta_u)\|_{n,\varpi} & = \|Z^a\sum_{k=1}^{\ceil{p/s}} (\hat\beta^\lambda_u-\hat\beta_u)_{T_k}\|_{n,\varpi}\\
& \leq \|Z^a(\hat\beta^\lambda_u-\hat\beta_u)_{T_1}\|_{n,\varpi} + \sum_{k\geq 2}\|Z^a(\hat\beta^\lambda_u-\hat\beta_u)_{T_k}\|_{n,\varpi}\\
 & \leq \sqrt{\semaxtilde{s,\varpi}}\|(\hat\beta^\lambda_u-\hat\beta_u)_{T_1}\|_{2,\varpi}+\sqrt{\semaxtilde{s,\varpi}}\sum_{k\geq 2}\|(\hat\beta^\lambda_u-\hat\beta_u)_{T_k}\|_{2,\varpi}\\
&  \leq  \sqrt{\semaxtilde{s,\varpi}}\lambda\sqrt{s} + \sqrt{\semaxtilde{s,\varpi}}\sum_{k\geq 1}\|(\hat\beta^\lambda_u-\hat\beta_u)_{T_k}\|_{1,\varpi}/\sqrt{s}\\
& =  \sqrt{\semaxtilde{s,\varpi}}\lambda\sqrt{s} + \sqrt{\semaxtilde{s,\varpi}} \|\hat\beta^\lambda_u-\hat\beta_u\|_{1,\varpi}/\sqrt{s} \\
&  \leq  \sqrt{\semaxtilde{s,\varpi}} \{ 2\lambda\sqrt{s} + \|\hat\beta_u-\beta_u\|_{1,\varpi}/\sqrt{s}\} \\ \end{array}$$

here the last inequality follows from the first result and the triangle inequality.

\end{proof}

\begin{lemma}[Supremum of Sparse Vectors on Symmetrized Random Matrices]\label{lemma:RV34}
Let $\hat \UU$ denote a finite set and $(X_{iu})_{u\in \hat \UU}$, $i=1,\ldots, n$, be fixed vectors such that $X_{iu} \in \RR^p$ and $\max_{1\leq i\leq n}\max_{u\in \hat\UU}\|X_{iu}\|_\infty \leq K$. Furthermore define $$\delta_n:= \bar C K \sqrt{k}\left(\sqrt{\log |\hat\UU|} + \sqrt{1+\log p} + \log k \sqrt{\log (p\vee n)} \sqrt{\log n} \right)/\sqrt{n},$$ where $\bar C$ is a universal constant. Then,
$$ \Ep \left[ \sup_{\|\theta\|_0\leq k, \|\theta\| =1}\max_{u\in\hat\UU} \left| \En [ \varepsilon(\theta'X_{u})^2] \right| \right] \leq \delta_n \sup_{\|\theta\|_0\leq k, \|\theta\| =1, u\in\hat \UU} \sqrt{\En[(\theta'X_{u})^2]}. $$
\end{lemma}

\begin{proof}
See \cite{BCCW-ManyProcesses} for the proof.
\end{proof}

\begin{corollary}[Supremum of Sparse Vectors on Many Random Matrices]\label{thm:RV34}
Let $\hat \UU$ denote a finite set and $(X_{iu})_{u\in \hat \UU}$, $i=1,\ldots, n$, be independent (across i) random vectors such that $X_{iu} \in \RR^p$ and $$\sqrt{\Ep[ \max_{1\leq i\leq n}\max_{u\in \hat\UU}\|X_{iu}\|_\infty^2]} \leq K.$$ Furthermore define $$\delta_n:= \bar C K \sqrt{k}\left(\sqrt{\log |\hat\UU|} + \sqrt{1+\log p} + \log k \sqrt{\log (p\vee n)} \sqrt{\log n} \right)/\sqrt{n},$$ here $\bar C$ is a universal constant. Then,
$$ \Ep\left[ \sup_{\|\theta\|_0\leq k, \|\theta\| =1}\max_{u\in\hat\UU} \left| \En\[ (\theta'X_{u})^2 - \Ep[(\theta'X_{u})^2] \]\right|\right] \leq \delta_n^2 + \delta_n \sup_{\|\theta\|_0\leq k, \|\theta\| =1, u\in\hat \UU} \sqrt{\En[\Ep[(\theta'X_{u})^2]]}. $$
\end{corollary}

We will also use the following result of \cite{chernozhukov2012gaussian}.
\begin{lemma}[Maximal Inequality]
\label{lemma:CCK}  Work with the setup above.  Suppose that $F\geq \sup_{f \in \mathcal{F}}|f|$ is a measurable envelope for $\mathcal{F}$
with $\| F\|_{P,q} < \infty$ for some $q \geq 2$.  Let $M = \max_{i\leq n} F(W_i)$ and $\sigma^{2} > 0$ be any positive constant such that $\sup_{f \in \mathcal{F}}  \| f \|_{P,2}^{2} \leq \sigma^{2} \leq \| F \|_{P,2}^{2}$. Suppose that there exist constants $a \geq e$ and $v \geq 1$ such that
\begin{equation*}
\log \sup_{Q} N(\epsilon \| F \|_{Q,2}, \mathcal{F},  \| \cdot \|_{Q,2}) \leq  v \log (a/\epsilon), \ 0 <  \epsilon \leq 1.
\end{equation*}
Then
\begin{equation*}
\Ep_P [ \sup_{f\in \mathcal{F}} | \Gn(f)| ] \leq K  \left( \sqrt{v\sigma^{2} \log \left ( \frac{a \| F \|_{P,2}}{\sigma} \right ) } + \frac{v\| M \|_{P, 2}}{\sqrt{n}} \log \left ( \frac{a \| F \|_{P,2}}{\sigma} \right ) \right),
\end{equation*}
here $K$ is an absolute constant.  Moreover, for every $t \geq 1$, with probability $> 1-t^{-q/2}$,
\begin{multline*}
\sup_{f\in \mathcal{F}} | \Gn(f)| \leq (1+\alpha) \Ep_P [ \sup_{f\in \mathcal{F}} | \Gn(f)| ] + K(q) \Big [ (\sigma + n^{-1/2} \| M \|_{P,q}) \sqrt{t}
+  \alpha^{-1}  n^{-1/2} \| M \|_{P,2}t \Big ], \
\end{multline*}
$\forall \alpha > 0$ where $K(q) > 0$ is a constant depends only on $q$.  In particular, setting $a \geq n$ and $t = \log n$,
with probability $> 1- c(\log n)^{-1}$,
\begin{equation} \label{simple bound}
\sup_{f\in \mathcal{F}} | \Gn(f)| \leq K(q,c) \left ( \sigma \sqrt{v \log \left ( \frac{a \| F \|_{P,2}}{\sigma} \right ) } + \frac{v
 \| M \|_{P,q} } {\sqrt{n}}\log \left ( \frac{a \| F \|_{P,2}}{\sigma} \right ) \right),
\end{equation}
here $  \| M \|_{P,q}  \leq n^{1/q} \| F\|_{P,q}$ and  $K(q,c) > 0$ is a constant depending only on $q$ and $c$.

\end{lemma}

 \section{Confidence Regions for Function-Valued Parameters Based on Moment Conditions}\label{sec: general}

For completeness, in this section we collect an adaptation of the results of \cite{BCCW-ManyProcesses} that are invoked in our proofs. The main difference is the weakening of the identification condition (which is allowed to decrease to zero, see the parameter $j_n$ in Condition \ref{ass: S1} below). We are interested in function-valued target parameters indexed  by $u \in \mathcal{U} \subset \mathbb{R}^{d_u}$.  The true value of the target parameter is denoted by   $$\theta^0 = (\theta_{u j})_{u \in \UU, j\in[\pp]}, \ \ \text{where}  \ \ \theta_{u j} \in \Theta_{u j} \text{ for each }
 u \in \mathcal{U} \ \ \mbox{and} \ \ j\in[\pp].$$
For each $u \in \mathcal{U}$ and $j\in [\pp]$, the parameter $\theta_{uj}$ is characterized as the solution to the following moment condition:
 \begin{equation}\label{eq:ivequation}
 \Ep[ \psi_{u j}(W_{uj}, \theta_{uj}, \eta_{uj} )] = 0,
 \end{equation}
where   $W_{uj}$ is a random vector that takes values in a Borel set $\mathcal{W}_{uj} \subset \mathbb{R}^{d_w}$, $\eta^0=(\eta_{uj})_{u\in\UU,j\in[\pp]}$ is a nuisance parameter where $\eta_{uj} \in T_{uj}$ a convex set, and the moment function
\begin{equation}\label{def:psi_u}
\psi_{u j}:  \mathcal{W}_{uj} \times \Theta_{uj} \times T_{ujn} \mapsto \mathbb{R},  \ \  (w, \theta, t) \mapsto \psi_{uj}(w, \theta, t)\end{equation}
 is a Borel measurable map.

We assume that the (continuum) nuisance parameter $\eta^0$ can be
modelled and estimated by $\hat\eta = (\hat\eta_{uj})_{u\in\UU, j\in[\pp]}$. We will discuss examples where the corresponding $\eta^0$ can be estimated using modern regularization and post-selection methods such as Lasso and Post-Lasso (although other procedures can be applied). The estimator $\check \theta_{u j}$ of $\theta_{u j}$ is constructed as any approximate $\epsilon_n$-solution in $\Theta_{uj}$ to a sample analog of the moment condition (\ref{eq:ivequation}), i.e.,
\begin{equation}\label{eq:analog}
\max_{j\in[\pp]} \sup_{u \in \UU} \left\{ |\En[ \psi_{uj}(W_{uj}, \check \theta_{uj}, \hat \eta_{uj} ) ] | - \inf_{\theta_j \in \Theta_{uj}}|\En[ \psi_{u j}(W_{uj}, \theta, \hat \eta_{uj} ) ] | \right\} \leq \epsilon_n = o_P(n^{-1/2}\delta_n).
\end{equation}

As discussed before, we rely on an orthogonality condition for regular estimation of $\theta_{uj}$, which we will state next.
\begin{definition}[\textbf{Near Orthogonality Condition}]\label{Def:Orthog} For each $u \in \mathcal{U}$ and $j\in[\pp]$, we say that $\psi_{uj}$ obeys a general form of orthogonality with respect to $\mT_{uj}$ uniformly in $u \in \mathcal{U}$, if the following conditions hold: the G{\^a}teaux derivative map
 $$
  \mathrm{D}_{u,j,\bar r}[\tilde \eta_{uj} - \eta_{uj}]:=  \left. \partial_r  \Ep \Bigg (  \psi_{u j} \Big\{ W_{uj}, \theta_{u j}, \eta_{uj}+ r \Big [\tilde \eta_{uj} - \eta_{uj}\Big] \Big\}   \Bigg )\right|_{r=\bar r}
  $$
  exists for all $r \in [0,1)$, $\tilde \eta \in \mT_{uj}$, $j\in\pp$, and $u \in \UU$ and vanishes at $r=0$, namely,
  \begin{equation}\label{eq:cont}
 |\mathrm{D}_{u,j,0}[\tilde \eta_{uj} - \eta_{uj}]|\leq \delta_n n^{-1/2}  \ \  \text{ for all } \tilde \eta_{uj} \in \mT_{uj}.
\end{equation}
\end{definition}

In what follows, we shall denote by $c_0$, $c$, and $C$ some positive constants.

\begin{assumption}[Moment Condition]\label{ass: S1}
Consider a random element $W$,  taking values in a measure space $(\mathcal{W}, \mathcal{A}_\mathcal{W})$, with law determined by a probability measure $P \in \mP_n$.
The observed data $((W_{iu})_{u \in \mathcal{U}})_{i=1}^{n}$ consist of $n$ i.i.d. copies of a random element $(W_{u})_{u \in \mathcal{U}}$ which is  generated as a suitably measurable transformation with respect to $W$ and $u$.  Uniformly for all $n \geq n_0$ and $P \in \mathcal{P}_n$, the following conditions hold: (i) The true parameter value $\theta_{u j}$ obeys (\ref{eq:ivequation}) and is interior relative to $\Theta_{u j}$, namely there is a ball of radius $C n^{-1/2}\un \log n $ centered at $\theta_{u j}$ contained in $\Theta_{u j}$ for all $u \in \mathcal{U}$, $j\in[\pp]$ with $\un:= \Ep[\sup_{u \in \UU, j\in [\pp]}|\sqrt{n}\En[\psi_{uj}(W_{uj}, \theta_{uj}, \eta_{uj} )]|]$;   (ii) For each  $u \in \mathcal{U}$ and $j \in [\pp]$,  the map $ (\theta,\eta) \in \Theta_{uj} \times \mT_{uj}  \mapsto \Ep[\psi_{uj}(W_{uj}, \theta,\eta )]|$ is twice continuously differentiable; (iii) For all $u \in \mU$ and $j \in [\pp]$, the moment function $\psi_{uj}$ obeys the orthogonality condition given in Definition \ref{Def:Orthog} for the set $\mT_{uj} =\mT_{ujn}$ specified in Assumption \ref{ass: AS}; (iv) The following identifiability condition holds: $|\Ep[\psi_{u j}(W_{uj}, \theta, \eta_{uj})]| \geq \frac{1}{2}|J_{uj} (\theta- \theta_{u j})| \wedge c_0\ \text{  for all } \theta \in \Theta_{uj},$  with $J_{u j} :=  \left.\partial_\theta \Ep[ \psi_{u j} (W_{uj}, \theta, \eta_{uj})]\right|_{\theta=\theta_{uj}}$ satisfies $0<j_n<|J_{uj}| <C <\infty $ for all $u \in \mathcal{U}$ and $j\in [\pp]$; (v) The following smoothness conditions holds \begin{itemize}
\item[(a)] $\sup_{u \in \mathcal{U}, j \in [\pp],(\theta, \bar \theta) \in \Theta_{uj}^2, (\eta,\bar\eta) \in \mT_{ujn}^2} \ \ \frac{\Ep[  \{ \psi_{uj}(W_{uj}, \theta,\eta) - \psi_{uj}(W_{uj}, \bar\theta,\bar\eta)\}^2]}{ \{ |\theta-\bar\theta|\vee \| \eta - \bar \eta\|_e\}^{\alpha}}\leq C$,
 
     \item [(b)]
 $  \sup_{u \in \mathcal{U}, (\theta, \eta) \in \Theta_{uj}\times \mT_{ujn}, r\in[0,1)}  \ \ \ \ \ | \partial_{r} \Ep \left [ \psi_{uj}(W_{uj}, \theta,\eta_{uj}+r\{\eta-\eta_{uj}\}) \right ]|/\|\eta-\eta_{uj}\|_e  \leq \bar B_{1n}$,
 \item[(c)] $\sup_{u \in \mathcal{U}, j \in [\pp], (\theta,\eta) \in \Theta_{uj}\times \mT_{ujn}, r\in[0,1)} \frac{|\partial_{r}^2 \Ep[\psi_{uj}(W_{uj}, \theta_{uj}+r\{\theta-\theta_{uj}\},\eta_{uj}+r\{\eta - \eta_{uj}\})]|}{\{|\theta-\theta_{uj}|^2 \vee \|\eta-\eta_{uj}\|_{e}^2 \}} \leq  \bar B_{2n}.$

    \end{itemize}
\end{assumption}

Next we state assumptions on the nuisance functions. In what follows, let $\Delta_n \searrow 0$, $\delta_n \searrow 0$, and $\tau_n \searrow 0$ be sequences of constants approaching zero from above at a speed at most polynomial in $n$ (for example, $\delta_n \geq 1/n^c$ for some $c > 0$).  \\

\begin{assumption}[Estimation of Nuisance Functions]\label{ass: AS}
The following conditions hold for each $n \geq n_0$ and all $P \in  \mathcal{P}_n$.  The estimated functions $\hat \eta_{uj} \in \mT_{ujn}$ with probability at least $1- \Delta_n$,
$\mT_{ujn}$ is the set of measurable maps $\tilde \eta_{uj}$ such that
$$
\sup_{u\in\UU}\max_{j\in[\pp]}\| \tilde \eta_{uj} - \eta_{uj}\|_{e} \leq \tau_n,
$$
here the $e$-norm is the same as in Assumption \ref{ass: S1}, and whose complexity does not grow too quickly in the sense that
$\mathcal{F}_1 = \{  \psi_{uj}(W_{uj}, \theta, \eta): u \in \mathcal{U},  j  \in [\pp], \theta \in \Theta_{uj}, \eta \in \mT_{ujn} \cup \{\eta_{uj}\} \}$
is suitably measurable and its uniform covering entropy obeys:
 $$
\sup_Q  \log N(\epsilon \|F_1\|_{Q,2}, \mathcal{F}_1, \| \cdot \|_{Q,2}) \leq s_{n(\UU,\pp)} ( \log (a_n/\epsilon)) \vee 0,
$$
where $F_1(W)$ is an envelope for $\mathcal{F}_1$ which is measurable with respect to  $W$ and satisfies $F_1(W)\geq  \sup_{u\in \UU, j\in[\pp], \theta\in \Theta_{uj}, \eta\in \mT_{ujn}}|\psi_{uj}(W_{uj},\theta,\eta)|$ and $\|F_1\|_{P,q}\leq K_{n}$ for $q\geq 2$. The complexity characteristics $a_n \geq \max(n, K_n, \mathrm{e}) $ and $s_{n(\UU,\pp)} \geq 1$ obey the growth conditions:
$$
\begin{array}{rl}
n^{-1/2}  \sqrt{ s_{n(\UU,\pp)} \log (a_n) } + n^{-1} s_{n(\UU,\pp)} n^{\frac{1}{q}} K_{n} \log (a_n) \leq \tau_n \\
\{(1\vee \bar B_{1n})(\tau_n/j_n)\}^{\alpha/2} \sqrt{ s_{n(\UU,\pp)} \log (a_n)}  +  s_{n(\UU,\pp)}n^{\frac{1}{q}-\frac{1}{2}}K_n\log (a_n) \log n  \leq \delta_n,\\
\text{ and }  \ \ \sqrt{n} \bar B_{2n} (1\vee \bar B_{1n}) (\tau_n/j_n)^2 \leq \delta_n
\end{array}$$
here $\bar B_{1n}$, $\bar B_{2n}$, $j_n$, $q$ and $\alpha$ are defined in Assumption \ref{ass: S1}.
\end{assumption}

\begin{theorem}[{Uniform Bahadur representation for a Continuum of Target Parameters}] \label{theorem:semiparametricMain} Under  Assumptions \ref{ass: S1} and \ref{ass: AS}, for an estimator $(\check \theta_{uj})_{u \in \mathcal{U},j\in[\pp]}$ that obeys equation (\ref{eq:analog}),
$$ \sqrt{n}\sigma_{uj}^{-1}(\check\theta_{uj} - \theta_{uj}) =   \Gn    \bar \psi_{uj}   + O_P(\delta_n) \text{ in } \ell^\infty(\mathcal{U}\times[\pp]), \text{ uniformly in $P \in  \mathcal{P}_n$},$$
here  $\bar \psi_{uj}(W):= - \sigma_{uj}^{-1}J^{-1}_{uj}  \psi_{uj}(W_{uj}, \theta_{uj}, \eta_{uj})$ and  $\sigma_{uj}^2 = \Ep[J^{-2}_{uj}  \psi_{uj}^2(W_{uj}, \theta_{uj}, \eta_{uj})]$.
\end{theorem}

The uniform Bahadur representation derived in Theorem \ref{theorem:semiparametricMain} is useful in the construction of simultaneous confidence bands for $(\theta_{uj})_{u\in\UU, j\in[\pp]}$. This is achieved by new high-dimensional central limit theorems that have recently been developed in \cite{chernozhukov2013gaussian} and \cite{chernozhukov2012gaussian}. We will make use of the following regularity condition. In what follows $\bar{\delta}_n$ and $\Delta_n$ are fixed sequences going to zero, and we denote $\hat \psi_{uj}(W):= - \hat\sigma_{uj}^{-1}\hat J_{uj}^{-1} \psi_{uj}(W_{uj}, \check \theta_{uj}, \hat \eta_{uj})$ be the estimators of $\bar \psi_{uj}(W)$, with $\hat J_{uj}$ and $\hat\sigma_{uj}$ being suitable estimators of $J_{uj}$ and $\sigma_{uj}$.
In what follows,  $\|\cdot\|_{\Pn,2}$ denotes the empirical $L_2(\Pn)$-norm with $\Pn$ as the empirical measure of the data.

\begin{assumption}[Score Regularity]\label{ass: OSR}
The following conditions hold for each $n \geq n_0$ and all $P \in  \mathcal{P}_n$. (i) The class of function induced by the score $\mathcal{F}_0 = \{ \bar \psi_{uj}(W): u \in \mathcal{U}, j  \in [\pp] \}$
is suitably measurable and its uniform covering entropy obeys:
 $$
\sup_Q  \log N(\epsilon \|F_0\|_{Q,2}, \mathcal{F}_0, \| \cdot \|_{Q,2}) \leq \varrho_{n} ( \log (A_n/\epsilon)) \vee 0,
$$
here $F_0(W)$ is an envelope for $\mathcal{F}_0$ which is measurable with respect to  $W$ and satisfies $F_0(W)\geq  \sup_{u\in \UU, j\in[\pp]}|\bar \psi_{uj}(W)|$ and $\|F_0\|_{P,q}\leq L_{n}$ for $q\geq 4$. Furthermore,  $c \leq \sup_{u\in \UU, j\in [\pp]} \Ep[ |\bar \psi_{uj}(W)|^k] \leq CL_n^{k-2}$ for $k=2,3,4$. (ii) The set $\widehat{\mathcal{F}}_0 = \{  \bar \psi_{uj}(W)-\hat\psi_{uj}(W): u \in \mathcal{U}, j  \in [\pp] \}$ satisfies the conditions
$
\log N(\epsilon, \widehat{\mathcal{F}}_0, \|\cdot\|_{\Pn,2}) \leq \bar \varrho_{n} ( \log (\bar A_n/\epsilon)) \vee 0,
$ and $\sup_{u\in\UU,j\in[\pp]}\En[ \{ \bar \psi_{uj}(W) - \hat \psi_{uj}(W)\}^2] \leq \bar\delta_n\{\rho_n\bar\rho_n\log(A_n\vee n)\log(\bar A_n\vee n)\}^{-1}$ with probability $1-\Delta_n$.
\end{assumption}

Assumption \ref{ass: OSR} imposes condition on the class of functions induced by $\bar\psi_{uj}$ and on its estimators $\hat\psi_{uj}$. Typically the bound $L_n$ on the moment of the envelope is smaller than $K_n$, and in many settings $\bar\rho_n=\rho_n \lesssim \dn$ the dimension of $\UU$.

Next let $\mathcal{N}$ denote a mean zero Gaussian process indexed by $\UU \times [\pp]$ with covariance operator given by $\Ep[ \bar\psi_{uj}(W)\bar\psi_{u'j'}(W)]$ for $j,j' \in [\pp]$ and $u,u'\in\UU$. Because of the high-dimensionality, indeed $\pp$ can be larger than the sample size $n$, the central limit theorem will be uniformly valid over ``rectangles". This class of sets are rich enough to construct many confidence regions of interest in applications accounting for multiple testing. Let $\mathcal{R}$ denote the set of rectangles $R = \{ z \in \RR^{\pp} : \max_{j\in A} z_j \leq t,  \max_{j\in B} (-z_j) \leq t\}$ for all $A, B \subset [\pp]$ and $t\in \RR$. The following result is a consequence of Theorem \ref{theorem:semiparametricMain} above and Corollary 2.2 of \cite{chernozhukov2012comparison}.

\begin{corollary} Under  Assumptions \ref{ass: S1}, \ref{ass: AS} and Assumption \ref{ass: OSR}(i),  with $\delta_n = o(\{ \rho_n\log(A_n\vee n)\}^{-1/2})$,  and $\rho_n\log(A_n\vee n) =o(\{ (n/L_n^2)^{1/7}\wedge (n^{1-2/q}/L_n^2)^{1/3}\})$, we have that
$$ \sup_{P \in  \mathcal{P}_n} \sup_{R \in \mathcal{R}} \left| \Pr_P\left(  \{\sup_{u\in\UU}n^{1/2}\sigma_{uj}^{-1}(\check\theta_{uj}-\theta_{uj})\}_{j=1}^{\pp} \in R \right) - \Pr_P(\mathcal{N} \in R) \right| = o(1).  $$
\end{corollary}

In order to derive a method to build confidence regions we approximate the process $\mathcal{N}$ by the Gaussian multiplier bootstrap based on estimates $\hat\psi_{uj}$ of $\bar\psi_{uj}$, namely
$$ \widehat{\mathcal{G}} = (\widehat{\mathcal{G}}_{uj})_{u\in\UU, j\in [\pp]} = \left\{ \frac{1}{\sqrt{n}}\sum_{i=1}^ng_i\hat\psi_{uj}(W_i)\right\}_{u\in\UU, j\in [\pp]}$$
here $(g_i)_{i=1}^n$ are independent standard normal random variables which are independent from the data $(W_i)_{i=1}^n$. Based on Theorem 5.2 of \cite{chernozhukov2013gaussian}, the following result shows that the multiplier bootstrap provides a valid approximation to the large sample probability law of $\sqrt{n}(\check \theta_{uj}- \theta_{uj})_{u \in \mathcal{U},j\in[\pp]}$ over rectangles.

\begin{corollary}[\textbf{Uniform Validity of Gaussian Multiplier Bootstrap}]\label{theorem: general bsMain}
Under Assumptions \ref{ass: S1}, \ref{ass: AS} and Assumption \ref{ass: OSR}, with $\delta_n = o(\{ (1+d_{W})\rho_n\log(A_n\vee n)\}^{-1/2})$ and $\rho_n\log(A_n\vee n) = o(\{ (n/L_n^2)^{1/7}\wedge (n^{1-2/q}/L_n^2)^{1/3}\})$, we have that
$$ \sup_{P \in  \mathcal{P}_n} \sup_{R \in \mathcal{R}} \left| \Pr_P\left( \{\sup_{u\in\UU}n^{1/2}\sigma_{uj}^{-1}(\check\theta_{uj}-\theta_{uj})\}_{j=1}^{\pp} \in R \right) - \Pr_P(\widehat{\mathcal{G}} \in R\mid (W_i)_{i=1}^n) \right| = o(1)  $$
\end{corollary}

\section{Continuum of $\ell_1$-Penalized M-Estimators}\label{FunctionalLassoSection}

For the reader's convenience, this section collects results on the estimation of a continuum of  estimation of high-dimensional models via $\ell_1$-penalized estimators. We refer to \cite{BCCW-ManyProcesses} for the proofs.

Consider a data generating process with a response variable $(Y_{u})_{u\in \UU}$ and observable covariates $(X_u)_{u\in\UU}$ satisfies for each $u\in \UU$,
\begin{equation}\label{A:EqMainFunc}\theta_u \in \arg\min_{\theta \in \RR^p}\Ep[M_u(Y_{u},X_u,\theta,a_u)], \end{equation}
here $\theta_u$ is a $p$-dimensional vector, $a_u$ is a nuisance function that capture the misspecification of the model, $M_u$ is a pre-specified function, and the $p_u$-dimensional ($p_u\leq p$) covariate $X_u$ could have been constructed based on transformations of other variables. This implies that
$$\partial_\theta\Ep[M_u(Y_{u},X_u,\theta_u,a_u)]=0 \ \ \mbox{for all} \ u\in\UU.$$ The solution $\theta_u$ is assumed to be sparse in the sense that for some process $(\theta_u)_{u\in\UU}$ satisfies
$$\|\theta_u\|_0\leq s  \ \mbox{for all} \ u\in\UU.$$
Because of the nuisance function, such sparsity assumption is very mild and formulation (\ref{A:EqMainFunc}) encompasses several cases of interest including approximate sparse models. We focus on the estimation of $(\theta_u)_{u\in \UU}$ and we assume that an estimate $\hat a_u$ of the nuisance function $a_u$ is available and the criterion $M_u(Y_{u},X_u,\theta_u):=M_u(Y_{u},X_u,\theta_u,\hat a_u)$ is used as a proxy for $M_u(Y_{u},X,\theta_u,a_u)$.

 In the case of linear regression we have $M_u(y,x,\theta) = \frac{1}{2}(y-x'\theta)^2$. In the logistic regression case, we have
$M_u(y,x,\theta) = -\{1(y=1)\log \G( x'\theta ) + 1(y=0)\log(1-\G( x'\theta))\}$
with $\G$ is the logistic link function $\G(t)=\exp(t)/\{1+\exp(t)\}$. Additional examples include quantile regression models for $u\in (0,1)$.

\begin{example}[Quantile Regression Model]
Consider a data generating process $Y = F^{-1}_{Y\mid X}(U) = X'\theta_U + r_U(X)$, with $U\sim {\rm Unif}(0,1)$, and  $X$ is a $p$-dimensional vector of covariates. The criterion $M_u(y,x,\theta) = ( u - 1\{ y \leq x'\theta\} )(y - x'\theta)$ with the (trivial) estimate $\hat a_u = 0$ for the nuisance parameter $a_u = r_u$.
\end{example}

\begin{example}[Lasso with Estimated Weights]
We consider a linear model defined as
$ f_uY = f_uX'\theta_u + \bar r_u + \zeta_u, \ \ \Ep[f_uX\zeta_u] = 0$, here $X$ are $\bar p$-dimensional covariates, $\theta_u$ is a $s$-sparse vector, and $\bar r_u$ is an approximation error satisfies $\sup_{u\in\UU}\En[\bar r_u^2] \lesssim_P s\log \bar p/n$. In this setting, $(Y,X)$ are observed and only an estimator $\hat f_u$ of $f_u$ is available. This corresponds to nuisance parameter  $a_u = (f_u,\bar r_u)$ and $\hat a_u = (\hat f_u,0)$ so that $\En[ M_u(Y,X,\theta,a_u)] = \En[ f_u^2(Y-X'\theta-\bar r_u)^2]$ and $\En[M_u(Y,X,\theta)]=\En[\hat f_u^2(Y-X'\theta)^2]$.
\end{example}

We assume that $n$ i.i.d. observations from dgps with (\ref{A:EqMainFunc}) holds, $\{( Y_{iu}, X_{iu})_{u\in\mathcal{U}}\}_{i=1}^n$, are observed to estimate  $(\theta_u)_{u\in\mathcal{U}}$. For each $u\in\mathcal{U}$, a penalty level $\lambda$, and a diagonal matrix of penalty loadings $\hat\Psi_u,$ we define the $\ell_1$-penalized $M_u$-estimator (Weighed-Lasso) as
 \begin{equation}\label{Adef:LassoFunc} \hat\theta_u \in \arg\min_{\theta} \En[M_u(Y_{u},X_u,\theta)] + \frac{\lambda}{n}\|\hat\Psi_u\theta\|_1. \end{equation}
Furthermore, for each $u\in\mathcal{U}$, the post-penalized estimator (Post-Lasso) based on a set of covariates $\widetilde T_u$ is then defined as
 \begin{equation}\label{Adef:PostFunc}\widetilde\theta_u \in \arg\min_{\theta} \En[M_u(Y_{u},X_u,\theta)] \ \ : \ \ \supp(\theta)\subseteq \widetilde T_u.\end{equation}
Potentially, the set $\widetilde T_u$ contains $\supp(\hat\theta_u)$ and possibly additional variables deemed as important (although in that case the total number of additional variables should also obey the same growth conditions that $s$ obeys). We will set $\widetilde T_u = \supp(\hat\theta_u)$ unless otherwise noted.

In order to handle the functional response data, the penalty level $\lambda$ and penalty loading $\hat\Psi_u= \diag(\{\hat l_{u k}, k=1,\ldots,p\})$ need to be set to control selection errors
uniformly over $u\in\mathcal{U}$. The choice of loading matrix is problem specific and we suggest to mimic the following ``ideal" choice $\widehat\Psi_{u0} = \diag(\{ l_{u k}, k=1,\ldots,p\})$ with
\begin{equation}\label{IdealLoading} l_{uk} = \{\En\left[\{\partial_{\theta_k} M_u(Y_{u},X_u,\theta_u,a_u)\}^2\right]\}^{1/2} \end{equation} which is motivated by the use of self-normalized moderate deviation theory. In that case, it is suitable to set
$\lambda$ so that with high probability
\begin{equation}\label{Eq:reg} \frac{\lambda}{n} \geq c \sup_{u\in \mathcal{U}}\left\|\hat \Psi^{-1}_{u0} \En\left[\partial_\theta M_u(Y_{u},X_u,\theta_u, a_u) \right] \right\|_\infty,
\end{equation} here $c>1$ is a fixed constant. Indeed, in the case that $\mathcal{U}$ is a singleton  the choice above is similar to \cite{BickelRitovTsybakov2009}, \cite{BC-PostLASSO}, and \cite{BCW-SqLASSO}. This approach was first employed for a continuum of indices $\UU$ in the context of $\ell_1$-penalized quantile regression processes by \cite{BC-SparseQR}.

To implement (\ref{Eq:reg}), we propose setting the penalty level as
\begin{equation}\label{Eq:Def-lambda}\lambda =   c \sqrt{n} \Phi^{-1}(1-\xi/\{2p N_n \}),
 \end{equation}
here $N_n$ is a measure of the class of functions indexed by $\UU$, $1-\xi$ (with $\xi = o(1)$) is a confidence level associated with the probability of event (\ref{Eq:reg}), and $c>1$ is a slack constant. In many settings we can take $N_n = n^{\dn}$. If the set $\mathcal{U}$ is a singleton, $N_n = 1$ suffices which corresponds to what is used in \cite{BCH2014inference}.

\subsection{Generic Finite Sample Bounds}

In this subsection we derive finite sample bounds based on Assumption \ref{ass: M} below. This assumption provides sufficient conditions that are implied by a variety of settings including generalized linear models.

\begin{assumption}[M-Estimation Conditions]\label{ass: M}
Let $\{(Y_{iu}, X_{iu}, u\in \UU),i=1,\ldots, n\}$ be $n$ i.i.d. observations of the model (\ref{A:EqMainFunc}) and let $T_u = \supp(\theta_u)$, here $\Vert T_u \Vert _0\leq s$, $u\in\UU$.  With probability $1-\Delta_n$ we have that for all $u\in \UU$ there are weights $w_u=w_u(Y_u,X_u)$ and $C_{un}$ such that:
\begin{itemize}
\item[(a)] $|\En[\partial_{\theta} M_u(Y_u,X_u,\theta_u)-\partial_{\theta} M_u(Y_u,X_u,\theta_u,a_u)]'\delta|\leq  C_{un}\|\sqrt{w_{u}} X_u'\delta\|_{\Pn,2} $;
\item[(b)] $\ell \widehat\Psi_{u0} \leq \widehat\Psi_u \leq L\widehat\Psi_{u0}$ for $\ell  > 1/c$, and let $\tilde c = \frac{Lc+1}{\ell c-1}\sup_{u\in \UU}\|\widehat \Psi_{u0}\|_\infty\|\widehat \Psi_{u0}^{-1}\|_\infty$;
\item[(c)] for all $\delta \in A_u$ there is $\bar q_{A_u}>0$ such that
$$\begin{array}{c}\En[M_u(Y_u,X_u,\theta_u + \delta)] - \En[M_u(Y_u,X_u,\theta_u)] -\En[\partial_{\theta} M_u(Y_u,X_u,\theta_u)]'\delta +2C_{un}\|\sqrt{w_{u}} X_u'\delta\|_{\Pn,2} \\ \geq \left\{\|\sqrt{w_{u}} X_u'\delta\|_{\Pn,2}^2\right\} \wedge \left\{ \bar q_{A_u}\|\sqrt{w_{u}} X_u'\delta\|_{\Pn,2}\right\}. \end{array}$$
\end{itemize}
\end{assumption}
In many applications we take the weights to be $w_u = w_u(X_u) = 1$ but we allow for more general weights. Assumption \ref{ass: M}(a) bounds the impact of estimating the nuisance functions uniformly over $u\in\UU$. In the setting with $s$-sparse estimands, we typically have $C_{un} \lesssim \{n^{-1}s \log (pn)\}^{1/2}$. The loadings $\hat\Psi_u$ are assumed larger (but not too much larger) than the ideal choice $\hat\Psi_{u0}$ defined in (\ref{IdealLoading}). This is formalized in Assumption \ref{ass: M}(b). Assumption \ref{ass: M}(c) is an identification condition that will be imposed for specific choices of $A_u$ and $q_{A_u}$. It relates to conditions in the literature derived for the case of a singleton $\UU$ and no nuisance functions, see  the restricted strong convexity\footnote{Assumption \ref{ass: M} (a) and (c) could have been stated with $\{C_{un}/\sqrt{s}\}\|\delta\|_1$ instead of $C_{un}\|\sqrt{w_{u}} X_u'\delta\|_{\Pn,2}$.} used in \cite{negahban2012unified} and the non-linear impact coefficients used  in \cite{BC-SparseQR} and \cite{BCK2013robustQR}.

The following results establish rates of convergence for the $\ell_1$-penalized solution with estimated nuisance functions (\ref{Adef:LassoFunc}), sparsity bounds and rates of convergence for the post-selection refitted estimator (\ref{Adef:PostFunc}). They are based on restricted eigenvalue type conditions and sparse eigenvalue conditions. With the restricted eigenvalue is defined as $\bar\kappa_{u,2\tilde\cc} = \inf_{\delta\in \Delta_{u, 2\tilde\cc}} \|\sqrt{w_{u}}  X_u'\delta\|_{\Pn,2}/\|\delta_{T_u}\|$  In the results for sparsity and post-selection refitted models, the minimum and maximum sparse eigenvalues,
$$ \semin{m,u} = \min_{1\leq \|\delta\|_0\leq m}\frac{\|\sqrt{w_u}X_u'\delta\|_{\Pn,2}^2}{\|\delta\|^2} \ \ \mbox{and} \ \ \semax{m,u} = \max_{1\leq \|\delta\|_0\leq m}\frac{\|\sqrt{w_{u}}X_u'\delta\|_{\Pn,2}^2}{\|\delta\|^2},$$
are also relevant quantities to characterize the behavior of the estimators.

\begin{lemma}\label{Lemma:LassoMRateRaw}
Suppose that Assumption \ref{ass: M} holds with $\delta \in A_u = \{ \delta : \|\delta_{T^c_u}\|_1\leq 2\tilde\cc \|\delta_{T_u}\|_1\} \cup \{ \delta : \|\delta\|_1 \leq \frac{6c\|\widehat\Psi_{u0}^{-1}\|_\infty}{\ell c - 1}\frac{n}{\lambda}C_{un}\|\sqrt{w_{u}} X_u'\delta\|_{\Pn,2}\}$ and $\bar q_{A_u}>3\left\{(L+\frac{1}{c})\|\widehat\Psi_{u0}\|_\infty\frac{\lambda\sqrt{s}}{n\bar\kappa_{u,2\tilde \cc}}+ 9\tilde \cc C_{un}\right\}.$ Suppose that $\lambda$ satisfies condition (\ref{Eq:reg}) with probability $1-\Delta_n$. Then, with probability $1-2\Delta_n$  we have uniformly over $u\in \UU$
$$\begin{array}{rl}
\|\sqrt{w_{u}}  X_u'(\hat \theta_u - \theta_u)\|_{\Pn,2} & \leq 3\left\{(L+\frac{1}{c})\|\widehat\Psi_{u0}\|_\infty\frac{\lambda\sqrt{s}}{n\bar\kappa_{u,2\tilde \cc}}+ 9\tilde \cc C_{un}\right\}\\
\|\hat \theta_u - \theta_u\|_1 & \leq  3\left\{ \frac{(1+2\tilde\cc)\sqrt{s}}{\bar\kappa_{u,2\tilde \cc}}+ \frac{6c\|\widehat\Psi_{u0}^{-1}\|_\infty}{\ell c - 1}\frac{n}{\lambda}C_{un}\right\} \left\{(L+\frac{1}{c})\|\widehat\Psi_{u0}\|_\infty\frac{\lambda\sqrt{s}}{n\bar\kappa_{u,2\tilde \cc}}+ 9\tilde \cc C_{un}\right\}
\end{array}
$$
\end{lemma}

\begin{lemma}[M-Estimation Sparsity]\label{Lemma:LassoMSparsity}
In addition to conditions of Lemma \ref{Lemma:LassoMRateRaw}, assume that with probability $1-\Delta_n$ for all $u\in\UU$ and $\delta \in \RR^p$ we have
 $$ |\{\En[\partial_\theta M_u(Y_u,X_u,\hat\theta_u)-\partial_\theta M_u(Y_u,X_u,\theta_u)]\}'\delta| \leq L_{un}\|\sqrt{w_{u}}X_u'\delta\|_{\Pn,2}.$$
Let $\mathcal{M}_u=\{ m \in \mathbb{N} :  m \geq 2 \semax{m,u} L_u^2 \}$ with $L_u =\frac{c\|\widehat\Psi_{u0}^{-1}\|_\infty }{c\ell-1} \frac{n}{\lambda}\left\{ C_{un} + L_{un} \right\}$, then with probability $1-3\Delta_n$ we have that
$$ \hat s_u \leq \min_{m\in \mathcal{M}_u} \semax{m,u} L_u^2 \ \ \mbox{ for all } \ u\in \UU.$$
\end{lemma}

\begin{lemma}\label{Lemma:PostLassoMRateRaw}
Let $\widetilde T_u, u\in \UU,$ be the support used for post penalized estimator (\ref{Adef:PostFunc}) and $\tilde s_u = \Vert \widetilde T_u\Vert_0 $ its cardinality. In addition to conditions of Lemma \ref{Lemma:LassoMRateRaw}, suppose that Assumption \ref{ass: M}(c) holds also for $A_u=\{ \delta : \|\delta\|_0 \leq \tilde s_u + s\}$ with probability $1-\Delta_n$, $\bar q_{A_u}>2\left\{ \frac{\sqrt{\tilde s_u+s_u}\|\En[S_u]\|_\infty}{\sqrt{\semin{\tilde s_u+s_u,u}}} + 3C_{un}\right\}$ and $\bar q_{A_u}>2\{\En[M_u(Y_u,X_u,\tilde \theta_u)] - \En[M_u(Y_u,X_u,\theta_u)]\}_+^{1/2}$. Then, we have uniformly over $u\in\UU$
$$ \|\sqrt{w_{u}} X_u'(\tilde \theta_u-\theta_u)\|_{\Pn,2} \leq \{\En[M_u(Y_u,X_u,\tilde \theta_u)] - \En[M_u(Y_u,X_u,\theta_u)]\}_+^{1/2} +  \frac{\sqrt{\tilde s_u+s_u}\|\En[S_u]\|_\infty}{\sqrt{\semin{\tilde s_u+s_u,u}}} + 3C_{un}.$$
\end{lemma}

In Lemma \ref{Lemma:PostLassoMRateRaw}, if  $\widetilde T_u=\supp(\hat \theta_u)$, we have that
$$\En[M_u(Y_u,X_u,\tilde \theta_u)] - \En[M_u(Y_u,X_u,\theta_u)] \leq \En[M_u(Y_u,X_u,\hat \theta_u)] - \En[M_u(Y_u,X_u,\theta_u)] \leq \lambda C' \|\hat\theta_u - \theta_u\|_1$$ and $\sup_{u\in\UU}\|\En[S_u]\|_\infty \leq C' \lambda$ with high probability, $C' \leq L \sup_{u\in\UU}\|\widehat \Psi_{u0}\|_\infty$.

These results generalize important results of the $\ell_1$-penalized estimators to the case of functional response data and estimated of nuisance functions.
A key assumption in Lemmas \ref{Lemma:LassoMRateRaw}-\ref{Lemma:PostLassoMRateRaw} is that the choice of $\lambda$ satisfies (\ref{Eq:reg}). We next provide a set of simple generic conditions that will imply the validity of the proposed choice. These generic conditions can be verified in many applications of interest.

{\bf Condition WL.} {\it For each $u\in \UU$, let  $S_{u} = \partial_\theta M_u(Y_u,X_u,\theta_u,a_u)$, suppose that:\\
 (i)  ${\displaystyle\sup_{u\in\UU} \max_{k\leq p} } \{\Ep[|S_{uk}|^3]\}^{1/3}/\{\Ep[|S_{uk}|^2]\}^{1/2}\Phi^{-1}(1-\xi/\{2p N_n\}) \leq \delta_n n^{1/6}$, for all $u\in \UU$, $k\in[p]$;\\
(ii)  $N_n \geq N(\epsilon,\mathcal{U},d_\mathcal{U})$, here $\epsilon$ is such that with probability $1-\Delta_n$:\\   ${\displaystyle \sup_{d_\mathcal{U}(u,u')\leq \epsilon}  \max_{k\leq p}} \frac{\| \En[ S_{u}-S_{u'} ]\|_\infty}{\Ep[|S_{uk}|^2]^{1/2}}\leq \delta_n n^{-\frac{1}{2}}$, and
${\displaystyle \sup_{d_\mathcal{U}(u,u')\leq \epsilon}  \max_{k\leq p}} \ \frac{|\Ep[S_{uk}^2-S_{u'k}^2]| +|(\En-\Ep)[S_{uk}^2]|}{\Ep[|S_{uk}|^2]}\leq \delta_n$.
}

The following technical lemma justifies the choice of penalty level $\lambda$. It is based on self-normalized moderate deviation theory.

\begin{lemma}[Choice of $\lambda$]\label{Thm:ChoiceLambda}
Suppose Condition WL holds, let $c'>c>1$ be constants, $\xi \in [1/n,1/\log n]$, and $\lambda = c'\sqrt{n}\Phi^{-1}(1-\xi/\{2pN_n\})$. Then for $n \geq n_0$ large enough depends only on Condition WL,
$$\Pr\left ( \lambda/n \geq c \sup_{u\in\mathcal{U}}\|\hat  \Psi^{-1}_{u0}\En[ \partial_\theta M_u(Y_u,X_u,\theta_u,a_u) ]\|_\infty \right ) \geq 1-\xi - o(\xi)-\Delta_n.$$
\end{lemma}

We note that Condition WL(ii) contains high level conditions. See \cite{BCFH2013program} for examples that satisfy these conditions. The following corollary summarizes these results for many applications of interest in well behaved designs.

\begin{corollary}[Rates under Simple Conditions]
Suppose that with probability $1-o(1)$ we have that $C_{un} \vee L_{un} \leq C\{n^{-1}s\log(pn)\}^{1/2}$, $(Lc+1)/(\ell c-1)\leq C$, $w_u = 1$, and Condition WL holds with $\log N_n \leq C\log(pn)$. Further suppose that with probability $1-o(1)$ the sparse minimal and maximal eigenvalues are well behaved, $c\leq \semin{s \ell_n,u } \leq \semax{s \ell_n,u } \leq C$ for some $\ell_n\to \infty$ uniformly over $u\in\UU$. Then with probability $1-o(1)$ we have
$$ \sup_{u\in\UU}\|X_u'(\hat \theta_u - \theta_u)\|_{\Pn,2} \lesssim \sqrt{\frac{s\log(p n)}{n}}, \ \ \sup_{u\in\UU}\|\hat \theta_u - \theta_u\|_1 \lesssim \sqrt{\frac{s^2\log(p n)}{n}}, \ \ \mbox{and} \ \  \sup_{u\in\UU}\|\hat \theta_u \|_0 \lesssim s.$$
Moreover, if  $\widetilde T_u=\supp(\hat \theta_u)$, we have that
$$ \sup_{u\in\UU}\|X_u'(\tilde \theta_u - \theta_u)\|_{\Pn,2} \lesssim \sqrt{\frac{s\log(p n)}{n}}$$
\end{corollary}

\section{Bounds on Covering entropy}

Let $ ( W_i)_{i=1}^n$ be a sequence of independent copies of a random element $ W$  taking values in a measurable space $({\mathcal{W}}, \mathcal{A}_{{\mathcal{W}}})$ according to a probability law $P$. Let $\mathcal{F}$ be a set  of suitably measurable functions $f\colon {\mathcal{W}} \to \mathbb{R}$, equipped with a measurable envelope $F\colon \mathcal{W} \to \mathbb{R}$. The proofs for the following lemmas can be found in \cite{BCFH2013program}.

\begin{lemma}[Algebra for Covering Entropies] \text{  } \label{lemma: andrews}Work with the setup above. \\
(1) Let $\F$ be a VC subgraph class with a finite VC index $k$ or any
other class whose entropy is bounded above by that of such a VC subgraph class, then
the uniform entropy numbers of $\mF$ obey
\begin{equation*}
 \sup_{Q} \log  N(\epsilon \|F\|_{Q,2}, \F,  \| \cdot \|_{Q,2}) \lesssim \{1+ k \log (1/\epsilon)\}\vee 0
\newline
\end{equation*}
(2) For any measurable classes of functions $\F$ and $\F^{\prime
} $ mapping $\mathcal{W}$ to $\Bbb{R}$,
\begin{align*}
&\log N(\epsilon \Vert F+F^{\prime }\Vert _{Q,2},\F+\F^{\prime
}, \| \cdot \|_{Q,2})\leq \log   N\left(\mbox{$ \frac{\epsilon }{2}$}\Vert F\Vert _{Q,2},\F, \| \cdot \|_{Q,2}\right)
+ \log N\left( \mbox{$ \frac{\epsilon }{2}$}\Vert F^{\prime }\Vert _{Q,2},\F^{\prime
}, \| \cdot \|_{Q,2}\right), \\
&\log  N(\epsilon \Vert F\cdot F^{\prime }\Vert _{Q,2},\F\cdot \F^{\prime
}, \| \cdot \|_{Q,2})\leq \log   N\left( \mbox{$ \frac{\epsilon }{2}$}\Vert F\Vert _{Q,2},\F, \| \cdot \|_{Q,2}\right)
+ \log N\left( \mbox{$ \frac{\epsilon }{2}$}\Vert F^{\prime }\Vert _{Q,2},\F^{\prime
}, \| \cdot \|_{Q,2}\right), \\
& N(\epsilon \Vert F\vee F^{\prime }\Vert _{Q,2},\F\cup \F^{\prime
}, \| \cdot \|_{Q,2})\leq   N\left(\epsilon\Vert F\Vert _{Q,2},\F, \| \cdot \|_{Q,2}\right)
+ N\left( \epsilon\Vert F^{\prime }\Vert _{Q,2},\F^{\prime
}, \| \cdot \|_{Q,2}\right).
\end{align*}
(3)  For any measurable class of functions $%
\mathcal{F}$ and a fixed function $f$ mapping $\mathcal{W}$ to $\Bbb{R}$,
\begin{equation*}
 \log \sup_{Q} N(\epsilon \Vert |f|\cdot F\Vert _{Q,2},f\cdot\F, \| \cdot \|_{Q,2})\leq \log \sup_{Q} N\left(
\epsilon /2\Vert F\Vert _{Q,2},\F, \| \cdot \|_{Q,2}\right)
\end{equation*}
(4)  Given measurable classes $\F_j$ and envelopes $F_j$, $j=1,\ldots,k$, mapping $\mathcal{W}$ to $\Bbb{R}$, a function $\phi\colon\Bbb{R}^k\to\Bbb{R}$ such that for $f_j,g_j\in\F_j$,
$ |\phi(f_1,\ldots,f_k) - \phi(g_1,\ldots,g_k) | \leq \sum_{j=1}^k L_j(x)|f_j(x)-g_j(x)|$, $L_j(x)\geq 0$, and fixed functions $\bar f_j \in \F_j$,  the class of functions $\mathcal{L}=\{\phi(f_1,\ldots,f_k)-\phi(\bar f_1,\ldots,\bar f_k)\colon f_j \in\mathcal{F}_j, j=1,\ldots,k\}$ satisfies
\begin{equation*}
 \log \sup_Q N\left(\epsilon\Big\|\sum_{j=1}^kL_jF_j\Big\|_{Q,2},\mathcal{L}, \| \cdot \|_{Q,2}\right)\leq \sum_{j=1}^k\log  \sup_Q  N\left(
\mbox{$\frac{\epsilon}{k}$}\|F_j\|_{Q,2},\F_j, \| \cdot \|_{Q,2}\right).
\end{equation*}
\end{lemma}

\begin{proof}
See Lemma L.1 in \cite{BCFH2013program}.
\end{proof}

\begin{lemma}[Covering Entropy for Classes obtained as Conditional Expectations]\label{Lemma:PartialOutCovering}
 Let $\mathcal F$ denote a class of measurable functions $f\colon \mathcal{W}\times \mathcal{Y} \to \Bbb{R}$ with a measurable envelope $F$. For a given $f \in \mathcal{F}$, let $\bar f\colon \mathcal{W} \to \Bbb{R}$ be the function $\bar f (w) := \int f(w,y) d\mu_{w}(y)$ here $\mu_{w}$ is a regular conditional probability distribution over $y \in \mathcal{Y}$ conditional on $w\in\mathcal{W}$. Set $\bar{\mathcal{F}} = \{ \bar f \colon f \in \mathcal{F}\}$ and let $\bar F(w):=\int F(w,y) d\mu_w(y)$ be an envelope for $\bar{\mathcal{F}}$. Then, for $r, s \geq 1$,
     $$
    \log \sup_{Q} N(\epsilon \| \bar F\Vert _{Q,r}, \bar{\mathcal{F}}, \| \cdot \|_{Q,r}) \leq \log \sup_{\widetilde Q} N((\epsilon/4)^r \| F\Vert _{\widetilde Q,s},  \mathcal \F , \| \cdot \|_{\widetilde Q,s}),
    $$ here $Q$ belongs to the set of finitely-discrete probability measures over $\mathcal{W}$ such that  $0<\| \bar F\Vert _{Q,r}< \infty$, and $\widetilde Q$ belongs to the set of finitely-discrete probability measures over $\mathcal{W}\times \mathcal{Y}$ such that $0<\|  F\Vert _{\widetilde Q,s}< \infty$. In particular, for every $\epsilon > 0$ and any $k\geq 1$,
        $$
    \log \sup_{Q} N(\epsilon, \bar{\mathcal{F}}, \| \cdot \|_{Q,k}) \leq \log \sup_{\widetilde Q} N(\epsilon/2,  \mathcal \F , \| \cdot \|_{\widetilde Q,k} ).
    $$
\end{lemma}
\begin{proof}
See Lemma L.2 in \cite{BCFH2013program}.
\end{proof}

\end{appendix}
\end{document}